%% file: main.tex
\documentclass[11pt]{amsart}

\usepackage{lmodern}
\usepackage{stmaryrd}
\usepackage[utf8]{inputenc}
\usepackage{bm}
\newcommand{\VAN}[3]{#2}
\newcommand{\VANDEN}[3]{#2}

\usepackage{csquotes}
\usepackage[dvipsnames]{xcolor}
\usepackage{tikz-cd}
\usepackage[font=small,format=plain,labelfont=bf,up,textfont=it,up]{caption}
\usepackage{adjustbox}
\usepackage[margin=1.5in]{geometry}

\usepackage{fancyhdr}
\usepackage[shortlabels]{enumitem}
\usepackage{amsmath}
\usepackage{mathtools} 
\usepackage{amsfonts}
\usepackage{amssymb} 
\usepackage{amsthm}
\usepackage{mathrsfs}
\usepackage{calligra}

\usepackage[linktocpage=true, hyperindex,pagebackref]{hyperref}

\mathtoolsset{showonlyrefs}
\usepackage[OT2,T1]{fontenc}
\usepackage{scalerel, stackengine}
\stackMath
\newcommand\widecheck[1]{%
\savestack{\tmpbox}{\stretchto{%
  \scaleto{%
    \scalerel*[\widthof{\ensuremath{#1}}]{\kern-.6pt\bigwedge\kern-.6pt}%
    {\rule[-\textheight/2]{1ex}{\textheight}}%WIDTH-LIMITED BIG WEDGE
  }{\textheight}% 
}{0.5ex}}%
\stackon[1pt]{#1}{\scalebox{-1}{\tmpbox}}%
}

\begingroup
    \makeatletter
    \@for\theoremstyle:=definition,remark,plain\do{%
        \expandafter\g@addto@macro\csname th@\theoremstyle\endcsname{%
            \addtolength\thm@preskip\parskip
            }%
        }
\endgroup

\numberwithin{equation}{section}

\newtheorem{Thm}{Theorem}[section]
\newtheorem*{MainTheorem}{Main Theorem}
\newtheorem{Lem}[Thm]{Lemma}
\newtheorem{Prop}[Thm]{Proposition}
\newtheorem{Cor}[Thm]{Corollary}
\newtheorem{Conj}[Thm]{Conjecture}
\newtheorem*{Thm*}{Theorem}
\newtheorem{mainThm}{Theorem}

\newtheorem{Claim}[Thm]{Claim}

\theoremstyle{definition}
\newtheorem{Def}[Thm]{Definition}
\newtheorem{Example}[Thm]{Example}
\newtheorem{Rem}[Thm]{Remark}
\newtheorem{Assump}[Thm]{Assumption}

\newcommand{\spec}{\operatorname{Spec}}
\newcommand{\spf}{\operatorname{Spf}}
\newcommand{\spa}{\operatorname{Spa}}
\newcommand{\spd}{\operatorname{Spd}}
\newcommand{\gal}{\operatorname{Gal}}

\newcommand{\Hom}{\operatorname{Hom}}
\newcommand{\Aut}{\operatorname{Aut}}
\newcommand{\End}{\operatorname{End}}
\newcommand{\isom}{\underline{\operatorname{Isom}}}
\newcommand{\GL}{\mathrm{GL}}

\newcommand{\ql}{{\mathbb{Q}_\ell}}
\newcommand{\zl}{{\mathbb{Z}_\ell}}
\newcommand{\fl}{\mathbb{F}_{\ell}}
\newcommand{\qlbar}{\overline{\mathbb{Q}}_\ell}
\newcommand{\zlbar}{\overline{\mathbb{Z}}_\ell}
\newcommand{\flbar}{{\overline{\mathbb{F}}_{\ell}}}

\newcommand{\qp}{{\mathbb{Q}_p}}
\newcommand{\zp}{{\mathbb{Z}_{p}}}
\newcommand{\fp}{{\mathbb{F}_{p}}}
\newcommand{\qpbar}{\overline{\mathbb{Q}}_p}

\newcommand{\fpbar}{{\overline{\mathbb{F}}_{p}}}
\newcommand{\qpbr}{\breve{\mathbb{Q}}_{p}}
\newcommand{\zpbr}{\breve{\mathbb{Z}}_{p}}

\newcommand{\ebreve}{\breve{E}}
\newcommand{\afp}{{\mathbb{A}_f^p}}
\newcommand{\af}{{\mathbb{A}_f}}
\newcommand{\zlocp}{\mathbb{Z}_{(p)}}
\newcommand{\ebar}{\overline{\mathsf{E}}}

\newcommand{\CO}{\mathcal{O}}
\newcommand{\calg}{\mathcal{G}}

\newcommand{\calgcirc}{\mathcal{G}^{\circ}}

\newcommand{\mP}{\mathcal{P}}
\newcommand{\smP}{\sigma^{\ast} \mP}
\newcommand{\bdtimes}{\buildrel{\boldsymbol{.}}\over\times}

\newcommand{\perf}{\operatorname{Perf}}
\newcommand{\affperf}{\mathbf{Aff}^{\mathrm{perf}}}

\newcommand{\shtg}{\mathrm{Sht}_{\mathcal{G}}}
\newcommand{\shtgcirc}{\mathrm{Sht}_{\calgcirc}}
\newcommand{\shtgmu}{\mathrm{Sht}_{\mathcal{G},\mu}}

\newcommand{\shtgcircmu}{\mathrm{Sht}_{\calgcirc,\mu}}
\newcommand{\shtgmuone}{\mathrm{Sht}_{\mathcal{G},\mu, \delta=1}}
\newcommand{\shtgmub}{\mathrm{Sht}_{\mathcal{G},\mu}^{[b]}}
\newcommand{\shtgmuonerat}{\mathrm{Sht}_{\mathcal{G},\mu,\delta=1,E}}
\newcommand{\shtgvmu}{\mathrm{Sht}_{\mathcal{G}_V,\mu_V}}
\newcommand{\shthmuone}{\mathrm{Sht}_{\mathcal{H},\mu, \delta=1}}

\newcommand{\shtgloc}{\mathrm{Sht}^{\mathrm{W}}_{\mathcal{G}}}
\newcommand{\shtgcircloc}{\mathrm{Sht}^{\mathrm{W}}_{\calgcirc}}
\newcommand{\shtglocmu}{\mathrm{Sht}^{\mathrm{W}}_{\mathcal{G},\mu}}
\newcommand{\shtgcirclocmu}{\mathrm{Sht}^{\mathrm{W}}_{\calgcirc,\mu}}
\newcommand{\shtglocmumu}{\mathrm{Sht}^{\mathrm{W}}_{\mathcal{G},\mu | \mu}}
\newcommand{\shtglocmuone}{\mathrm{Sht}^{\mathrm{W}}_{\mathcal{G},{\mu},\delta=1}}

\newcommand{\shtgvlocmu}{\mathrm{Sht}^{\mathrm{W}}_{\mathcal{G}_V,\mu_V}}

\newcommand{\mintgmu}{\mathcal{M}^{\mathrm{int}}_{\mathcal{G},b,{\mu}}}

\newcommand{\bun}{\mathrm{Bun}}
\newcommand{\bung}{\bun_{G}}
\newcommand{\bungk}{\bun_{G,k}}
\newcommand{\bungmu}{\bun_{G,\mu^{-1}}}
\newcommand{\bungmuk}{\bun_{G,\mu^{-1},k}}
\newcommand{\admu}{\operatorname{Adm}(\mu^{-1})}
\newcommand{\bgmu}{B(G,\mu^{-1})}

\newcommand{\gisoc}{{G\text{-}\mathrm{Isoc}}}
\newcommand{\gvisoc}{{G_V\text{-}\mathrm{Isoc}}}
\newcommand{\gisocmu}{\gisoc_{\le {\mu^{-1}}}}

\newcommand{\gafp}{\mathsf{G}(\afp)}
\newcommand{\gaf}{\mathsf{G}(\af)}

\newcommand{\gx}{{(\mathsf{G}, \mathsf{X})}}

\newcommand{\gv}{\mathsf{G}_{V}}
\newcommand{\gvx}{(\mathsf{G}_{V},\mathsf{H}_{V})}

\newcommand{\gxp}{(\mathsf{G}', \mathsf{X}')}

\newcommand{\scrs}{\mathscr{S}}
\newcommand{\scrsg}{\mathscr{S}_K\gx}
\newcommand{\scrsginf}{\mathscr{S}_{K_p}\gx}

\newcommand{\hatscrsg}{\widehat{\mathscr{S}}_K\gx}
\newcommand{\hatscrsginf}{\widehat{\mathscr{S}}_{K_p}\gx}
\newcommand{\hatscrsgv}{\widehat{\mathscr{S}}_M\gvx}
\newcommand{\hatscrsgvinf}{\widehat{\mathscr{S}}_{M_p}\gvx}

\newcommand{\scrsd}{\mathscr{S}_K\gx^{\diamond}}
\newcommand{\scrsdinf}{\mathscr{S}_{K_p}\gx^{\diamond}}
\newcommand{\scrsdpreinf}{\mathscr{S}_{K_p}\gx^{\diamond,\mathrm{pre}}}

\newcommand{\scrsdvinf}{\mathscr{S}_{M_p}\gvx^\diamond}
\newcommand{\scrsdvpreinf}{\mathscr{S}_{M_p}\gvx^{\diamond,\mathrm{pre}}}

\newcommand{\Fl}{\mathscr{F}\!\ell}
\newcommand{\grgloc}{\mathrm{Gr}_{\mathcal{G}}^{\mathrm{W}}}

\newcommand{\igs}{\mathrm{Igs}_{K^p}\gx}
\newcommand{\igsk}{\mathrm{Igs}_{K^p,k}}
\newcommand{\igsinf}{\mathrm{Igs}\gx}

\newcommand{\ig}{\mathrm{Ig}}
\newcommand{\iginfbgx}{\mathrm{Ig}_{}^b\gx}
\newcommand{\igvinfbgx}{\mathrm{Ig}_{}^{b,\mathrm{v}}\gx}

\newcommand{\crys}{\mathrm{crys}}
\newcommand{\shg}{{\scrs_{K}\gx_{k_{E}}^{\mathrm{perf}}}}

\newcommand{\shginf}{{\scrs_{K_p}\gx_{k_{E}}^{\mathrm{perf}}}}
\newcommand{\shginfd}{{\scrs_{K_p}\gx_{k_{E}}^{\mathrm{perf},\diamond}}}
\newcommand{\shgvinf}{\scrs_{M_p}\gvx_{\fp}^{\mathrm{perf}}}

\newcommand{\g}{\mathsf{G}}
\newcommand{\gab}{\mathsf{G}^{\mathrm{ab}}}
\newcommand{\gder}{\mathsf{G}^{\mathrm{der}}}

\newcommand{\gtwoder}{\mathsf{G}_2^{\mathrm{der}}}
\newcommand{\gxtwo}{(\mathsf{G}_2, \mathsf{X}_2)}

\newcommand{\dualgrp}[1]{\widehat{#1}}
\newcommand{\loc}{X_{\dualgrp{G}}}
\newcommand{\Div}{\mathrm{Div}^{1}_{E}}
\newcommand{\Divk}{\mathrm{Div}^{1}_{E,k}}
\newcommand{\cocycle}{Z^1(W_\qp,\dualgrp{G})}

\newcommand{\smallmcG}{\scaleobj{.8}{{\mathcal{G}}}}
\newcommand{\IgsQuot}{q_\mathrm{Igs}}
\newcommand{\IgsQuotInt}{q_{\mathrm{Igs},\smallmcG}}

\makeatletter
\newcommand*{\da@rightarrow}{\mathchar"0\hexnumber@\symAMSa 4B }
\newcommand*{\da@leftarrow}{\mathchar"0\hexnumber@\symAMSa 4C }
\newcommand*{\xdashrightarrow}[2][]{%
  \mathrel{%
    \mathpalette{\da@xarrow{#1}{#2}{}\da@rightarrow{\,}{}}{}%
  }%
}
\newcommand{\xdashleftarrow}[2][]{%
  \mathrel{%
    \mathpalette{\da@xarrow{#1}{#2}\da@leftarrow{}{}{\,}}{}%
  }%
}
\newcommand*{\da@xarrow}[7]{%
  \sbox0{$\ifx#7\scriptstyle\scriptscriptstyle\else\scriptstyle\fi#5#1#6\m@th$}%
  \sbox2{$\ifx#7\scriptstyle\scriptscriptstyle\else\scriptstyle\fi#5#2#6\m@th$}%
  \sbox4{$#7\dabar@\m@th$}%
  \dimen@=\wd0 %
  \ifdim\wd2 >\dimen@
    \dimen@=\wd2 %   
  \fi
  \count@=2 %
  \def\da@bars{\dabar@\dabar@}%
  \@whiledim\count@\wd4<\dimen@\do{%
    \advance\count@\@ne
    \expandafter\def\expandafter\da@bars\expandafter{%
      \da@bars
      \dabar@ 
    }%
  }%  
  \mathrel{#3}%
  \mathrel{%   
    \mathop{\da@bars}\limits
    \ifx\\#1\\%
    \else
      _{\copy0}%
    \fi
    \ifx\\#2\\%
    \else
      ^{\copy2}%
    \fi
  }%   
  \mathrel{#4}%
}
\makeatother

\newcommand\restr[2]{{% we make the whole thing an ordinary symbol
  \left.\kern-\nulldelimiterspace % automatically resize the bar with \right
  #1 % the function
  \vphantom{\big\vert} % pretend it's a little taller at normal size
  \right\rvert_{#2} % this is the delimiter
  }}

\iftrue
\newcommand{\commentDaniel}[1]{\textcolor{Blue}{Daniel: #1}}
\newcommand{\commentPatrick}[1]{\textcolor{violet}{Patrick: #1}}
\newcommand{\commentPol}[1]{\textcolor{red}{Pol: #1}}
\newcommand{\commentMingjia}[1]{\textcolor{ForestGreen}{Mingjia: #1}}
\else
\newcommand{\commentDaniel}[1]{}
\newcommand{\commentPatrick}[1]{}
\newcommand{\commentPol}[1]{}
\newcommand{\commentMingjia}[1]{}
\fi

\def\lozenge{\diamondsuit}

\author{Patrick Daniels}
\address{Department of Mathematics and Statistics, Skidmore College, 815 N Broadway, Saratoga Springs, NY, 12866, USA}
\email{pdaniels@skidmore.edu}

\author{Pol van Hoften} 
\address{School of Mathematical Sciences, Zhejiang University, 866 Yuhangtang Rd, Hangzhou, 310058, P. R. China}
\email{pvhoften@zju.edu.cn}

\author{Dongryul Kim}
\address{Department of Mathematics, Stanford University, 450 Jane Stanford Way
(Building 380), Stanford, California, USA}
\email{dkim04@stanford.edu}

\author{Mingjia Zhang}
\address{Department of Mathematics, Princeton University, Fine Hall, Washington Road,
Princeton, NJ, 08544-1000, USA}
\email{mz9413@princeton.edu}

\title[Igusa stacks and the cohomology of Shimura varieties]{Igusa stacks and the cohomology of Shimura varieties}
\begin{document}
\sloppy %should keep things from extending into the margins
\input{AbstractIgusaFirst}

\maketitle\thispagestyle{empty}

\setcounter{tocdepth}{2}

\bigskip

\tableofcontents

\newpage

\section{Introduction}

\input{Intro_NewV2}

\subsection{Notation and conventions} \label{Sec:Conventions} Our notation and conventions will be introduced and recalled throughout the body of this text. We want to emphasize a number of conventions below, especially regarding Hodge cocharacters.  

We will denote Shimura data and their reflex fields by sans serif fonts, e.g., $\gx$ for a Shimura datum and $\mathsf{E}$ for its reflex field. For a prime $p$ and a place $v$ of $\mathsf{E}$ above $p$, we will write $E$ for the completion of $\mathsf{E}$ at $v$, which has ring of integers $\mathcal{O}_E$ and residue field $k_E$. We will then write $G$ for the base change of $\mathsf{G}$ to $\qp$, and $\mathcal{G}$ for a parahoric group scheme over $\zp$ with generic fiber $G$.

Our (conjugacy class of) Hodge cocharacter(s), normalized as in \cite[Section 1.3.1]{KisinPoints}, is denoted by $\mu$ and is defined over $\mathsf{E}$. We consider this as a $G(\qpbar)$-conjugacy class using the place $v$ of $\mathsf{E}$. We use $\mathrm{Gr}_G$ to denote the $\mathbf{B}_\mathrm{dR}^+$-affine Grassmannian of Scholze--Weinstein, with its stratification by Schubert cells as defined in \cite[Definition 19.2.2]{ScholzeWeinsteinBerkeley}. The Shimura variety over $E$ with infinite level at $p$ has a Hodge--Tate period map with target in $\mathrm{Gr}_{G,{\mu^{-1}}}$. Its Newton stratification is moreover indexed by the $\mu^{-1}$-admissible set $\bgmu \subset B(G)$. The diamond $\scrsd$ has a crystalline period map with target $\shtgmu$, the stack of $\mathcal{G}$-shtukas bounded by $\mu$. These conventions all agree with those of \cite{PappasRapoportShtukas}. 

The perfect special fiber of $\scrs_K\gx$ has a crystalline period map with target $\shtglocmu$, the stack of Witt vector $\mathcal{G}$-shtukas bounded by $\mu$. The stack $\shtglocmu$ does not quite agree with either $\mathrm{Sht}_{\mu,K}^{\mathrm{loc}}$ of \cite[Definition 4.1.3]{ShenYuZhang} or $\mathrm{Sht}_{\mu}^{\mathrm{loc}}$ of \cite[Definition 5.2.1]{XiaoZhu}, see Remark \ref{Rem:ShtukasInversion}.

\subsection{Acknowledgements} We are grateful to Ana Caraiani, Jean-Francois Dat, Laurent Fargues, Ian Gleason, Thomas Haines, Linus Hamann, Thibaud van den Hove, Christian Johansson, Kai-Wen Lan, Si Ying Lee, James Newton, Emanuel Reinecke, Peter Scholze, Jack Sempliner, Matteo Tamiozzo, Richard Taylor, Alex Youcis, Bogdan Zavyalov, Zhiyou Wu, and Xinwen Zhu for helpful conversations and correspondences. We thank the referee for their detailed comments and suggestions. Various ideas for this project were conceived at the 2022 IHES Summer School on the Langlands Program and the HIM Trimester Program on the Arithmetic of the Langlands Program; we thank the organizers of those programs for their efforts as well as the IHES and HIM for their hospitality. Thanks also to the IAS for hosting the first, third, and fourth authors during a portion of the editing of this manuscript. The second-named author gives special thanks to Yau Mathematical Sciences Center and the Morningside Center for Mathematics for inviting him to give a number of talks about this project.

\section{Preliminaries} \label{Sec:Preliminaries}

\subsection{The Fargues--Fontaine curve and v-sheaves}
This section is intended primarily for the purpose of establishing notation. For a more detailed account, we refer the reader to \cite[Section 2.1]{Companion} and the references therein. Throughout this section, we fix a perfect field $k$ of characteristic $p$. For such a $k$, we denote by $\perf_{k}$ the category of affinoid perfectoid spaces over $k$. If $k=\fp$ we will write $\perf$ rather than $\perf_{\fp}$.

\subsubsection{v-sheaves associated with adic spaces}\label{Sub:AdicSpaces} If $X$ is a pre-adic space over $\spa\zp$ in the sense of \cite[Section 3.4]{ScholzeWeinsteinBerkeley}, we let $X^\lozenge$ denote its associated v-sheaf, i.e.,
\begin{align}
    X^\lozenge(S) = \{(S^\sharp, f)\} / \, \text{isom.}
\end{align}
for any $S$ in $\perf$, where $S^\sharp$ is an untilt of $S$ and $f:S^\sharp \to X$ is a morphism of pre-adic spaces, see \cite[Section 18.1]{ScholzeWeinsteinBerkeley}. For a Huber pair $(A,A^+)$ we write $\spd(A,A^+)$ in place of $\spa(A,A^+)^\lozenge$, and when $A^+ = A^\circ$ we write $\spd A$ instead of $\spd(A,A^+)$.

If $X$ is now a $\zp$-scheme, we follow the convention of \cite[Section 2.2]{AGLR} and denote by $X^\diamond$ and $X^\lozenge$ the two different v-sheaves naturally associated with $X$. When $X = \spec A$ is affine and $S = \spa(R,R^+)$ is affinoid perfectoid, these are given by
\begin{equation}
    X^\diamond(S) = \{ (\spa(R^\sharp, R^{\sharp+}), f: A \to R^{\sharp+}) \}/ \, \text{isom.},
\end{equation}
and 
\begin{equation}
    X^\lozenge(S) = \{ (\spa(R^\sharp, R^{\sharp+}), f: A \to R^\sharp)\} / \, \text{isom.},
\end{equation}
respectively, where $\spa(R^{\sharp}, R^{\sharp+})$ denotes an untilt of $S$, and in each case $f$ denotes a ring homomorphism.\footnote{Note that in \cite{PappasRapoportShtukas}, the notation $(-)^\blacklozenge$ is used in place of $(-)^\diamond$.} 

\subsubsection{Formal schemes} By a formal scheme, we mean a locally topologically ringed
space that is locally isomorphic to $\spf A$, where $A$ has a complete and
separated $I$-adic topology for $I$ a finitely generated ideal (but $A$ need not
be Noetherian). As noted in \cite[Section~2.2]{ScholzeWeinsteinModuli}, formal
schemes over $\spf \zp$ can be equivalently thought of as functors on
$\zp$-algebras in which $p$ is nilpotent, and moreover there is a fully faithful
functor from the category of formal schemes over $\spf \zp$ to the category of pre-adic spaces
over $\spa \zp$.

For a formal scheme $\mathfrak{X}$ over $\spf \zp$, we write $\mathfrak{X}^\mathrm{ad}$ for the pre-adic space associated to $\mathfrak{X}$ as in \cite[Proposition 2.2.1]{ScholzeWeinsteinModuli}, and we often abbreviate $(\mathfrak{X}^\mathrm{ad})^\lozenge$ to $\mathfrak{X}^\lozenge$.

\begin{Def}
  For $\mathfrak{X}$ a formal scheme over $\spf \zp$% locally admitting a finitely generated ideal of definition
, we
  define $\mathfrak{X}^{\lozenge,\mathrm{pre}}$ to be the presheaf on $\perf$ sending $(R,
  R^+)$ to the equivalence class of triples $(R^\sharp, R^{\sharp+}, f)$, where
 $(R^\sharp, R^{\sharp+})$ is an untilt of $(R, R^+)$ and
 $f \colon \spf R^{\sharp+} \to \mathfrak{X}$.
\end{Def}
\begin{Lem} \label{Lem:DiamondOfFormalScheme}
Let $\mathfrak{X}$ be a formal scheme over $\spf \zp$, which locally admits a finitely generated ideal of definition. Then $\mathfrak{X}^{\lozenge} / \spd \zp$ is the analytic sheafification of the presheaf on $\perf$ given by $\mathfrak{X}^{\lozenge,\mathrm{pre}}$.
\end{Lem}
\begin{proof}
    It suffices to check in the case $\mathfrak{X}=\spf A$, where $A$ is some adic $\zp$-algebra with finitely generated ideal of definition containing $p$. Then by definition an $S=\spa(R,R^+)$-point of $\mathfrak{X}^{\lozenge}$ amounts to an untilt $(S^\sharp, \iota)$ of $S$ and a map $f: S^\sharp \rightarrow \mathfrak{X}^\mathrm{ad}$, but $\mathfrak{X}^\mathrm{ad}=\spa(A,A)$, so the datum of $f$ amounts to a map of affinoid $(\zp,\zp)$-algebras $(A,A)\rightarrow (R^{\sharp},R^{\sharp+})$. But such maps are determined by continuous maps $A \to R^{\sharp+}$, which are the same thing as morphisms $\spf R^{\sharp+} \to \spf A$ (see \cite[Lemma 0AN0]{stacks-project}).
\end{proof}

\subsubsection{The Fargues--Fontaine curve}
For any affinoid perfectoid space $S=\spa(R,R^+)$ over $\fp$ with fixed pseudouniformizer $\varpi \in R^+$, we write $S \bdtimes \zp$ for the analytic adic space
\[S \bdtimes \zp = \spa(W(R^+)) \setminus \{[\varpi]=0\}.\]
By \cite[Proposition 11.3.1]{ScholzeWeinsteinBerkeley}, there is a natural isomorphism 
\[(S \bdtimes \zp)^\lozenge \xrightarrow{\sim} S^\lozenge \times \spd \zp.\] 
Moreover, $S \bdtimes \spa\zp$ comes equipped with a Frobenius $\varphi$, and any untilt $S^\sharp$ of $S$ determines a closed Cartier divisor $S^\sharp \hookrightarrow S \bdtimes \spa\zp$, see \cite[Proposition 11.3.1]{ScholzeWeinsteinBerkeley}.

We also define 
\[\mathcal{Y}(S) = \spa(W(R^+)) \setminus \{[\varpi]=0, p = 0\},\] 
and for any interval $I = [a,b] \subset [0,\infty)$ with rational endpoints we have an open subset $\mathcal{Y}_I(S)$ of $S\bdtimes \spa\zp$, see \cite[Section 2.1]{PappasRapoportShtukas} for a precise definition\footnote{Note that $\mathcal{Y}_I(S)$ depends on the choice of pseudouniformizer $\varpi$.}. In particular, we have the identities $\mathcal{Y}_{[0,\infty)}(S) = S \bdtimes \spa\zp$ and $\mathcal{Y}_{(0,\infty)}(S) = \mathcal{Y}(S)$.

For any $S$ in $\perf$, the relative adic Fargues--Fontaine curve over $S$ is the quotient
\begin{align}
    X_S = \mathcal{Y}(S) / \varphi^\mathbb{Z}.
\end{align}
This quotient is well-defined by \cite[Proposition II.1.16]{FarguesScholze}. 

Let $G$ be a reductive group over $\qp$. Following \cite{FarguesScholze}, we denote by $\bun_G(S)$ the groupoid of $G$-torsors on $X_S$. By \cite[Theorem~III.0.2]{FarguesScholze}, the presheaf of groupoids $\bun_G$ on $\perf$ sending $S$ to $\bun_G(S)$ is a small v-stack.

\subsubsection{The Kottwitz set and admissible loci} \label{Sec:BGMU} For a choice of algebraic closure $\fpbar$ of $\fp$ we set $\zpbr=W(\fpbar)$ and $\qpbr=W(\fpbar)[1/p]$; here $W$ denotes the $p$-typical Witt vectors of perfect $\fp$-algebra. Let $\sigma$ be the automorphism of $\qpbr$ induced by the absolute Frobenius on $\fpbar$. Let $B(G)$ be the set of $\sigma$-conjugacy classes in $G(\qpbr)$, equipped with the topology coming from the \emph{opposite} of the partial order defined in \cite[Section 2.3]{RapoportRichartz}. By \cite[Theorem 1]{Viehmann}, there is a homeomorphism
\begin{align}
    \lvert \bun_G \rvert \to B(G). 
\end{align}
Here, while \cite{FarguesScholze} and \cite{Viehmann} work with $(\bung)_{\fpbar}$, we note that the group $\gal(\fpbar/\fp) = \gal(\qpbr/\qp)$ acts trivially on $B(G)$ because $g \in G(\qpbr)$ and $\sigma(g)$ are $\sigma$-conjugate, and therefore $\lvert \bun_G \rvert \simeq B(G)$ before base changing to $\fpbar$.

If $\mu$ is a $G(\qpbar)$-conjugacy class of minuscule cocharacters, we let $\bgmu \subset B(G)$ be the set of $\mu^{-1}$-admissible elements, as defined in \cite[Section 1.1.5]{KMPS}; note that this set is closed in the partial order and thus defines an open substack $$\bungmu \subset \bung$$ via \cite[Proposition 12.9]{EtCohDiam}. 

For $b \in G(\qpbr)$, we let $[b] \in B(G)$ denote the $\sigma$-conjugacy class of $b$. By \cite[Theorem 5.3]{AnschutzIsocrystal}, there is an element $\mathcal{E}_{b} \in \bun_G(\spd \fpbar)$ associated to $b$, and we write $\tilde{G}_b = \Aut(\mathcal{E}_b)$ for its sheaf of automorphisms. We recall the following result. 

\begin{Thm}{\cite[Theorem~III.0.2]{FarguesScholze}}\label{Thm:GeometryOfBunG}
The subfunctor     
\[
        \bun_{G}^{[b]} = \bun_G \times_{\lvert \bun_G \rvert} \lbrace [b] \rbrace
        \subseteq \bun_G
\]
      is locally closed. Moreover its base change to $\spd \fpbar$ is isomorphic to $\left[\spd \fpbar / \tilde{G}_b\right]$.
\end{Thm}

For any v-stack $\mathscr{Y}$ on $\perf$ equipped with a morphism $\mathscr{Y} \to \bung$, we write
\begin{align}\label{Eq:NewtonStrat}
    \mathscr{Y}^{[b]} = \mathscr{Y} \times_{\bung} \bung^{[b]}.
\end{align}

\subsubsection{Absolute Frobenii} \label{Sec:AbsoluteFrobenii}
For every v-stack $X$, there is a canonical automorphism
\[
  \phi_X \colon X \to X; \quad X(S) \ni f \mapsto \phi_S^\ast f \in X(S),
\]
which we shall call the \textit{absolute Frobenius} on $X$. When $X =
Y^\lozenge$ for a pre-adic space $Y / \spa \fp$, the automorphism $\phi_X \colon
X \to X$ agrees with the map of v-sheaves induced from the absolute Frobenius on
$Y$. Moreover, every map of v-stacks $f \colon X \to Y$ is canonically
$\phi$-equivariant in the sense that there exists a canonical 2-morphism filling
in the diagram
\[ \begin{tikzcd}
  X \arrow{r}{\phi_X} \arrow{d}{f} & X \arrow{d}{f} \\ Y \arrow{r}{\phi_Y} & Y,
\end{tikzcd} \]
that behaves well under composition. The absolute Frobenius on a fiber product
is canonically identified with the fiber product of absolute Frobenii.

\begin{Prop} \label{Prop:AbsoluteFrobeniusBunG}
  For every connected reductive group $G$ over $\qp$, there exists a 2-morphism
  $\phi_{\bun_G} \simeq \mathrm{id}_{\bun_G}$.
\end{Prop}

\begin{proof}
  For $S \in \perf$, we regard a $G$-bundle on the Fargues--Fontaine curve as a
  $G$-bundle $\mathscr{P}$ on $\mathcal{Y}_{(0,\infty)}(S)$ together with an
  automorphism $\alpha \colon \varphi^\ast \mathscr{P}
  \xrightarrow{\sim} \mathscr{P}$. Then $\alpha$ defines the 2-morphism
  $\phi_{\bun_G} \xrightarrow{\sim} \mathrm{id}_{\bun_G}$.
\end{proof}

\subsection{Some Bruhat--Tits theory } \label{Sec:BruhatTits}  Let $G$ be a connected reductive group over $\qp$. We write $\Gamma_p$ for the absolute Galois group $\gal(\qpbar/\qp)$, and let $I_{p} \subset \Gamma_p$ be the inertia subgroup. Let $\pi_1(G)$ be the algebraic fundamental group of $G$, see \cite{Borovoi}. We denote by $\tilde{\kappa}_G$ the Kottwitz homomorphism 
\begin{align}
    \tilde{\kappa}_G: G(\qpbr) \to \pi_1(G)_{I_p},
\end{align}
which is constructed originally in \cite[Section 7]{Kottwitz2} (but see also \cite[Section 11.5]{KalethaPrasad} for an exposition of its construction).
Denote the composition of $\tilde{\kappa}_G$ with $\pi_1(G)_{I_p} \to \pi_1(G)_{\Gamma_p}$ by $\kappa_G$. We define 
\begin{align*}
    G(\qpbr)^0 = \ker(\tilde{\kappa}_G) \qquad \text{ and } \qquad G(\qpbr)^1 = \tilde{\kappa}_G^{-1}(\pi_1(G)_{I_p, \mathrm{tors}}),
\end{align*}
where $\pi_1(G)_{I_p, \mathrm{tors}}$ denotes the torsion subgroup of $\pi_1(G)_{I_p}$. 

\subsubsection{Bruhat--Tits buildings} \label{Sec:Parahorics} Let $\mathcal{B}(G,\qp)$ (resp.\ $\mathcal{B}(G,\qpbr)$) denote the (reduced) Bruhat--Tits building of $G$ (resp.\ of $G_{\qpbr})$; it is a contractible metric space with an action of $G(\qp)$ (resp.\ $G(\qpbr)$) by isometries, see \cite[Axiom 4.1.1, Corollary 4.2.9]{KalethaPrasad}. It also naturally has the structure of a polysimplicial complex (see \cite[Definition 1.5.1]{KalethaPrasad}) with facets denoted by $\mathcal{F} \subset \mathcal{B}(G,\qp)$ (resp.\ $\mathcal{F} \subset \mathcal{B}(G,\qpbr)$). Note that there is a $G(\qp)$-equivariant inclusion $\mathcal{B}(G,\qp) \subset \mathcal{B}(G,\qpbr)$ identifying $\mathcal{B}(G,\qp)$ with the fixed points of $\mathcal{B}(G,\qpbr)$ under the Frobenius $\sigma$, see \cite[Theorem 9.2.7]{KalethaPrasad}.

Given a subset $\Omega$ of $\mathcal{B}(G,\qpbr)$ we consider the pointwise stabilizers $G(\qpbr)^0_{\Omega}$ and $G(\qpbr)^1_{\Omega}$ of $\Omega$ inside of $G(\qpbr)^0$ and $G(\qpbr)^1$, respectively. When $\Omega = \{x\}$ is a point, we will write $G(\qpbr)^0_x$ and $G(\qpbr)^1_x$ instead of $G(\qpbr)^0_{\{x\}}$ and $G(\qpbr)^1_{\{x\}}$. If $\Omega = \mathcal{F}$ is a facet, a subgroup of the form  $G(\qpbr)^0_{\mathcal{F}}$ is called a \emph{parahoric subgroup}. Following \cite[Section 2.2]{PappasRapoportRZSpaces}, we will define a \emph{quasi-parahoric subgroup} $\breve{K} \subset G(\qpbr)$ to be any subgroup for which there exists (necessarily uniquely) a facet $\mathcal{F}$ of $\mathcal{B}(G,\qpbr)$ such that 
\begin{align} \label{Eq:FacetRelation}
    G(\qpbr)^0 \cap \operatorname{Stab}_{\mathcal{F}} \subset \breve{K} \subset G(\qpbr)^1 \cap \operatorname{Stab}_{\mathcal{F}},
\end{align}
where now $\operatorname{Stab}_{\mathcal{F}}$ is the stabilizer of $\mathcal{F}$ in $G(\qpbr)$ (rather than the pointwise stabilizer). 

\subsubsection{Quasi-parahorics} For a quasi-parahoric subgroup $\breve{K}$ there is a unique
smooth affine group scheme $\mathcal{G}$ over $\zpbr$ together with an
isomorphism $\mathcal{G}_{\qpbr} \xrightarrow{\sim} G_{\qpbr}$ which identifies
$\mathcal{G}(\zpbr)$ with $\breve{K}$, called the \emph{quasi-parahoric group
scheme} associated to $\breve{K}$. If $\breve{K}$ is moreover stable under
$\sigma$, then the corresponding quasi-parahoric group scheme canonically
descends to $\zp$.

For such $\breve{K}$ with corresponding quasi-parahoric group scheme
$\mathcal{G}$ over $\zp$, the inclusion $G(\qpbr)^0_\mathcal{F} = \breve{K} \cap
G(\qpbr)^0 \subset \breve{K}$ induces an open immersion $\mathcal{G}^{\circ} \to
\mathcal{G}$, where $\mathcal{G}^{\circ}$ is the parahoric group scheme
corresponding to $G(\qpbr)^0_\mathcal{F}$. Moreover, the induced map
$\calgcirc_{\fp} \to \mathcal{G}_{\fp}$ is the inclusion of the identity
component of $\mathcal{G}_{\fp}$, see \cite[Theorem~8.3.13]{KalethaPrasad}. By
\cite[Corollary~11.6.3]{KalethaPrasad}, the finite \'etale group scheme over
$\fp$
\[
  \pi_0(\mathcal{G}):=\pi_0(\mathcal{G}_{\fp})
\]
can be identified, as a finite group with an action of $\gal(\fpbar/\fp)$, with
the image of $\breve{K}$ in $\pi_1(G)_{I_p, \mathrm{tors}}$ under the Kottwitz
map $\tilde{\kappa}_{G}$.

\begin{Def} \label{Def:QpStabilizerQuasiParahoric}
  We say that a smooth
  affine model $\mathcal{G}/\zp$ of $G$ is a \textit{stabilizer
  Bruhat--Tits group scheme} when it is the quasi-parahoric group scheme
  attached to a subgroup of the form $\breve{K} = G(\qpbr)_x^1$ for a point $x
  \in \mathcal{B}(G, \qp) \subseteq \mathcal{B}(G, \qpbr)$. A \textit{stabilizer parahoric
  group scheme} is a parahoric group scheme that is also a stabilizer
  Bruhat--Tits group scheme.
\end{Def}
If $\mathcal{G}$ is a stabilizer Bruhat--Tits group scheme then $\mathcal{G}(\zpbr)$ is the stabilizer in $G(\qpbr)$ of a point $x$ in the extended (or enlarged) Bruhat--Tits building $\mathcal{B}^e(G,\qp)$ of \cite[Section 4.3]{KalethaPrasad}. We will use these extended buildings in the proof of Lemma \ref{Lem:LandVogt} below.

\begin{Lem} \label{Lem:StabilizerToPointwiseStabilizer}
  Let $\mathcal{F}$ be a facet in $\mathcal{B}(G, \qp)$, regarded as a subset of
  $\mathcal{B}(G, \qpbr)$. Then there exists a point $x \in \mathcal{F}$ such
  that $G(\qpbr)^1_x = G(\qpbr)^1_{\mathcal{F}}$.
\end{Lem}

\begin{proof}
  Since facets in $\mathcal{B}(G, \qp)$ are intersections of facets in
  $\mathcal{B}(G, \qpbr)$ with $\mathcal{B}(G, \qp)$, see
  \cite[Section~9.2.4]{KalethaPrasad}, there exists a facet of $\mathcal{B}(G,
  \qpbr)$ containing $\mathcal{F}$. This implies that $G(\qpbr)$ acts on
  $\mathcal{F}$ through affine-linear automorphisms.

  The affine-linear automorphism group $\Aut_\mathrm{aff}(\mathcal{F})$ of
  $\mathcal{F}$ is finite, because such an automorphism is determined by its
  action on vertices. For each nontrivial element $\mathrm{id} \neq \alpha \in
  \Aut_\mathrm{aff}(\mathcal{F})$, the fixed point set $\mathcal{F}^\alpha$ is the
  intersection of $\mathcal{F}$ with a lower-dimensional linear subspace. Take
  $x \in \mathcal{F}$ to be any point in the complement $\mathcal{F} \setminus
  \bigcup_{\mathrm{id} \neq \alpha} \mathcal{F}^\alpha$, so that the stabilizer
  of $x$ in $\Aut_\mathrm{aff}(\mathcal{F})$ is trivial. Then an element $g \in
  G(\qpbr)$ fixes $x$ if and only if it fixes $\mathcal{F}$ pointwise, because
  the action of $g$ on $\mathcal{F}$ is affine-linear as noted above.
\end{proof}

\begin{Lem} \label{Lem:LandVogt}
  Let $f \colon G \to G'$ be a morphism of connected reductive groups over
  $\qp$. Then there exist points $x \in \mathcal{B}(G,\qp)$ and $y \in
  \mathcal{B}(G',\qp)$ with corresponding stabilizer Bruhat--Tits group schemes
  $\mathcal{G}$ and $\mathcal{G}'$ such that $f$ extends $($necessarily uniquely$)$
  to a morphism $f \colon \mathcal{G} \to \mathcal{G}'$.
\end{Lem}
\begin{proof}
  By \cite[Proposition~1.7.6]{BruhatTitsII}, it suffices to show that we can
  find $x$ and $y$ such that $f(G(\qpbr)^1_x) \subset G'(\qpbr)^1_y$ (see also
  \cite[Corollary~2.10.10]{KalethaPrasad}). Combining \cite[Proposition~2.1.3,
  Theorem~2.2.1]{Landvogt}, we see that there is a $G(\qpbr)$-equivariant,
  Frobenius-equivariant, and continuous map $g \colon \mathcal{B}^e(G,\qpbr) \to
  \mathcal{B}^e(G',\qpbr)$. Thus if we pick a point $\tilde{x} \in
  \mathcal{B}^e(G,\qp)$, then $g(\tilde{x}) \in \mathcal{B}^e(G',\qp)$ and
  moreover the stabilizer in $G(\qpbr)$ of $\tilde{x}$ maps to the stabilizer in
  $G'(\qpbr)$ of $g(\tilde{x})$. Then the lemma is proved by taking $x \in
  \mathcal{B}(G,\qp)$ to be the projection of $\tilde{x}$ and $y \in
  \mathcal{B}(G',\qp)$ to be the projection of $g(\tilde{x})$.
\end{proof}

\subsection{Local Shtukas} \label{Sub:LocalShtukas} Let $\mathcal{G}$ be a quasi-parahoric group scheme over $\zp$ with generic fiber $G$. We let $\mathrm{Gr}_\mathcal{G} \to \spd \zp$ denote the corresponding Beilinson--Drinfeld Grassmannians as in \cite[Definition 20.3.1]{ScholzeWeinsteinBerkeley}. Its base change to $\spd \qp$ is the $\mathbf{B}^+_\mathrm{dR}$ affine Grassmannian, and we will denote it by $\mathrm{Gr}_{G} \to \spd \qp$, see \cite[Lecture XIX]{ScholzeWeinsteinBerkeley}. 

For $\mu$ a $G(\qpbar)$-conjugacy class of minuscule cocharacters of $G$ with reflex field $E$, we denote by $\mathrm{Gr}_{G,{\mu}} \subset \mathrm{Gr}_{G,E}$ the closed Schubert cell determined by $\mu$, see \cite[Section 19.2]{ScholzeWeinsteinBerkeley}. Note that since $\mu$ is minuscule, this can be identified with the (diamond associated to) the flag variety $\mathscr{F}\ell_\mu$ using the Bia\l{}ynicki-Birula map. The v-sheaf theoretic closure of $\mathrm{Gr}_{G,{\mu}}$ inside $\mathrm{Gr}_{\mathcal{G},\spd \mathcal{O}_E}$ is the \emph{v-sheaf local model} $\mathbb{M}_{\mathcal{G},{\mu}}^{\mathrm{v}} \subset \mathrm{Gr}_{\mathcal{G},\spd \mathcal{O}_E}$. By \cite[Theorem 1.2]{AGLR} and \cite[Corollary 1.4]{GleasonLourencoLocalModel}, this is the diamond associated to a flat normal projective scheme over $\mathcal{O}_E$.

Recall that a $\mathcal{G}$-shtuka over a perfectoid space $S$ with leg at an untilt $S^\sharp$ is defined to be a $\mathcal{G}$-torsor $\mathscr{P}$ over $S \bdtimes \spa \zp$, together with an isomorphism of $\mathcal{G}$-torsors
\begin{align}
    \phi_{\mathscr{P}}:\restr{\varphi^{\ast} \mathscr{P}}{S \bdtimes \zp \setminus S^{\sharp}} \to \restr{\mathscr{P}}{S \bdtimes \zp \setminus S^{\sharp}}
\end{align}
which is meromorphic along $S^\sharp$ in the sense of \cite[Definition 5.3.5]{ScholzeWeinsteinBerkeley}. For $\mu$ a $G(\qpbar)$-conjugacy class of minuscule cocharacters, we say that a $\mathcal{G}$-shtuka $(\mathscr{P}, \phi_{\mathscr{P}})$ is bounded by $\mu$ if the relative position of $\varphi^\ast \mathscr{P}$ and $\mathscr{P}$ at $S^\sharp$ is bounded by the v-sheaf local model $\mathbb{M}_{\mathcal{G},{\mu}}^{\mathrm{v}} \subset \mathrm{Gr}_{\mathcal{G},\spd \mathcal{O}_E}$ (see \cite[Section 2.3.4]{PappasRapoportShtukas}).

For $S$ in $\perf$, denote by $\shtg(S)$ the groupoid of triples $(S^{\sharp}, \mathscr{P}, \phi_{\mathscr{P}})$, where
$S^{\sharp}$ is an untilt of $S$ and where $(\mathscr{P}, \phi_{\mathscr{P}})$ is a $\mathcal{G}$-shtuka over $S$ with leg $S^\sharp$. By \cite[Proposition~2.1.2]{ScholzeWeinsteinBerkeley}, the assignment $S \mapsto \shtg(S)$ defines a v-stack $\shtg$ on $\perf$ (note that $S \bdtimes \zp$ is sousperfectoid by the proof of \cite[Proposition~11.2.1]{ScholzeWeinsteinBerkeley}).\footnote{Note that in \cite[Definition 6.5]{GleasonIvanovZillinger}, the notation $\shtg$ is used to denote the stack which would be $\shtg \otimes_{\spd \zp} \spd \fp$ in our notation. }

If ${\mu}$ is a $G(\qpbar)$-conjugacy class of cocharacters of $G_{\qpbar}$ with field of definition $E$, we let $\shtgmu \subset \shtg \times_{\spd \zp} \spd \mathcal{O}_E$ be the closed substack whose $S$-points consist of $\mathcal{G}$-shtukas over $S$ with one leg at $S^\sharp$ which are bounded by $\mu$. 

Let $S = \spa(R,R^+) \to \spd \zp$ be an object in $\perf$ together with an untilt $S^\sharp$, and let $(\mathscr{P}, \phi_{\mathscr{P}})$ be a shtuka over $S$ with one leg at $S^\sharp$. We can choose $r$ sufficiently large such that $\mathcal{Y}_{[r,\infty)}$ does not meet the divisor of $\mathcal{Y}_{[0,\infty)}$ defined by $S^\sharp$. The restriction of $(\mathscr{P},\phi_{\mathscr{P}})$ determines a $\phi$-equivariant $\mathcal{G}$-torsor on $\mathcal{Y}_{[r,\infty)}$. By spreading out via the Frobenius (see \cite[Proposition 22.1.1]{ScholzeWeinsteinBerkeley}), the bundle $\mathscr{P}$ descends to a $G$-bundle $\mathcal{E}(\mathscr{P},\phi_{\mathscr{P}})$ on $X_S$. In this way we obtain a morphism of v-stacks on $\perf$ 
\begin{align}\label{Eq:ShtukaMorphism}
    \shtg \to \bung, \quad (\mathscr{P},\phi_{\mathscr{P}}) \mapsto \mathcal{E}(\mathscr{P},\phi_{\mathscr{P}})
\end{align}
and we will denote by $\mathrm{BL}^{\circ}$ its restriction to $\shtgmu$. 

\begin{Lem} \label{Lem:ZeroTruncated}
  The map $\shtg \to \bun_G$ is $0$-truncated, i.e., for every $\spa(R, R^+) \in \perf$, the map of groupoids
  \[
    \shtg(R, R^+) \to \bun_G(R, R^+)
  \]
  is faithful.
\end{Lem}

\begin{proof}
  Let $(\mathscr{P}, \phi_\mathscr{P})$ be a $\mathcal{G}$-shtuka with leg at an
  untilt $(R^\sharp, R^{\sharp+})$. Fix a pseudouniformizer $\varpi \in R^+$
  which we use to define $\mathcal{Y}_{I}(R,R^+)$ for intervals $I$. Choose an
  integer $r$ sufficiently large such that $\mathcal{Y}_{[r,\infty)}$ does not
  meet the divisor of $\mathcal{Y}_{[0,\infty)}$ defined by $S^\sharp$.

  Let $\alpha$ be an automorphism of $(\mathscr{P}, \phi_\mathscr{P})$. It
  suffices to prove
  that if $\alpha$ restricts to the identity on $\mathcal{Y}_{[r,\infty)}(R,
  R^+)$, then $\alpha$ is the identity map. By pushing out along $\mathcal{G}
  \hookrightarrow \mathrm{GL}(V)$, we may reduce the statement to the case of
  vector bundle shtukas. Choose a larger integer $s > r$. We first check that
  the ring homomorphism
  \[
    A = H^0(\mathcal{Y}_{[0,s]}(R, R^+),
    \mathscr{O}_{\mathcal{Y}_{[0,s]}(R,R^+)}) \to
    H^0(\mathcal{Y}_{[r,s]}(R,R^+),
    \mathscr{O}_{\mathcal{Y}_{[r,s]}(R,R^+)}) = B
  \]
  is injective. We may write
  \[
    A = W(R^+)\langle p^s/[\varpi] \rangle[\tfrac{1}{[\varpi]}], \quad B =
    W(R^+)\langle p^s/[\varpi], [\varpi]/p^r \rangle[\tfrac{1}{p[\varpi]}],
  \]
  and it suffices to instead verify the injectivity of
  \[
    A_0 = W(R^+)\langle p^s/[\varpi] \rangle \to W(R^+)\langle p^s/[\varpi],
    [\varpi]/p^r \rangle = B_0.
  \]
  Recall that both sides are $(p, [\varpi])$-adic completions of
  \[
    A_1 = W(R^+)[p^s/[\varpi]] \subseteq W(R^+)[p^s/[\varpi], [\varpi]/p^r] =
    B_1;
  \]
  the ring $A_0$ is definitionally the $[\varpi]$-adic completion of $A_1$, but
  $(p, [\varpi])$ and $([\varpi])$ generate the same topology on $A_1$.
  Injectivity of $A_0 \to B_0$ now follows from the claim that for every $n
  \ge 0$ we have
  \[
    ((p, [\varpi])^n B_1) \cap A_1 = (p, [\varpi])^n A_1.
  \]
  The right hand side is clearly contained in the left hand side. In the
  other direction, we need to verify that if $x_1, \dotsc, x_N \in W(R^+)$
  satisfies
  \[
    x = x_1 \frac{[\varpi]^{n+1}}{p^r} + \dotsb + x_N
    \frac{[\varpi]^{n+N}}{p^{Nr}} \in A_1
  \]
  then $x \in (p, [\varpi])^n A_1$. Because $x_N ([\varpi]^{n+N} / p^{Nr}) \in
  p^{-(N-1)r} W(R)$ we have $x_N \in p^r W(R) \cap W(R^+) = p^r W(R^+)$, and
  hence by induction on $N$ we conclude that $x \in [\varpi]^{n+1} W(R^+)
  \subseteq (p, [\varpi])^n A_1$ as desired.

  A vector bundle on $\mathcal{Y}_{[0,s]}$
  corresponds to a finite projective module $M$ over $A$ by
  \cite[Theorem~2.7.7]{KedlayaLiu}, and then the injectivity of $A \to B$
  implies that
  \[
    \Hom_A(M, M) \to \Hom_A(M, M) \otimes_A B = \Hom_B(M \otimes_A B, M
    \otimes_A B)
  \]
  is injective. Since $\alpha$ restricted to
  $\mathcal{Y}_{[r,s]}(R, R^+)$ is the identity map, $\alpha$ restricted to
  $\mathcal{Y}_{[0,s]}(R, R^+)$ is also the identity map. As
  $\mathcal{Y}_{[0,s]}(R, R^+)$ and $\mathcal{Y}_{[r,\infty)}(R, R^+)$ cover all
  of $\mathcal{Y}_{[0,\infty)}(R, R^+)$, the map $\alpha$ is the identity map on
  all of $\mathcal{Y}_{[0,\infty)}(R, R^+)$.
\end{proof}

Recall from \cite[Section 3.3.7]{Companion} the open and closed substack
\begin{align}
    \shtgmuone \subset \shtgmu,
\end{align}
which is the image of $\shtgcircmu \to \shtgmu$, where $\calgcirc \subset \calg$ is the relative identity component of $\calg$ (a parahoric model of $G$). It follows from \cite[Corollary 3.3.9]{Companion} that there is an isomorphism 
\begin{align} \label{Eq:GenericFibreShtukas}
    c:\shtgmuonerat \to \left[\mathrm{Gr}_{G,{\mu^{-1}}} / \underline{\calg(\zp)}\right],
\end{align}
where $\underline{\mathcal{G}(\zp)}$ is as in \cite[the discussion before Definition 10.12]{EtCohDiam}.
Recall the Beauville--Laszlo map $\mathrm{BL}:\mathrm{Gr}_{G,{\mu^{-1}}} \to \bung$ from \cite[Proposition III.3.1]{FarguesScholze}, which factors through a map $\mathrm{BL}:\left[\mathrm{Gr}_{G,{\mu^{-1}}} / \underline{\mathcal{G}(\zp)}\right] \to \bun_G$. 
\begin{Lem} \label{Lem:BLCircVSBL}
    We have an equality $$\mathrm{BL} \circ c=\mathrm{BL}^{\circ} \times_{\spd \mathcal{O}_E} \spd E.$$
\end{Lem}
\begin{proof}
This is explained in \cite[Proposition 11.16, Proposition 11.17]{ZhangThesis} for $\calgcirc$, and now follows from the definition of $\shtgmuone$, see \cite[Section 3.3.7]{Companion}.
\end{proof}

\subsection{Dieudonn\'e theory} \label{Sec:Dieudonne} In this section we will recall some Dieudonn\'e theory in order to make our conventions clear.

Let $R$ be a semiperfect $\fp$-algebra and let 
$\mathbf{A}_\mathrm{crys}(R)$ denote Fontaine's crystalline period ring
equipped with its Frobenius $\phi$, see \cite[Proposition~4.1.3]{ScholzeWeinsteinModuli}. Let $Y$ be a $p$-divisible group over $\spec R$ and let 
$\mathbb{D}(Y)$ denote its contravariant Dieudonn\'e crystal over $\spec R$ in the sense of \cite{BerthelotBreenMessing}. Its evaluation on $\mathbf{A}_\mathrm{crys}(R)$ is a finite projective $\mathbf{A}_\mathrm{crys}(R)$-module also denoted by $\mathbb{D}(Y)$. The relative Frobenius map $F_{Y/R}:Y \to Y^{(p)}$ induces a morphism $\phi_Y:\mathbb{D}(Y^{(p)})=\phi^{\ast} \mathbb{D}(Y) \to \mathbb{D}(Y)$ which becomes an isomorphism after inverting $p$. The pair $(\mathbb{D}(Y), \phi_Y)$ is functorial for morphisms of $p$-divisible groups over $R$.

We will often use the $\mathbf{A}_\mathrm{crys}(R)$-linear dual $\mathbb{D}^{\ast}(Y)$ of $\mathbb{D}(Y)$, which is equipped with a map
\begin{align}
    \phi_{Y}:\phi^{\ast}\mathbb{D}^{\ast}(Y)[1/p] \to \mathbb{D}^{\ast}(Y)[1/p]
\end{align}
given by the inverse of the dual of the isomorphism $\phi_Y:\phi^{\ast} \mathbb{D}(Y)[1/p] \to \mathbb{D}(Y)[1/p]$.\footnote{The pair $(\mathbb{D}^{\ast},\phi_{Y})$ is the covariant Dieudonn\'e module of $Y$ normalized as in \cite{CaraianiScholzeCompact}.} 

If $R$ is perfect, then $\phi$ is an automorphism of $W(R)$ and we can consider
\begin{align}
    \mathbb{D}^{\natural}(Y)=(\phi^{-1})^{\ast} \mathbb{D}^{\ast}(Y),
\end{align}
which is equipped with the Frobenius $\phi_{Y}^{\natural}=(\phi^{-1})^{\ast}\phi_{Y}$, see \cite[Remark 2.3.10]{PappasRapoportShtukas} for the notation. 

\label{subsub:Rsharpmap}
Let $(R,R^+)$ be a perfectoid Huber pair of characteristic $p$ and let $\varpi \in R^{+}$ be a pseudouniformizer. We will consider $W(R^{+})$ equipped with its Frobenius $\phi$. Given an untilt $(R^{\sharp}, R^{\sharp+})$ there is a natural map $\theta:W(R^+) \to R^{\sharp+}$ whose kernel is generated by an element $\xi$. After possibly replacing $\varpi$ by a $p^n$-th root, there is an induced natural map 
\begin{align} \label{Eq:NaturalMapPseudoUniformizer}
    R^{\sharp+} \to R^{+}/\varpi
\end{align}
given by realizing $R^{\sharp+}$ as the quotient of $W(R^{+})$ by the kernel of $\theta$ and then taking the quotient by $p$ and $\theta([\varpi])$.
Let $Y$ be a $p$-divisible group over $R^{\sharp+}$, then by \cite[Theorem 17.5.2]{ScholzeWeinsteinBerkeley}, there is a finite projective $W(R^{+})$-module $M(Y)$ together with an isomorphism 
\begin{align}
    \phi_Y:\phi^{\ast} M(Y)[1/\phi(\xi)] \to M(Y)[1/\phi(\xi)].
\end{align}
The pair $(M(Y), \phi_Y)$ is functorial for morphisms of $p$-divisible groups over $R^{\sharp+}$. We will often consider the pullback under $\phi^{-1}$ of the pair $(M(Y), \phi_Y)$, and denote it by $(M^{\natural}(Y), \phi^{\natural}_Y)$. The pair $(M^{\natural}(Y), \phi^{\natural}_Y)$ is a BKF module over $R^{\sharp+}$ in the sense of \cite[Definition 2.2.4]{PappasRapoportShtukas}, and gives rise to a shtuka $\mathscr{V}(Y)$ over $\spd(R,R^{+})$ via \cite[Definition 2.2.6]{PappasRapoportShtukas}, see \cite[Example 2.3.4]{PappasRapoportShtukas}.\footnote{In \cite{PappasRapoportShtukas}, this shtuka is denoted by $\mathcal{E}(Y)$. We will however use $\mathcal{E}(Y)$ for the induced vector bundle on $X_S$, see Lemma \ref{Lem:Compatibility}.} 

There is a natural map $W(R^{+}) \to \mathbf{A}_\mathrm{crys}(R^{+}/\varpi)$, and there is a natural isomorphism
\begin{align}
    (M(Y), \phi_Y) \otimes_{W(R^{+})} \mathbf{A}_\mathrm{crys}(R^{+}/\varpi) \xrightarrow{\sim} (\mathbb{D}^{\ast}(Y \otimes_{R^{\sharp+}} R^{+}/\varpi), \phi_Y),
\end{align}
see \cite[Theorem 17.5.2]{ScholzeWeinsteinBerkeley}. 

\subsection{Formal quasi-isogenies} \label{Sub:FormalQuasiIsogenies}

Suppose $\varpi$ is a pseudouniformizer of $R^+$ for which the natural map $W(R^+) \to R^{\sharp+}$ induces a map
\begin{align}
    R^{\sharp+} \to R^{+}/\varpi
\end{align}
as in equation \eqref{Eq:NaturalMapPseudoUniformizer}. Using this map, we will define a notion of
formal quasi-isogeny between $p$-divisible groups (or formal abelian schemes up to prime-to-$p$ isogeny\footnote{Throughout this paper we will often work with the category of (formal) abelian schemes up to prime-to-$p$ isogeny, which is obtained from the category of (formal) abelian schemes by tensoring the homomorphism groups with $\zlocp$.}) that live over different untilts of $R^+$. 

\label{Sec:FormalQuasiIsogenyDifferentUntiltsDef} 
Assume that we are given two untilts $(R^{\sharp_1}, R^{\sharp_1+})$ and
$(R^{\sharp_2}, R^{\sharp_2+})$ of $(R, R^+)$. Let $Y_1, Y_2$ be $p$-divisible
groups over $\spf R^{\sharp_1+}, \spf R^{\sharp_2+}$. A \emph{formal
quasi-isogeny}
\[
  Y_1 \dashrightarrow Y_2
\]
is a quasi-isogeny $Y_1 \otimes_{R^{\sharp_1+}} R^+/\varpi \dashrightarrow Y_2
\otimes_{R^{\sharp_2+}} R^+/\varpi$. We note that this notion does not depend on
the choice of pseudouniformizer $\varpi$, by Serre--Tate lifting
\cite[Lemma~1.1.3]{KatzSerreTate}. If $A_1, A_2$ are formal abelian schemes over
$\spf R^{\sharp_1+}, \spf R^{\sharp_2+}$ (up to prime-to-$p$ isogeny), we similarly define a \emph{formal
quasi-isogeny}
\[
  A_1 \dashrightarrow A_2
\]
to be a quasi-isogeny $A_1 \otimes_{R^{\sharp_1+}} R^+/\varpi \dashrightarrow
A_2 \otimes_{R^{\sharp_2+}} R^+/\varpi$.

\label{subsub:WeaklyPolarized} Suppose now $(R^\sharp, R^{\sharp+})$ is an untilt of $(R,R^+)$, and suppose $A$ is a formal abelian scheme up to prime-to-$p$ isogeny over $\spf R^{\sharp+}$. A \textit{weak principal polarization} of $A$ is a prime-to-$p$ quasi-isogeny $\lambda: A \dashrightarrow A^\vee$ that is equal to $c \cdot \mu$ where $\mu:A \to A^{\vee}$ is a principal polarization and where $c$ is a locally constant $\zlocp^{\times}$-valued function on $\spec R^+/\varpi$. 

If $(A_1, \lambda_1)$ and $(A_2, \lambda_2)$ are weakly polarized formal abelian schemes up to prime-to-$p$ isogeny over two untilts $(R^{\sharp_1}, R^{\sharp_1+})$ and $(R^{\sharp_2}, R^{\sharp_2+})$, respectively, then a formal quasi-isogeny
\begin{align*}
    f \colon (A_1, \lambda_1) \dashrightarrow (A_2, \lambda_2)
\end{align*}
is a formal quasi-isogeny $A_1 \dashrightarrow A_2$ such that 
\begin{align*}
    f^\vee \circ \lambda_2 \circ f = c \, \lambda_1,
\end{align*}
where $c$ is a locally constant $\mathbb{Q}^\times$-valued function on $\spec R^+/\varpi$. 

\label{Sec:FormalQuasiIsogeniesFF} Let $S = \spa(R,R^+)$ be in $\perf$ and let $(R^{\sharp}, R^{\sharp+})$ be an untilt. Let $\varpi$ be a choice of pseudouniformizer of $R^+$. For a $p$-divisible group $Y$ over $R^{+}/\varpi$ we will write $\mathbb{D}^{\ast}(Y)$ for the finite projective $\mathbf{A}_\mathrm{crys}(R^{+}/\varpi)$-module defined in Section \ref{Sec:Dieudonne}, which comes equipped with an isomorphism
\begin{align}
    \phi_{Y}:\phi^{\ast} \mathbb{D}^{\ast}(Y)[1/p] \to \mathbb{D}^{\ast}(Y)[1/p]. 
\end{align}
Define $\mathbf{B}_\mathrm{crys}^+(R^+/\varpi)=\mathbf{A}_\mathrm{crys}(R^+/\varpi)[1/p]$ and note that $\mathbb{D}^{\ast}(Y)[1/p]$ is a module over $\mathbf{B}_\mathrm{crys}^+(R^+/\varpi)$. Recall that for a large enough rational number $r \gg 0$, there is a natural ring homomorphism
\[
  \mathbf{B}_\mathrm{crys}^+(R^+/\varpi) \to \Gamma(\mathcal{Y}_{[r, \infty]}(R, R^+),
  \mathscr{O}_{\mathcal{Y}_{[r,\infty]}(R,R^+)}),
\]
for instance, as in \cite[Section~6.2]{FarguesConjecture}. Base changing along
this map, we obtain from $\mathbb{D}^{\ast}(Y)[1/p]$ a vector bundle $\mathcal{E}(Y)$ on
$\mathcal{Y}_{[r, \infty)}(R, R^+)$, together with an isomorphism $\phi^\ast
\mathcal{E}(Y) \xrightarrow{\sim} \mathcal{E}(Y)$ of vector bundles. Descending along this
isomorphism, we obtain a vector bundle $\mathcal{E}(Y)$ on $X_S$. As the category of $\phi$-equivariant vector bundles on $\mathcal{Y}_{[r,\infty]}(R,R^+)$ does not depend on $r$, this construction is independent of the choice of pseudouniformizer $\varpi$ as well as the rational number $r$.

If $Y$ is a $p$-divisible group over $R^{\sharp+}$, then we can define a vector bundle over $X_S$ in two ways: either we take $\mathcal{E}(Y_{R^+/\varpi})$ as above or we consider the vector bundle shtuka $\mathscr{V}(Y)$ from Section \ref{Sec:Dieudonne} and apply the map $\mathrm{BL}^{\circ}$ from equation \eqref{Eq:ShtukaMorphism}.
\begin{Lem} \label{Lem:Compatibility}
    There is a functorial (in $Y$) isomorphism $\mathcal{E}(Y_{R^+/\varpi}) \to \mathrm{BL}^{\circ}(\mathscr{V}(Y))$.
\end{Lem}
\begin{proof}
The shtuka $\mathscr{V}(Y)$ comes from the BKF-module $(M^{\natural}(Y), \phi^{\natural}_Y)$ over $R^{\sharp+}$, while $\mathcal{E}(Y_{R^+/\varpi})$ comes from the finite projective $\mathbf{A}_\mathrm{crys}(R^{+}/\varpi)$-module given by $\mathbb{D}^{\ast}(Y_{R^+/\varpi})$. The natural comparison isomorphism in Section \ref{Sec:Dieudonne}
\begin{align}
    \phi^{\ast} (M^{\natural}(Y), \phi^{\natural}_Y) \otimes_{W(R^{+})} \mathbf{A}_\mathrm{crys}(R^{+}/\varpi) \to (\mathbb{D}^{\ast}(Y \otimes_{R^{\sharp+}} R^+/\varpi), \phi_Y),
\end{align}
induces an isomorphism of the induced vector bundles on $X_S$.
\end{proof}

\label{Sec:FormalQuasiIsogenyBundle} Assume that we are given untilts $(R^{\sharp_i}, R^{\sharp_i+})$ of $(R,R^+)$ and $p$-divisible groups $Y_i$ over $\spf R^{\sharp_i+}$ for $i=1,2$. A formal quasi-isogeny $f:Y_1 \dashrightarrow Y_2$ determines an isomorphism of $\mathbf{B}_\mathrm{crys}^+(R^+/\varpi)$-modules
\begin{align}
    \mathbb{D}^{\ast}(f): \mathbb{D}^{\ast}(Y_1)[1/p] \xrightarrow{\sim} \mathbb{D}^{\ast}(Y_2)[1/p],
\end{align}
compatibly with the two Frobenius maps. It in turn induces an isomorphism
\begin{align}
    \mathcal{E}(f): \mathcal{E}(Y_1) \to \mathcal{E}(Y_2)
\end{align}
of vector bundles on $X_S$. If $\varpi$ is replaced with another
pseudouniformizer $\varpi^\prime$ dividing $\varpi$, the quasi-isogeny $f$
restricted to $R^+/\varpi^\prime$ induces the same isomorphism of vector bundles
$\mathcal{E}(f) : \mathcal{E}(Y_1) \to \mathcal{E}(Y_2)$, because we may choose $r$
large enough in the construction so that $\Gamma(\mathcal{Y}_{[r,\infty]}(R,
R^+), \mathscr{O}_{\mathcal{Y}_{[r,\infty]}(R,R^+)})$ contains
$\mathbf{B}_\mathrm{crys}^+(R^+/\varpi^\prime)$ as well. This shows that $\mathcal{E}(f)$
does not depend on the choice of pseudouniformizer $\varpi$.

\section{Witt vector shtukas and isocrystals} \label{Sec:PerfecAlgGeom}
The goal of this section is to introduce the Witt vector (local) shtukas of \cite{XiaoZhu, ShenYuZhang}, and to compare them with the shtukas of Pappas--Rapoport. We also prove new fundamental technical results in the theory of Witt vector shtukas, see Proposition \ref{Prop:VSurjective}.

\subsection{Perfect algebraic geometry} Let $\affperf$ denote the category of perfect $\fp$-algebras. We will equip this with the Grothendieck (pre-)topology where covers are given by v-covers in the sense of \cite[Definition 2.1]{BhattScholze}. By \cite[Theorem 4.1]{BhattScholze}, perfect $\fp$-schemes define v-sheaves on $\affperf$, and for $X$ a scheme over $\fp$ we will write $X^{\mathrm{perf}}$ for the perfection of $X$ (the inverse limit over the Frobenius of $X$). Note that the natural map $X^{\mathrm{perf}} \to X$ is a universal homeomorphism. 

Recall that an integral domain $R$ is called \emph{absolutely integrally closed} if the fraction field $\operatorname{Frac} R$ is algebraically closed and if $R$ is integrally closed in $\operatorname{Frac} R$. The following result is \cite[Lemma 11.2]{AnschuetzExtension}.

\begin{Lem} \label{Lem:SchemeTheoreticProductOfPoints}
    Let $R \in \affperf$, then there is a v-cover $S \to \spec R$ with $S$ the spectrum of a product of absolutely integrally closed valuation rings.
\end{Lem}

\begin{Lem} \label{Lem:ExistenceSection}
    Let $R=\prod_{i \in I} V_i$ be a product of absolutely integrally closed
    valuation rings $V_i$, let $S=\spec R$ and let $f:X \to S$ be the perfection of a proper surjective morphism of schemes. Then $f$ admits a section.
\end{Lem}

\begin{proof}
Note that $f$ is quasi-compact and universally closed because it is the perfection of a proper morphism.

    For $i \in I$ let $s_i=\spec V_i \to S$ be the natural closed immersion, and
    let $f_i:X_i \to s_i$ be the restriction of $X$ to $s_i$. Since the fraction
    field of $V_i$ is algebraically closed, since $f$ is surjective, and $f$ is
    the perfection of a finite type morphism, it follows that $X_i$ has a $\operatorname{Frac}
    V_i$-point. By the valuative criterion for quasi-compact universally closed
    morphisms, see \cite[Proposition 01KF]{stacks-project}, this extends to a section $t_i$ of $f_i$. Since $V_i$ is local, it follows that $t_i:s_i \to X_i \to X$ factors through an affine open
  $U_i \subset X$. \smallskip 
    
    Since $X$ is quasi-compact, there are finitely many affine opens $U_1, \cdots, U_k$ of $X$ such that each $t_i$ factors through $U_j$ for some $1 \le j \le k$. Choosing such a $U_j$ for each $i \in I$, we get a partition $I=I_1 \coprod \cdots \coprod I_k$. 
    
    If we write $U_j=\spec A_j$, then for $i \in I_j$ the map $t_i$ corresponds to a map of $R$-algebras $A_j \to V_i$. We can assemble these to a map of $R$-algebras $A_j \to \prod_{i \in I_j} V_i$ using the universal property of the product, which corresponds to a map of $S$-schemes $f_j:\spec \prod_{i \in I_j} V_i \to U_j$. We can then take the disjoint union over $1 \le j \le k$ to get a map of $S$-schemes
    \begin{align}
      f \colon S=\coprod_{j=1}^k \spec ({\textstyle \prod_{i \in I_j} V_i}) \to X,
    \end{align}
    which gives the desired section.
\end{proof}

\label{Sec:Reduction} Recall that attached to a perfect $\fp$-algebra $R$ there is a v-sheaf $\spd R = \spd(R,R)$ on $\perf$. The following lemma is \cite[Proposition 3.7]{GleasonSpecialization}.
\begin{Lem} \label{Lem:Vcover}
  If $A \to B$ is a morphism of perfect $\fp$-algebras that is a
  v-cover, then the induced map $\spd B \to \spd A$ is v-surjective.
\end{Lem}

Let $\mathcal{F} \colon \perf^\mathrm{op} \to \mathsf{Grpd}$ be a v-stack. Following \cite[Definition 3.12]{GleasonSpecialization}, we define the \textit{reduction} of $\mathcal{F}$ to be the functor $\mathcal{F}^\mathrm{red} \colon (\affperf)^\mathrm{op} \to \mathsf{Grpd}$ defined by
\begin{align*}
    \mathcal{F}^\mathrm{red}(A) = \Hom(\spd A, \mathcal{F}).
\end{align*}
It follows immediately from Lemma \ref{Lem:Vcover} that $\mathcal{F}^\mathrm{red}$ is a v-stack on $\affperf$, and it is clear from the construction that the assignment $\mathcal{F} \mapsto \mathcal{F}^\mathrm{red}$ commutes with limits. We will implicitly use the following lemma below.

\begin{Lem} \label{Lem:ReductionOfScheme}
  Let $X$ be a scheme over $\zp$. Then the natural map
  \[
    X_{\fp}^\mathrm{perf} \to (X^\diamond)^\mathrm{red}
  \]
  is an isomorphism.
\end{Lem}
\begin{proof}
  For any map $f \colon \spd R \to X^\diamond$, the composition $\spd R \to
  X^\diamond \to \spd \zp$ factors through $\spd \fp \subset \spd \zp$ by
  \cite[Lemma~2.30]{GleasonSpecialization}. Therefore $f$ factors through
 $(X_{\fp}^\mathrm{perf})^\diamond = X_{\fp}^\diamond \subseteq X^\diamond$.
  We now apply \cite[Proposition~18.3.1]{ScholzeWeinsteinBerkeley}.
\end{proof}

\subsection{Witt vector shtukas}
Let $\mathcal{G}$ be a quasi-parahoric group scheme over $\zp$ with generic fiber $G$ and relative identity component $\calgcirc$. In this section we study the relationship between the reduction of the stack of $\mathcal{G}$-shtukas and the stack of Witt vector shtukas defined in \cite{XiaoZhu}. Let us first recall some definitions. 

\label{Sec:Loops} For an object $R$ of $\affperf$ we set
\begin{align}
  D_R=\spec W(R), \qquad D_R^{\ast}=\spec W(R)[1/p],
\end{align}
where $W(R)$ denotes the ring of $p$-typical Witt vectors of $R$. The Frobenius on $R$ induces a Frobenius morphism $\sigma=\sigma_R:W(R) \to W(R)$. For an affine scheme $X$ over $\zp$ we define its loop space $L^+X$ as the functor on $\affperf$ by
\begin{align}
    L^+X(R)=X(D_R).
\end{align}
By \cite[Section 1.1]{Zhu1}, the functor $L^+X$ is representable by an affine perfect scheme. 

Let $R \in \affperf$ and let $\mathcal{P},\mathcal{Q}$ be $\mathcal{G}$-torsors over $D_R$. Recall from \cite[Section 3.1.3]{XiaoZhu} that a \emph{modification} $\beta: \mathcal{P} \dashrightarrow \mathcal{Q}$ is an isomorphism of $G$-torsors
\begin{align}
  \beta: \restr{\mP}{D_R^{\ast}} \to \restr{\mathcal{Q}}{D_R^{\ast}}.
\end{align}
We define the \emph{Witt vector (partial) affine flag variety} $\grgloc$ to be the functor sending $R$ to the set of isomorphism classes of modifications
\begin{align}
  \alpha:\mP \dashrightarrow \mP^0,
\end{align}
where $\mP$ is a $\calg$-torsor over $D_R$, and where $\mP^0$ is the trivial $\calg$-torsor over $D_R$. 

Recall that $\mathrm{Gr}_\mathcal{G}$ denotes the Beilinson--Drinfeld Grassmannian over $\spd \zp$ as in \cite[Definition 20.3.1]{ScholzeWeinsteinBerkeley}. We have the following comparison between $\grgloc$ and $\mathrm{Gr}_\mathcal{G}$.

\begin{Lem} \label{Lem:ReductionGrassmannian}
    There is a natural isomorphism 
    \begin{align}
        \grgloc \to \mathrm{Gr}_{\mathcal{G}}^{\mathrm{red}}.
    \end{align}
\end{Lem}

\begin{proof}
It follows from the definition, see \cite[Section 20.3]{ScholzeWeinsteinBerkeley}, that $\mathrm{Gr}_{\mathcal{G},\fp}$ is isomorphic to $\left(\grgloc\right)^{\diamond}$. But the natural map $\mathrm{Gr}_{\mathcal{G}, \fp}^{\mathrm{red}} \to  \mathrm{Gr}_{\mathcal{G}}^{\mathrm{red}}$ is an isomorphism and then we can use Lemma \ref{Lem:ReductionOfScheme} to conclude.
\end{proof}
\label{Sec:LocalShtukas} Let $\shtgloc$ denote the stack on $\affperf$ sending $R$ to the groupoid of $\calg$-torsors $\mP$ over $\spec W(R)$ together with a modification
\begin{align}
    \beta:\smP \dashrightarrow \mP.
\end{align}
The pair $(\mP, \beta)$ is called a \emph{(Witt vector) $\mathcal{G}$-shtuka} over $\spec R$, and was first defined in \cite{XiaoZhu}, cf.\ \cite{ShenYuZhang}. There is a closely related stack given by $\shtg^{\mathrm{red}}$.

\begin{Lem} \label{Lem:ShtukasLocvsShtukas}
There is a natural isomorphism $\shtgloc \to \shtg^{\mathrm{red}}$ of stacks on $\affperf$.
\end{Lem}

\begin{proof}
  This is \cite[Theorem~2.3.8, Example~2.4.9]{PappasRapoportShtukas}, combined with \cite[Lemma~2.30]{GleasonSpecialization}. Note that \cite[Example~2.4.9]{PappasRapoportShtukas} is only stated for parahoric groups, but the proof works verbatim for quasi-parahoric groups.
\end{proof}

\subsubsection{Shtukas and admissible loci} Let $\mu$ be a $G(\qpbar)$-conjugacy class of cocharacters of $G$ with reflex field $E \subset \qpbar$. We will write $\mathcal{O}_E$ for the ring of integers of $E$ and $k_E$ for its residue field. 

If $(\mathcal{P},\beta)$ is a $\calgcirc$-shtuka over $\spec R$, then any choice of trivialization $\alpha$ for $\mathcal{P}$ will determine a point of the Witt vector affine Grassmannian $\mathrm{Gr}_{\calgcirc}^{\mathrm{W}}$ for $\calgcirc$. We say $(\mathcal{P},\beta)$ is bounded by $\mu$ if this point lies in the admissible locus 
\[ \bigcup_{w \in \admu_{\calgcirc}} \mathrm{Gr}_{\calgcirc,w}^{\mathrm{W}} \subset \mathrm{Gr}_{\calgcirc,k_E}^{\mathrm{W}},\]
see \cite[Section 2.4.3]{PappasRapoportShtukas} for the notation, and see \cite[Remark 3.3.2]{PappasRapoportShtukas} for further discussion of this condition (which is independent of the choice of trivialization). We denote by \[\shtgcirclocmu \subset \shtgcircloc \times_{\spec \fp} \spec k_E\] the substack of Witt vector $\calgcirc$-shtukas bounded by $\mu$. By the discussion in \cite[Section 2.2.16]{vH}, this is a closed substack.

\begin{Lem} \label{Lem:ReductionShtukasMu}
Under the isomorphism $\shtgcircloc \to \shtgcirc^{\mathrm{red}}$ of Lemma \ref{Lem:ShtukasLocvsShtukas}, the substack $\shtgcirclocmu$ is identified with $\shtgcircmu^{\mathrm{red}}$.
\end{Lem}

\begin{proof}
This follows from the discussion in \cite[Section 2.4.3]{PappasRapoportShtukas}, as we will now explain.

The natural map $\shtgcircloc \to \shtgcirc^{\mathrm{red}}$ of Lemma \ref{Lem:ShtukasLocvsShtukas} is given by sending a Witt vector $\calgcirc$-shtuka $(\mP, \beta)$ over $\spec R$, to the $\spd R$-point of $\shtgcirc$ given by sending $\spd(A,A^+) \to \spd R$ to the $\calgcirc$-shtuka obtained from $\mathcal{P} \otimes_{W(R)} W(A^+)$ by pulling back along the map of locally ringed spaces
\begin{align}
    \spa(A,A^+) \bdtimes \zp \to \spa W(A^+) \to \spec W(A^+).
\end{align}
Note that the resulting $\calgcirc$-shtuka has a leg at the trivial untilt $A^+$ of $A^{+}$. We have $(\mP, \beta) \in \shtgcirclocmu(A)$ if and only if $\beta$ is bounded by
\begin{align}
    \bigcup_{w \in \admu_{\calgcirc}} \mathrm{Gr}_{\calgcirc,w}^{\mathrm{W}} \subset \mathrm{Gr}_{\calgcirc,k_E}^{\mathrm{W}},
\end{align}
(see \cite[Section 2.4.3]{PappasRapoportShtukas} for the notation).

By \cite[Theorem 6.16]{AGLR}, this is the same as the reduction of the inclusion
\begin{align}
    \mathbb{M}^{\mathrm{v}}_{\calgcirc,{\mu}} \subset \mathrm{Gr}_{\calgcirc, \spd \mathcal{O}_E}.
\end{align}
Thus the image of $(\mP, \beta)$ under $\shtgcircloc \to \shtgcirc^{\mathrm{red}}$ lies in $\shtgcircmu^{\mathrm{red}}$ if and only if $(\mP, \beta) \in \shtgcirclocmu$.
\end{proof}

\begin{Rem} \label{Rem:ShtukasInversion}
    If $b \in G(\qpbr)$, then there is a $\calgcirc$-shtuka $(\mP^0,b)$, where $\mP^0$ is the trivial $\calgcirc$-torsor and where $b$ is considered as an automorphism of the trivial $G$-torsor over $\spec \qpbr$. The shtuka $(\mP^0, b)$ is bounded by $\mu$ if and only if
    \begin{align}
        b \in \bigcup_{w \in \mathrm{Adm}(\mu^{-1})_{\calgcirc}} \mathrm{Gr}_{\calgcirc,w}^{\mathrm{W}}(\fpbar),
    \end{align}
    see \cite[Remark 4.2.3]{PappasRapoportShtukas}. In particular, the stack that we denote by $\shtgcirclocmu$ is denoted by $\mathrm{Sht}_{\mu^{-1},K}^{\mathrm{loc}}$ in \cite[Definition 4.1.3]{ShenYuZhang} and by $\mathrm{Sht}_{G,K,\mu^{-1}}$ in \cite[Section 2.2.16]{vH} (here the symbol $K$ records the choice of parahoric). The stacks of local shtukas of \cite[Definition 5.2.1]{XiaoZhu} have yet another definition; in loc.\ cit.\ a shtuka is a pair $(\mathcal{P}, \beta)$ with $\beta:\mathcal{P} \dashrightarrow \sigma^{\ast} \mathcal{P}$ a modification.
\end{Rem}

Motivated by Lemma \ref{Lem:ReductionShtukasMu}, we define the substack
$\shtglocmu \subseteq \shtgloc$ to be the reduction of $\shtgmu$.

\subsubsection{Shtukas and Dieudonn\'e modules} Let $R$ be a perfect ring and $Y$ a $p$-divisible group over $\spec R$ of height $h$, and let $\mathbb{D}^{\natural}(Y)$ be the finite projective $W(R)$-module (of rank $h$) from Section \ref{Sec:Dieudonne} which comes equipped with a Frobenius 
\begin{align}
    \phi_{Y}^{\natural}:\phi^{\ast} \mathbb{D}^{\natural}(Y)[1/p] \to \mathbb{D}^{\natural}(Y).
\end{align}
Taking frame bundles, we get a $\mathrm{GL}_h$-shtuka. For $0 \le d \le h$ let $\mu_d$ be the cocharacter of $\mathrm{GL}_{h}$ sending $\lambda \mapsto \operatorname{diag}(\lambda, \dotsc, \lambda, 1, \dotsc, 1)$, where $\lambda$ occurs $d$ times and $1$ occurs $h-d$ times. The following result is due to Gabber.

\begin{Lem} \label{Lem:Dieudonne1}
The functor $Y \mapsto (\mathbb{D}^{\natural}(Y),\phi_{Y}^{\natural})$ induces an equivalence between the groupoid of $p$-divisible groups $Y$ over $\spec R$ of dimension $d$, height $h$ and $\mathrm{Sht}^\mathrm{W}_{\mathrm{GL}_{h}, \mu_d}(R)$, which is functorial in $R$.
\end{Lem}
\begin{proof}
That the functor is an equivalence of categories follows from a result of Gabber, see \cite[Theorem 1.2]{Lau}. The fact that the resulting $\mathrm{GL}_{h}$-shtuka is bounded by $\mu_d$ follows from Lemma \ref{Lem:ReductionShtukasMu} in combination with \cite[Example 2.3.4, Remark 2.3.10, Lemma 2.4.4]{PappasRapoportShtukas}. 
\end{proof}

Recall that a principal quasi-polarization of a $p$-divisible group $Y$ over a perfect ring $R$ is an isomorphism $\lambda:Y \to Y^t$, where $Y^t$ is the Serre-dual $p$-divisible group, such that $\lambda^t=-\lambda$. An isomorphism of quasi-polarized $p$-divisible groups $f:(Y_1,\lambda_1) \to (Y_2, \lambda_2)$ is an isomorphism $f:Y_1 \to Y_2$ such that there is an element $c \in \mathbb{Z}_p^{\times}$ such that $f^{\ast} \lambda_2 = c \lambda_1$. 

We note that $\lambda$ induces a perfect alternating pairing $\lambda^{\natural}_Y$ on $\mathbb{D}^{\natural}(Y)$. Indeed, this follows from the explicit description of $\mathbb{D}(Y^t)$, see \cite[Proposition 4.6.9]{AnschuetzLeBras}. If $Y$ has height $h=2g$, then this means that the $\mathrm{GL}_h$-shtuka corresponding to $\mathbb{D}^{\natural}(Y)$ naturally upgrades to a $\mathrm{GSp}_{2g}$-shtuka. Note furthermore that for $h=2g$ and $d=g$, the cocharacter $\mu_g$ factors through $\mathrm{GSp}_{2g} \subset \mathrm{GL}_{2g}$. 

\begin{Lem} \label{Lem:DieudonneTheory}
The functor $Y \mapsto (\mathbb{D}^{\natural}(Y), \lambda^{\natural}_Y)$ induces an equivalence between the groupoid of quasi-polarized $p$-divisible groups of height $2g$ over $\spec R$ and the groupoid $\mathrm{Sht}^\mathrm{W}_{\mathrm{GSp}_{2g},\mu_g}(R)$, functorially in $R$. 
\end{Lem}
\begin{proof}
This is a direct consequence of Lemma \ref{Lem:Dieudonne1} and the description of $\mathbb{D}(Y^t)$ in \cite[Proposition 4.6.9]{AnschuetzLeBras}. 
\end{proof}

\subsection{v-local existence of lattices}

We continue to assume $G$ is a reductive group over $\qp$ with quasi-parahoric model $\mathcal{G}$ over $\zp$. 

\label{Sec:Isocrystals} Let $\gisoc$ denote the stack on $\affperf$ sending $R$ to the groupoid of (\'etale) $G$-torsors $P$ over $\spec W(R)[1/p]$ together with an isomorphism
\begin{align}
    \beta:\sigma^{\ast}P \to P.
\end{align}
This is a v-stack, see \cite[Proposition 6.3]{GleasonIvanovZillinger},\footnote{In \cite{GleasonIvanovZillinger}, the stack we call $\gisoc$ is denoted by $\mathfrak{B}(G)$. They also have an object denoted by $\operatorname{Isoc}_{G}$, which corresponds to $(\gisoc)^{\lozenge}$ in our notation.} and there is a homeomorphism $\lvert\gisoc\rvert \simeq B(G)$, where now $B(G)$ is equipped with the order topology for the partial order defined in \cite[Section 2.3]{RapoportRichartz}. We define the closed substack $\gisocmu \subset \gisoc$ to be the closed substack corresponding to $\bgmu \subset B(G)$.

By \cite[Theorem 1.11]{GleasonIvanovZillinger}, there is a natural isomorphism
\begin{align}
    \gisoc \xrightarrow{\sim} \bun_G^{\mathrm{red}}.
\end{align}
Moreover, it follows from the arguments in \cite[Section 7.2]{GleasonIvanovZillinger} that under this identification $\bungmu^{\mathrm{red}}$ is identified with $\gisocmu$.

There is a natural morphism $\shtgloc \to \gisoc$ sending $(\mP, \beta)$ to $(\restr{\mP}{W(R)[1/p]}, \beta)$. We can identify this with the reduction of the natural morphism $\shtg \to \bun_G$. We define $\shtglocmuone \subset \shtglocmu$ to be the reduction of the open and closed substack $\shtgmuone \subset \shtgmu$ of \cite[Section 3.3.7]{Companion}. 

It follows from \cite[Proposition 3.1.10]{Companion} that $\shtglocmuone \to \shtgloc \to \gisoc$ factors through $\gisocmu \subset \gisoc$. Our goal in the remainder of this section is to prove the following proposition. 

\begin{Prop} \label{Prop:VSurjective}
If $S$ is the spectrum of a product of absolutely integrally closed valuation rings, then the map $\shtglocmuone(S) \to \gisocmu(S)$ is essentially surjective. In particular, the map $\shtglocmuone \to \gisocmu$ is v-surjective.
\end{Prop}

Since the natural map $\shtgcircmu \to \shtgmu$ factors through $\shtgmuone$, see \cite[Section~3.3.7]{Companion}, it suffices after reduction to show that $\shtgcirclocmu(S) \to \gisocmu(S)$ is essentially surjective. We therefore assume for the remainder of the section that $\mathcal{G}=\calgcirc$, so that $\shtglocmuone=\shtglocmu$. 

We start with some preliminaries. For any perfect ring $R$ and element $b \in G(W(R)[1/p])$, there is a morphism
\begin{align}
	\Psi_b:\mathrm{Gr}_{\mathcal{G},R}^{\mathrm{W}}\to \mathrm{Sht}_{\mathcal{G},R}^\mathrm{W}
\end{align}
sending a modification $\alpha:\mP^0 \dashrightarrow \mP_1$ over $\spec R'$ to the (Witt vector) $\mathcal{G}$-shtuka $(\mP_1, \beta_1)$, where $\beta_1$ is defined to be the modification $\beta_1$ which completes the diagram
\begin{equation}
	\begin{tikzcd}
		\smP^0 \arrow[r, dashrightarrow, "b"] \arrow[d, dashrightarrow, "\sigma^{\ast}\alpha"] & \mP^0 \arrow[d, dashrightarrow, "\alpha"] \\
		\smP_1 \arrow[r, dashrightarrow, "\beta_1"] & \mP_1.
	\end{tikzcd}
\end{equation}
We denote by $Z(b)_{{\mu}}$ the inverse image of $\mathrm{Sht}_{\mathcal{G},{\mu},R}^\mathrm{W}$ in $\mathrm{Gr}_{\mathcal{G},R}^\mathrm{W}$ under $\Psi_b$. Since $\mathrm{Sht}_{\mathcal{G},{\mu},R}^\mathrm{W} \subset \mathrm{Sht}_{\mathcal{G},R}^\mathrm{W}$ is closed, we see that $Z(b)_{{\mu}} \subset \mathrm{Gr}_{\mathcal{G},R}^{\mathrm{W}}$ is closed. Therefore $Z(b)_{{\mu}}$ is an inductive limit of perfections of projective $R$-schemes along closed embeddings by \cite[Corollary~9.6]{BhattScholze}.

In the lemma below, following our convention \eqref{Eq:NewtonStrat}, we use the notation $\shtgmub \subset \shtgmu$ and $\gisoc^{[b]} \subset \gisoc$ for the locally closed Newton stratum corresponding to $[b] \in B(G)$, see also \cite[Theorem I.2.1]{FarguesScholze}. Recall also that an affine scheme $X=\spec R$ is called w-contractible if every pro-\'etale cover of $X$ splits, see \cite[Definition 1.4]{BhattScholzeProEtale}. 

\begin{Lem} \label{Lem:SurjectiveIII}
  Let $R \in \affperf$ and let $b \in G(W(R)[1/p])$ be an arbitrary element.
  Then there exists a closed subscheme $Z \subseteq Z(b)_{\mu}$ such that the
  projection $Z \to \spec R$ is surjective.
\end{Lem}

In the proof of this lemma, we will make use of the affine Deligne--Lusztig variety $X_{\mathcal{G}}(b',{\mu^{-1}})$ of level $K=\mathcal{G}(\zp)$ associated to a representative $b' \in [b']$ and $\mu^{-1}$, see \cite[Definition 3.3.1]{PappasRapoportShtukas}. Its set of $\fpbar$ points can be identified with the subset of elements $h \in G(\qpbr)/G(\zpbr)$ such that the $\mathcal{G}$-shtuka $(\mP^0, h b' \sigma(h)^{-1}) \in \shtgloc(\fpbar)$ lies in $\shtglocmu(\fpbar)$. By \cite[Theorem A]{He2}, it is nonempty if and only if $[b'] \in \bgmu$. 

\begin{proof}[Proof of Lemma \ref{Lem:SurjectiveIII}]
  We first note that we are free to base change along any surjective map $\spec
  R^\prime \to \spec R$. Indeed, if there exists a closed subscheme $Z^\prime
  \subseteq Z(b)_\mu \times_{\spec R} \spec R^\prime$ with surjective projection
  map $Z^\prime \to \spec R^\prime$, then the map $Z^\prime \to Z(b)_\mu$
  factors through a closed subscheme $Z \subseteq Z(b)_\mu$, and it follows that
  $Z \to \spec R$ is surjective.

  The given element $b$ defines a map $\spec R \to \gisoc$. Using the above, we
  first reduce to the case when this map factors through a single Newton stratum
  $\gisoc^{[b^\prime]}$. To do this, we observe that $\spec R$ has only
  finitely many Newton strata, and moreover the strata $(\spec R)^{[b']}$ are
  quasi-compact schemes because they are open complements of finitely presented
  closed immersions inside closed subschemes of $\spec R$, see
  \cite[Theorem~3.6]{RapoportRichartz}.\footnote{Note that
  \cite[Theorem~3.6]{RapoportRichartz} assumes that $\spec R$ is connected for the
  conclusion of that theorem to hold. However, this is not necessary as the
  result of Katz they cite, see \cite[Theorem~2.3.1]{KatzSlopeFiltrations}, does
  not require $\spec R$ to be connected.} We also reduce to the case when $R$ is
  w-contractible, e.g., using \cite[Lemma~2.4.9]{BhattScholzeProEtale}.

  We will now show that $Z(b)_\mu \to \spec R$ has a section. Let $b^\prime \in
  G(\qpbr)$ be a representative of $[b'] \in B(G)$ and let
  $X_{\mathcal{G}}(b',{\mu^{-1}})$ be the affine Deligne--Lusztig variety as above. By
  \cite[Theorem~I.2.1]{FarguesScholze}, we can identify 
  \begin{align*}
    \gisoc^{[b']} \xrightarrow{\sim} \left[\underline{G_{b'}(\qp)}\backslash
    \spec \fpbar \right]
  \end{align*}
  over $\fpbar$. Since $\spec R$ has no nonsplit pro-\'etale covers, all
  $\underline{G_{b'}(\qp)}$-torsors over $\spec R$ are trivial (see the proof of
  \cite[Lemma~III.2.6]{FarguesScholze}). It follows that $\gisoc^{[b']}(R)$
  consists of a single isomorphism class, namely the one corresponding to $b'$.
  By assumption, $b \in G(W(R)[1/p])$ determines an element of
  $\gisoc^{[b']}(R)$, so $b' =gb\sigma(g)^{-1}$ for some $g$ in $G(W(R)[1/p])$.
    
  Since $[b'] \in \bgmu$, \cite[Theorem~1.1]{He2} implies that
  $X_{\mathcal{G}}(b',{\mu^{-1}})$ is nonempty. Let $h \in
  X_{\mathcal{G}}(b',{\mu^{-1}})(\fpbar)$, and let $\alpha: \mP^0 \dashrightarrow
  \mP^0$ be the modification defined by $gh$. By definition $\Psi_b(\alpha)$ is
  the $\mathcal{G}$-shtuka $(\mathcal{P}^0, hb'\sigma(h)^{-1})$, which lies in
  $\mathrm{Sht}_{\mathcal{G},{\mu}}^\mathrm{W}(R)$ since $h \in
  X_{\mathcal{G}}(b',{\mu^{-1}})(\fpbar)$. Thus $\alpha$ lies in
  $Z(b)_{{\mu}}(R)$, and we are done.
\end{proof}

\begin{proof}[Proof of Proposition~\ref{Prop:VSurjective}]
  As remarked before, it suffices to prove the case when $\calgcirc=\calg$ and
  thus $\shtglocmuone=\shtglocmu$. 

  Let $R$ be the product of absolutely integrally closed valuation rings $R_i$
  and let $\spec R \to  \gisocmu$ be a morphism corresponding to the
  $G$-isocrystal $(P,\beta_{P}) \in \gisoc(R)$. By
  \cite[Proposition~11.5]{AnschuetzExtension} the $G$-torsor $P$ on $\spec
  W(R)[1/p]$ is trivial. Choose a trivialization, and let $b \in
  LG(R)=G(W(R)[1/p])$ be the element corresponding to $\beta_P$. 

  For a morphism $f \colon \spec R \to \mathrm{Gr}_{\mathcal{G},R}^{\mathrm{W}}$
  over $\spec R$, the composition
  \begin{align}
    \spec R \xrightarrow{f} \mathrm{Gr}_{\mathcal{G},R}^{\mathrm{W}}
    \xrightarrow{\Psi_b} \mathrm{Sht}_{\mathcal{G},R}^\mathrm{W} \to \gisoc_R
  \end{align}
  determines a $G$-isocrystal which is isomorphic to $(P, \beta_{P})$. Our goal
  is to show that there is a section $\spec R \to
  \mathrm{Gr}_{\mathcal{G},R}^{\mathrm{W}}$ whose image under $\Psi_b$ lies in
  $\mathrm{Sht}_{\mathcal{G},{\mu},R}^\mathrm{W}$. In other words, we want to construct a
  section of $Z(b)_{{\mu}} \to \spec R$.

  By Lemma~\ref{Lem:SurjectiveIII}, there exists a closed subscheme $Z \subset
  Z(b)_{\mu}$ for which $Z \to \spec R$ is surjective. On the other hand, $Z \to
  \spec R$ is the perfection of a projective morphism, and it now follows from
  Lemma~\ref{Lem:ExistenceSection} that the map admits a section. The
  v-surjectivity of $\shtglocmu \to \gisocmu$ immediately follows from
  Lemma~\ref{Lem:SchemeTheoreticProductOfPoints}.
\end{proof}

\section{The Igusa stack conjecture} \label{Sec:Conjecture} In this section, we recall the conjectural canonical integral models for Shimura varieties of parahoric level of \cite{PappasRapoportShtukas}, and the quasi-parahoric generalization of \cite{Companion}. We then state a conjecture on the existence of Igusa stacks, and deduce consequences. 

\subsection{Canonical integral models for Shimura varieties} Let $\gx$ be a Shimura datum with reflex field $\mathsf{E}$, let $p$ be a prime and write $G=\mathsf{G}_{\qp}$. Let $\mathcal{G}$ be a quasi-parahoric model of $G$ over $\zp$, and let $K_p = \mathcal{G}(\zp)$. Choose a prime $v$ of $\mathsf{E}$ above $p$, and let $E$ denote the completion of $\mathsf{E}$ at $v$. The local reflex field $E$ is also the field of definition of $\mu$, the $G(\qpbar)$-conjugacy class of cocharacters of $G$ corresponding to $\mathsf{X}$ and $v$, see Section \ref{Sec:Conventions}. We will write $\mathcal{O}_E$ for the ring of integers of $E$ and $k_E$ for its residue field. When $K^p \subset \gafp$ is a neat compact open subgroup, associated to $\gx$ and $K=K_pK^p$ is the Shimura variety $\mathbf{Sh}_K\gx$, which we view as an $E$-scheme (i.e., we take the base change to $E$ of the canonical model over $\mathsf{E}$). 

We will often consider Shimura varieties at infinite level. That is, we consider
\begin{equation}\label{Eq:SVTower}
    \mathbf{Sh}_{K^p}\gx = \varprojlim_{K_p' \subset K_p} \mathbf{Sh}_{K_p'K^p}\gx
\end{equation}
as $K_p'\subset K_p$ varies over all compact open subgroups of the fixed $K_p$, as well as
\begin{equation}
    \mathbf{Sh}_{K_p}\gx = \varprojlim_{K^p\subset \gafp} \mathbf{Sh}_{K_pK^p}\gx
\end{equation}
as $K^p$ varies over all neat (see \cite[Definition 1.4.1.8]{LanThesis}) compact open subgroups $K^p \subset \gafp$.

Let $\mathsf{Z}^\circ$ denote the connected component of the center of $\mathsf{G}$. We will assume that $\gx$ satisfies
\begin{equation}\label{Eq:SV5}
    \operatorname{rank}_\mathbb{Q}(\mathsf{Z}^\circ) = \operatorname{rank}_\mathbb{R}(\mathsf{Z}^\circ).
\end{equation}
This equality is equivalent to Milne's axiom SV5 \cite[p.63]{Milne} by \cite[Lemma 1.5.5]{KisinShinZhu}.
\begin{Rem}
    By \cite[Lemma 5.1.2.(i)]{KisinShinZhu}, the assumption \eqref{Eq:SV5} is satisfied whenever $\gx$ is of Hodge type, which will be the main case of interest to us.
\end{Rem}
Fix a neat compact open subgroup $K^p \subset \gafp$, and let $K = K_pK^p$. Each finite level Shimura variety $\mathbf{Sh}_{K_p'K^p}\gx$ is a smooth variety over $E$, and the transition maps in the tower \eqref{Eq:SVTower} are finite \'etale. We denote by $\mathbb{P}_K$ the pro-\'etale $\mathcal{G}(\zp)$-cover 
\begin{align}
    \mathbf{Sh}_{K^p}\gx \to \mathbf{Sh}_K\gx.
\end{align}
Let $\mathbf{Sh}_K\gx^{\lozenge}$ denote the v-sheaf over $\spd E$ associated to $\mathbf{Sh}_K\gx$ as in Section \ref{Sub:AdicSpaces}. By \cite[Corollary 3.3.9]{Companion}, cf.\ \cite[Proposition 4.1.2]{PappasRapoportShtukas}, there is a $\mathcal{G}$-shtuka with one leg $\mathscr{P}_{K,E}$ on $\mathbf{Sh}_K\gx^\lozenge$ which is bounded by $\mu$, and is associated to $\mathbb{P}_K$ in the sense of \cite[Definition 2.6.6]{PappasRapoportShtukas}.

Let $K^p$ be a neat compact open subgroup of $\gafp$, and let $K = K_pK^p$. Pappas and Rapoport conjecture that for each $K^p$ there is a flat normal integral model $\scrs_K\gx$ over $\spec \mathcal{O}_E$ of $\mathbf{Sh}_K\gx$, satisfying certain properties, see \cite[Conjecture~4.2.2]{PappasRapoportShtukas} for the conjecture in the case where $\mathcal{G}$ is a stabilizer Bruhat--Tits group scheme, and see \cite[Conjecture~4.1.4]{Companion} for the case where $\mathcal{G}$ is quasi-parahoric. They then show, see \cite[Theorem~4.2.4]{PappasRapoportShtukas} and \cite[Corollary~4.1.13]{Companion}, that there is at most one system of normal flat models satisfying their list of properties. A consequence of their list of properties is the existence of a map 
\begin{equation}\label{Eq:picrys}
    \pi_\crys:\scrs_K\gx^\diamond \to \shtgmu,
\end{equation}
compatible with changing $K^p$, which induces a $\gafp$-equivariant map $$\pi_{\crys}:\scrsdinf=\varprojlim_{K^p \subset \gafp} \scrs_K\gx^\diamond \to \shtgmu.$$ In fact, it follows from \cite[Remark 4.1.5]{Companion} that both these maps factor through the open and closed substack
\begin{align}
    \shtgmuone \subset \shtgmu,
\end{align}
see \cite[Section 3.3.7]{Companion}. 

By \cite[Theorem 4.2.3]{Companion}, cf.\ \cite[Theorem 4.5.2]{PappasRapoportShtukas}, a system of canonical integral models exists in the case of a Hodge-type Shimura datum. The conjecture is also known to hold if $\gx$ is of toral type (i.e., if $\mathsf{G} = \mathsf{T}$ is a torus) by \cite[Theorem A]{Daniels}. It is moreover known to hold when $\gx$ is of abelian type and $p>2$ by \cite[Theorem A]{DanielsYoucis}.\footnote{In fact, a version of the conjecture holds even when SV5 is not satisfied; see \cite[Conjecture 4.5]{Daniels} for the statement in this generality.}

\subsection{Igusa stacks} 
The following conjecture is the version of \cite[Conjecture 1.1.(4)]{ZhangThesis} for the good reduction locus, although we make the additional assumption that $\gx$ satisfies the equality of \eqref{Eq:SV5} and we allow quasi-parahoric models $\mathcal{G}$ of $G$. 
\begin{Conj} \label{Conj:IgusaMain}
  There is a small v-sheaf $\igsinf$ over $\spd \mathbb{F}_p$ equipped with an
  action of $\gafp$ and a $\gafp$-invariant map 
  \begin{align*}
      \overline{\pi}_\mathrm{HT} \colon \igsinf \to \bungmu.
  \end{align*}
  Moreover for every quasi-parahoric model $\calg$ of $G$ with $K_p=\calg(\zp)$, there is a $\gafp$-equivariant
  map 
  \begin{align*}
      \IgsQuot \colon \scrsdinf \to \igsinf
  \end{align*}
  such that the following diagram is
  $2$-commutative and $2$-Cartesian
  \begin{equation} \begin{tikzcd} \label{Eq:ConjectureCartesianDiagram}
    \scrsdinf \arrow{r}{\pi_{\mathrm{crys}}} \arrow{d}{\IgsQuot} & \shtgmuone \arrow{d}{\mathrm{BL}^{\circ}} \\
    \igsinf \arrow{r}{\overline{\pi}_\mathrm{HT}} & \bungmu.
  \end{tikzcd} \end{equation}
    Furthermore, the formation of $\igsinf$ is functorial in morphisms of
    Shimura data and compatible with the morphism $\overline{\pi}_\mathrm{HT}$ and the
    $2$-Cartesian diagram above.
\end{Conj}

\begin{Rem}
  We emphasize that the v-sheaf $\igsinf$ is defined over $\fp$ and not $k_E$.
  In fact, the bottom row of the diagram is defined over $\spd \fp$ while the
  top row is defined over $\spd \mathcal{O}_E$.
\end{Rem}

\begin{Rem}
    Taking the generic fiber of $\mathscr{S}_{K_p}\gx^\diamond$, we obtain the good reduction locus of the diamond $\mathbf{Sh}_{K_p}\gx^\lozenge$ attached to the generic fiber of the Shimura variety, while Scholze's original conjecture is above uniformizing the whole Shimura variety. In this sense, the Igusa stack in the conjecture above is only the good reduction locus of a larger Igusa stack. It is denoted $\operatorname{Igs}^\circ\gx$ in \cite{ZhangThesis}. Since we do not discuss other versions of Igusa stacks in this paper, we will suppress the superscript ${}^\circ$ from notation. In fact, it is conjectured more generally that there is a ``big'' Igusa stack living over the whole $\bung$ instead of just the locus bounded by $\mu$. However, for the time being we can only make use of the Shimura variety, so we are content with the statement of the current conjecture.
\end{Rem}

\subsection{Perfect Igusa varieties} \label{Sec:PerfectIgusa}

Let $\gx$ be a Shimura datum satisfying \eqref{Eq:SV5} and let $K_p \subseteq G(\qp)$ be a parahoric subgroup with integral model $\mathcal{G} = \mathcal{G}^\circ$. We will assume in the rest of this section that there is a stack $\igsinf$ sitting in a $2$-Cartesian diagram as in \eqref{Eq:ConjectureCartesianDiagram}. We let $\igsinf^{\mathrm{red}}$ be its reduction. 
\subsubsection{Igusa varieties} \label{Sec:ReductionDiagram}
We use $\shginf$ to denote the perfect special fiber of $\scrsginf$, which is
naturally identified with $(\scrsdinf_{k_E})^{\mathrm{red}}=(\scrsdinf)^{\mathrm{red}}$ (see Lemma \ref{Lem:ReductionOfScheme}). The reduction of the morphism $\pi_{\mathrm{crys}}:\scrsdinf \to \shtgmu$ gives rise to a morphism
$\pi_{\mathrm{crys}}^{\mathrm{red}}:\shginf \to \shtglocmu$, see Lemma~\ref{Lem:ReductionShtukasMu}. Its
composition with the reduction of $\mathrm{BL}^{\circ}:\shtgmu \to \bung$ gives rise to a
morphism $\shginf \to \gisoc$, see Section~\ref{Sec:Isocrystals}. Since reduction is a right adjoint (see \cite[Definition 3.12]{GleasonSpecialization}), it commutes with limits. Therefore, we get a Cartesian diagram
\begin{equation} \label{Eq:ReductionDiagram}
\begin{tikzcd}
        \shginf\arrow{r}{\pi_{\mathrm{crys}}^{\mathrm{red}}} \arrow{d}{\IgsQuot^\mathrm{red}} & \shtglocmu \arrow{d}{\mathrm{BL}^{\circ, \mathrm{red}}} \\
        \igsinf^{\mathrm{red}} \arrow{r}{\overline{\pi}_\mathrm{HT}^{\mathrm{red}}} & \gisocmu.    
\end{tikzcd}
\end{equation}
\label{Sec:DefPerfectIgusa}
Let $b:\spec \fpbar \to \gisocmu$ be a morphism. Then we define the perfect Igusa variety $\iginfbgx$ to be the fiber product
\[ \begin{tikzcd}
  \iginfbgx \arrow{r} \arrow{d} & \igsinf^{\mathrm{red}} \arrow{d}{\overline{\pi}_\mathrm{HT}^{\mathrm{red}}} \\ \spec \fpbar \arrow{r}{b} &
  \gisocmu.
\end{tikzcd} \]
If we modify $b$ by a $\sigma$-conjugate
\[
  g^{-1} b \sigma(g) = b^\prime,
\]
where $g \in G(\qpbr)$, then $g$ induces a $2$-isomorphism between the two maps
$b, b^\prime \colon \spec \fpbar \to \gisocmu$. It therefore induces an isomorphism
\[
   \iginfbgx \xrightarrow{\sim} \mathrm{Ig}^{b^\prime}\gx.
\]
In particular, this induces an action of the $\sigma$-centralizer $G_b(\qp) \subset G(\qpbr)$ of $b$ on $\iginfbgx$. 

After possibly replacing $b$ by a $\sigma$-conjugate, we may arrange using Proposition~\ref{Prop:VSurjective} that there is a lift of $b \colon \spec \fpbar \to \gisocmu$ to $b:\spec \fpbar \to \shtglocmu$. Then from the definition of $\iginfbgx$ together with the $2$-Cartesian diagram in \eqref{Eq:ReductionDiagram}, we have a $2$-Cartesian square
\begin{equation}\label{Eq:IgbDiagram} 
  \begin{tikzcd}
    \iginfbgx  \arrow{d} \arrow{r} & \spec \fpbar \arrow{d}{b}
    \\ \shginf \arrow{r}{\pi_{\mathrm{crys}}^{\mathrm{red}}} & \shtglocmu.
  \end{tikzcd} 
  \end{equation}
It follows from this 2-Cartesian diagram that $\iginfbgx$ agrees with the perfect Igusa cover corresponding to $\pi_\mathrm{crys}^\mathrm{red}$ defined in \cite[Section 2.14]{HamacherKimPointCounting}.

\begin{Lem} \label{Lem:RepresentabilityIgusa}
    The perfect Igusa variety $\iginfbgx$ is representable by a perfect scheme.
\end{Lem}
\begin{proof}
 This follows from \cite[Proposition 2.15.(1)]{HamacherKimPointCounting}, which says that $\iginfbgx \to \shginf$ is representable by an affine morphism of perfect schemes. 
\end{proof}

\subsubsection{Central Leaves}\label{Sec:CentralLeaves} Let $b: \spec \fpbar \to \shtglocmu$ be as above. By \eqref{Eq:IgbDiagram} and \cite[Proposition 2.15.(3)]{HamacherKimPointCounting}, there is a locally closed \emph{central leaf} $C^{b}_{K_p} \subset \shginf$ such that $C^{b}_{K_p}$ is closed in the Newton stratum $\shginf \times_{\gisoc} \gisoc^{[b]}$, and such that the morphism $\iginfbgx \to \shginf$ constructed in \eqref{Eq:IgbDiagram} factors through a morphism $\iginfbgx \to C^{b}_{K_p}$. The $\fpbar$-points of $C^{b}_{K_p}$ can be characterized as the subset of $\shginf(\fpbar)$ consisting of those $x$ such that $x \to \shginf \to \shtglocmu$ is isomorphic to $b$. By \cite[Lemma 2.17]{HamacherKimPointCounting}, the morphism $\iginfbgx \to C^{b}_{K_p}$ is a pro-\'etale torsor for the profinite group $\Gamma_b = G_b(\qp) \cap \mathcal{G}(\zpbr)$.

Since the morphism $\shginf \to \shtglocmu$ is $\gafp$-equivariant, we find that $C^{b}_{K_p} \subset \shginf$ is $\gafp$-stable and that $\gafp$ acts on $\iginfbgx$ such that the morphism $\iginfbgx \to C^{b}_{K_p}$ is $\gafp$-equivariant. Thus we see that for each $K^p \subset \gafp$ there is an Igusa variety $\mathrm{Ig}^{b}_{K^p}\gx \to C^{b}_K \subset \shg$. Note that $C^{b}_K$ is the perfection of a finite type scheme.

\begin{Cor} \label{Cor:InverseLimitOfFiniteType}
    The Igusa variety $\iginfbgx$ is an inverse limit of finite \'etale covers of $C^{b}_K$. 
\end{Cor}
\begin{proof}
Since the $G_b(\qp)$-action on $\iginfbgx$ commutes with the $\gafp$-action, we see that $\iginfbgx \to C^{b}_K$ is a pro-\'etale torsor for the profinite group $\Gamma_b \times K^p$, and the result follows.
\end{proof}

\subsection{Igusa varieties in the perfectoid setting} \label{Sub:VSheafIgusa} Let the notation and assumptions be as in Section \ref{Sec:PerfectIgusa}. For $b \in G(\qpbr)$ with $[b] \in \bgmu$, there is a corresponding map
\[
  \spd \fpbar \to \bun_G,
\]
see \cite[Theorem 5.3]{AnschutzIsocrystal}. We then define the v-sheaf Igusa variety $\igvinfbgx$
as the fiber product
\[ \begin{tikzcd}
  \igvinfbgx \arrow{r} \arrow{d} & \igsinf \arrow{d}{\overline{\pi}_\mathrm{HT}} \\ \spd \fpbar \arrow{r}{b} &
  \bun_G.
\end{tikzcd} \]
Since $b:\spd \fpbar \to \bun_G$ factors through $\bun_{G}^{[b]} \to \bun_G$ via a $\tilde{G}_b$-torsor
\begin{align}
    b:\spd \fpbar \to \bun_{G}^{[b]},
\end{align}
 we see that $ \igvinfbgx \to \igsinf^{[b]}$ is a $\tilde{G}_b$-torsor. Note that the $\gafp$-action on $\igsinf$ induces an
$\gafp$-action on $\igvinfbgx$, which commutes with the $\tilde{G}_b$-action since $\igsinf \to \bun_G$ is $\gafp$-equivariant. As in Section \ref{Sec:DefPerfectIgusa}, if $g^{-1} b \sigma(g) = b^\prime$, then there is a $\gafp$-equivariant isomorphism
\[
  \igvinfbgx \xrightarrow{\sim} \mathrm{Ig}^{b^\prime,\mathrm{v}}\gx.
\]
By Proposition~\ref{Prop:VSurjective} we may, after replacing $b$ by a $\sigma$-conjugate, lift the point $b:\spd \fpbar \to \bungmu$ to a point $b:\spd \fpbar \to \shtgmu$. Then from the definition of $\igvinfbgx$ together with the $2$-Cartesian diagram in \eqref{Eq:ConjectureCartesianDiagram}, we have a $2$-Cartesian square
\begin{equation} \begin{tikzcd} \label{Eq:CartesianVSheaf}
    \igvinfbgx  \arrow{d} \arrow{r} & \spd \fpbar \arrow{d}{b}
    \\ \scrsdinf \arrow{r}{\pi_{\mathrm{crys}}} & \shtgmu.
\end{tikzcd} \end{equation}
Recall from \cite[Section 3.2]{GleasonSpecialization} that there is a natural map $\shginfd \to \scrsdinf$.

\begin{Lem} \label{Lem:CanonicalCompactification}
  Let $X$ be a perfect scheme together with a morphism $f \colon X \to
  \shtgloc$, which, by the adjunction between reduction and diamond functors and Lemma~\ref{Lem:ShtukasLocvsShtukas}, corresponds to $f^\diamond \colon X^\diamond \to \shtg$. Choose an element $b \in G(\qpbr)$, which induces maps $\spec \fpbar \to \shtgloc$ and $\spd \fpbar \to \shtg$ as in Remark \ref{Rem:ShtukasInversion}, and
  consider the fiber products
  \[ \begin{tikzcd}
    Y \arrow{r} \arrow{d} & \spec \fpbar \arrow{d}{b} \\ X \arrow{r}{f} &
    \shtgloc,
  \end{tikzcd} \quad \begin{tikzcd}
    \mathscr{Y} \arrow{r} \arrow{d} & \spd \fpbar \arrow{d}{b} \\ X^\diamond
    \arrow{r}{f^\diamond} & \shtg.
  \end{tikzcd} \]
  Then $\mathscr{Y}$ is the canonical compactification, in the sense of
  \cite[Proposition~18.6]{EtCohDiam}, of the map $Y^\diamond \to X^\diamond$.
\end{Lem}

Here, recall that \cite[Proposition~2.15.(1)]{HamacherKimPointCounting} implies
that $Y$ is representable by a perfect scheme and $Y \to X$ is affine.

\begin{proof}
  Since the question is local in $X$, we may as well assume that $X = \spec A$
  is affine. Let us write $(\mP, \phi_\mP)$ for the $\mathcal{G}$-shtuka on $X$
  corresponding to $f$, where $\mP$ is a $\mathcal{G}$-torsor on $W(A)$. For
  each $S \in \affperf$, an $S$-point of $Y$ is the data of an $S$-point $x
  \colon \spec S \to X_{\fpbar}$ together with a trivialization $x^\ast \mP
  \xrightarrow{\sim} \mathcal{G}_{W(S)}$ under which $\phi_{\mP}$ on the
  left hand side agrees with $b$ on the right hand side.

  On the other hand, for each $(R, R^+) \in \perf$, an $(R, R^+)$-point of
  $\mathscr{Y}$ is the data of an $R^+$-point $x \colon \spec R^+ \to
  X_{\fpbar}$ together with a trivialization
  \[
    x^\ast \mP \vert_{\mathcal{Y}_{[0,\infty)}(R, R^+)} \xrightarrow{\sim}
    \mathcal{G} \times \mathcal{Y}_{[0,\infty)}(R, R^+)
  \]
  under which $\phi_\mP$ corresponds to $b$. The category of
  $\mathcal{G}$-torsors on $\mathcal{Y}_{[0,\infty)}$ together with a
  meromorphic action of $\phi$ does not change when we replace $R^+$ with
  $R^\circ$; this follows from applying the Tannakian formalism to
  \cite[Proposition~2.1.1]{PappasRapoportShtukas}. Therefore we may instead
  parametrize isomorphisms
  \[
    x^\ast \mP \vert_{\mathcal{Y}_{[0,\infty)}(R, R^\circ)} \xrightarrow{\sim}
    \mathcal{G} \times \mathcal{Y}_{[0,\infty)}(R, R^\circ)
  \]
  intertwining $\phi_\mP$ and $b$.

  We also know from \cite[Proposition~2.2.7]{PappasRapoportShtukas} that the
  restriction from $\spec W(R^\circ)$ to $\mathcal{Y}_{[0,\infty)}(R, R^\circ)$
  induces a fully faithful functor from Breuil--Kisin--Fargues modules for
  $R^\circ$ to shtukas for $(R, R^\circ)$. Hence an isomorphism between the two
  shtukas uniquely lifts to an isomorphism
  \[
    x^\ast \mathscr{P} \vert_{\spec W(R^\circ)} \xrightarrow{\sim} \mathcal{G}
    \times \spec W(R^\circ).
  \]
  This shows that $(R, R^+)$-points of $\mathscr{Y}$ correspond to diagrams of
  the form
  \[ \begin{tikzcd}
    \spec R^\circ \arrow{r} \arrow{d} & Y \arrow{d} \\ \spec R^+ \arrow{r} &
    X_{\fpbar},
  \end{tikzcd} \]
  which means that $\mathscr{Y}$ is the canonical compactification of
  $Y^\diamond \to X_{\fpbar}^\diamond$. This is also the same as the canonical
  compactification of $Y^\diamond \to X^\diamond$, as $X_{\fpbar} \to X$ is
  integral.
\end{proof}

\begin{Cor} \label{Cor:CanonicalCompactification}
  The map $\igvinfbgx \to \scrsdinf$ factors through $\shginfd$.
  Moreover, the induced map $\igvinfbgx \to \shginfd$ is isomorphic to
  the canonical compactification of the map $\iginfbgx^{\diamond} \to
  \shginfd$ from the proof of Lemma~\ref{Lem:RepresentabilityIgusa}. Furthermore, the map $\iginfbgx^{\diamond} \to \shginfd$ is compactifiable, and therefore $\iginfbgx^{\diamond} \to \igvinfbgx$ is an open immersion.
\end{Cor}

\begin{proof}
  Since the bottom row of \eqref{Eq:CartesianVSheaf} lies over $\spd
  \mathcal{O}_E$, we can further regard $\igvinfbgx$ as the fiber product
  \[ \begin{tikzcd}
    \igvinfbgx \arrow{d} \arrow{r} & \spd \fpbar \arrow{d}{b} \\ \shginfd
    \arrow{r} & \shtgmu.
  \end{tikzcd} \]
  We now apply Lemma~\ref{Lem:CanonicalCompactification} to $X = \shginf \to
  \shtgloc$ to conclude.

  It remains to check compactifiability of $\iginfbgx^{\diamond} \to
  \shginfd$. We have a factorization $\iginfbgx \to C_{K_p}^b \to \shginf$, where the first map is pro-(finite \'etale) and the second map is a locally closed immersion. As compactifiable maps are stable under composition \cite[Proposition~22.3.(iv)]{EtCohDiam} and both proper maps and \'etale maps are compactifiable by \cite[Proposition~22.3.(vi)]{EtCohDiam}, we see that $\iginfbgx^{\diamond} \to
  \shginfd$ is compactifiable by decomposing $C_{K_p}^b \to \shginf$ into a composition of an open embedding and a closed embedding. To deduce that $\iginfbgx^{\diamond} \to \igvinfbgx$ is an open immersion, we apply \cite[Proposition 22.3.(i)]{EtCohDiam}.
\end{proof}

We consider the inverse image $\mathbf{Sh}\gx^{\circ, \lozenge}$ under $\mathbf{Sh}\gx^{\lozenge} \to \mathbf{Sh}_{K_p}\gx^{\lozenge}$ of the open subspace $\scrs_{K_p}\gx^{\diamond} \times_{\spd \mathcal{O}_E} \spd E \subset \mathbf{Sh}_{K_p}\gx^{\lozenge}$, and we write $\pi_{\mathrm{HT}}^{\circ}$ for the restriction of $\pi_{\mathrm{HT}}$ to this inverse image.\footnote{The definition of $\mathbf{Sh}\gx^{\circ, \lozenge}$ depends a priori on the choice of parahoric $K_p$. In the case of Hodge type Shimura varieties, we will show that it is in fact independent of $K_p$, see Lemma \ref{Lem:PotentiallyCrystalline}.} Let $C$ be an algebraically closed complete non-archimedean field containing $E$ with ring of integers $\CO_C$, and let $x$ be a $\spd(C,\CO_C)$-point of the flag variety $\mathrm{Gr}_{G,{\mu^{-1}}}$. Then the image of $x$ under the Beauville--Laszlo map lies in $\bun_{G}^{[b]} \subset \bung$ for some $[b] \in \bgmu$, see \cite[Proposition~3.5.3]{CaraianiScholzeCompact}. Choose an element $b \in [b]$ and let $\mathrm{Ig}^{b} = \iginfbgx$ be the perfect Igusa variety over $\fpbar$ associated to $b$. Let $\ig^{b}_{\zpbr}$ be the Witt vector lift of $\ig^{b}$, which is a formal scheme over $\spf \zpbr$. Denote by $\ig^{b}_{C}$ the adic generic fiber of $\ig^{b}_{\zpbr}\times_{\spf \zpbr}\spf \CO_C$. 
\begin{Prop}\label{Prop: HTfiber}
There is a monomorphism $\ig_{C}^{b,\lozenge} \to \pi_{\mathrm{HT}}^{\circ,-1}(x)$ inducing an isomorphism on canonical compactifications. 
\end{Prop}
\begin{proof}
There is a Cartesian diagram
\begin{equation} \label{Eq:CartesianDiagramInfiniteLevel}
\begin{tikzcd}
    \mathbf{Sh}\gx^{\circ, \lozenge} \arrow{r}{\pi_{\mathrm{HT}}^{\circ}} \arrow{d} & \operatorname{Gr}_{G, \mu^{-1}} \arrow{d} \\
    \igsinf \arrow{r} & \bun_G,
\end{tikzcd}
\end{equation}
and it follows from this Cartesian diagram in \eqref{Eq:CartesianDiagramInfiniteLevel} that $\igvinfbgx \times_{\spd \fpbar} \spd(C, \mathcal{O}_C)$ is the fiber of $\pi_{\mathrm{HT}}^{\circ}$ over $x$. We now make the identification \[ \iginfbgx^{\diamond} \times_{\spd \fpbar} \spd (C, \mathcal{O}_C) \simeq \ig_{C}^{b,\lozenge}, \] and the result now follows from Corollary \ref{Cor:CanonicalCompactification}, together with the fact that the formation of the canonical compactification is stable under base change, see \cite[Proposition 18.7]{EtCohDiam}.
\end{proof}

\subsection{A product formula} \label{Sec:ProductFormula}  
In this section we develop an almost product structure for perfectoid Igusa varieties, generalizing \cite[Corollary 11.26]{ZhangThesis}, \cite[Proposition 4.3]{CaraianiScholzeCompact}, \cite[Proposition 11]{MantovanPEL}. Let $b \in G(\qpbr)$ such that the triple $(G,b,\mu)$ is a local Shimura datum in the sense of \cite[Definition 24.1.1]{ScholzeWeinsteinBerkeley}. Then associated with $(\mathcal{G},b,\mu)$ is the integral local Shimura variety $\mintgmu$ as in \cite[Section 25.1]{ScholzeWeinsteinBerkeley}. We have the following lemma, see \cite[Lemma 3.1.6]{Companion}. 
\begin{Lem} \label{Lem:LocalUniformisation}
    There is a Cartesian diagram
    \begin{equation}
        \begin{tikzcd}
            \mintgmu \arrow{r} \arrow{d} & \shtgmub \arrow{d}{\mathrm{BL}^{\circ}} \\
            \spd \fpbar \arrow{r} & \bun_{G}^{[b]}.
        \end{tikzcd}
    \end{equation}
\end{Lem}

Let the notation and assumptions be as in Section \ref{Sec:PerfectIgusa}. For an element $b \in G(\qpbr)$ with $[b] \in \bgmu$, we can restrict the Cartesian diagram to the corresponding Newton stratum. This results in the Cartesian diagram
\[ \begin{tikzcd}
  \mathscr{S}_{K_p}\gx^{\diamond,[b]} \arrow{r}
  \arrow{d} & \shtgmub \arrow{d} \\
  \mathrm{Igs} \gx^{[b]} \arrow{r} &
  \bun_{G}^{[b]},
\end{tikzcd} \]
where $\mathscr{S}_{K_p}\gx^{\diamond,[b]} = \scrs_{K_p}\gx^\diamond \times_{\bung} \bung^{[b]}$ as in \eqref{Eq:NewtonStrat}. Once we further base change along the map $\spd \fpbar \to \bun_{G}^{[b]}$ induced by a choice of $b \in [b]$, we obtain (using Lemma \ref{Lem:LocalUniformisation} above)
\[ \begin{tikzcd}
 \widetilde{\mathscr{S}_{K_p}\gx^\diamond_{[b]}} \arrow{r} \arrow{d} & \mintgmu \arrow{d} \\ \igvinfbgx \arrow{r} &
  \spd \fpbar.
\end{tikzcd} \]
By Theorem~\ref{Thm:GeometryOfBunG}, the entirety of the second diagram is a $\tilde{G}_b$-torsor over the first diagram. The following corollary generalizes \cite[Corollary 11.26]{ZhangThesis}, \cite[Proposition 4.3]{CaraianiScholzeCompact}, \cite[Proposition 11]{MantovanPEL}.

\begin{Cor}\label{Cor:ProductFormula}
  There exists a canonical $\gafp$-equivariant map of v-sheaves
  \[
    \igvinfbgx \times_{\spd \fpbar} \mintgmu \to
    \mathscr{S}_{K_p}\gx_{\mathcal{O}_{\ebreve}}^{\diamond,[b]},
  \]
  that is a $\tilde{G}_b$-torsor in the v-topology.
\end{Cor}

\section{Shimura varieties of Hodge type} \label{Sec:HodgeType}
The goal of this section is to recall the integral models of Shimura varieties of Hodge type constructed in \cite{PappasRapoportShtukas} and \cite{Companion}. We then discuss the universal abelian scheme (up to prime-to-$p$ isogeny) and its extra structure, and finally give a definition of $\igsinf$.

\subsection{Recollection on integral models}
For a symplectic space $(V, \psi)$ over $\mathbb{Q}$ we write $\gv=\mathrm{GSp}(V, \psi)$ for the group of symplectic similitudes of $(V,\psi)$ over $\mathbb{Q}$. It admits a Shimura datum $\mathsf{H}_V$ consisting of the union of the Siegel upper and lower half spaces. For a self-dual $\zp$-lattice $V_{\zp} \subset V_{\qp}$ we write $M_p=\mathrm{GSp}(V_\zp)(\zp)$. For a neat compact open subgroup $M^p \subset \gv(\afp)$ we write $M=M_pM^p$ and we consider the Shimura variety $\mathbf{Sh}_{M}\gvx$ of level $M$ as a scheme over $\qp$. It has an integral model $\scrs_{M}\gvx$ over $\zp$ which is the moduli space of (weakly) polarized abelian schemes $(A,\lambda)$ up to prime-to-$p$ isogeny with $M^p$-level structure $\eta^p$, see \cite[Section 4]{DeligneTravaux}. By \cite[Theorem 4.5.2]{PappasRapoportShtukas}, the system of integral models $\{\scrs_{M}\gvx\}_{M^p}$ is a canonical integral model of $\{\mathbf{Sh}_{M}\gvx\}_{M^p}$.

Let $\gx$ be a Shimura datum of Hodge type with reflex field $\mathsf{E}$, let $p$ be a prime and write $G=\mathsf{G}_{\qp}$. Fix a place $v$ above $p$ of the reflex field $\mathsf{E}$, and let $E$ be the completion of $\mathsf{E}$ at $v$ with ring of integers $\mathcal{O}_E$ and residue field $k_E$. For any neat compact open subgroup $K \subset \gaf$ we will consider the Shimura variety $\mathbf{Sh}_{K}\gx$ as a scheme over $E$. We will also consider the inverse limit
\begin{align}
    \mathbf{Sh}\gx=\varprojlim_{K \subset \gaf} \mathbf{Sh}_{K}\gx,
\end{align}
which is an $E$-scheme (because the transition morphisms are affine) and is equipped with an action of $\gaf$. We will moreover consider the v-sheaf $\mathbf{Sh}_K\gx^{\lozenge}$, which is a diamond by \cite[Lemma 15.6]{EtCohDiam} because the $(-)^\lozenge$-functor factors through rigid analytification, see \cite[Lemma 2.11]{AGLR}.

\subsubsection{Canonical integral models} \label{subsub:ChoicesIntegral}
Let $\mathcal{G}$ be a stabilizer Bruhat--Tits model of $G$ over $\zp$ and let $K_p=\mathcal{G}(\zp)$. We let $\mathcal{H} \subset \mathcal{G}$ be an open quasi-parahoric subgroup (thus with $\calgcirc \subset \mathcal{H}$) and let $K_p'=\mathcal{H}(\zp)$. For $K^p \subset \gafp$ a neat compact open subgroup we write $K=K_pK^p$ and $K'=K_p'K^p$. Then by \cite[Theorem~4.2.3]{Companion}, there are canonical integral models $\{\scrs_{K'}\gx\}_{K^p}$ and $\{\scrs_{K}\gx\}_{K^p}$ over $\mathcal{O}_E$. Applying \cite[Proposition~4.1.10]{Companion} we see that there is a $2$-Cartesian diagram
\begin{equation} \label{Eq:DiagramQuasiParahoricStabilizer}
    \begin{tikzcd}
        \scrs_{K'}\gx^{\diamond} \arrow{d} \arrow[r, "\pi_{\mathrm{crys}, \mathcal{H}}"] & \shthmuone \arrow{d} \\
         \scrs_{K}\gx^{\diamond} \arrow{r}{\pi_{\mathrm{crys}, \mathcal{G}}} & \shtgmuone.
    \end{tikzcd}
\end{equation}
We will also consider the integral model with infinite level away from $p$
\begin{align}
\scrs_{K_p}\gx=\varprojlim_{K^p \subset \gafp} \scrs_{K_pK^p}\gx,
\end{align}
which is an $\mathcal{O}_E$-scheme equipped with an action of $\gafp$. Similarly, we define $\hatscrsginf$, the formal integral model with infinite level away from $p$, to be the $p$-adic completion of $\scrs_{K_p}\gx$, which represents the limit
\begin{align}
\varprojlim_{K^p \subset \gafp} \widehat{\scrs}_{K_pK^p}\gx
\end{align}
in the category of formal schemes over $\spf \mathcal{O}_E$. We will use similar notation for $K'$ in place of $K$.

\label{subsub:Zarhin} The integral models of \cite[Theorem~4.2.3]{Companion} are constructed as follows. Because $\mathcal{G}$ is the stabilizer of a point in the extended building $\mathcal{B}^e(G,\qp)$, it follows from the discussion in \cite[Section~1.3.2]{KMPS} that there exists a Hodge embedding $\iota:\gx \to \gvx$ and a $\zp$-lattice $V_{\zp} \subset V_{\qp}$ on which $\psi$ is $\zp$-valued, such that $\mathcal{G}(\zpbr)$ is the stabilizer in $G(\qpbr)$ of $V_{\zp} \otimes_{\zp} \zpbr$. By Zarhin's trick, see \cite[Remark 2.2.4]{ShenYuZhang}, we may moreover assume (after possibly changing $\iota$ and the symplectic space) that $V_{\zp}$ is a self-dual lattice.

The fact that $\mathcal{G}(\zpbr)$ stabilizes $V_{\zpbr}$ implies, by \cite[Corollary 2.10.10]{KalethaPrasad}, that $G \to G_V$ extends to a morphism $\mathcal{G} \to \mathcal{G}_{V}=\mathrm{GSp}(V_\zp)$. By \cite[Lemma 2.1.2]{KisinModels}, for $K^p \subset \gafp$ we can find $M^p \subset \gv(\afp)$ containing the image of $K^p$ such that the natural map
\begin{align}
    \mathbf{Sh}_{K}\gx \to \mathbf{Sh}_{M} \gvx \otimes_{\mathbb{Q}_p} E
\end{align}
is a closed immersion, where $M_p=\mathrm{GSp}(V_\zp)(\zp)$. We then define $\scrsg$ to be the normalization of the Zariski closure of $\mathbf{Sh}_{K}\gx$ in $\scrs_{M}\gvx \otimes_{\zp} \mathcal{O}_E$. We further define $\scrs_{K'}\gx$ to be the normalization of $\scrsg$ in $\mathbf{Sh}_{K'}\gx$.

\begin{Rem} \label{Rem:IntModelPresheaf} By Lemma \ref{Lem:DiamondOfFormalScheme},
$\scrs_K\gx^\diamond$ is the sheafification with respect to the analytic topology of the presheaf on $\perf$
\begin{align*}
\scrs_K\gx^{\diamond, \mathrm{pre}}: S \mapsto \{(S^\sharp, x)\}
\end{align*}
which assigns to $S = \spa(R,R^+)$ the set of pairs $(S^\sharp, x)$, where $S^\sharp=\spa(R^\sharp, R^{\sharp+})$ is an untilt of $S$ over $\mathcal{O}_E$ and $x: \spf R^{\sharp+} \to \hatscrsg$ is a morphism of formal schemes over $\spf \mathcal{O}_E$. The analogous statement holds when $K$ is replaced by $K'$.
\end{Rem}

\subsubsection{Potentially crystalline loci} \label{Sec:PotCrys} Let $\mathbf{Sh}_K\gx^\mathrm{an}$ denote the rigid analytification of $\mathbf{Sh}_K\gx$, i.e., $\mathbf{Sh}_K\gx^\mathrm{an}$ is the fiber product
\begin{align*}
    \mathbf{Sh}_K\gx^\mathrm{an} = \mathbf{Sh}_K\gx \times_{\spec E} \spa E
\end{align*}
in the sense of \cite[Proposition 3.8]{Huber}. By \cite{ImaiMieda}, there is an open immersion $\mathbf{Sh}_K\gx^\circ \subset \mathbf{Sh}_K\gx^{\mathrm{an}}$ of rigid spaces over $E$, see \cite[Theorem 5.17]{ImaiMieda}. The rigid analytic space $\mathbf{Sh}_K\gx^\circ$ is called the \emph{potentially crystalline locus}. The formation of $\mathbf{Sh}_K\gx^\circ$ is compatible with changing $K$, see \cite[Corollary 5.29]{ImaiMieda}, and we will also consider
\begin{align}
    \mathbf{Sh}\gx^{\circ,\lozenge} \coloneqq\varprojlim_{K \subset \gaf} \mathbf{Sh}_{K}\gx^{\circ,\lozenge}.
\end{align}
\begin{Lem} \label{Lem:PotentiallyCrystalline}
There is a unique isomorphism 
\begin{align}\label{Eq:GoodReductionLocus}
    \mathbf{Sh}_K\gx^{\circ, \lozenge} \xrightarrow{\sim} \scrsd \times_{\spd \mathcal{O}_E} \spd E
\end{align}
compatible with the two open immersions into $\mathbf{Sh}_K\gx^{\lozenge}$.
\end{Lem}
\begin{proof}
    The inclusions of $\mathbf{Sh}_K\gx^{\circ, \lozenge}$ and $\scrsd \times_{\spd \mathcal{O}_E} \spd E$ into $\mathbf{Sh}_K\gx^{\lozenge}$ each identify their respective subdiamond with the pullback of the good reduction locus of the Siegel Shimura variety along the closed immersion
    \begin{align}
    \mathbf{Sh}_{K}\gx^\lozenge \to \mathbf{Sh}_{M} \gvx_E^\lozenge.
\end{align}
Indeed, the statement for $\mathbf{Sh}_K\gx^{\circ,\lozenge}$ follows from its construction, cf.\ \cite[Proposition 5.16.(iii)]{ImaiMieda}, while for $\scrsd$ it follows from Lemma \ref{Lem:FiniteMorphism} below.
\end{proof}
\begin{Lem}\label{Lem:FiniteMorphism}
    Let $f:X\to Y$ be an integral morphism of schemes over $\CO_E$, then the diagram
    \[
    \begin{tikzcd}
        X^\diamond \ar[r]\ar[d,"f^\diamond"] & X^\lozenge\ar[d, "f^\lozenge"]\\
        Y^\diamond \ar[r] & Y^\lozenge
    \end{tikzcd}
    \]
    is Cartesian.
\end{Lem}
\begin{proof}
    Since the statement is compatible with Zariski localizations, we may without loss of generality assume $X=\spec A$ and $Y=\spec B$ are affine. Then given a test object $S=\spa(R,R^+)$ in $\perf$, a map from $S$ to $Y^\diamond\times_{Y^\lozenge}X^\lozenge$ amounts to an isomorphism class of pairs consisting of an untilt $S^\sharp=\spa(R^\sharp,R^{\sharp+})$ over $\CO_E$, and an $\CO_E$-algebra homomorphism $A\to R^\sharp$, such that the composition $B\to A\to R^\sharp$ factors through $R^{\sharp+}$. But since $B\to A$ is integral, while $R^{\sharp+}\subset R^\sharp$ is integrally closed, this amounts to the pair consisting of $S^\sharp=\spa(R^\sharp,R^{\sharp+})$, and an $\CO_E$-algebra homomorphism $A\to R^{\sharp+}$, i.e., an $S$-point of $X^\diamond$. 
\end{proof}

\subsubsection{\'Etale local systems} \label{subsub:EtaleLocalSystems} For a scheme $X$, we write $X_\mathrm{proet}$ for the pro-\'etale site of $X$ as in \cite{BhattScholzeProEtale}. We will denote by $\underline{\mathbb{A}}_{f}^{p}$ the sheaf of topological groups on $X_\mathrm{proet}$ obtained by sheafifying the presheaf
\[ U \mapsto \lbrace \text{continuous maps } \lvert U \rvert \to \afp \rbrace. \]
More generally, for a topological space $T$ and any site whose objects have underlying topological spaces, we use the notation $\underline{T}$ for the sheaf attached to the presheaf sending an object $U$ to the set of continuous maps from $\lvert U \rvert$ to $T$. Following \cite[Definition 1.1]{HansenOberwolfach}, we define $\afp\mathrm{Loc}(X)$ to be the category of $\underline{\mathbb{A}}_{f}^{p}$-modules that are pro-\'etale locally on $X$ isomorphic to $\underline{\mathbb{A}}_{f}^{p,\oplus n}$ for some integer $n$.
 
 We write $\pi:A \to \scrsg$ for the pullback of the universal abelian scheme up to prime-to-$p$ isogeny over $\scrs_M\gvx$ along $\scrsg \to \scrs_M\gvx$. We denote by $\mathcal{V}^p$ the \emph{dual} of $R^1 \pi_\mathrm{proet,\ast}\underline{\mathbb{A}}_{f}^{p}$. For a map $x: \spec R \to \scrsg$, we will write $A_x$ for the pullback of $A$, and similarly $\mathcal{V}^p_x$ for the pullback along $x$ of $\mathcal{V}^p$.
 
As in \cite[Section 5.1.4]{KisinShinZhu}, there is an exact tensor functor
\begin{equation}\label{Eq:LisseAfp}
\mathbb{L}^p:\operatorname{Rep}_{\mathbb{Q}}(\mathsf{G}) \to \afp\mathrm{Loc}(\scrs_K\gx),
\end{equation}
obtained by descending the continuous $K^p$-representation $V\otimes_\mathbb{Q} \afp$ along the pro-\'etale Galois cover
\[\scrs_{K_p}\gx \to \scrs_K\gx.\]
We have that $\mathbb{L}^p(V)=\mathcal{V}^p$, where $V$ is considered as a rational representation of $\mathsf G$ by restricting the standard representation of $\gv$ on $V$ along $\mathsf{G} \to \gv$. Note that we use $\mathcal{V}^{p}$ for $\mathbb{L}^p(V)$ whereas in \cite[Section 5.1.4]{KisinShinZhu} they use $\mathcal{V}^p$ to denote $\mathbb{L}^p(V^{\ast})$.

\subsubsection{\'Etale tensors} \label{Sec:EtaleTensors} Write $V^{\otimes}$ for the direct sum of $V^{\otimes n} \otimes (V^{\ast})^{\otimes m}$ for all pairs of integers $m, n \ge 0$. We will also use this notation later for modules over commutative rings, or over sheaves of rings, and we refer to this construction as the tensor space of $V$. By \cite[Lemma 1.3.2]{KisinModels} and its improvement \cite{DeligneLetter}, we can fix tensors $\{t_{\alpha}\}_{\alpha \in \mathscr{A}} \subset V^{\otimes}$ whose schematic stabilizer in $\mathrm{GL}(V)$ is $\g$. These tensors can be viewed as morphisms $t_{\alpha}:1 \to V^{\otimes}$ and thus the tensor functor \eqref{Eq:LisseAfp} gives global sections
\begin{align}
   \{t_{\alpha,\afp}\}_{\alpha \in \mathscr{A}} \subset H^0(\scrs_{K}\gx, \mathcal{V}^{p,\otimes}).
\end{align}
Over $\scrs_{K_p}\gx$, there is a canonical isomorphism
\begin{align}
    \eta: \mathcal{V}^p\xrightarrow{\sim} V \otimes \underline{\mathbb{A}}_{f}^{p}
\end{align}
which takes $t_{\alpha, \afp}$ to $t_{\alpha} \otimes 1$ for all $\alpha \in \mathscr{A}$, see \cite[Lemma 5.1.9]{KisinShinZhu}. Denote by $$\isom^{\otimes}(\mathcal{V}^p,V \otimes \underline{\mathbb{A}}_{f}^{p})$$
the pro-\'etale sheaf of isomorphisms from $\mathcal{V}^p$ to $V \otimes \underline{\mathbb{A}}_{f}^{p}$ sending $t_{\alpha, \afp}$ to $t_{\alpha} \otimes 1$ for all $\alpha \in \mathscr{A}$. This is a left $\underline{\gafp}$-torsor, where the action is induced by the natural $\underline{\gafp}$-action on $V \otimes \underline{\mathbb{A}}_{f}^{p}$. Note that $\eta$ gives a section of this sheaf over $\scrs_{K_p}\gx$ and there is an induced section
\begin{align}
    \bar{\eta} \in \Gamma(\scrsg, \, \isom^{\otimes}(\mathcal{V}^p, V \otimes \underline{\mathbb{A}}_{f}^{p})/\underline{K}^p).
\end{align}

\subsubsection{Comparison of vector bundles} \label{Sec:CrystallineTensors} Let $S=\spa(R,R^+)$ be an object in $\perf$ with an untilt $S^\sharp=\spa(R^\sharp, R^{\sharp+})$, and let $x$ be a morphism $x:\spf R^{\sharp+} \to \hatscrsg$. By pulling back the universal abelian scheme along
\begin{align}
    \spf R^{\sharp+} \to \hatscrsg \to \hatscrsgv \otimes_{\zp} \mathcal{O}_E,
\end{align}
we get a formal abelian scheme $A_x \to \spf R^{\sharp+}$. As explained in Section \ref{Sec:FormalQuasiIsogeniesFF}, this gives rise to a vector bundle $\mathcal{E}(A_x)$ on the relative Fargues--Fontaine curve $X_S$. \smallskip 

The untilt $S^\sharp$ along with the composition
\begin{align}
    \spa(R^{\sharp},R^{\sharp+}) \to \spa R^{\sharp+} \to (\hatscrsg)^\mathrm{ad}
\end{align}
determines a point $\tilde{x}$ of $\scrs_K\gx^\diamond(S)$. In turn, via the composition
\begin{equation}\label{Eq:MapToBunG}
    \scrs_K\gx^\diamond \xrightarrow{\pi_\mathrm{crys}} \shtgmu \xrightarrow{\mathrm{BL}^{\circ}} \bungmu,
\end{equation}
$\tilde{x}$ determines a $G$-torsor $P$ over $X_S$. By the Tannakian description of $G$-torsors, see \cite[Appendix to Lecture 19]{ScholzeWeinsteinBerkeley}, this gives rise to an exact tensor functor
\begin{align}
    \mathbb{L}_{\mathrm{crys},x}:\operatorname{Rep}_{\mathbb{Q}}(\mathsf{G}) \to
    \{ \text{vector bundles on } X_S\}.
\end{align}
Let $Y_x := A_x[p^\infty]$ be the $p$-divisible group over $R^{\sharp+}$ associated with $A_x$. By construction, the composition
\begin{align*}
\spa(R^{\sharp},R^{\sharp+}) \to \scrs_K\gx^\diamond \xrightarrow{\pi_\mathrm{crys}} \shtgmu \to \mathrm{Sht}_{\mathrm{GL}_V}
\end{align*}
corresponds to the shtuka $\mathscr{V}(Y_x)$ defined in Section \ref{Sec:Dieudonne}, see \cite[4.6.3]{PappasRapoportShtukas}. It then follows from Lemma \ref{Lem:Compatibility} that there is a canonical isomorphism
\[
  \mathbb{L}_{\mathrm{crys},x}(V) \simeq \mathcal{E}(Y_x).
\]

\subsection{The definition of the Igusa stack} \label{Subsec:IgusaDefinition}
{
\def\igspre{\mathrm{Igs}^{\mathrm{pre}}_{K}\gx}
\def\igspreinf{\mathrm{Igs}^{\mathrm{pre}}_{K_p}\gx}
\def\igsv{\mathrm{Igs}_{M}\gvx}
\def\igsvinf{\mathrm{Igs}_{M_p}\gvx}
\def\igsvpre{\mathrm{Igs}^{\mathrm{pre}}_{M}\gvx}
\def\igsvpreinf{\mathrm{Igs}^{\mathrm{pre}}_{M_p}\gvx}

% Temporary renew
\def\igs{\mathrm{Igs}_{\Xi,K^p}\gx}
\def\igsinf{\mathrm{Igs}_{\Xi}\gx}
\def\igspreinf{\mathrm{Igs}^{\mathrm{pre}}_{\Xi}\gx}

In this section, we construct the Igusa stack $\mathrm{Igs}_{K^p}\gx$ for a
Hodge type Shimura datum $\gx$, appearing in the statement of
Theorem~\ref{Thm:IntroIgusa}. We first construct the Igusa stack
$\mathrm{Igs}\gx$ at infinite prime-to-$p$ level and then define the Igusa stack
at finite level $\mathrm{Igs}_{K^p}\gx$ as a pro-\'{e}tale quotient of $\mathrm{Igs}\gx$, see
Definition~\ref{Def:IgusaStackFiniteLevel}. Our construction of the infinite
level Igusa stack $\mathrm{Igs}\gx$ will initially depend on some choices, but
we will later verify that the Igusa stack is independent of these choices, see Proposition \ref{Prop:IndependenceOfParahoric}.

\begin{itemize}
  \item We choose a stabilizer Bruhat--Tits (quasi-parahoric) model $\calg$ over
    $\zp$ and set $K_p = \calg(\zp)$ as in Section~\ref{subsub:ChoicesIntegral}.
  \item We choose a Hodge embedding $\iota \colon \gx \to \gvx$ together with a
    self-dual lattice $V_\zp \subseteq V_{\qp}$ such that $\calg(\zpbr)$ is the
    stabilizer of $V_\zp \otimes_\zp \zpbr$, see Section~\ref{subsub:Zarhin}.
\end{itemize}

Let $\Xi$ denote the triple $(\mathcal{G}, \iota, V_\zp)$. We will denote the resulting Igusa stack by $\igsinf$.

\subsubsection{The Igusa stack} When $\gx = \gvx$ for a symplectic vector space $(V, \psi)$
over $\mathbb{Q}$, for each self-dual lattice $V_\zp \subseteq V_{\qp}$ there is a
hyperspecial subgroup $M_p = \mathrm{GSp}(V_\zp)(\zp) \subseteq G(\qp)$ and an
associated natural choice
\[
  \Xi = (\mathrm{GSp}(V_\zp), \mathrm{id}, V_\zp)
\]
of auxiliary data. We shall denote the resulting Igusa stack by $\igsvinf$
instead of $\mathrm{Igs}_\Xi\gvx$.

Let $(R,R^+)$ be a perfectoid Huber pair of characteristic $p$. Let $(R^{\sharp_1}, R^{\sharp_1+})$ and $(R^{\sharp_2}, R^{\sharp_2+})$ be a pair of untilts and let $\varpi \in R^+$ be a pseudouniformizer for which there exist maps $R^{\sharp_i+} \to R^+/\varpi$ for $i=1,2$ (this can always be arranged, see Section \ref{subsub:Rsharpmap}). Let $A_1$ and $A_2$ be formal abelian schemes up to prime-to-$p$ isogeny over $\spf R^{\sharp_1+}$ and $\spf R^{\sharp_2+}$, respectively. We recall from Section~\ref{Sec:FormalQuasiIsogenyDifferentUntiltsDef} that a formal quasi-isogeny $A_1 \dashrightarrow A_2$ is defined to be a quasi-isogeny $A_1 \otimes_{R^{\sharp_1+}} R^+/\varpi \dashrightarrow A_2 \otimes_{R^{\sharp_2+}} R^+/\varpi$.

Given morphisms $x:\spf R^{\sharp_1+} \to \hatscrsg$, $y:\spf R^{\sharp_2+} \to \hatscrsg$ and a formal quasi-isogeny $f:A_x \dashrightarrow A_y$, there is an induced isomorphism
\begin{align}
  f:\mathcal{E}(A_x) \xrightarrow{\sim} \mathcal{E}(A_y)
\end{align}
of vector bundles on $X_S$ with $S=\spa(R,R^+)$, see Section \ref{Sec:FormalQuasiIsogenyBundle}. There is moreover an induced isomorphism of pro-\'etale sheaves over $\spec R^+/\varpi$
\begin{align}
  f:\mathcal{V}^p_x \xrightarrow{\sim} \mathcal{V}^p_y.
\end{align}

\begin{Def} \label{Def:IgusaStack}
  We define $\igspreinf$ to be the presheaf of groupoids on $\perf$, whose
  value on $S=\spa(R,R^+)$ is the following groupoid: 
  \begin{itemize}
    \item an object is a pair $(S^\sharp, x)$ where $S^\sharp = \spa(R^\sharp,
      R^{\sharp+})$ is an untilt of $S$ over $\mathcal{O}_E$ and $x$ is a map
      $\spf R^{\sharp+} \to \hatscrsginf$ of formal schemes over
      $\mathcal{O}_E$,\footnote{We could, of course, equivalently say that an
      object is an untilt $S^\sharp$ of $S$ together with a map $x \colon \spf
      R^{\sharp+} \to \hatscrsginf$, as the $\mathcal{O}_E$-structure on
      $R^{\sharp+}$ is automatically determined.}
    \item a morphism $f \colon (S^{\sharp_1},x) \to (S^{\sharp_2},y)$ is a
      formal quasi-isogeny 
      \[
        f \colon (A_x,\lambda_x) \dashrightarrow (A_y,\lambda_y)
      \]
      (see Section~\ref{subsub:WeaklyPolarized}) such that the diagram
      \[ \begin{tikzcd}
        \mathcal{V}^p_{x} \arrow{r}{f} \arrow{d}{\eta_x} &  \mathcal{V}^p_{y} \arrow{d}{\eta_y} \\
        V \otimes \underline{\mathbb{A}}_f^p \arrow[r, equals] & V \otimes \underline{\mathbb{A}}_f^p,
      \end{tikzcd} \]
      of pro-\'etale sheaves on $\spec (R^{+}/\varpi)$ commutes for some
      (equivalently, any) choice of pseudouniformizer $\varpi$ of $R^+$ (see
      Section~\ref{Sec:EtaleTensors}) and such that the induced isomorphism (see
      Section~\ref{Sec:FormalQuasiIsogeniesFF})
      \[
        f \colon \mathcal{E}(A_x) \to \mathcal{E}(A_y)
      \]
      of vector bundles on $X_S$ is induced by a (necessarily unique) isomorphism of $G$-bundles $\mathbb{L}_{\mathrm{crys},x} \xrightarrow{\sim} \mathbb{L}_{\mathrm{crys},y}$.
  \end{itemize}
\end{Def}

\begin{Rem}
  There is \emph{no} canonical morphism $\igspreinf \to \spd \mathcal{O}_E$,
  since there could be an isomorphism in $\igspreinf(R)$ between
  $(S^{\sharp_1},x)$ and $(S^{\sharp_2},y)$ where $S^{\sharp_1}$ and
  $S^{\sharp_2}$ are different untilts over $\mathcal{O}_E$.
\end{Rem}

\begin{Lem} \label{Lem:Discrete}
    The presheaf of groupoids $\igspreinf$ is $0$-truncated, i.e., objects in
    the groupoid $\igspreinf(R,R^+)$ do not have nontrivial automorphisms.
\end{Lem}

\begin{proof}
It is enough to show that for a $\zp$-algebra $B$ and an abelian scheme $\pi \colon A \to \spec B$, the natural map (morphisms of group schemes on the left)
\begin{align}
    \Hom_B(A,A) \otimes \mathbb{Q} \to \Hom_B(V^p A, V^p A)
\end{align}
is injective, where $V^p A$ is the prime-to-$p$ adelic Tate module. In turn, it is enough to show that the map
\begin{align}
    \Hom_B(A,A) \to \Hom_B(T^p A, T^p A)
\end{align}
is injective, where $T^p A$ is the prime-to-$p$ Tate module of $A$ over $\spec B$. Let $f$ be an element in the kernel of $\Hom_B(A,A) \to \Hom_B(T^p A, T^p A)$.

Let us write $B=\varinjlim_{i \in I} B_i$ as a filtered colimit of finitely generated $\mathbb{Z}_p$-algebras $B_i$. Then by spreading out, see \cite[Tag 0C0C, Tag 01ZM]{stacks-project} and \cite[Th\'eor\`eme 8.10.5.(xii)]{EGA4III}, there is an index $i \in I$, an abelian scheme $\pi_i:A_i \to \spec B_i$ and a morphism $f_i:A_i \to A_i$ over $\spec B_i$ such that, when we base change along $\spec B \to \spec B_i$, we obtain $\pi:A \to \spec B$ and $f:A \to A$. Since the topological space of $\spec B$ is the inverse limit over the topological spaces of the $B_i$, see \cite[Tag 0CUF]{stacks-project}, we may moreover assume that the image of $\spec B \to \spec B_i$ nontrivially intersects each (of the finitely many) connected components of $\spec B_i$. It now suffices to show that $f_i$ is the zero morphism.

Since the result is well known over algebraically closed fields, see \cite[Theorem 12.10]{MoonenVdGeerEdixhoven},  we see that $f_s$ is zero for any $s:\spec k \to \spec B$ with $k$ an algebraically closed field. By our assumption on $\spec B \to \spec B_i$, we see that there is a geometric point $s:\spec k \to \spec B_i$ in each connected component of $\spec B_i$ such that $f_{i,s}$ is zero. Then by \cite[Proposition 6.1]{MumfordGIT}, there is a section $\rho:\spec B_i \to A_i$ of $\pi_i$ such that $f_i=\rho \circ \pi_i$. Since $f_i$ is a group homomorphism this implies that $\rho$ is the zero section of $A_i$ (e.g.\ because $[n] \circ \rho = \rho$ for all integers $n$). It follows that $f_i = 0$ and thus that $f=0$. 
\end{proof}

\begin{Def}
  The \textit{Igusa stack} $\igsinf$ is the v-sheafification\footnote{To avoid set-theoretic issues, we strictly speaking have to work with the site of $\kappa$-small perfectoid spaces, where $\kappa$ is an uncountable strong limit cardinal. See \cite[Section 4]{EtCohDiam} for a detailed discussion of the set-theoretic issues regarding v-sheafification. We will later see that the Igusa stack does not depend on the choice of $\kappa$, see Remark~\ref{Rem:ChangeOfCutoffCardinal}.} of the presheaf sending $(R,R^+)$ to the set of isomorphism classes in $\igspreinf(R,R^+)$.
\end{Def} 

\subsubsection{Hecke action} \label{Sec:ActionGafp}
Recall from Section~\ref{subsub:ChoicesIntegral} that the integral model
at infinite level $\scrsginf = \varprojlim_{K^p \subset \gafp}
\scrs_{K_pK^p}\gx$ has a natural action of $\gafp$. Since the action on each
finite level factors through a discrete quotient, the pro-\'{e}tale sheaf of
topological groups $\underline{\gafp}$ acts on $\scrsginf$.

Taking the attached v-sheaves, we get an action of $\underline{\gafp}$ on $\scrsdinf$. This induces a $\underline{\gafp}$-action on $\igsinf$. Indeed, on the level of presheaves, for $g \in \underline{\gafp}(S)$, we get a morphism $\igspreinf(S) \to \igspreinf(S)$ sending $x \mapsto gx$ and a formal quasi-isogeny $f:A_x \dashrightarrow A_y$ to $g_y \circ f \circ g_x^{-1}$. Here $g_x$ is the unique quasi-isogeny $g_x:(A_x,\lambda_x) \dashrightarrow (A_{gx},\lambda_{gx})$ preserving the polarization up to a scalar such that the following diagram commutes
\begin{equation}
    \begin{tikzcd}
    \mathcal{V}^p_{x} \arrow{r}{g_x} \arrow{d}{\eta_{x}} &  \mathcal{V}^p_{gx} \arrow{d}{\eta_{gx}} \\
    V \otimes \underline{\mathbb{A}}_{f,R}^{p} \arrow[r, "g"] & V \otimes \underline{\mathbb{A}}_{f,R}^{p},
    \end{tikzcd}
\end{equation}
and $g_y$ is defined similarly. 

\subsubsection{The reduction map}
As in \cite[Proposition 8.3, Remark 8.4]{ZhangThesis}, we have a map of v-sheaves 
\begin{align}
    \IgsQuot \colon \scrs_{K_p}\gx^\diamond \to \igsinf,
\end{align}
obtained by sheafifying the obvious map 
\[\scrs_{K_p}\gx^{\diamond, \mathrm{pre}} \to \igspreinf, \ (S^\sharp, x)\mapsto (S^\sharp, x)\] (see Remark \ref{Rem:IntModelPresheaf}). 
One checks readily that this map is $\underline{\gafp}$-equivariant. By construction, the map
\begin{align}
\scrs_{K_p}\gx^{\diamond, \mathrm{pre}} \to \bungmu
\end{align}
from \eqref{Eq:MapToBunG} factors through $\scrs_{K_p}\gx^{\diamond, \mathrm{pre}} \to \igspreinf$ and thus $\scrsdinf \to \bungmu$ factors through $\IgsQuot \colon \scrsdinf \to \igsinf$. In this way, we obtain a $2$-commutative cube
\begin{equation} \label{Eq:TheCubeInt}
  \begin{tikzcd}[column sep=tiny, row sep=tiny]
    & \scrsdvinf \arrow[rr] \arrow[dd] & &  \shtgvmu \arrow{dd} \\
    \scrsdinf \arrow[rr, crossing over] \arrow{ur} \arrow{dd}{\IgsQuot} & &\shtgmu \arrow{ur}\\
    & \igsvinf \arrow{rr} & & \bun_{G_V,\mu_V^{-1}} \\
     \igsinf \arrow{rr}{\overline{\pi}_{\mathrm{HT}}} \arrow{ur} & & \bungmu. \arrow{ur} \arrow[from=uu, crossing over]
  \end{tikzcd}
\end{equation}
We will prove that the front square of \eqref{Eq:TheCubeInt} (i.e., the one consisting of objects with index $\gx$, $G$ or $\mathcal{G}$) is $2$-Cartesian, see Theorem \ref{Thm:HodgeMain}. 

\subsubsection{The finite level Igusa stack} \label{Sec:Defitionfinitelevel} 

Fix a neat compact open subgroup $K^p \subset \gafp$.

\begin{Def} \label{Def:IgusaStackFiniteLevel}
We define the Igusa stack of level $K^p$ to be the quotient stack \[\igs = \left[\igsinf/\underline{K^p}\right]. \]
\end{Def}

\begin{Rem} \label{Rem:FiniteLevelIgs}
  \def\igspre{\mathrm{Igs}_{\Xi,K^p,\mathrm{fin}}^\mathrm{pre}}
We could also consider the v-sheafification of the following presheaf of groupoids $\igspre$ on $\perf$. For an affinoid perfectoid $S=\spa (R,R^+)$, the objects in $\igspre(S)$ are pairs $(S^\sharp, x)$, where $S^\sharp = \spa(R^\sharp, R^{\sharp+})$ is an untilt of $S$ over $\mathcal{O}_E$, and $x$ is a morphism $\spf R^{\sharp+} \to \hatscrsg$ of formal schemes over $\mathcal{O}_E$. A morphism $f:(S^{\sharp_1},x) \to (S^{\sharp_2},y)$ is a formal quasi-isogeny
\begin{align}
    f:(A_x,\lambda_x) \dashrightarrow (A_y,\lambda_y)
\end{align}
such that the isomorphism $f:\mathcal{V}^p_{x} \to \mathcal{V}^p_{y}$ of pro-\'etale sheaves on $\spec R^{+}/\varpi$ sends $t_{\alpha,\afp,x}$ to $t_{\alpha,\afp,y}$ for all $\alpha \in \mathscr{A}$, the induced isomorphism
\begin{align}
    \isom^{\otimes}(\mathcal{V}^p_{x},V \otimes \underline{\mathbb{A}}_{f}^{p})/K^p \to \isom^{\otimes}(\mathcal{V}^p_y,V \otimes \underline{\mathbb{A}}_{f}^{p})/K^p
\end{align}
sends $\bar{\eta}_x$ to $\bar{\eta}_y$, and the induced isomorphism (see Section \ref{Sec:FormalQuasiIsogeniesFF})
\begin{align}
    f:\mathcal{E}(A_x) \to \mathcal{E}(A_y)
\end{align}
of vector bundles on $X_S$ comes from a (necessarily unique) isomorphism of $G$-bundles $\mathbb{L}_{\mathrm{crys},x} \to \mathbb{L}_{\mathrm{crys},y}$. In other words, this is the sheafification of the coequalizer in presheaves of groupoids
\[\underline{K^p}\times \igspreinf \rightrightarrows \igspreinf.\]
This defines the same v-stack as $\igs$, since sheafification preserves colimits.
\end{Rem}

\begin{Lem} \label{Lem:ComparisonZhang}
  The v-stack $\igsv = [\igsvinf / M^p]$ is isomorphic to the v-stack
  $\mathrm{Igs}$ of level $M^p$ defined in \cite[Definition~8.1]{ZhangThesis}. 
\end{Lem}

\begin{proof}
\def\igsvprefin{\mathrm{Igs}_{\Xi,M^p,\mathrm{fin}}^\mathrm{pre}}

Fix $S = \spa(R,R^+)$. It suffices to define a functorial equivalence $\igsvprefin(R,R^+)\xrightarrow{\sim} \mathrm{Igs}^\mathrm{pre}(R,R^+)$, where $\Xi = (\mathrm{GSp}(V_\zp), \mathrm{id}, V_\zp)$ and $\igsvprefin$ is defined as in Remark \ref{Rem:FiniteLevelIgs}. Making the definition explicit, we see that $\igsvprefin(R,R^+)$ is the presheaf of groupoids whose objects are pairs $(S^\sharp, x)$, where $S^\sharp = \spa(R^\sharp, R^{\sharp+})$ is an untilt of $S$ over $\zp$ and $x$ is a $\spf R^{\sharp+}$-point of $\hatscrsgv $ over $\zp$, where morphisms are formal quasi-isogenies $A_x \dashrightarrow A_y$ preserving the polarizations up to a locally constant $\mathbb{Q}^\times$-valued function on $\spec R^+/\varpi$. Define $A_{x,0} = A_x \otimes_{R^{\sharp+}} R^+/\varpi$. Then the assignment 
\begin{align*}
    (S^\sharp, x) \mapsto (A_{x,0}, \iota, \lambda_x, \bar{\eta}_x)
\end{align*}
defines a functorial map $\igsvprefin(R,R^+) \to \mathrm{Igs}^\mathrm{pre}(R,R^+)$, where $\iota: \mathbb{Z} \to \operatorname{End}(A_x) \otimes {\mathbb{Z}_{(p)}}$ is the obvious map. 

Full faithfulness of $\igsvprefin(R,R^+) \to \mathrm{Igs}^\mathrm{pre}(R,R^+)$ is immediate from the definition of morphisms on both sides, so it remains only to show essential surjectivity. But this follows from formal smoothness of $\mathscr{S}_M\gvx$ over $\spec \zp$, which implies that any map $\spec R^+/\varpi \to \scrs_M\gvx$ over $\spec \zp$ lifts to a map $\spf R^{\sharp+} \to \hatscrsgv$ over $\spf \zp$.
\end{proof}

\subsubsection{Frobenius on the Igusa stack} We record a result that follows easily from our definition of
the Igusa stack. This will be used later in Section~\ref{Sec:Cohomology}. Recall
from Section~\ref{Sec:AbsoluteFrobenii} that every v-stack has an absolute Frobenius.

\begin{Prop} \label{Prop:FrobeniusTrivial}
  The absolute Frobenius on the Igusa stack $\igsinf$ is the identity map.
  Moreover, under the isomorphism $\phi_{\bung} \simeq \mathrm{id}_{\bung}$ of
  Proposition~\ref{Prop:AbsoluteFrobeniusBunG}, the diagram
  \[ \begin{tikzcd}
    \igsinf \arrow{r} \arrow[bend right]{d}[']{\phi} \arrow[bend
    left]{d}{\mathrm{id}} \arrow[phantom]{d}[description]{=} & \bung \arrow[bend
    right]{d}[']{\phi} \arrow[bend left]{d}{\mathrm{id}} \\ \igsinf \arrow{r} &
    \bung
  \end{tikzcd} \]
  is 2-commutative, i.e., the two isomorphisms between $\igsinf
  \xrightarrow{\phi} \igsinf \to \bung$ and $\igsinf \to \bung
  \xrightarrow{\mathrm{id}} \bung$ agree.
\end{Prop}

\begin{proof}
  We instead check that the absolute Frobenius on the presheaf $\igspreinf$ is
  the identity. Let $S = \spa(R^\sharp, R^{\sharp+})$ be any affinoid
  perfectoid and fix a point $x \colon \spf R^{\sharp+} \to \hatscrsginf$. We
  need to verify that if we regard $R^\sharp$ as two different untilts of
  $R^{\sharp\flat}$, one Frobenius-twisted by another, then there exists a
  formal quasi-isogeny $(A_x, \lambda_x) \dashrightarrow (A_x, \lambda_x)$
  preserving the \'{e}tale trivializations and the $G$-structures on
  $\mathbb{L}_{\crys,x}$.

  Choose a pseudouniformizer $\varpi \in R^{\sharp+}$ that divides $p$.
  Unraveling the definition of a formal quasi-isogeny, cf.\
  Section~\ref{Sub:FormalQuasiIsogenies}, we see that it suffices to find a
  quasi-isogeny between the two abelian schemes $A_x \otimes_{R^{\sharp+}}
  R^+/\varpi$ and $\phi_{\spec R^+/\varpi}^\ast (A_x \otimes_{R^{\sharp+}}
  R^+/\varpi)$, preserving the extra structures. But the relative Frobenius on
  $A_x \otimes_{R^{\sharp+}} R^+/\varpi$ is such a quasi-isogeny.

  The second part directly follows from the constructions. After unraveling the
  definition, it remains to verify that for a $p$-divisible group $Y$ over
  $R^+/\varpi$, the isomorphism $\mathcal{E}(\phi^\ast Y) \simeq
  \mathcal{E}(Y)$ of vector bundles on $X_S$ induced from the relative Frobenius
  $\phi^\ast Y \dashrightarrow Y$ is equal to the isomorphism of
  Proposition~\ref{Prop:AbsoluteFrobeniusBunG}. This is clear from the
  construction of $\phi_{\bung} \simeq \mathrm{id}_{\bung}$ and the discussion
  of Section~\ref{Sec:FormalQuasiIsogeniesFF}.
\end{proof}

\subsection{Comparison of crystalline tensors} This is a standalone section whose goal is to compare the crystalline tensors constructed in \cite{Hamacher-Kim} with those coming from the morphism $\scrs_K\gx^{\diamond} \to \shtgmu$. This can be used to show that the (perfect) Igusa varieties we define in Section \ref{Sec:PerfectIgusa} agree with the ones defined by Hamacher--Kim in \cite{Hamacher-Kim} (under the hypotheses of loc.\ cit.\ that we will recall below). 

Let $\gx$ be a Shimura datum of Hodge type as before. Assume in addition that $p>2$, that $G$ is tamely ramified over $\qp$ and that $p$ is coprime to the order of $\pi_1(\gder)$. Let $\mathcal{G}$ be a stabilizer parahoric group scheme, and choose $\iota:\gx \to \gvx$ a Hodge embedding along with a self-dual $\zp$-lattice $V_\zp \subset V_{\qp}$ such that $\mathcal{G}(\zpbr)$ is the stabilizer in $G(\qpbr)$ of $V_\zp$. By \cite[Lemma 1.3.2]{KisinModels}, we can find tensors
\begin{align}
    \{t_{\alpha}\}_{\alpha \in \mathscr{A}'} \subset V_{\zp}^{\otimes}
\end{align}
such that their stabilizer can be identified with $\mathcal{G}$.\footnote{Here we use $\mathscr{A}'$ to denote the indexing set of the tensors to distinguish it from the set $\mathscr{A}$ from Section \ref{Sec:EtaleTensors}.}

By \cite[Lemma 2.1.2]{KisinModels}, for $K^p \subset \gafp$ we can find $M^p \subset \gv(\afp)$ containing the image of $K^p$ such that the natural map
\begin{align}
    \mathbf{Sh}_{K}\gx \to \mathbf{Sh}_{M} \gvx \otimes_{\mathbb{Q}_p} E
\end{align}
is a closed immersion, where $M_p=\mathrm{GSp}(V_\zp)(\zp)$. Then there is a corresponding finite morphism $\scrs_{K}\gx \to \scrs_{M}\gvx$ as in Section \ref{subsub:Zarhin}. We will also consider $\scrs_{K_p} \gx \to \scrs_{M_p}\gvx_{\mathcal{O}_E}$ and we let $\shginf \to \mathscr{S}_{M_p}(\mathsf{G}_V,\mathsf{H}_V)^\mathrm{perf}_{k_E}$ denote the induced morphism on perfect special fibers. 

Let $A$ be the universal abelian scheme up to prime-to-$p$ isogeny over $\shgvinf$ and let $A$ also denote its pullback to $\shginf$. Let $Y=A[p^{\infty}]$ be the $p$-divisible group of $A$, let $\mathbb{D}=\mathbb{D}(Y)^{\ast}$ be as in 
Section \ref{Sec:Dieudonne} and let $\mathbb{D}^{\otimes}$ be its tensor space. 

In \cite[Section 3, Proposition 3.3.1]{Hamacher-Kim} Hamacher--Kim construct $\phi_Y$-invariant tensors $\{t_{\alpha, \mathrm{HK}, \mathrm{crys}}\}_{\alpha \in \mathscr{A}'} \subset \mathbb{D}^{\otimes}$. As explained in the discussion before \cite[Theorem 4.5.2]{PappasRapoportShtukas}, there are \'etale tensors 
\begin{align}
    \{t_{\alpha, \text{\'et}}\}_{\alpha \in \mathscr{A}'} \subset \Gamma(\mathbf{Sh}_{K_p}\gx, \mathbb{V}_p^{\otimes}),
\end{align}
where $\mathbb{V}_p$ is the $\zp$-local system corresponding to the relative $p$-adic Tate module of $A$. The tensors $t_{\alpha,\mathrm{HK},\mathrm{crys}}$ satisfy the following property (see \cite[Corollary 3.3.7]{Hamacher-Kim}): For every point $\overline{z}:\spec \fpbar \to \shginf$ and every lift $z:\spec \mathcal{O}_F \to \scrs_{K_p}\gx$ of $\overline{z}$ where $F$ is a finite extension of $\ebreve$, the crystalline comparison isomorphism
\begin{align}
    \mathbb{V}_{p,z} \otimes_{\mathbb{Z}_{p}} \mathbf{B}_\mathrm{crys} \to \mathbb{D}_{\overline{z}} \otimes_{\zpbr} \mathbf{B}_\mathrm{crys}
\end{align}
matches $t_{\alpha, \text{\'et},z} \otimes 1$ with $t_{\alpha,\mathrm{HK},\mathrm{crys},\overline{z}} \otimes 1$ for all $\alpha \in \mathscr{A}'$. Note that this uniquely characterizes the sections $t_{\alpha,\mathrm{HK},\mathrm{crys},\overline{z}}$ for all $\overline{z}:\spec \fpbar \to \shginf$. \smallskip 

Let $\mathbb{D}^{\natural}=\mathbb{D}^{\natural}(Y)$ and let $\mathbb{D}^{\natural, \otimes}$ be its tensor space; note that $\mathbb{D}^{\natural, \otimes}=((\phi^{-1})^{\ast} \mathbb{D})^{\otimes}$. The tensors $\{t_{\alpha}\}_{\alpha \in \mathscr{A}'} \subset V_\zp^\otimes$ give rise to ($\phi_Y$-invariant) crystalline tensors $\{t_{\alpha, \mathrm{PR}, \mathrm{crys}}\}_{\alpha \in \mathscr{A}'} \subset  \mathbb{D}^{\natural, \otimes}$ 
on $\shginf$ using the Tannakian interpretation of torsors and the morphism $\shginf \to \shtglocmu$ defined by Pappas--Rapoport. The following proposition is the main result of this section.

\begin{Prop} \label{Prop:HamacherKimTensors}
    We have an equality $\phi^{\ast} t_{\alpha, \mathrm{PR}, \mathrm{crys}}=t_{\alpha,\mathrm{HK},\mathrm{crys}}$ for all $\alpha \in \mathscr{A}'$.
\end{Prop}

We will need the following statement from \cite{PappasRapoportShtukas}.

\begin{Lem} \label{Lem:PappasRapoportTensors}
    For every point $\overline{z}:\spec \fpbar \to \shginf$ and every lift $z:\spec \mathcal{O}_F \to \scrs_{K_p}\gx$ of $\overline{z}$, where $F$ is a finite extension of $\ebreve$, the crystalline comparison isomorphism
\begin{align}
    \mathbb{V}_{p,z} \otimes_{\zp} \mathbf{B}_\mathrm{crys} \to \mathbb{D}_{\overline{z}} \otimes_{\zpbr} \mathbf{B}_\mathrm{crys}
\end{align}
  matches $t_{\alpha,\mathrm{\acute{e}t},z} \otimes 1$ with $\phi^{\ast}t_{\alpha, \mathrm{PR},\mathrm{crys},\overline{z}} \otimes 1$.
\end{Lem}

\begin{proof}
This is explained in the proof of \cite[Lemma 4.6.5]{PappasRapoportShtukas}. 
\end{proof}

\begin{proof}[Proof of Proposition \ref{Prop:HamacherKimTensors}]
If $R$ is a perfect $\mathbb{F}_p$-algebra, let $\mathbb{D}_R$ denote the evaluation of $\mathbb{D}$ on the PD-thickening $W(R) \to R$. There is an affine morphism $D \to \shginf$ representing the functor $(R \to \shginf) \mapsto \mathbb{D}_R$, since $\mathbb{D}_R$ is \'etale locally on $R$ isomorphic to $W(R)^{\oplus h}$ and $R \mapsto W(R)^{\oplus h}$ is representable by the affine scheme $L^+ \mathbb{A}^h$, see Section \ref{Sec:Loops}.

It follows in the same way that there is a scheme $D^{\otimes}$ with an ind-affine morphism $D^{\otimes} \to \shginf$ representing the functor $(R \to \shginf) \mapsto \mathbb{D}_R^{\otimes}$. For each $\alpha$ in $\mathscr{A}'$, the tensors $\phi^{\ast} t_{\alpha, \mathrm{PR},\crys}$ and $t_{\alpha, \mathrm{HK},\crys}$ define morphisms $\shginf \to D^{\otimes}$, and we are trying to show that the two are equal. Since the source is separated and quasi-compact, and since the target is ind-affine, there is a closed subscheme of $\shginf$ where they agree. By Lemma \ref{Lem:PappasRapoportTensors} and the characterization of tensors constructed by Hamacher--Kim, this closed subscheme contains $\shginf(\fpbar)$ and is thus equal to $\shginf$.
\end{proof}

\section{Proof of Theorem \ref{Thm:IntroIgusa}} \label{Sec:Proof}

\input{proof.tex}

\subsection{Points of the Igusa stack} \label{Sec:PointsIgusa}
Recall that we have defined $\igsinf$ as the v-sheafification of $\igspreinf$.
This means that it is a priori difficult to explicitly describe the
points of $\igsinf$. However, we will prove that for $S$ a product of geometric
points in characteristic $p$, the set of $S$-points of $\igspreinf$ agrees with
the set of $S$-points of $\igsinf$. Recall from \cite[Proposition 3.1.10]{Companion} that the map $\operatorname{BL}^{\circ}:\shtgmuone \to \bung$ factors through $\bungmu$. Given a v-stack $Y$ over $\spd \mathcal{O}_E$ and a perfectoid space $S$ in characteristic $p$ with untilt $S^{\sharp}$ over $\mathcal{O}_E$, we will write $Y(S^{\sharp}) = Y(S) \times_{(\spd \mathcal{O}_E)(S)} \{S^{\sharp}\} \subset Y(S) $.

\begin{Prop} \label{Prop:EssentiallySurjectiveBL}
  Let $S=\spa(R,R^+)$ be a strictly totally disconnected perfectoid space in characteristic $p$, and let
  $S^\sharp$ be an untilt of $S$ over $E$. Then the map $\mathrm{BL} \colon
  \shtgmuonerat(S^{\sharp}) \to \bungmu(S)$ is essentially surjective.
\end{Prop}

\begin{proof}
By Lemma \ref{Lem:BLCircVSBL}, it suffices to show that $\mathrm{BL}:\mathrm{Gr}_{G,\mu^{-1}}(S^{\sharp}) \to \bungmu(S)$ is essentially surjective.  

We argue as in \cite[Proposition~III.3.1]{FarguesScholze}. Let $\mathcal{E}$
  be a $G$-bundle on $X_S$. For each connected component $\spa(C, C^+)
  \hookrightarrow S$, $C$ is an algebraically closed perfectoid
  field, see \cite[Proposition 7.16]{EtCohDiam}. Using
  \cite[Proposition~A.9]{RapoportAccessiblePeriodDomain}, we find an element of
 $\mathrm{Gr}_{G,\mu^{-1}}(C^\sharp, C^{\sharp+})$ that maps to
 $\mathcal{E} \vert_{\spa(C, C^+)}$ under $\mathrm{BL}$. By trivializing
 $\mathcal{E}$ at the completion at $S^\sharp$, we see as in the proof of
  \cite[Proposition~III.3.1]{FarguesScholze} that it suffices to prove that
 $\mathrm{Gr}_{G, \mu^{-1} }(S^\sharp) \to \mathrm{Gr}_{G,
  \mu^{-1}}(C, C^+)$ is surjective.

  At this point, we can use the argument of
  \cite[Lemma~III.3.2]{FarguesScholze}. Using the Cartan decomposition but
  restricting to the double coset $G(\mathbf{B}_\mathrm{dR}^+(C)) \mu^{-1}(\xi)
  G(\mathbf{B}_\mathrm{dR}^+(C))$, we similarly reduce to proving that
  $G(\mathbf{B}_\mathrm{dR}^+(R)) \to G(\mathbf{B}_\mathrm{dR}^+(C))$ is
  surjective. The rest of the proof of \cite[Proposition~III.3.1]{FarguesScholze} now follows through.
\end{proof}

\begin{Cor} \label{Cor:VSurjectiveBL}
  The morphism $\mathrm{BL} \colon \shtgmuonerat \to \bungmu$ is a surjective map of v-stacks; in fact, of pro-\'etale stacks.
\end{Cor}

\begin{proof}
  It follows from Proposition~\ref{Prop:EssentiallySurjectiveBL} together with
  the fact that strictly totally disconnected spaces form a basis for the pro-\'etale topology.
\end{proof}

\begin{Cor} \label{Cor:ProdGeomPtsIgsQuotSurjective}
  For $S$ a strictly totally disconnected perfectoid space in characteristic $p$, the map $\IgsQuot
  \colon \scrsdinf(S) \to \igsinf(S)$ is essentially surjective.
\end{Cor}

\begin{proof}
  It follows from Proposition~\ref{Prop:EssentiallySurjectiveBL}
  in combination with Lemma~\ref{Lem:BLCircVSBL} that the map
 $\mathrm{BL}^{\circ}: \shtgmu(S) \to \bungmu(S)$ of groupoids is essentially surjective. The result
  then follows from Theorem~\ref{Thm:HodgeMain}.
\end{proof}

Denote by
\[
  \mathcal{R} = \scrsdinf \times_{\igsinf} \scrsdinf \subseteq \scrsdinf \times
  \scrsdinf
\]
the equivalence relation. Note that this is the v-sheafification of
\[
  \mathcal{R}^\mathrm{pre} = \scrsdpreinf \times_{\igspreinf} \scrsdpreinf
  \subseteq \scrsdpreinf \times \scrsdpreinf,
\]
as sheafification commutes with fiber products.
Since $\scrsdinf(S) \to \igsinf(S)$ is essentially surjective, as shown in
Corollary~\ref{Cor:ProdGeomPtsIgsQuotSurjective}, we see that
\[
  \igsinf(S) = [\scrsdinf(S) / \mathcal{R}(S)].
\]
By the definition of $\igspreinf$ we also have
\[
  \igspreinf(S) = [\scrsdpreinf(S) / \mathcal{R}^\mathrm{pre}(S)].
\]
\begin{Thm} \label{Thm:ProductOfPointsIgs}
  Let $S$ be a product of geometric points in characteristic $p$. Then the natural map
 $\igspreinf(S) \to \igsinf(S)$ is an equivalence.
\end{Thm}

\begin{proof}
  As noted in the proof of Proposition~\ref{Prop:UniqueLiftsShimura}, the
  natural map $\scrsdpreinf(S) \to \scrsdinf(S)$ is a bijection by
  Lemma~\ref{Lem:DiamondOfFormalScheme} and
  \cite[Proposition~1.6]{GleasonSpecialization}. Hence it suffices to prove that
  the map $\mathcal{R}^\mathrm{pre}(S) \to \mathcal{R}(S)$ is a bijection.

  We note that Theorem~\ref{Thm:HodgeMain} implies that the natural map
  \[
    \mathcal{R} \to \scrsdinf \times_{\bun_G} \shtgmu
  \]
  is an isomorphism. On the other hand, the map $\scrsdpreinf \to
  \igspreinf \times_{\bun_G} \shtgmu = F^\mathrm{pre}$ induces a map
  \[
    \mathcal{R}^\mathrm{pre} \to \scrsdpreinf \times_{\igspreinf} F^\mathrm{pre}
    = \scrsdpreinf \times_{\bun_G} \shtgmu.
  \]
  By Proposition~\ref{Prop:ProdOfGeomPointsBijective}, this map induces an
  isomorphism at the level of $S$-points. That is, we have an isomorphism
  \begin{align}
    \mathcal{R}^\mathrm{pre}(S) &\simeq \scrsdpreinf(S) \times_{\bun_G(S)}
    \shtgmu(S) \\ & \simeq \scrsdinf(S) \times_{\bun_G(S)} \shtgmu(S) \simeq
    \mathcal{R}(S).
  \end{align}
  This proves that $\igspreinf(S) \to \igsinf(S)$ is an equivalence.
\end{proof}

\begin{Rem} \label{Rem:ChangeOfCutoffCardinal}
  A consequence of Theorem~\ref{Thm:ProductOfPointsIgs} is that the
  v-sheafification $\igsinf$ of $\igspreinf$ does not depend on the choice of
  the cutoff cardinal $\kappa$ in the definition of the v-site $\perf$. Given
  an affinoid perfectoid space $S$ in characteristic $p$, we may choose
  v-covers $S_0 \to S$ and $S_1 \to S_0 \times_S S_0$, where both $S_0, S_1$ are
  products of geometric points. Then
  \[
    \igsinf(S) \to \igspreinf(S_0) \rightrightarrows \igspreinf(S_1)
  \]
  is an equalizer diagram, and this gives a description of $\igsinf(S)$ that
  does not depend on the choice of $\kappa$.
\end{Rem}

\subsection{The reduction of the Igusa stack} \label{Sec:ReductionIgusaStack}
The goal of this section is to identify the reduction $\igsinf^{\mathrm{red}}$,
see Section~\ref{Sec:Reduction}. As in Section~\ref{Sec:ReductionDiagram}, we
use $\shginf$ to denote the perfect special fiber of $\scrsginf$, which is
naturally identified with $(\scrsdinf_{k_E})^{\mathrm{red}}=(\scrsdinf)^{\mathrm{red}}$ (see Lemma \ref{Lem:ReductionOfScheme}). The reduction of the morphism $\pi_{\mathrm{crys}}:\scrsdinf \to \shtgmu$ gives rise to a morphism
$\pi_{\mathrm{crys}}^{\mathrm{red}}:\shginf \to \shtglocmu$, see Lemma~\ref{Lem:ReductionShtukasMu}. Its
composition with the reduction of $\mathrm{BL}^{\circ}:\shtgmu \to \bung$ gives rise to a
morphism $\shginf \to \gisoc$, see Section~\ref{Sec:Isocrystals}. We let $A$ denote the pullback along $i:\shginf \to
\shgvinf$ of the universal abelian scheme up to prime-to-$p$ isogeny $A$. For $x:\spec R
\to \shginf$ we let $A_x$ be the pullback of $A$ along $x$. We let $\mathcal{E}(A_{x})$ be the
$\mathrm{GL}_V$-isocrystal corresponding to Dieudonn\'e-module
$\mathbb{D}^{\natural}(A_x[p^{\infty}])$ of the pullback along $i \circ x$. We will denote by $\mathbb{L}_{\mathrm{crys},x}$ the $G$-isocrystal over $\spec R$ induced by the morphism $\shginf \to \gisoc$. It
follows from the discussion in Section~\ref{Sec:CrystallineTensors} that there
is a natural isomorphism
\begin{align}
    \mathbb{L}_{\mathrm{crys},x} \times^{G} \mathrm{GL}_V \simeq \mathcal{E}(A_{x}).
\end{align}
The main result of this section is the following. 

\begin{Thm} \label{Thm:ReductionIgusa}
  The map of v-sheaves on $\affperf$
  \[
    \IgsQuot^\mathrm{red} \colon \shginf \to \igsinf^\mathrm{red}
  \]
  identifies $\igsinf^\mathrm{red}$ with the quotient v-sheaf of $\shginf$
  under the following equivalence relation: Two morphisms $x, y \colon \spec R
  \to \shginf$ are equivalent when there exists a quasi-isogeny $f:A_x
  \dashrightarrow A_y$ such that
  \begin{equation} \begin{tikzcd} \label{Eq:CommutativeDiagramEtaleIII}
    \mathcal{V}^p_{x} \arrow{r}{f} \arrow{d}{\eta_x} &  \mathcal{V}^p_{y}
    \arrow{d}{\eta_y} \\ V \otimes \underline{\mathbb{A}}_f^p \arrow[r, equals]
    & V \otimes \underline{\mathbb{A}}_f^p
  \end{tikzcd} \end{equation}
  commutes and the induced isomorphism of $\mathrm{GL}_V$-isocrystals
 $\mathcal{E}(A_{x}) \to \mathcal{E}(A_{y})$ is induced by a
  $($necessarily unique$)$ isomorphism of $G$-isocrystals
 $\mathbb{L}_{\mathrm{crys},x} \to \mathbb{L}_{\mathrm{crys},y}$.
\end{Thm}

\begin{Rem}
  When such a quasi-isogeny $A_x \dashrightarrow A_y$ exists, it is unique by
  Lemma~\ref{Lem:Discrete}. We will refer to quasi-isogenies $f:A_x \dashrightarrow A_y$ for which the diagram \eqref{Eq:CommutativeDiagramEtaleIII} commutes as quasi-isogenies \emph{preserving the \'etale trivializations}. 
\end{Rem}

\subsubsection{Reduction of the diagram} By taking the reduction of the diagram in Theorem~\ref{Thm:HodgeMain} and using the discussion before Theorem \ref{Thm:ReductionIgusa}, we
obtain a Cartesian diagram \eqref{Eq:ReductionDiagram}
\begin{equation}
\begin{tikzcd} \label{Eq:ReductionDiagram2}
  \shginf \arrow{r}{\pi_{\mathrm{crys}}^{\mathrm{red}}} \arrow{d}{ \IgsQuot^\mathrm{red}} & \shtglocmu \arrow{d}{\mathrm{BL}^{\circ, \mathrm{red}}} \\ \igsinf^\mathrm{red}
  \arrow{r}{\overline{\pi}_{\mathrm{HT}}^{\mathrm{red}}} & \gisocmu.
\end{tikzcd}
\end{equation}

\begin{Prop} \label{Prop:ReductionOfIgusaVSurjective}
  The map $\shginf \to \igsinf^\mathrm{red}$ is v-surjective.
\end{Prop}

\begin{proof}
  This follows from the Cartesian diagram and Proposition~\ref{Prop:VSurjective}.
\end{proof}

As in Section~\ref{Sec:PointsIgusa}, we consider the sub-v-sheaf of relations
\[
  \mathcal{R} = \scrsdinf \times_{\igsinf} \scrsdinf \subseteq
  \scrsdinf \times \scrsdinf.
\]
Since reduction commutes with fiber products, we identify its reduction as
\[
  \mathcal{R}^\mathrm{red} = \shginf \times_{\igsinf^\mathrm{red}} \shginf
  \subseteq \shginf \times \shginf,
\]
using Lemma~\ref{Lem:ReductionOfScheme}. On the other hand, it follows from
Proposition~\ref{Prop:ReductionOfIgusaVSurjective} that 
\[
  \igsinf^\mathrm{red} = [\shginf / \mathcal{R}^\mathrm{red}].
\]
Our goal is now to compute $\mathcal{R}^\mathrm{red}$.

As sheafification commutes with fiber products, the v-sheaf
$\mathcal{R}$ is the sheafification of
\[
  \mathcal{R}^\mathrm{pre} = \scrsdpreinf \times_{\igspreinf} \scrsdpreinf \subseteq
  \scrsdpreinf \times \scrsdpreinf.
\]
This presheaf associates to $(R, R^+)$ the set of pairs of points $x \in
\scrsginf(\spf R^{\sharp_1+})$ and $y \in \scrsginf(\spf R^{\sharp_2+})$
for which there exists a formal quasi-isogeny $f \colon
A_x \dashrightarrow A_y$ preserving the \'{e}tale trivialization, and such that $\mathcal{E}(f):\mathcal{E}(A_x) \to \mathcal{E}(A_y)$ is
induced by an isomorphism $\mathbb{L}_{\mathrm{crys},x} \to
\mathbb{L}_{\mathrm{crys},y}$ of $G$-torsors on $X_{(R, R^+)}$. Note that we have a
commutative diagram
\[ \begin{tikzcd}
  \mathcal{R}^\mathrm{pre} \arrow[hook]{r} \arrow{d} & \scrsdpreinf \times \scrsdpreinf
  \arrow{d} \\ \mathcal{R}_V^\mathrm{pre} \arrow[hook]{r} & \scrsdvpreinf \times
  \scrsdvpreinf,
\end{tikzcd} \]
where $\mathcal{R}_V^\mathrm{pre}$ is the analogue of $
\mathcal{R}^\mathrm{pre}$ for $\gvx$.

\begin{Lem} \label{Lem:RelationGFromGV}
  Writing
  \[
    i^\ast \mathcal{R}_V^\mathrm{pre} = \mathcal{R}_V^\mathrm{pre}
    \times_{\scrsdvpreinf \times \scrsdvpreinf} (\scrsdpreinf \times \scrsdpreinf),
  \]
  the diagram
  \[ \begin{tikzcd}
    \mathcal{R}^\mathrm{pre} \arrow{r} \arrow{d} & i^\ast
    \mathcal{R}_V^\mathrm{pre} \arrow{d} \\ \bun_G \arrow{r}{\Delta} & \bun_G
    \times_{\bun_{G_V}} \bun_G
  \end{tikzcd} \]
  is Cartesian.
\end{Lem}

\begin{proof}
  This is a reinterpretation of the definition of $\mathcal{R}^\mathrm{pre}$. The upper
  right corner parametrizes a quintuple $(R^{\sharp_1}, R^{\sharp_2}, x,y,f)$ where $x \in
\scrsginf(\spf R^{\sharp_1+})$ and $y \in \scrsginf(\spf R^{\sharp_2+})$ and $f:A_x \dashrightarrow A_y$ is a formal quasi-isogeny of the underlying formal abelian schemes preserving the \'{e}tale trivializations. The lower right corner parametrizes
  two $G$-torsors with an isomorphism between the underlying $G_V$-torsors. The
  diagonal imposes precisely the condition that this isomorphism of
 $G_V$-torsors comes from an isomorphism of $G$-torsors. Thus the fiber product parametrizes those quintuples $(R^{\sharp_1}, R^{\sharp_2}, x,y,f)$ where $\mathcal{E}(f):\mathcal{E}(A_x) \to \mathcal{E}(A_y)$ is
induced by an isomorphism $\mathbb{L}_{\mathrm{crys},x} \to
\mathbb{L}_{\mathrm{crys},y}$ of $G$-torsors on $X_{(R, R^+)}$.
\end{proof}
To simplify notation, we set 
\[(i^{\ast} \mathcal{R}_V)^{\mathrm{red}}=\mathcal{R}_V^\mathrm{red}
    \times_{\shgvinf \times \shgvinf} (\shginf \times \shginf).\] 
\begin{Cor} \label{Cor:RelationGFromGVRed}
  The diagram
  \[ \begin{tikzcd}
    \mathcal{R}^\mathrm{red} \arrow{r} \arrow{d} & (i^{\ast} \mathcal{R}_V)^{\mathrm{red}} \arrow{d}
    \\ \gisoc \arrow{r}{\Delta} & \gisoc \times_{\gvisoc} \gisoc
  \end{tikzcd} \]
  is Cartesian.
\end{Cor}

\begin{proof}
  We apply sheafification and then reduction to the Cartesian diagram of
  Lemma~\ref{Lem:RelationGFromGV}.
\end{proof}

Hence to understand $\mathcal{R}^\mathrm{red}$, it is enough to understand
$\mathcal{R}_V^\mathrm{red}$.

\begin{Prop} \label{Prop:ReductionIgusaSiegel}
  Let $B$ be a perfect $\fp$-algebra. A point $(x, y) \in (\shgvinf \times
  \shgvinf)(B)$ lies inside
  \[
    \mathcal{R}_V^\mathrm{red} \subseteq \shgvinf \times \shgvinf
  \]
  if and only if there exists a quasi-isogeny between the corresponding abelian
  varieties $A_x, A_y$ preserving the
  \'{e}tale trivializations.
\end{Prop}

\begin{proof}
  Denote by $\mathcal{S}$ the sub-presheaf 
  \[
    \mathcal{S} \subseteq \shgvinf \times \shgvinf
  \]
  with $B$-points given by those pairs $(x, y)$ for which there exists a quasi-isogeny
 $A_x \dashrightarrow A_y$ preserving the \'{e}tale
  trivializations. (Note that a quasi-isogeny preserving the \'{e}tale
  trivializations automatically preserves the polarizations up to a scalar that is locally constant on $B$.) We first
  note that there is a containment
  \[
    \mathcal{S} \subseteq \mathcal{R}_V^\mathrm{red},
  \]
  because if there exists a quasi-isogeny $f \colon A_x \dashrightarrow A_y$,
  then for any perfectoid $B$-algebra $(R, R^+)$, the base change of $f$ is a
  quasi-isogeny over $R^+/\varpi$, again preserving the extra structures.

  On the other hand, we have an isomorphism
  \[
    \mathcal{R}_V^\mathrm{red} \simeq \shgvinf \times_{\gvisoc} \shtgvlocmu
  \]
  coming from the fiber product diagram \eqref{Eq:ReductionDiagram2} for $\gvx$. The composition
  \[
    q \colon \mathcal{S} \hookrightarrow \mathcal{R}_V^\mathrm{red} \xrightarrow{\sim}\shgvinf \times_{\gvisoc} \shtgvlocmu
  \]
  sends a $B$-point $(x, y)$ coming from a quasi-isogeny $f \colon A_x
  \dashrightarrow A_y$ to the pair $(x, \mathbb{D}^{\natural}(A_y))$ together with the
  isomorphism 
  \begin{equation}
    \mathcal{E}(f) \colon \mathcal{E}(A_{x})= \mathbb{D}^{\natural}(A_x)[p^{-1}] \xrightarrow{\sim}
    \mathbb{D}^{\natural}(A_y)[p^{-1}] = \mathcal{E}(A_{y})
  \end{equation}
  Here we use Lemma \ref{Lem:DieudonneTheory} to identify $\shtgvlocmu$ with the
  groupoid of quasi-polarized $p$-divisible groups of height $\dim V$. 

  For the converse, suppose that we are given a $B$-point of the target of $q$. That is, we are given a $B$-point $x$ of $\shgvinf$, a Witt vector $\mathcal{G}_V$ shtuka $\mathcal{P}_V$ with corresponding $W(B)$-module with Frobenius $(M, \phi_M)$ and a quasi-isogeny
 $f:\mathbb{D}^{\natural}(A_x)[p^{-1}] \xrightarrow{\sim} (M[1/p], \phi_M)$ preserving the polarization up to a scalar. By Lemma \ref{Lem:DieudonneTheory}, there is a quasi-polarized $p$-divisible group $(Y,\lambda_Y)$ over $B$ with Dieudonn\'e module corresponding to $(M, \phi_M)$. The quasi-isogeny induces a polarization preserving quasi-isogeny $f:A_x[p^{\infty}] \dashrightarrow Y$, which uniquely determines a weakly polarized abelian scheme up to prime-to-$p$ isogeny $A_y$ over $B$ together with a quasi-isogeny $f: A_x \dashrightarrow A_y$ and an isomorphism $A_y[p^{\infty}]=Y$. There is moreover a unique prime-to-$p$ level structure on $A_y$ making the diagram \eqref{Eq:CommutativeDiagramEtaleIII} commute. This lifts $A_y$ to a point $y:\spec B \to \shgvinf$ and this describes an inverse to the map $q$.
  
  Hence the map $q$ is an isomorphism, and therefore $\mathcal{S} =\mathcal{R}_V^\mathrm{red}$ as subsheaves of $\shgvinf \times \shgvinf$.
\end{proof}

We can now prove Theorem~\ref{Thm:ReductionIgusa}.

\begin{proof}[Proof of Theorem~\ref{Thm:ReductionIgusa}]
  In view of the presentation 
  \[
    \igsinf^\mathrm{red} = [\shginf / \mathcal{R}^\mathrm{red}],
  \]
  it suffices to prove that the subsheaf $\mathcal{R}^\mathrm{red} \subseteq
  \shginf \times \shginf$ consists of those pairs $(x, y) \in (i^{\ast} \mathcal{R}_V)^{\mathrm{red}}$ such that there
  exists a quasi-isogeny $f \colon A_x \dashrightarrow A_y$ for which the
  induced isomorphism of $\mathrm{GL}_V$-torsors $\mathcal{E}(A_{x}) \simeq
  \mathcal{E}(A_{y})$ comes from a (necessarily unique) isomorphism of $G$-torsors
 $\mathbb{L}_{\mathrm{crys},x} \simeq
\mathbb{L}_{\mathrm{crys},y}$. This follows
  immediately from Corollary~\ref{Cor:RelationGFromGVRed} and
  Proposition~\ref{Prop:ReductionIgusaSiegel}.
\end{proof}

\subsubsection{Points of the reduction} Another consequence of the previous analysis is an explicit description of points of $\igsinf^\mathrm{red}$ on a basis of the v-topology on $\affperf$ (see Lemma \ref{Lem:SchemeTheoreticProductOfPoints}).

\begin{Cor} \label{Cor:ReductionIgusaGeometricPoint}
  Let $R$ be a product of absolutely integrally closed valuation rings over $\fp$. Then the
  natural map $\IgsQuot^{\mathrm{red}} \colon \shginf(R) \to
  \igsinf^\mathrm{red}(R)$ identifies $\igsinf^\mathrm{red}(R)$ with the
  quotient of $\shginf(R)$ with respect to the following equivalence relation:
  Two elements $x, y \in \shginf(R)$ are equivalent when there exists a
  quasi-isogeny $f \colon A_x \dashrightarrow A_y$ such that $f$ respects the
  \'{e}tale trivializations $\eta_x, \eta_y$, and the induced isomorphism of
  isocrystals $\mathcal{E}(A_x) \simeq \mathcal{E}(A_y)$ is induced by a
  $($necessarily unique$)$ isomorphism of $G$-isocrystals
 $\mathbb{L}_{\mathrm{crys},x} \simeq \mathbb{L}_{\mathrm{crys},y}$.
\end{Cor}

\begin{proof}
  We proceed as in Section~\ref{Sec:PointsIgusa}. We first note that the map
 $\shtglocmu(R) \to \gisocmu(R)$ is essentially surjective by
  Proposition~\ref{Prop:VSurjective}. From the Cartesian diagram \eqref{Eq:ReductionDiagram2} it
  follows that $\shginf(R) \to \igsinf^\mathrm{red}(R)$ is surjective. Therefore
  we have
  \[
    \igsinf^\mathrm{red}(R) = [\shginf(R) / \mathcal{R}^\mathrm{red}(R)].
  \]
  We now observe from Corollary~\ref{Cor:RelationGFromGVRed} and
  Proposition~\ref{Prop:ReductionIgusaSiegel} that $\mathcal{R}^\mathrm{red}(R)$
  defines the equivalence relation given in the statement.
\end{proof}

\begin{Rem} \label{Rem:IsogenyClasses} For $k=\fpbar$, we see that the quotient $\igsinf(\fpbar)/\gafp$ can be identified with the quotient of $\shginf(\fpbar)$ by the following equivalence relation: Two points $x, y \in \shginf(\fpbar)$ are equivalent when there exists a
  quasi-isogeny $f \colon A_x \dashrightarrow A_y$ such that $f:\mathcal{V}^p_x \to \mathcal{V}^p_y$ respects the \'etale tensors, and such that the induced isomorphism of
  isocrystals $\mathcal{E}(A_x) \simeq \mathcal{E}(A_y)$ is induced by an isomorphism of $G$-isocrystals
 $\mathbb{L}_{\mathrm{crys},x} \simeq \mathbb{L}_{\mathrm{crys},y}$. So we can reasonably think of $\igsinf(\fpbar)/\gafp$ as the set of mod $p$ isogeny classes in $\shginf(\fpbar)$. \footnote{We expect that our definition agrees with the one of \cite[Section 2.3]{KMPS}.} Using Proposition \ref{Prop:IndependenceOfParahoric}, we see that this set does not depend on $K_p$. This is in agreement with the Langlands--Rapoport conjecture, see  \cite{LanglandsRapoport, RapoportReduction, KisinPoints}, which gives a prediction of the set of mod $p$ isogeny classes which does not depend on $K_p$.
\end{Rem}

\begin{Rem} \label{Rem:XiaoZhuShtukasCorrespondence}
    Following \cite[Definition~5.2.8]{XiaoZhu}, we define\footnote{Our definition can readily be seen to agree with the one in loc.\ cit., keeping in mind Remark~\ref{Rem:ShtukasInversion}.}
    \begin{align}
        \shtglocmumu:=\shtglocmu \times_{\gisoc} \shtglocmu.
    \end{align}
    It follows from Theorem \ref{Thm:ReductionIgusa} that there is a natural map
    \begin{align}
        \mathcal{R}^{\mathrm{red}}:=\shginf \times_{\igsinf^{\mathrm{red}}} \shginf \to \shtglocmumu
    \end{align}
    which fits in a $2$-commutative diagram
    \begin{equation}
        \begin{tikzcd}
            \shginf \arrow{d}{\pi_{\mathrm{crys}}^{\mathrm{red}}} & \mathcal{R}^{\mathrm{red}} \arrow{d} \arrow{l} \arrow{r} & \shginf \arrow{d}{\pi_{\mathrm{crys}}^{\mathrm{red}}} \\
            \shtglocmu & \shtglocmumu \arrow{l} \arrow{r} & \shtglocmu
        \end{tikzcd}
    \end{equation}
    where both squares are $2$-Cartesian. Thus $\mathcal{R}^{\mathrm{red}}$ corresponds to the sheaf $\operatorname{Sh}_{\mu | \mu}$ of \cite[Remark 7.3.3]{XiaoZhu}. One can also interpret their conjecture on the existence of exotic Hecke correspondences, see \cite[Hypothesis 7.3.2]{XiaoZhu}, in terms of (the reductions of) Igusa stacks. For details, see \cite[Section 2.7]{vHSemplinerII}. 
\end{Rem}

\section{Functoriality} \label{Sec:Functoriality}
  \input{Section7_Current}
}

\section{Cohomological consequences} \label{Sec:CohomologicaConsequences}
\input{Section8}

\section{The Eichler--Shimura relation} \label{Sec:EichlerShimura}
\input{Eichler-Shimura}

\section{Generic part of the cohomology}\label{Sec:TorsionVanishing} 
\input{Torsion-vanishing}

\bibliographystyle{amsalpha}
\renewcommand{\VAN}[3]{#3}
\renewcommand{\VANDEN}[3]{#3}
\renewcommand\MR[1]{}
\bibliography{references}

\end{document}

%% file: AbstractIgusaFirst.tex
\begin{abstract}
We construct functorial Igusa stacks for all Hodge-type Shimura varieties, proving a conjecture of Scholze and extending earlier results of the fourth-named author for PEL-type Shimura varieties. Using the Igusa stack, we construct a sheaf on $\mathrm{Bun}_G$ that controls the cohomology of the corresponding Shimura variety. We use this sheaf and the spectral action of Fargues--Scholze to prove a compatibility between the cohomology of Shimura varieties of Hodge type and the semisimple local Langlands correspondence of Fargues--Scholze and to generalize the Eichler--Shimura relation conjectured by Blasius--Rogawski to Iwahori level at $p$. When the given Shimura variety is proper, we show moreover that the sheaf is perverse, which allows us to prove new torsion vanishing results for the cohomology of Shimura varieties.
\end{abstract}

%% file: Intro_NewV2.tex
Fix a prime $p$. In their seminal work \cite{CaraianiScholzeCompact}, Caraiani and Scholze prove torsion vanishing results for the generic part of the cohomology of compact unitary Shimura varieties. A key geometric ingredient in their proof is the Hodge--Tate period morphism 
\begin{align*}
	\pi_{\mathrm{HT}}\colon \mathbf{Sh}_{K^p}\gx \to \Fl,
\end{align*}
which was originally defined and studied in \cite{ScholzeTorsion}. Here $\mathbf{Sh}_{K^p}\gx$ is the $p$-adic perfectoid Shimura variety with infinite level at $p$, and $\Fl$ is the corresponding flag variety. As a crucial ingredient in their proof, Caraiani and Scholze show that the pushforward $R \pi_{\mathrm{HT},\ast} \fl$ is perverse in a suitable sense. 

Subsequently, Fargues conjectured in \cite[Section 7]{FarguesConjecture} that the sheaf $R\pi_{\mathrm{HT}, \ast} \ql$ on $\Fl$ can be obtained by pulling back a sheaf on $\bun_G$ along the Beauville--Laszlo map $\mathrm{BL} \colon
\Fl^{\lozenge} \to \bun_G$.\footnote{In fact, Fargues gives a conjectural description of this sheaf on $\bun_G$.}  Scholze, in turn, suggested that this can be explained by the existence of a so-called \emph{Igusa stack} $\mathbf{Igs}_{K^p}\gx$ sitting in a Cartesian diagram
\begin{equation} \label{Eq:IntroCartesianIntro} \begin{tikzcd} 
    \mathbf{Sh}_{K^p}\gx^{\lozenge} \arrow{r}{\pi_{\mathrm{HT}}^\lozenge} \arrow{d}{\IgsQuot} & \Fl^\lozenge \arrow{d}{\mathrm{BL}} \\
    \mathbf{Igs}_{K^p}\gx \arrow{r}{\overline{\pi}_{\mathrm{HT}}} & \bun_G.
\end{tikzcd}
\end{equation}
This conjectural v-stack receives its name because, in a sense, it interpolates Igusa varieties associated to different Newton strata, which are themselves important geometric objects in the study of the global Langlands conjecture, see e.g., \cite{HarrisTaylor}, \cite{MantovanPEL}, \cite{ShinIgusa}, \cite{CaraianiScholzeCompact, CaraianiScholzeNoncompact}. Given the Cartesian diagram \eqref{Eq:IntroCartesianIntro}, the aforementioned conjecture of Fargues follows by applying base change to see that the sheaf $R\pi_{\mathrm{HT},\ast} \ql$ is the pullback of $R\overline{\pi}_{\mathrm{HT},\ast}\ql$ along $\mathrm{BL}$. 

In this paper, we realize this vision by proving the existence of $\igs$ for all Hodge-type Shimura varieties, without restriction on $p$, generalizing results of one of us (MZ) for PEL-type AC Shimura varieties \cite{ZhangThesis}. 
\begin{MainTheorem}[Theorem \ref{Thm:IntroIgusaGeneric}] \label{Thm:IntroIgusaGenericAbbreviated}
For all Shimura data $\gx$ of Hodge type, there is a v-stack $\mathrm{Igs}_{K^p}\gx$ on $\perf$ fitting into the Cartesian diagram \eqref{Eq:IntroCartesianIntro} with $\mathbf{Sh}_{K^p}\gx^{\lozenge}$ replaced by the good reduction locus $\mathbf{Sh}_{K^p}\gx^{\circ,\lozenge} \subset \mathbf{Sh}_{K^p}\gx^{\lozenge}$.\footnote{It would be appropriate to write $\operatorname{Igs}_{K^p}\gx^{\circ}$ for our $\operatorname{Igs}_{K^p}\gx$. However, since only this variant of the Igusa stack appears in this paper, we have opted to omit the superscript.}
\end{MainTheorem}
From this theorem we obtain several cohomological applications. In particular, we prove a compatibility between the cohomology of Shimura varieties of Hodge type and the Fargues--Scholze local Langlands correspondence, and we generalize the classical Eichler--Shimura relation conjectured by Blasius and Rogawski at hyperspecial level. When the given Shimura variety is compact, we moreover show that $R\overline{\pi}_{\mathrm{HT},\ast} \fl$ is perverse for the perverse $t$-structure on $\bungmu$ given by its Harder--Narasimhan stratification. In combination with results of Hamann and Hamann--Lee, see \cite{Hamann22, Hamann-Lee}, this leads to new torsion vanishing results for the cohomology of some compact Shimura varieties. Let us now explain our results in detail.

\subsection{Main geometric results} \label{sub:GeometricResults} Let $\gx$ be a Shimura datum of Hodge type with reflex field $\mathsf{E}$. Let $p$ be a prime number, let $v$ be a prime of $\mathsf{E}$ above $p$ and set $E=\mathsf{E}_v$ and $G=\mathsf{G} \otimes \qp$. Let $K^p \subset \gafp$ be a neat compact open subgroup. For a compact open subgroup $K_p \subset G(\qp)$, we consider the Shimura variety $\mathbf{Sh}_K\gx$ of level $K=K_pK^p$ over $E$. There is an open subspace (the ``good reduction locus'')
\begin{align}
    \mathbf{Sh}_K\gx^{\circ, \mathrm{an}} \subset \mathbf{Sh}_K\gx^{\mathrm{an}}
\end{align}
of the adic space $\mathbf{Sh}_K\gx^{\mathrm{an}}$ associated with $\mathbf{Sh}_K\gx$, which is compatible with changing the level $K$. Let $\perf$ denote the category of perfectoid spaces of characteristic $p$. We let $\mathbf{Sh}_{K^p}\gx^{\circ, \lozenge} \subset \mathbf{Sh}_{K^p}\gx^{\lozenge}$  be the corresponding objects with infinite level at $p$, considered as v-sheaves on $\perf$. These come equipped with a Hodge--Tate period map 
\begin{align}
    \pi_{\mathrm{HT}}^\lozenge\colon \mathbf{Sh}_{K^p}\gx^{\lozenge} \to \operatorname{Gr}_{G, \mu^{-1}},
\end{align}
where $\operatorname{Gr}_{G, \mu^{-1}}$ is the Schubert cell in the $\mathbf{B}_\mathrm{dR}^+$-affine Grassmannian corresponding to the inverse of the Hodge cocharacter $\mu$ (this is $\Fl^\lozenge$ from \eqref{Eq:IntroCartesianIntro}). Let $\bung$ denote the v-stack on $\perf$ of $G$-bundles on the Fargues--Fontaine curve. Finally we need the Beauville--Laszlo map $\operatorname{BL}:\operatorname{Gr}_{G, \mu^{-1}} \to \bun_G$, whose image is an open substack $\bungmu$ of $\bung$, see Section \ref{Sec:BGMU}. 
\begin{mainThm}[Theorem \ref{Thm:HodgeMain}, \ref{Thm:IgusaDualizingComplex}, \ref{Thm:RationalFunctoriality}] \label{Thm:IntroIgusaGeneric}
There is an Artin v-stack $\mathrm{Igs}_{K^p}\gx$ on $\perf$ sitting in a Cartesian diagram of v-stacks on $\perf$
    \begin{equation}
        \begin{tikzcd}
        \mathbf{Sh}_{K^p}\gx^{\circ, \lozenge} \arrow{r}{\pi_{\mathrm{HT}}^\lozenge} \arrow{d}{\IgsQuot}& \operatorname{Gr}_{G, \mu^{-1}} \arrow{d}{\mathrm{BL}} \\
        \mathrm{Igs}_{K^p} \gx \arrow{r}{\overline{\pi}_\mathrm{HT}} & \bungmu.
        \end{tikzcd}
\end{equation}
Furthermore, for all $\ell \not=p$, the stack $\mathrm{Igs}_{K^p}\gx$ is $\ell$-cohomologically smooth of dimension $0$, and $\mathrm{Igs}_{K^p}\gx \times_{\spd \fp} \spd \fpbar$ has constant dualizing sheaf. Moreover, the formation of $\mathrm{Igs}_{K^p}\gx$ is functorial for morphisms of tuples $(\mathsf{G}, \mathsf{X}, v, K^p)$.\footnote{The Igusa stack depends on the choice of place $v |p$ of $\mathsf{E}$, although we omit this from the notation. The functoriality we prove is for morphisms $f:\gx \to \gxp$ with $f(K^p) \subset K^{p'}$ such that under the induced inclusion $\mathsf{E}' \subset \mathsf{E}$ the place $v$ of $\mathsf{E}$ induces the place $v'$ of $\mathsf{E}'$.}
\end{mainThm}
This confirms a conjecture of Scholze on the good reduction locus, cf. \cite[Conjecture 1.1.(4)]{ZhangThesis}. For $\gx$ a Shimura datum of PEL type such that $p$ is an unramified prime for the PEL datum, this result was proved by one of us (MZ) as part of \cite[Theorem 1.3]{ZhangThesis}.\footnote{Although strictly speaking the results in \cite{ZhangThesis} are for the PEL type moduli spaces constructed by Kottwitz, which are generally disjoint unions of finitely many Shimura varieties.} 

If $\mathbf{Sh}_K\gx$ is proper, then $\mathbf{Sh}_{K^p}\gx^{\circ, \lozenge}=\mathbf{Sh}_{K^p}\gx^{\lozenge}$ and we recover the Cartesian diagram \eqref{Eq:IntroCartesianIntro}. For non-proper Shimura varieties, we expect that there is a version of Theorem \ref{Thm:HodgeMain} for the minimal compactification of the Shimura variety; see \cite[Theorem 1.3]{ZhangThesis} for the PEL type case.

\subsection{Main cohomological results} Let $\gx$ be a Shimura datum of Hodge type with reflex field $\mathsf{E}$, fix a place $v$ of $\mathsf{E}$ above a rational prime number $p$ and set $E=\mathsf{E}_v$ and $G=\mathsf{G} \otimes \qp$ as before. Choose a prime $\ell \not=p$,\footnote{Although our current applications only concern the $\ell$-adic/mod $\ell$ cohomology of Shimura varieties with constant coefficients, Theorem~\ref{Thm:IntroIgusaGeneric} is a geometric result and hence is robust under changing coefficient systems. We therefore expect that our method can be applied to more general \'etale local systems. In fact, it might be possible to take coherent sheaves as coefficients. The latter might shed light on $p$-adic/mod-$p$ local Langlands, provided that a suitable framework for these sheaves is available.} let $\Lambda=\zlbar$, choose $\sqrt{p} \in \Lambda$. Let $K^p \subset \gafp$ be a neat compact open subgroup. For a compact open subgroup $K_p \subset G(\qp)$ we let $K=K_pK^p$ and consider the cohomology complex $R\Gamma_{\text{\'et}}(\mathbf{Sh}_{K}\gx_{\overline{E}}, \Lambda)$. This has a right action of the local Hecke algebra $\mathcal{H}_{K_p}=H_{K_{p}}(G(\mathbb{Q}_{p}), \Lambda)$; we write $\mathcal{Z}_{K_{p}}$ for its center. Let $L \in \{\flbar, \qlbar\}$, let $\chi:\mathcal{Z}_{K_{p}} \to L$ be a character and consider the
representation of the Weil group $W_E$ of $E$ given by (where $d=\dim \mathsf{X}$ and $(d/2)$ denotes a half Tate-twist)
\begin{align*}
    W^i(\chi):=H^i_{\text{\'et}}(\mathbf{Sh}_{K}\gx_{\overline{E}}, \Lambda)(\tfrac{d}{2}) \otimes_{\mathcal{Z}_{K_{p}}, \chi} L
\end{align*}
for some $i$. Finally let us write 
\begin{align*}
    \phi_{\chi}:W_{\qp} \to \dualgrp{G}(L) \rtimes W_{\qp} 
\end{align*}
for the Fargues--Scholze $L$-parameter associated to $\chi$, and 
\begin{align*} 
    r_{\mu}: \dualgrp{G}(L) \rtimes W_E\to \operatorname{GL}(V_{\mu})(L)
\end{align*}
for the representation determined by $\mu$, see Section \ref{subsub:RMu}. The following is our first main cohomological result, and we remark that its statement and proof are inspired by the analogous result for local Shimura varieties, see \cite[Theorem 6.2]{Koshikawa}.
\begin{mainThm}[Theorem~\ref{Thm:CompatibilityFarguesScholze}]\label{Thm:IntroCompatibilityFarguesScholze}
If the order of $\pi_0(Z(G))$ is coprime to $\ell$, then each irreducible $L$-linear $W_E$-representation occurring as a subquotient of $W^i(\chi)$ also occurs as a subquotient of $r_{\mu} \circ \restr{\phi_{\chi}}{W_E}$. 
\end{mainThm}
\begin{Rem}
The assumption on the order of $\pi_0(Z(G))$ in Theorem \ref{Thm:IntroCompatibilityFarguesScholze} is very mild. Many typical examples of Shimura varieties of Hodge type have connected center, in which case the assumption is vacuous. 
\end{Rem}

Theorem \ref{Thm:IntroCompatibilityFarguesScholze} generalizes \cite[Theorem 6.2]{Koshikawa}, which additionally assumes that $\mathbf{Sh}_K\gx$ is proper and that Mantovan's formula holds. If we specialize Theorem \ref{Thm:IntroCompatibilityFarguesScholze} to hyperspecial $K_p$ and $L=\qlbar$, then this is closely related to the Eichler--Shimura relation conjectured by Blasius--Rogawski (see \cite[Section 6]{BlasiusRogawski}), which states that for a certain Hecke polynomial $H_{\mu}(X) \in \mathcal{H}_{K_p}[X]$ the evaluation $H_{\mu}(\operatorname{Frob})$ in a Frobenius element of $W_E$ acts trivially on $H^i_{\text{\'et}}(\mathbf{Sh}_{K}\gx_{\overline{E}}, \qlbar)$ for all $i$. The conjecture of Blasius--Rogawski was proved earlier by Wu, see \cite[Corollary 1.2]{WuSEqualsT} by proving the $S=T$ conjecture of Xiao--Zhu, see \cite[Conjecture 7.3.13]{XiaoZhu}. There are many earlier partial results towards the conjecture, see e.g.\ \cite{LeeEichlerShimura}, \cite{Koshikawa}, \cite{LiEichlerShimura}, \cite{KoskivirtaEichlerShimura}, \cite{BueltelWedhorn}, \cite{MoonenSerreTate}, \cite{BueltelEichlerShimura}, and \cite{ChaiFaltings}. 

We record the following version of the Eichler--Shimura relation at Iwahori level (cf.\ \cite[Theorem 1.1]{Koshikawa}, which is the corresponding result for local Shimura varieties). To state it, we note that for an Iwahori subgroup $K_p' \subset K_p$ there is a natural Bernstein morphism $\mathcal{H}_{K_p} \to \mathcal{H}_{K_p'}$, see e.g.\ \cite{Boumasmoud}, identifying the image with the center of $\mathcal{H}_{K_p'}$. In particular, we can consider $H_{\mu}$ as an element of $\mathcal{H}_{K_{p}'}[X]$; this does not depend on the choice of hyperspecial $K_p$ containing $K_p'$. 

\begin{mainThm}[Theorem~\ref{Cor:EichlerShimura}] \label{Thm:IntroEichlerShimura}
Suppose that $G$ is unramified and let $K_p' \subset G(\qp)$ be an Iwahori subgroup. If the order of $\pi_0(Z(G))$ is coprime to $\ell$, then the inertia group $I_E \subset W_E$ acts unipotently on $H^i_\mathrm{\acute{e}t}(\mathbf{Sh}_{K'}\gx_{\overline{E}}, \Lambda)$ for all $i$. Moreover, the action of a Frobenius element $\operatorname{Frob} \in W_E$ on $R \Gamma_\mathrm{\acute{e}t}(\mathbf{Sh}_{K'}\gx_{\overline{E}}, \Lambda)$ satisfies $H_{\mu}(\operatorname{Frob})=0$.
\end{mainThm}
The assumption that $G$ is unramified means that $G_{\qp}$ is quasi-split and split over an unramified extension. Both of these assumptions have been removed in recent work of van den Hove, see \cite{vdHoveEichlerShimura} and Remark \ref{Rem:UnramifiedNecessary}.
\begin{Rem} \label{Rem:DatLanard}
    If $G$ is quasi-split and tamely ramified, then we can combine Theorem \ref{Thm:IntroCompatibilityFarguesScholze} with \cite[Corollary 1.1.1]{DatLanard}\footnote{We note that the assumption of \cite[Corollary 1.1.1]{DatLanard} holds by \cite[Corollary 6.3 and its proof]{ScholzeMotivicGeometrization}.} to get the following result in the context of Theorem \ref{Thm:IntroCompatibilityFarguesScholze}: If $\chi$ is the character corresponding to a depth zero irreducible smooth representation of $G(\qp)$ over $L$, then the $W_E$ action on $W^i(\chi)$ is tamely ramified. 
\end{Rem}

\subsubsection{Proof of cohomological results} Let us sketch the proof of our cohomological results. We fix a torsion $\zl$-algebra $\Lambda$ and a square root of $p$ in $\Lambda$. We define a complex of sheaves in $D(\bun_{G,\mu^{-1}, \fpbar}, \Lambda)$ by 
\begin{align*}
    \mathcal{F}:=R \overline{\pi}_{\mathrm{HT},\fpbar, \ast} \Lambda.
\end{align*}
We write $j_{\mu, \fpbar}$ for the inclusion of $\bun_{G,\mu^{-1},\fpbar} \hookrightarrow \bun_{G,\fpbar}$; then $j_{\mu,\fpbar,!} \mathcal{F}$ defines an object of $D(\bun_{G,\fpbar},\Lambda)$.

The representation $V_{\mu}$ defines a Hecke operator
\begin{align}
    T_{\mu}:D(\bun_{G, \fpbar}, \Lambda) \to D(\bun_{G, \fpbar}, \Lambda)^{B W_E},
\end{align}
where the superscript $B W_E$ roughly denotes objects with an action of $W_E$. Pulling back along the inclusion $i_1$ of the neutral stratum $\bun_{G, \fpbar}^{1}$ into $\bun_{G,\fpbar}$, we get a complex of $\Lambda$-modules with commuting $G(\qp)$ and $W_E$-actions. We then prove the following theorem.

\begin{mainThm}[Theorem \ref{Thm: WeilCohoShiVar}] \label{Thm:IntroWeilCohShimVar}
There is a $G(\qp) \times W_E$-equivariant isomorphism
\begin{align}
  R \Gamma_\mathrm{\acute{e}t}(\mathbf{Sh}_{K^p}\gx_{\overline{E}}, \Lambda) \simeq i_{1,\fpbar}^{\ast} T_\mu j_{\mu,\fpbar,!} (\mathcal{F}[-d])(-\tfrac{d}{2}),
\end{align}
where $[-d]$ denotes a shift and $(-\tfrac{d}{2})$ a Tate twist. 
\end{mainThm}

To be more precise, Theorem \ref{Thm:IntroWeilCohShimVar} is proved in two steps. First we compare $i_{1,\fpbar}^{\ast} T_\mu j_{\mu,\fpbar,!} \mathcal{F}$ to the cohomology of $R \Gamma_{\text{\'et}}(\mathbf{Sh}_{K^p}\gx^{\circ, \lozenge}_{C}, \Lambda)$, where $C$ is the completion of $\overline{E}$, using the six functor formalism of \cite{EtCohDiam} and \cite{FarguesScholze}. We then compare the result to the cohomology of the Shimura variety using the work of Lan--Stroh \cite{LanStrohII}, see Lemma \ref{Prop:CohoGoodRed}. We remark that the main subtlety in the first comparison lies, somewhat surprisingly, in proving the $W_E$-equivariance, or more precisely the equivariance for the action of Frobenius.

\subsubsection{Proofs of Theorems \ref{Thm:IntroCompatibilityFarguesScholze} and \ref{Thm:IntroEichlerShimura}} We now sketch our proofs of Theorems \ref{Thm:IntroCompatibilityFarguesScholze} and \ref{Thm:IntroEichlerShimura}. Let $\mu$ be the $G(\qpbar)$-conjugacy class of cocharacters determined by $\mathsf{X}$ and the place $v$ of $\mathsf{E}$ over $p$. The cocharacter $\mu$ determines a $W_E$-equivariant vector bundle $\mathcal{V}_{\mu}$ on the stack $\loc$ of $\dualgrp{G}$-valued $L$-parameters of $W_{\qp}$ (as in \cite{DHKMModuli}), where $\dualgrp{G}$ is the Langlands dual group of $G$ over $\Lambda$ equipped with its action by the Weil group $W_{\qp}$. Using Theorem \ref{Thm:IntroWeilCohShimVar} and the spectral action of \cite[Section IX]{FarguesScholze}, we deduce that the algebra of endomorphisms of $\mathcal{V}_{\mu}$ acts on $R\Gamma_{\text{\'et}}(\mathbf{Sh}_{K}\gx_{\overline{E}}, \Lambda)$ compatibly with the Hecke action and Weil group action, see Corollary~\ref{Cor:EllAdicSpectralAction} for the precise statement. That such an action exists is predicted by the philosophy of Zhu, see \cite[Conjecture 4.60]{ZhuCoherent}. Given this action, we then prove Theorem~\ref{Thm:IntroCompatibilityFarguesScholze} by adapting arguments from \cite{Koshikawa}. 

To prove Theorem~\ref{Thm:IntroEichlerShimura}, we need to prove some results about the unramified local Langlands correspondence in families, which are presumably well known to experts; this happens in Section \ref{Sub:HeckePolynomials}.

We also prove a Mantovan product formula for $R \Gamma_{\text{\'et}}(\mathbf{Sh}_{K^p}\gx_{\overline{E}}, \Lambda)$, see Theorem \ref{Thm:MantovanFormula}. This expresses the cohomology of the Shimura variety in terms of the cohomology of Igusa varieties and the cohomology of local Shimura varieties, and generalizes \cite{Mantovan, MantovanPEL, Hamacher-Kim, KoshikawaGeneric, Hamann-Lee}. 

\subsection{Torsion vanishing} Recall that a crucial ingredient in the proof of the main torsion vanishing result in \cite{CaraianiScholzeCompact} is a perversity result for $R \pi_{\mathrm{HT}, \ast} \flbar$, see \cite[Proposition 6.1.3]{CaraianiScholzeCompact}. As a corollary to Theorem \ref{Thm:IntroIgusaGeneric}, we obtain a perversity result for $\mathcal{F}=R \overline{\pi}_{\mathrm{HT}, \fpbar, \ast} \flbar$, where $\overline{\pi}_\mathrm{HT}$ is the map $\mathrm{Igs}_{K^p}\gx \to \bungmu$ from Theorem \ref{Thm:IntroIgusaGeneric} and $\ell \neq p$. 

\begin{mainThm}[Theorem \ref{Thm: Perversity}] \label{Thm:IntroPerversity}
If $\mathbf{Sh}_K\gx \to \spec E$ is proper, then for $\ell \neq p$ the complex
    \begin{align}
        \mathcal{F} \in D(\bun_{G,\mu^{-1},\fpbar}, \flbar)
    \end{align}
    lies in the heart of the perverse $t$-structure on $D(\bun_{G,\mu^{-1},\fpbar}, \flbar)$.
\end{mainThm}
Here we use the perverse $t$-structure on $D(\bun_{G,\mu^{-1},\fpbar}, \flbar)$ given by its Harder--Narasimhan stratification, see Section~\ref{Sec:SheavesBunG} for details. To prove Theorem \ref{Thm:IntroPerversity}, we use the Cartesian diagram of Theorem \ref{Thm:IntroIgusaGeneric} to show that the restriction of $\mathcal{F}$ to a Newton stratum $\bun_{G}^{[b]} \subset \bun_{G}$ can be identified with the \'etale cohomology of a (perfect) Igusa variety for that Newton stratum. The first half of our perversity result then follows from Artin vanishing once we show that Igusa varieties are affine of the expected dimension, and the other half follows from our computation of the dualizing sheaf of $\mathrm{Igs}_K\gx$ in Theorem \ref{Thm:IntroIgusaGeneric} along with Verdier duality. The affineness is proved by Mao \cite{MaoCompact}. We compute the dimension of Igusa varieties using the Igusa stack diagram, see Proposition \ref{Prop:CentLeafDim}; this could be of independent interest.

\begin{Rem}
It is necessary to assume that $\mathbf{Sh}_K\gx \to \spec E$ is proper for the conclusion of Theorem \ref{Thm:IntroPerversity} to hold. For non-proper Shimura varieties, we only expect a semi-perversity result for $\mathcal{F}$. This is compatible with \cite[Conjecture 6.6]{Hamann-Lee} about vanishing results for the cohomology of Shimura varieties, see the next section. 
\end{Rem}

Recall from Theorem \ref{Thm:IntroWeilCohShimVar} that the \'etale cohomology of the Shimura variety can be computed from the perverse sheaf $R \overline{\pi}_{\mathrm{HT},\ast} \flbar$ by applying the Hecke operator $T_{\mu}$. The cohomology of the Shimura variety is typically not concentrated in a single degree, which reflects the fact that $T_{\mu}$ is typically not $t$-exact for the perverse $t$-structure. However, if we localize at a semisimple $L$-parameter $\phi$ such that $T_{\mu}$ \emph{is} $t$-exact on the $\phi$-localized category of sheaves on $\bun_G$, then we find that the cohomology of the Shimura variety, localized at $\phi$, \emph{is} concentrated in middle degree. We learned this observation from Hamann--Lee \cite{Hamann-Lee}, and we will make it precise below. \smallskip

Let $K_p \subset G(\qp)$ be a hyperspecial subgroup and let $\mathcal{H}_{K_p}$ denote the $\flbar$-valued spherical Hecke algebra for $G(\qp)$ with respect to $K_p$. Let $\ebar$ be an algebraic closure of $\mathsf{E}$, then $\mathcal{H}_{K_p}$ acts on the cohomology $H^i_{\text{\'et}}(\mathbf{Sh}_K\gx_{\ebar},\flbar)$ for all $K^p \subset \gafp$. Let $\mathfrak{m} \subset \mathcal{H}_{K_p}$ be a maximal ideal and let $\phi_{\mathfrak{m}}$ be the associated semisimple $L$-parameter. 
\begin{mainThm}[Theorem \ref{Thm: TorsionVanishing}] 
\label{Thm:IntroTorsionVanishing}
Assume moreover that $\mathbf{Sh}_K\gx \to \spec E$ is proper and that $\ell$ is coprime to the pro-order of $K_p$. If $\phi_{\mathfrak{m}}$ satisfies Assumption \ref{Assump:tExact}, then 
\begin{align}
    H^i_\mathrm{\acute{e}t}(\mathbf{Sh}_K\gx_{\ebar},\flbar)_{\mathfrak{m}}=0
\end{align}
unless $i=d$, where $d$ is the dimension of $\mathbf{Sh}_K\gx$.
\end{mainThm}
Assumption \ref{Assump:tExact} for $\phi_{\mathfrak{m}}$ states that the Hecke operator $T_{\mu}$ is $t$-exact on the $\phi_{\mathfrak{m}}$-localized category $D(\bun_G,\flbar)_{\phi_{\mathfrak{m}}}$, see Section \ref{Sec:LocalizedCategories}. This is conjectured to be true if $\phi_{\mathfrak{m}}$ is of weakly Langlands--Shahidi type in the sense of \cite[Definition 6.2]{Hamann-Lee}, see \cite[Conjecture 6.4]{Hamann-Lee}. 

The conjecture is proved in \cite[Theorem 4.20]{Hamann-Lee}, under a somewhat long list of assumptions on $\phi_{\mathfrak{m}}$, $G$ and $\ell$. If $G$ is a split reductive group, then (a version of) the conjecture holds under \cite[Assumption 7.5]{Hamann22}; we prove this following a suggestion of Hamann, see Proposition \ref{Prop:tExact}. The latter assumption is known for split groups of type $A$ and $C_2$.
\begin{Rem}
    Theorem \ref{Thm:IntroTorsionVanishing} generalizes the compact case of \cite[Theorem 1.17]{Hamann-Lee}. In particular, it is a generalization of \cite[Theorem 1.1]{CaraianiScholzeCompact}.
\end{Rem}
    As explained in \cite[Section 5.2]{Hamann-Lee}, Theorem \ref{Thm: TorsionVanishing} implies a torsion vanishing result for Shimura varieties $\gxtwo$ of abelian type with the same connected Shimura datum as $\gx$. For example, if $\gtwoder$ is simply connected, then there is always a Shimura datum of Hodge type $\gx$ with the same connected Shimura datum as $\gxtwo$. In Section \ref{Sec:Example}, we work out an explicit example of unconditional torsion vanishing results for compact abelian type Shimura varieties of type $C_2$.

\subsection{Integral refinement of Theorem \ref{Thm:IntroIgusaGeneric}} 
Rather than proving Theorem \ref{Thm:IntroIgusaGeneric} directly, we instead derive it from a stronger result on the level of integral models. For a parahoric model $\mathcal{G}$ of $G$ over $\zp$ we write $K_p=\mathcal{G}(\zp) \subset G(\qp)$, and for $K^p \subset \gafp$ a neat compact open subgroup we write $K=K_pK^p$. Let $\mathcal{O}_E$ be the ring of integers of $E$, then the Shimura variety $\mathbf{Sh}_K\gx$ of level $K$ over $\mathsf{E}$ has an integral model $\mathscr{S}_K\gx$ over $\mathcal{O}_E$, uniquely characterized by \cite[Conjecture 4.2.2]{PappasRapoportShtukas}, see \cite[Theorem I]{Companion} and \cite[Theorem 4.5.2]{PappasRapoportShtukas}.

Let $\perf$ denote the category of perfectoid spaces over $\fp$ and let $\shtgmu \to \spd \mathcal{O}_E$ denote the stack on $\perf$ of $\mathcal{G}$-shtukas bounded by $\mu$. As part of the characterization of $\mathscr{S}_K\gx$ there is a morphism
\begin{align}
    \pi_{\mathrm{crys}, \smallmcG}:\mathscr{S}_K\gx^{\diamond} \to \shtgmu,
\end{align}
where $\mathscr{S}_K\gx^{\diamond}$ is the v-sheaf over $\spd \mathcal{O}_E$ associated with the $p$-adic completion of $\mathscr{S}_K\gx$. We moreover recall the morphism $\mathrm{BL}^{\circ}:\shtgmu \to \bun_G$. 
\begin{mainThm}[Theorem~\ref{Thm:HodgeMain}, Remark~\ref{Rem:HodgeMainQuasiParahoric}] \label{Thm:IntroIgusa}
For every parahoric model $\mathcal{G}$ of $G$ over $\zp$ there is a morphism $\IgsQuotInt:\mathscr{S}_K\gx^{\diamond} \to \operatorname{Igs}_{K^p}\gx$ which fits in a Cartesian diagram
    \begin{equation} \label{Eq:IntroCartesianIntegral}
        \begin{tikzcd}
        \mathscr{S}_K\gx^{\diamond} \arrow{r}{\pi_{\mathrm{crys}, \smallmcG}} \arrow{d}{\IgsQuotInt}& \shtgmu \arrow{d}{\mathrm{BL}^{\circ}} \\
        \mathrm{Igs}_{K^p}\gx \arrow{r}{\overline{\pi}_\mathrm{HT}} & \bungmu.
        \end{tikzcd}
\end{equation}
\end{mainThm}
We deduce Theorem \ref{Thm:IntroIgusaGeneric} from Theorem \ref{Thm:IntroIgusa} by pulling back the top row of \eqref{Eq:IntroCartesianIntegral} along
\begin{align}
    \operatorname{Gr}_{G, \mu^{-1}} \to \left[\operatorname{Gr}_{G, \mu^{-1}}/\underline{K_p} \right] \simeq \operatorname{Sht}_{\mathcal{G},\mu, \spd E} \to \shtgmu,
\end{align}
and using the fact that $\mathscr{S}_K\gx^{\diamond} \times_{\spd \mathcal{O}_E} \spd E$ gives $\mathbf{Sh}_K\gx^{\circ, \lozenge}$, see Lemma \ref{Lem:PotentiallyCrystalline}. When $\gx$ is a Shimura datum of PEL type and $K_p$ is hyperspecial, then Theorem \ref{Thm:IntroIgusa} was proved by one of us (MZ) as part of \cite[Theorem 1.3]{ZhangThesis}.

\begin{Rem}
     Let us point out that Theorem~\ref{Thm:IntroIgusa} requires no restriction on the prime $p$ and no technical assumption on either of the groups $G$ or $\mathcal{G}$. In particular, we do not need to assume that $G$ splits over a tamely ramified extension or that $\mathcal{G}$ is a stabilizer parahoric (both of which are common assumptions in the literature on integral models of Hodge-type Shimura varieties). This level of generality is enabled by the construction of canonical integral models of all Shimura varieties of Hodge type. This is done for stabilizer parahorics $\mathcal{G}$ in \cite{PappasRapoportShtukas}, and in general in our companion paper \cite{Companion}. 
\end{Rem}

Let us briefly sketch the proof of Theorem~\ref{Thm:IntroIgusa}. For this, it will be convenient for us to work with $\scrsdinf$, the inverse limit of $\scrsd$ over compact open subgroups $K^p \subset \gafp$. We then define $\mathrm{Igs}\gx$ as the quotient of $\scrsdinf$ by an equivalence relation (so it is a sheaf, as opposed to $\mathrm{Igs}_{K^p}\gx$, which will genuinely be a stack), see Definition~\ref{Def:IgusaStack}.

To describe this equivalence relation, we make use of the family of abelian varieties $A \to \scrsdinf$ coming from a choice of Hodge embedding $\gx \to \gvx$, see Section \ref{Sec:HodgeType}. Roughly speaking, we say that two points $x,y:\spd (R,R^+) \to \scrsdinf$ are equivalent if there is a quasi-isogeny between the abelian varieties $A_x$ and $A_y$ over $R^+/\varpi$ respecting the extra structures arising from the embedding ${\sf G} \hookrightarrow {\sf G}_V$, for some pseudouniformizer $\varpi$ of $R^+$. 

To prove the fiber product formula for $\mathrm{Igs}\gx$, we make crucial use of Rapoport--Zink uniformization for $\scrsdinf$ as proved by 
Pappas--Rapoport and Gleason--Lim--Xu, see \cite{PappasRapoportShtukas, GleasonLimXu}. Roughly speaking, this uniformization result can be used to show the fiber product formula one Newton stratum at a time, or on the level of rank one geometric points, see Corollary~\ref{Cor:RankOneGeometricBijective}.

To complete the proof, it suffices to work locally in the v-topology, and so we reduce to checking the diagram is Cartesian on $S$-points, where $S$ is a so-called ``product of geometric points'', see Section \ref{Sec:ProductOfPointsProof}, in particular Proposition~\ref{Prop:ProdOfGeomPointsBijective}. At this point, we make important use of Theorem \ref{Thm:IntroIgusa} in the case of $\gx=\gvx$; this was proved in earlier work of one of us (MZ) in \cite{ZhangThesis}. We combine this fact with a subtle technical argument to reduce from $S$-points to the case of rank one geometric points. 

\begin{Rem}
    While the proof of Theorem \ref{Thm:IntroIgusa} is in many ways inspired by the proof of \cite[Theorem 1.3]{ZhangThesis}, the arguments differ significantly. The fundamental difficulty in moving to the Hodge-type case lies in the fact that, unlike the PEL-type case, integral models of Hodge-type Shimura varieties are not moduli spaces. As a result, we do not have access to explicit arguments using formal abelian schemes, and instead we must make use of recent advances in the $p$-adic geometry of Shimura varieties such as the results of \cite{PappasRapoportShtukas} and \cite{GleasonLimXu} mentioned above. 
\end{Rem}

\subsubsection{Functoriality} From its construction, it is not a priori obvious that the Igusa stack is independent of the choice of $K_p$ and the choice of Hodge embedding $\gx \to \gvx$. In order to prove independence from these choices we prove the functoriality stated in Theorem \ref{Thm:IntroIgusaGeneric}. 

For this we first prove functoriality in the case where the given morphism of Shimura data $\gx \to \gxp$ extends to a morphism of parahoric group schemes $\mathcal{G} \to \mathcal{G}'$ (see Corollary~\ref{Cor:IntegralFunctoriality}). We then prove that the transition morphisms between Igusa stacks at varying level are isomorphisms (Proposition \ref{Prop:Forgetful}) and apply contractibility of the Bruhat--Tits buildings to prove independence of $\mathrm{Igs}_{K^p}\gx$ of the choice of the parahoric subgroup, see Proposition~\ref{Prop:IndependenceOfParahoric}. This then allows us to extend functoriality to the case of an arbitrary morphism of Hodge-type Shimura data, see Theorem~\ref{Thm:RationalFunctoriality}. Since the functoriality holds for arbitrary choices of Hodge embeddings for $\gx$ and $\gxp$, we can then apply it to the identity $\gx \to \gx$ to deduce that the Igusa stack is independent of the choice of Hodge embedding. 

We expect that our functoriality results will be a crucial ingredient in extending our construction of Igusa stacks from Shimura varieties of Hodge type to Shimura varieties of abelian type. Indeed, Shimura varieties of abelian type can roughly be constructed by taking a quotient by Deligne's group $\mathscr{A}(\g)$ of a disjoint union of copies of a connected component of an associated Hodge-type Shimura variety, and the functoriality proven here allows us to construct an action of $\mathscr{A}(\g)$ on $\mathrm{Igs}\gx$. We remark that only a smaller integral variant acts on the integral models $\scrs_{K_p}\gx$.

\subsubsection{Reduction and isogeny classes} In Theorem \ref{Thm:ReductionIgusa}, we give an explicit description of the stack $\mathrm{Igs}\gx^{\mathrm{red}}$ on the category of perfect $\fp$-algebras, given by sending $R$ to $\mathrm{Igs}\gx(\spd R)$, see Theorem \ref{Thm:ReductionIgusa}. This stack fits in a Cartesian diagram involving the perfect special fiber of the Shimura variety, the stack of Witt vector $\mathcal{G}$-shtukas in the style of Xiao--Zhu (see \cite{XiaoZhu}, \cite{ShenYuZhang}) and the stack of $G$-isocrystals. Both these results are used in work of Sempliner and one of us (PvH) to reinterpret a conjecture of \cite{XiaoZhu} in terms of our Igusa stacks, see Remark \ref{Rem:XiaoZhuShtukasCorrespondence} and \cite{vHSemplinerII}. We moreover use it to show that $\gafp$-orbits in $\mathrm{Igs}\gx(\spd \fpbar)$ are in bijection with mod $p$ isogeny classes in $\scrs_K\gx(\fpbar)$, see Remark \ref{Rem:IsogenyClasses}. 

An important input in the proof of Theorem \ref{Thm:ReductionIgusa} is Proposition \ref{Prop:VSurjective}, which may be of independent interest. This proposition shows that $\shtgmu(\spd R) \to \bungmu(\spd R)$ is essentially surjective for $R$ a product of absolutely integrally closed valuation rings. These rings form a basis for the v-topology on the category of perfect $\fp$-algebras, see Lemma \ref{Lem:SchemeTheoreticProductOfPoints}.

\subsection{Organization of the paper} Let us provide a brief guide for navigating the paper and highlight some results not mentioned to this point. 

In Sections \ref{Sec:Preliminaries} and \ref{Sec:PerfecAlgGeom}, we recall some important notions from the study of the Fargues--Fontaine curve and Bruhat--Tits theory, review and establish properties of local shtukas for parahoric and quasi-parahoric group schemes over $\zp$, and specify our conventions regarding Dieudonn\'e theory.

In Section \ref{Sec:Conjecture}, we state a precise conjecture on the existence of an Igusa stack for any canonical integral model of a Shimura variety in the sense of \cite{PappasRapoportShtukas}. We also show how the existence of such an Igusa stack implies the existence of Igusa varieties in the setting of perfect and perfectoid geometry, and we establish some expected results in the perfectoid setting, including a comparison with fibers of the Hodge--Tate period morphism. In Section \ref{Sec:HodgeType}, we define the Igusa stack in the Hodge-type case, after recalling the construction of canonical integral models for Hodge-type Shimura varieties. 

Our main geometric results are proved in Sections \ref{Sec:Proof} and \ref{Sec:Functoriality}. In particular, Section \ref{Sec:Proof} is devoted primarily to the proof of Theorem \ref{Thm:IntroIgusa}, and it closes with an explicit computation of the reduction of the Igusa stack. In Section \ref{Sec:Functoriality}, we prove functoriality of the Igusa stack in morphisms of Shimura data, as well as independence of the choices of Hodge embedding and parahoric subgroup.

Sections \ref{Sec:CohomologicaConsequences} through \ref{Sec:TorsionVanishing} contain our cohomological consequences. In Section \ref{Sec:CohomologicaConsequences}, we prove the remaining statements in Theorem \ref{Thm:IntroIgusaGeneric}. Moreover, we prove Theorem \ref{Thm:IntroWeilCohShimVar} and prove a Mantovan-type product formula at infinite level. We close this section by proving Theorem \ref{Thm:IntroPerversity}. In Section \ref{Sec:EichlerShimura} we apply the results of Section \ref{Sec:CohomologicaConsequences} to prove Theorem
\ref{Thm:IntroCompatibilityFarguesScholze} as well as Theorem~\ref{Thm:IntroEichlerShimura}. Finally, in Section \ref{Sec:TorsionVanishing} we prove our cohomological torsion-vanishing result, Theorem \ref{Thm:IntroTorsionVanishing}. We conclude with an explicit example of our torsion vanishing result in an abelian-type case.

%% file: proof.tex
Let us return to the notation of Section \ref{Sec:Conjecture}. That is, let $\gx$ be a Shimura datum with reflex field $\mathsf{E}$, let $\mathcal{G}$ be a quasi-parahoric model of $G:=\mathsf{G}_\qp$ over $\zp$, and let $K_p = \mathcal{G}(\zp)$. Choose a prime $v$ of $\mathsf{E}$ above $p$, let $E = \mathsf{E}_v$, and 
let $\mu$ denote the $G(\qpbar)$-conjugacy class of cocharacters of $G$ corresponding to $\mathsf{X}$ and $v$.  

In this section we prove that for the integral models of Hodge type Shimura
varieties at stabilizer Bruhat--Tits level, the 2-commutative
diagram in Conjecture~\ref{Conj:IgusaMain} is 2-Cartesian. More precisely this is the following theorem.

\begin{Thm} \label{Thm:HodgeMain}
  Let $\gx$ be a Shimura datum of Hodge type and let $\mathcal{G}$ be a
  stabilizer Bruhat--Tits model of $G$ over $\zp$. For a choice of $\Xi
  = (\mathcal{G}, \iota, V_{\zp})$ as in Section~\ref{Subsec:IgusaDefinition},
  the diagram
  \begin{equation}\begin{tikzcd} \label{Eq:TheDiagram}
    \scrsdinf \ar[r,"{\pi_\mathrm{crys}}"] \ar[d, "\IgsQuot"] & \shtgmuone
    \ar[d,"{\mathrm{BL}^\circ}"] \\ \igsinf
    \arrow[r,"{\overline{\pi}_\mathrm{HT}}"] & \bungmu
  \end{tikzcd} \end{equation}
  of small v-stacks on $\perf$ is 2-Cartesian.
\end{Thm}

\begin{Rem} \label{Rem:HodgeMainQuasiParahoric}
 If $\mathcal{H}$ is another quasi-parahoric model of $G$ equipped with an open immersion $\mathcal{H} \to \mathcal{G}$ inducing $K_p':=\mathcal{H}(\zp) \subseteq K_p$, then it follows from \eqref{Eq:DiagramQuasiParahoricStabilizer} that we have a 2-Cartesian diagram
  \[\begin{tikzcd}
    \scrs_{K_p'}\gx^{\diamond} \ar[r,"{\pi_\mathrm{crys}}"] \ar[d, "\IgsQuot"] &
    \shthmuone \ar[d,"{\mathrm{BL}^\circ}"] \\ \igsinf
    \arrow[r,"{\overline{\pi}_\mathrm{HT}}"] & \bungmu.
  \end{tikzcd} \]
\end{Rem}

For convenience, we define a presheaf and sheaf version of the fiber product
\[
  F^\mathrm{pre} = \igspreinf \times_{\bungmu} \shtgmuone, \quad F = \igsinf
  \times_{\bungmu} \shtgmuone,
\]
so that there are natural maps $\scrsdpreinf \to F^\mathrm{pre} \to F$ of
presheaves of groupoids. It follows from Lemma \ref{Lem:Discrete} and Lemma \ref{Lem:ZeroTruncated} that $F^\mathrm{pre}$ and $F$ are $0$-truncated.

\subsection{Pappas--Rapoport uniformization}
{
\newcommand{\mintgcircbxmu}[1]{\mathcal{M}^{\mathrm{int}}_{\calgcirc,b_x,\mu,#1}}
\newcommand{\mintgbxmu}[1]{\mathcal{M}^\mathrm{int}_{\mathcal{G},b_x,\mu,#1}}
\newcommand{\mintgbxmuone}[1]{\mathcal{M}^\mathrm{int}_{\mathcal{G},b_x,\mu,\delta=1,#1}}
\newcommand{\mintgvbxmu}{\mathcal{M}^\mathrm{int}_{\mathcal{G}_V,b_{i(x)},\mu_V}}
\newcommand{\oeel}[1]{{W_{\mathcal{O}_E,#1}(l)}}

To prepare for the proof of Theorem~\ref{Thm:HodgeMain}, we discuss the
uniformization map constructed by Pappas--Rapoport \cite{PappasRapoportShtukas}.
Denote by $k_E$ the residue field of $\mathcal{O}_E$ as before. For a perfect field $l$ in characteristic $p$ together with a fixed embedding
$e \colon k_E \hookrightarrow l$, we write
\[
  \oeel{e} := \mathcal{O}_E \otimes_{W(k_E),e} W(l).
\]
We define the integral local Shimura variety as follows.

{
\newcommand{\mintgcircbmu}[1]{\mathcal{M}^\mathrm{int}_{\calgcirc,b,\mu,#1}}
\newcommand{\mintgbmu}[1]{\mathcal{M}^\mathrm{int}_{\mathcal{G},b,\mu,#1}}
\newcommand{\mintgbmuone}[1]{\mathcal{M}^\mathrm{int}_{\mathcal{G},b,\mu,\delta=1,#1}}
\begin{Def}[{\cite[Definition~25.1.1]{ScholzeWeinsteinBerkeley}, \cite[Definition~3.1.5]{Companion}}]
  Let $l$ be a perfect field of characteristic $p$ together with an embedding $e
  \colon k_E \hookrightarrow l$, and let $b \in G(W(l)[p^{-1}])$. The
  \emph{integral local Shimura variety}
  \[
    \mintgbmu{e} \to \spd \oeel{e}
  \]
  is the v-sheaf on $\perf$ that assigns to each affinoid perfectoid space $S$ the set of
  isomorphism classes of tuples $(S^\sharp, \mathscr{P}, \phi_\mathscr{P},
  \iota_r)$, where $S^\sharp$ is an untilt of $S$ over $\oeel{e}$, where
  $(\mathscr{P}, \phi_\mathscr{P})$ is a $\mathcal{G}$-shtuka with one leg along
  $S^\sharp$ bounded by $\mu$, and where $\iota_r$ is an isomorphism
  \[
    \iota_r \colon G \vert_{\mathcal{Y}_{[r,\infty)}(S)} \xrightarrow{\simeq}
    \mathscr{P} \vert_{\mathcal{Y}_{[r,\infty)}(S)}
  \]
  for $r \gg 0$, which satisfies $\iota_r \circ \phi_\mathscr{P} = (b \times
  \mathrm{Frob}_S) \circ \iota_r$.
\end{Def}

Let $\mathcal{G}^\circ$ denote the parahoric group scheme corresponding to $\mathcal{G}$, i.e., $\mathcal{G}^\circ$ is the identity component of $\mathcal{G}$. By \cite[Theorem~4.4.1]{PappasRapoportRZSpaces}, the natural map
$\mintgcircbmu{e} \to \mintgbmu{e}$ is finite \'{e}tale. We define
\[
  \mintgbmuone{e} := \operatorname{im}(\mintgcircbmu{e} \to \mintgbmu{e})
  \subseteq \mintgbmu{e},
\]
which is an open and closed sub-v-sheaf.
}

The image of any $x$ in $\scrs_{K_p}\gx(l)$ under
$\pi_\mathrm{crys}$ defines a $\spd l$-point of $\shtgmuone$. Choose a
representative $b_x \in G(W(l)[p^{-1}])$ for the associated $G$-isocrystal, see
Section~\ref{Sec:Isocrystals}. Let $e:k_E \to l$ be the map corresponding to $x:
\spec l \to \scrs_{K_p}\gx \to \spec k_E$, so that attached to $x$ is a base
point
\[
  x_0 \colon \spd l \to \mintgbxmuone{e},
\]
given by the $\spd l$-point of $\shtgmuone$ corresponding to $x$, together with
the isomorphism between its associated $G$-isocrystal and $(G_{W(l)[p^{-1}]},
b_x)$.

Pappas--Rapoport constructed a uniformization map under a
certain condition U, see \cite[Section 4.10.2]{PappasRapoportShtukas}. This
condition U was later verified by Gleason--Lim--Xu \cite{GleasonLimXu} through
studying the connected components of affine Deligne--Lusztig varieties. It was
extended to stabilizer Bruhat--Tits groups in \cite{Companion}.

\begin{Thm}[{\cite[Theorem~4.4.1]{Companion}, \cite[Corollary~6.3]{GleasonLimXu}}]
There exists a uniformization map
\[
  \Theta_{\mathcal{G},x} \colon \mintgbxmuone{e} \to
  \scrs_{K_p}\gx_\oeel{e}^\diamond
\]
sending the base point $x_0$ to $x$, which moreover restricts to an isomorphism
\[
  \Theta_{\mathcal{G},x} \colon \widehat{\mintgbxmuone{e}}_{/x_0}
  \xrightarrow{\simeq} (\widehat{\scrs_{K_p}\gx_\oeel{e}}_{/x})^\lozenge.
\]
\end{Thm}
Here, by $\widehat{\scrs_{K_p}\gx_\oeel{e}}_{/x}$ we mean the formal completion
of a Shimura variety $\scrs_{K_p K^p}\gx_\oeel{e}$ of finite level at the image
of $x$, which does not depend on the level as the transition maps are \'{e}tale.
In the Siegel case, the uniformization map
\[
  \mathrm{RZ}_{\mathcal{G}_V,b_x, \mu_V}^\lozenge \simeq
  \mathcal{M}_{\mathcal{G}_V,b_x, \mu_V }^\mathrm{int}
  \xrightarrow{\Theta_{\mathcal{G}_V,x}} \scrsdvinf_{W(l)}.
\]
can be identified with the functor $(-)^\lozenge$ applied to Rapoport--Zink
uniformization in the sense of \cite{RapoportZink}, where $\mathrm{RZ}_{\mathcal{G}_V,b_x,\mu_V}$ is
the formal scheme over $\spf W(l)$ representing the moduli of formal
quasi-isogenies $\mathcal{A}_x[p^\infty] \dashrightarrow \mathbb{X}$ preserving
the polarization up to a scalar in $\mathbb{Q}_p^\times$, see
\cite[Definition~3.21]{RapoportZink}, where $\mathbb{X}$ is the $p$-divisible group corresponding to $b_x$.

\begin{Lem}
\label{Lem:UnderlyingUniformization}
  The composition
  \[
    \mintgbxmuone{e} \xrightarrow{\Theta_{\mathcal{G},x}}
    \scrs_{K_p}\gx_\oeel{e}^\diamond \xrightarrow{\pi_\mathrm{crys}} \shtgmuone
  \]
  is given by the functor sending a tuple $(S^\sharp, \mathscr{P}, \phi_{\mathscr{P}}, \iota_r)$ to the $\mathcal{G}$-shtuka $(\mathscr{P},\phi_{\mathscr{P}})$ with one leg at $S^\sharp$.
\end{Lem}

\begin{proof} This is part of the statement of \cite[Theorem 4.4.1]{Companion}; we include a proof for convenience. The lemma follows from the construction of the map $\Theta_{\mathcal{G},x}$. In
  \cite[Lemma~4.10.2.(a)]{PappasRapoportShtukas}, the map $c \colon
  \mathrm{RZ}_{\mathcal{G}, \mu ,x,e}^\lozenge \to \mintgbxmuone{e}$ is
  constructed by putting a framing on the $\mathcal{G}$-shtuka obtained via
  \[
    \mathrm{RZ}_{\mathcal{G}, \mu ,x,e}^\lozenge
    \xrightarrow{\Theta_{\mathcal{G},x}^\mathrm{RZ,\lozenge}}
    \scrs_{K_p}\gx_\oeel{e}^\diamond \to \shtgmuone.
  \]
  As $c$ is an isomorphism and $\Theta_{\mathcal{G},x} =
  \Theta_{\mathcal{G},x}^{\mathrm{RZ},\lozenge} \circ c^{-1}$, the claim
  follows.
\end{proof}

By \cite[Equation~(4.10.4)]{PappasRapoportShtukas}, the above map sits in a
commutative diagram of v-sheaves
\begin{equation} \label{Eq:RZCompatible}
    \begin{tikzcd}
     \mintgbxmuone{e} \arrow{r}{\Theta_{\mathcal{G},x}} \arrow{d} &
      \scrsdinf_\oeel{e} \arrow{d}{i} \\ \mintgvbxmu \times_{\spd \zp} \spd
      \mathcal{O}_E \arrow{r}{\Theta_{\mathcal{G}_V, i(x)}} & \scrsdvinf_{W(l)
      \otimes_{\zp} \mathcal{O}_E}
    \end{tikzcd}
\end{equation}
over $\spd W(l)$. Here, for the Siegel local Shimura variety $\mintgvbxmu$, we
suppress the embedding $\mathbb{F}_p \hookrightarrow l$ since there is a unique
such one. 

\subsection{The case of a geometric point of rank one}

We first check that diagram \eqref{Eq:TheDiagram} is Cartesian when tested against rank $1$
geometric points. To make use of the Pappas--Rapoport uniformization map, we
need the following lemma.

\begin{Lem} \label{Lem:ResidueFieldSection}
  Let $V$ be an absolutely integrally closed valuation ring of rank $1$ in characteristic $p$ whose fraction
  field is algebraically closed. Then the ring homomorphism $V \to
  V/\mathfrak{m}_V$ has a section.
\end{Lem}

\begin{proof}
  Denote the residue field by $l = V/\mathfrak{m}_V$. By Zorn's lemma, it
  suffices to show that if there exists a subfield $l_0 \subsetneq l$, a section
  $s_0 \colon l_0 \to V$, and an element $x \in l \setminus l_0$, then the
  section $s_0$ can be extended to $s \colon l_0(x) \to V$.

  When $x$ is transcendental over $l_0$, we may pick an arbitrary lift
  $\tilde{x} \in V$ of $x \in l$ and extend $s_0$ to $s$ by setting $s(x) =
  \tilde{x}$. Therefore we assume that $x$ is algebraic over $l_0$. Let $f_0(t)
  \in l_0[t]$ be the minimal polynomial of $x \in l$. Consider the polynomial
  $f(t) \in V[t]$ obtained by applying $s_0$ to the coefficients of $f_0$. As
  $V$ is absolutely integrally closed, it factors as $f(t) = \prod_i (t -
  \alpha_i)$, where the reduction of the multiset $\lbrace \alpha_i \rbrace$
  under $V \to V/\mathfrak{m}_V$ recovers the roots of $f_0$ inside $l$. This
  implies that there is an $\alpha_i \in V$ reducing to $x$, and hence we may
  extend $s_0$ to $s \colon l_0(x) \to V$ by setting $s(x) = \alpha_i$.
\end{proof}

\begin{Prop} \label{Prop:RankOneGeometricSurjective}
  For $C$ an algebraically closed perfectoid field in characteristic $p$, the
  map
  \[
    \scrsdpreinf(C, \mathcal{O}_C) \to F^\mathrm{pre}(C, \mathcal{O}_C)
  \]
  is essentially surjective.
\end{Prop}

\begin{proof}
  Denote the residue field by $l = \mathcal{O}_C / \mathfrak{m}_C$. Using
  Lemma~\ref{Lem:ResidueFieldSection}, we fix a section $l \hookrightarrow
  \mathcal{O}_C$. Then for any untilt $C^\sharp$ of $C$, the map
  $W(\mathcal{O}_C) \twoheadrightarrow \mathcal{O}_{C^\sharp}$ induces a
  $W(l)$-algebra structure on $\mathcal{O}_{C^\sharp}$.

  Fix an element of $F^\mathrm{pre}(C, \mathcal{O}_C)$, which corresponds to a
  tuple $T=(C^{\sharp_1}, y, C^{\sharp_2}, \mathscr{Q}, \phi_\mathscr{Q}, \alpha)$
  where $C^{\sharp_1}, C^{\sharp_2}$ are untilts of $C$ over $\mathcal{O}_E$, where $y
  \colon \spf \mathcal{O}_{C^{\sharp_1}} \to \hatscrsginf$ is a map over
  $\mathcal{O}_E$, where $(\mathscr{Q}, \phi_\mathscr{Q})$ is a $\mathcal{G}$-shtuka
  with a leg at $\spa(C^{\sharp_2}, \mathcal{O}_{C^{\sharp_2}})$ bounded by
  $\mu$, and where $\alpha \colon \mathscr{P}_y \vert_{\mathcal{Y}_{[r, \infty)}(C,
  \mathcal{O}_C)} \simeq \mathscr{Q} \vert_{\mathcal{Y}_{[r, \infty)}(C,
  \mathcal{O}_C)}$ is a $\phi$-equivariant isomorphism of $G$-torsors ($r$ large enough). Our goal
  is to construct a new $\mathcal{O}_E$-morphism $z \colon \spf
  \mathcal{O}_{C^{\sharp_2}} \to \hatscrsginf$ such that the induced $(C,\mathcal{O}_C)$-point of $F^\mathrm{pre}$ is isomorphic to $T$. \smallskip

  We note that the two maps
  \[
    e_1 \colon k_E \hookrightarrow
    \mathcal{O}_{C^{\sharp_1}}/\mathfrak{m}_{C^{\sharp_1}} \simeq l, \quad e_2
    \colon k_E \hookrightarrow
    \mathcal{O}_{C^{\sharp_2}}/\mathfrak{m}_{C^{\sharp_2}} \simeq l
  \]
  may be different from each other. As $k_E$ is finite, the two embeddings are
  related by a finite power of the absolute Frobenius. Write $\phi_l^m \circ e_1
  = e_2$ for a nonnegative integer $m$, where $\phi_l$ is the absolute Frobenius
  on $l$.

  Let $x \in \scrs_{K_p}\gx(l)$ be the image of the reduced point under $y$.
  Using the $W(l)$-algebra structure on $\mathcal{O}_{C^{\sharp_1}}$, we may
  regard $\mathcal{O}_{C^{\sharp_1}}$ as an $\oeel{e_1}$-algebra to promote $y$
  to a map
  \[
    y_l \colon \spf \mathcal{O}_{C^{\sharp_1}} \to \hatscrsginf_{\oeel{e_1}}
  \]
  over $\spf \oeel{e_1}$. Note that the
  map $y_l$ factors through the formal completion at $x$. Identifying its associated v-sheaf
  with the formal neighborhood of the integral local Shimura variety,
  \begin{align*}
    \spd (C^{\sharp_1}, \mathcal{O}_{C^{\sharp_1}}) &\hookrightarrow \spd
    (\mathcal{O}_{C^{\sharp_1}}, \mathcal{O}_{C^{\sharp_1}}) \\ &\xrightarrow{y_l^\lozenge}
    (\widehat{\scrs_{K_p}\gx_{\oeel{e_1}}}_{/x})^\lozenge
    \xrightarrow{\Theta_{\mathcal{G},x}^{-1}} \widehat{ \mintgbxmuone{e_1}}_{/x_0},
  \end{align*}
  we obtain a framing of the $G$-shtuka induced by $y$
  \[
    \iota \colon (\mathscr{P}_y \vert_{\mathcal{Y}_{[r,\infty)}(C,
    \mathcal{O}_C)}, \phi_{\mathscr{P}_y}) \xrightarrow{\sim} (G \times
    \mathcal{Y}_{[r,\infty)}(C, \mathcal{O}_C), b_x).
  \]
  Composing with $\alpha^{-1}$ gives a framing
  \[
    (\mathscr{Q} \vert_{\mathcal{Y}_{[r,\infty)}(C,
    \mathcal{O}_C)}, \phi_\mathscr{Q}) \xrightarrow{\iota \circ \alpha^{-1}} (G
    \times \mathcal{Y}_{[r,\infty)}(C, \mathcal{O}_C), b_x),
  \]
  and therefore defines a map
  \[
    z_0 \colon \spd (C^{\sharp_2}, \mathcal{O}_{C^{\sharp_2}}) \to \mintgbxmuone{e_2}
  \]
  over $\spd \oeel{e_2}$.

  At this point, we wish to apply $\Theta_{\mathcal{G},x}$ to get a point of the
  Shimura variety, but the problem is that $x$ lies over $e_1$ and not $e_2$.
  Compose $x \colon \spec l \to \scrs_{K_p}\gx$ with $\phi_l^m$ to obtain a map
  $\phi^m(x) \colon \spec l \to \scrs_{K_p}\gx$ that now lies over
  $e_2$. Using the isomorphism of left $G$-torsors
  \[
    (\phi^m)^\ast (G_{W(l)[p^{-1}]}, b_x) = (G_{W(l)[p^{-1}]}, \phi^m(b_x))
    \simeq (G_{W(l)[p^{-1}]}, b_x)
  \]
  given by
  \[ \begin{tikzcd}[column sep=8em]
    \phi^\ast G_{W(l)[p^{-1}]} = G_{W(l)[p^{-1}]} \arrow{d}{\phi^m(b_x)}
    \arrow{r}{\phi^m(b_x) \dotsm \phi^2(b_x) \phi(b_x)} & \phi^\ast
    G_{W(l)[p^{-1}]} = G_{W(l)[p^{-1}]} \arrow{d}{b_x} \\ G_{W(l)[p^{-1}]}
    \arrow{r}{\phi^{m-1}(b_x) \dotsm \phi(b_x) b_x} & G_{W(l)[p^{-1}]},
  \end{tikzcd} \]
  we may identify $b_{\phi^m(x)} = b_x$. We can now compose
  \[
    z \colon \spd (C^{\sharp_2}, \mathcal{O}_{C^{\sharp_2}}) \xrightarrow{z_0}
    \mintgbxmuone{e_2} \simeq
    \mathcal{M}^\mathrm{int}_{\mathcal{G},b_{\phi^m(x)},\mu,\delta=1,e_2}
    \xrightarrow{\Theta_{\mathcal{G},\phi^m(x)}}
    \scrs_{K_p}\gx_{\oeel{e_2}}^\diamond
  \]
  to get a map defined over $\spd \oeel{e_2}$. Forgetting about the
  $l$-structure finally yields a map 
  \[
    z \colon \spd (C^{\sharp_2}, \mathcal{O}_{C^{\sharp_2}}) \to \scrsdinf
  \]
  lying over $\spd \mathcal{O}_E$. This corresponds to a map $\spa
  (C^{\sharp_2}, \mathcal{O}_{C^{\sharp_2}}) \to
  \hatscrsginf^\mathrm{ad}$, which uniquely extends to a morphism $z
  \colon \spf \mathcal{O}_{C^{\sharp_2}} \to \hatscrsginf$.

  We now verify that the image of $z$ under $\scrsdpreinf \to F^\mathrm{pre}$
  is isomorphic to $T$. Lemma~\ref{Lem:UnderlyingUniformization} implies that applying
  $\pi_\mathrm{crys}$ to $z$ returns back the underlying $\mathcal{G}$-shtuka,
  which is $(\mathscr{Q}, \phi_\mathscr{Q})$ by construction. Next, to check
  that $z$ induces the $(C, \mathcal{O}_C)$-point of the Igusa stack
  corresponding to $T$, we need to construct a
  formal quasi-isogeny between the formal abelian schemes $\mathcal{A}_y$ and
  $\mathcal{A}_z$ preserving the $G$-structure on the associated vector bundles
  on $X_{\spa(C, \mathcal{O}_C)}$, as well as the away-from-$p$ level structure.

  Choose $\varpi$ large enough so that there are surjective ring homomorphisms
  $\mathcal{O}_{C^{\sharp_1}} \twoheadrightarrow \mathcal{O}_C/\varpi$ and
  $\mathcal{O}_{C^{\sharp_2}} \twoheadrightarrow \mathcal{O}_C/\varpi$ induced
  from the isomorphisms $(C^{\sharp_1})^\flat \simeq C$ and $(C^{\sharp_2})^\flat
  \simeq C$ as in \eqref{Eq:NaturalMapPseudoUniformizer}. Since $\scrs_{K_p}\gx$
  is a limit of Noetherian schemes along \'{e}tale transition maps, we may
  further increase $\varpi$ so that $y \vert_{\spec \mathcal{O}_C/\varpi}$
  factors through $x \in \scrs_{K_p}\gx(l)$. Then compatibility with the Siegel
  case implies that the framing on the $\mathcal{G}_V$-shtuka $\mathscr{P}_y
  \times^{\mathcal{G}} \mathcal{G}_V$ is obtained from taking the $p$-divisible
  group $\tilde{\iota}[p^\infty]$ of the canonical isomorphism
  \[
    \tilde{\iota} \colon \mathcal{A}_y \times_{\spf \mathcal{O}_{C^{\sharp_1}}}
    \spec \mathcal{O}_C/\varpi \simeq \mathcal{A}_x \times_{\spec l} \spec
    \mathcal{O}_C/\varpi.
  \]

  On the other hand, by \cite[Proposition~2.2.7]{PappasRapoportShtukas}, the
  $G_V$-shtuka $\mathscr{Q} \times^\mathcal{G} \mathcal{G}_V$ corresponds to a
  principally polarized $p$-divisible group $\mathbb{X}_\mathscr{Q}$ over $\spf
  \mathcal{O}_{C^{\sharp_2}}$, and by
  \cite[Proposition~2.1.4]{PappasRapoportShtukas} and
  \cite[Theorem~A]{ScholzeWeinsteinModuli}, the isomorphism $\alpha$ corresponds
  to a formal quasi-isogeny
  \[
    \tilde{\alpha} \colon \mathcal{A}_y[p^\infty] \dashrightarrow
    \mathbb{X}_\mathscr{Q}
  \]
  preserving the polarization up to a scalar. It follows that the composition
  $\spa C^{\sharp_2} \xrightarrow{z_0} \mintgbxmuone{e_2} \to \mintgvbxmu$ is induced
  by a map
  \[
    \tilde{z}_0 \colon \spf \mathcal{O}_{C^{\sharp_2}} \to
    \mathrm{RZ}_{\mathcal{G}_V,b_x,\mu_V}
  \]
  corresponding to the quasi-isogeny
  \begin{align}
    \mathbb{X}_\mathscr{Q} \times_{\spf \mathcal{O}_{C^{\sharp_2}}} \spec
    \mathcal{O}_C/\varpi
    &\xdashrightarrow{\tilde{\alpha}^{-1}} \mathcal{A}_y[p^\infty] \times_{\spf
    \mathcal{O}_{C^{\sharp_1}}} \spec
    \mathcal{O}_C/\varpi \\
    &\xdashrightarrow{\tilde{\iota}[p^\infty]} \mathcal{A}_x[p^\infty] \times_{\spec
    l} \spec \mathcal{O}_C/\varpi.
  \end{align}

  To access $z$, we first need to understand what the isomorphism
  $(\phi_l^m)^\ast (G_l, b_x) \simeq (G_l, b_x)$ is doing. If we push out to an
  isomorphism of $\mathcal{G}_V$-torsors, Lemma~\ref{Lem:RelativeFrobeniusIsocrystal} below
  shows that it is the inverse of the iterated relative Frobenius
  \[
    F_l^m[p^\infty] \colon \mathcal{A}_x[p^\infty] \dashrightarrow
    \mathcal{A}_x[p^\infty] \times_{\spec l, \phi_l^m} \spec l.
  \]
  We can finally describe the image of $z$ under $\scrs_{K_p}\gx \to
  \scrs_{M_p}\gvx$ as the modification of the abelian scheme
  \[
    \mathcal{A}_x \times_{\spec l, \phi_l^m} \spec \mathcal{O}_C/\varpi
  \]
  along the quasi-isogeny 
  \begin{align}
    \mathbb{X}_\mathscr{Q} \times_{\spf \mathcal{O}_{C^{\sharp_2}}} \spec
    \mathcal{O}_C/\varpi
    &\xdashrightarrow{\tilde{\alpha}^{-1}} \mathcal{A}_y[p^\infty] \times_{\spf
    \mathcal{O}_{C^{\sharp_1}}} \spec
    \mathcal{O}_C/\varpi \\
    &\xdashrightarrow{\tilde{\iota}[p^\infty]} \mathcal{A}_x[p^\infty] \times_{\spec
    l} \spec \mathcal{O}_C/\varpi \\
    &\xdashrightarrow{F_l^m[p^\infty]} \mathcal{A}_x[p^\infty] \times_{\spec l,
    \phi_l^m} \spec \mathcal{O}_C/\varpi,
  \end{align}
  lifted according to Serre--Tate theory so that its $p$-divisible group is
  $\mathbb{X}_\mathscr{Q}$. Because both quasi-isogenies
  $\tilde{\iota}[p^\infty]$ and $F_l^m[p^\infty]$ actually come from
  quasi-isogenies $\tilde{\iota}$ and $F_l^m$ of abelian schemes, we may also
  describe it as the modification of the abelian scheme
  \[
    \mathcal{A}_y \times_{\spf \mathcal{O}_{C^{\sharp_1}}} \spec
    \mathcal{O}_C/\varpi
  \]
  along the quasi-isogeny 
  \[
    \mathbb{X}_\mathscr{Q} \times_{\spf \mathcal{O}_{C^{\sharp_2}}} \spec
    \mathcal{O}_C/\varpi
    \xdashrightarrow{\tilde{\alpha}^{-1}} \mathcal{A}_y[p^\infty] \times_{\spf
    \mathcal{O}_{C^{\sharp_1}}} \spec
    \mathcal{O}_C/\varpi
  \]
  lifted according to Serre--Tate theory so that its $p$-divisible group is
  $\mathbb{X}_\mathscr{Q}$. That is, we have 
  \[
    \mathcal{A}_z \times_{\spf \mathcal{O}_{C^{\sharp_2}}} \spec
    \mathcal{O}_C/\varpi \simeq (\mathcal{A}_y
    \times_{\spf \mathcal{O}_{C^{\sharp_1}}} \spec
    \mathcal{O}_C/\varpi) / \ker \tilde{\alpha}.
  \]    
  Note that if $\tilde{\alpha}$ is not an honest isogeny then we cannot literally quotient out by $\ker \tilde{\alpha}$. To remedy this, we argue as in \cite[Remark~8.12]{ZhangThesis}.
  This gives the desired formal quasi-isogeny between $\mathcal{A}_y$ and
  $\mathcal{A}_z$.

  From the construction, it is clear that the induced isomorphism
  $\mathcal{E}(\mathcal{A}_y) \simeq \mathcal{E}(\mathcal{A}_z)$ of vector
  bundles on $X_{(C, \mathcal{O}_C)}$ is the one associated to $\alpha$. Hence the
  $G$-structures on $\mathbb{L}_\mathrm{crys}$ are preserved under the formal
  quasi-isogeny. Moreover, as we are only modifying the abelian variety along a
  quasi-isogeny of the associated $p$-divisible group, the prime-to-$p$
  level structure is preserved. Therefore
  $\alpha$ defines a morphism between the two objects $y, z$ of the groupoid
  $\igspreinf(C, \mathcal{O}_C)$. This finishes the proof of surjectivity.
\end{proof}

\begin{Rem}
  Even though the construction of the inverse $F^\mathrm{pre}(C, \mathcal{O}_C)
  \to \scrsdinf(C, \mathcal{O}_C)$ involves choosing a section $l \to
  \mathcal{O}_C$, we see a posteriori from
  Proposition~\ref{Prop:RankOneGeometricInjective} that the result is independent of the choice.
\end{Rem}

\begin{Lem} \label{Lem:RelativeFrobeniusIsocrystal}
  Let $\mathbb{X}$ be a $p$-divisible group over a perfect ring $R$ in
  characteristic $p$. Consider $(\mathbb{D}(\mathbb{X})^{\natural}, \phi_{\mathbb{X}}^\natural)$ as in Section \ref{Sec:Dieudonne} and write $(\mathbb{D}(\mathbb{X})^{\natural}, \phi_{\mathbb{X}}^\natural)=(M, \beta)$. The map of
  Dieudonn\'{e} modules $(M, \beta) \to ((\phi^n)^\ast M, (\phi^n)^\ast \beta)$
  associated to the iterated relative Frobenius
  \[
    \mathbb{X} \to \phi^\ast \mathbb{X} \to \dotsb \to (\phi^n)^\ast \mathbb{X}
  \]
  is given by $(\phi^{n-1})^\ast \beta^{-1} \circ \dotsb \circ \phi^\ast
  \beta^{-1} \circ \beta^{-1} \colon M \to (\phi^n)^\ast M$.
\end{Lem}

\begin{proof}
  This follows from the fact that for a single relative Frobenius, the induced
  map is
  \[ \begin{tikzcd}
    \phi^\ast M \arrow{d}{\beta} \arrow{r}{\phi^\ast \beta^{-1}} & (\phi^2)^\ast
    M \arrow{d}{\phi^\ast \beta} \\ M \arrow{r}{\beta^{-1}} & \phi^\ast M.
  \end{tikzcd} \]
  Composing appropriate Frobenius pullbacks of the relative Frobenius then gives
  the result.
\end{proof}

\begin{Prop} \label{Prop:RankOneGeometricInjective}
  For $C$ an algebraically closed perfectoid field in characteristic $p$, the
  map
  \[
    \scrsdpreinf(C, \mathcal{O}_C) \to F^\mathrm{pre}(C, \mathcal{O}_C)
  \]
  is fully faithful.
\end{Prop}

\begin{proof}
  Let $y, z \colon \spf \mathcal{O}_{C^\sharp} \to \scrs_{K_p}\gx$ be two maps over
  $\mathcal{O}_E$. We wish to show that if there exists a quasi-isogeny
  \[
    f \colon \mathcal{A}_y \times_{\spf \mathcal{O}_{C^\sharp}} \spec
    \mathcal{O}_C/\varpi \dashrightarrow \mathcal{A}_z
    \times_{\spf \mathcal{O}_{C^\sharp}} \spec
    \mathcal{O}_C/\varpi
  \]
  preserving the \'{e}tale trivialization and an isomorphism
  \[
    \alpha \colon (\mathscr{P}_y, \phi_{\mathscr{P}_y}) \simeq (\mathscr{P}_z,
    \phi_{\mathscr{P}_z})
  \]
  inducing the same isomorphism of vector bundles on $X_{(C, \mathcal{O}_C)}$, then
  $y = z$.

  We first note that $\alpha \times^\mathcal{G} \mathcal{G}_V$ corresponds to an isomorphism
  $\alpha \colon \mathcal{A}_y[p^\infty] \simeq \mathcal{A}_z[p^\infty]$ of
  $p$-divisible groups, by \cite[Proposition~2.2.7]{PappasRapoportShtukas} and \cite[Theorem 14.4.1]{ScholzeWeinsteinBerkeley}.
  Moreover, compatibility with $f$ says that $\alpha$ restricted to $\spec
  \mathcal{O}_C/\varpi$ agrees with $f[p^\infty]$.  This implies
  that $f$ in fact is an isomorphism of abelian schemes over $\spec
  \mathcal{O}_C/\varpi$, and moreover lifts to an isomorphism
  $\mathcal{A}_y \simeq \mathcal{A}_z$ over $\spf \mathcal{O}_{C^\sharp}$ by Serre--Tate theory. That
  is, the images of $y$ and $z$ under
  \[
    i \colon \scrs_{K_p}\gx \to \scrs_{M_p}\gvx
  \]
  agree.

  We now argue as in \cite[Proposition~4.10.3]{PappasRapoportShtukas}. Denote by
  $l$ the residue field of $C$, and choose a section $l \hookrightarrow
  \mathcal{O}_C$ using Lemma~\ref{Lem:ResidueFieldSection}. We first show that
  the restrictions $x, x_z \in \scrs_{K_p}\gx(l)$ of $y, z$ agree. As
  the $\spf \mathcal{O}_C^{\sharp}$-point $y$
  factors through the formal completion of $x$, via
  $\Theta_{\mathcal{G},x}^{-1}$ we obtain a framing $(\mathscr{P}_y,
  \phi_{\mathscr{P}_y}) \simeq (G, b_x)$. Through the isomorphism $\alpha$, we
  also obtain a framing of $(\mathscr{P}_z, \phi_{\mathscr{P}_z})$. As $\alpha$
  is an isomorphism of $\mathcal{G}$-shtukas, they define the same point $s
  \colon \spd \mathcal{O}_{C^\sharp} \to \mintgbxmuone{e}$.

Let $\mathscr{S}_{K_p}^-$ be the image of $\scrs_{K_p}\gx$ in $\scrs_{M_p}\gvx$. By compatibility of
  $\Theta_{\mathcal{G}}$ with $\Theta_{\mathcal{G}_V}$ \eqref{Eq:RZCompatible}, we deduce that the
  images of
  \[
    \widehat{\scrs_{K_p}\gx_{W(l)}^\diamond}_{/x},
    \widehat{\scrs_{K_p}\gx_{W(l)}^\diamond}_{/x_z} \rightarrow
    \widehat{\scrs_{M_p}\gvx_{W(l) \otimes_{W(k_E)}
    \mathcal{O}_E}^\diamond}_{/i(x)=i(x_z)}
  \]
  agree. Since both are irreducible components of the normalization of the image $\widehat{\mathscr{S}_{K_p}^-}_{/i(x)}$, it follows that $x = x_z$. We now observe that $y, z$ correspond
  to $s$ under the isomorphism
  \[
    \widehat{\scrs_{K_p}\gx_{\oeel{e}}^\diamond}_{/x} \simeq
    \widehat{ \mintgbxmuone{e}}_{/s_0},
  \]
  where $s_0$ is the restriction of $s$ to $\spd l$. It follows that $y = z$ as
  desired.
\end{proof}

Combining Proposition~\ref{Prop:RankOneGeometricSurjective} and
Proposition~\ref{Prop:RankOneGeometricInjective}, we obtain the following.

\begin{Cor} \label{Cor:RankOneGeometricBijective}
  For $C$ an algebraically closed perfectoid field in characteristic $p$, the
  map
  \[
    \scrsdpreinf(C, \mathcal{O}_C) \to F^\mathrm{pre}(C, \mathcal{O}_C)
  \]
  is an equivalence of categories. 
\end{Cor}
}

\subsection{The case of a product of geometric points}
\label{Sec:ProductOfPointsProof}

{
\def\RkOne{\coprod_i s_i}
\def\ProdPts{S}

The v-topology on $\perf$ has a particularly simple basis, given by the so-called
``products of (geometric) points,'' see \cite[Definition~1.2,
Remark~1.3]{GleasonSpecialization}. Thus we only need to prove Theorem
\ref{Thm:HodgeMain} for these test objects. The basic strategy is to reduce to
the case of a single rank one geometric point, which is achieved by comparing
with Theorem \ref{Thm:HodgeMain} for Siegel Shimura varieties, established in
\cite{ZhangThesis}.

\subsubsection{Lifting properties} We recall the following definition.

\begin{Def} \label{Def:ProductOfPoints}
  A \textit{product of geometric points} is a perfectoid Huber pair of the form $((\prod_i
  C_i^+)[\varpi^{-1}], \prod_i C_i^+)$. Here $I$ is a set and for each $i \in I$ we have an algebraically closed perfectoid field $C_i$ of characteristic $p$ together with an open and bounded valuation
  subring $C_i^+$ and a pseudouniformizer $\varpi_i$; we set $\varpi = (\varpi_i)$ and give $\prod_i
  C_i^+$ the $\varpi$-adic topology. 
\end{Def}

We introduce the following definition, see also \cite[Definition 2.1.8]{Companion}.

\begin{Def} \label{Def:LiftsProductOfPoints}
  Let $f \colon \mathcal{F} \to \mathcal{G}$ be a map of presheaves of groupoids on $\perf$. For $\ProdPts = \spa((\prod_i C_i^+)[\varpi^{-1}], \prod_i C_i^+)$ a product of geometric points in characteristic $p$, there is a natural map
  \[
    \lambda_f \colon \mathcal{F}(\ProdPts) \to \mathcal{F}({\textstyle\RkOne})
    \times_{\mathcal{G}(\RkOne)} \mathcal{G}(\ProdPts),
  \]
  where we write $\RkOne = \coprod_i \spa(C_i, \mathcal{O}_{C_i})$. We say that
  \begin{itemize}
    \item $f$ is proper* when the map $\lambda_f$ is an equivalence of groupoids for every $\ProdPts$,
    \item $f$ is separated* when the map $\lambda_f$ is a fully faithful morphism of groupoids for every $\ProdPts$.
  \end{itemize}
\end{Def}

Informally, the properties of being proper* and separated* should be understood as lifting properties along diagrams of the following shape.
\begin{equation} \begin{tikzcd}
  \RkOne = \coprod_i \spa(C_i, \mathcal{O}_{C_i}) \arrow{d} \arrow{r} & \mathcal{F}
  \arrow{d} \\ \ProdPts = \spa((\prod_i C_i^+)[\varpi^{-1}], \prod_i C_i^+) \arrow{r}
  \arrow[dashed]{ru} & \mathcal{G} \label{Eq:LiftingDiagram}
\end{tikzcd} \end{equation}
We record some formal properties.

\begin{Lem} \label{Lem:ProperStarIsomorphismRankOne}
    If $f \colon \mathcal{F} \to \mathcal{G}$ is a map of v-stacks that is proper*, then $f$ is an isomorphism if and only if $f$ induces equivalences when evaluated at rank one points. 
\end{Lem}
\begin{proof}
    This is a direct consequence of Definition \ref{Def:LiftsProductOfPoints} and the fact that products of points are a basis of the v-topology on $\perf$, see \cite[Remark~1.3]{GleasonSpecialization}.
\end{proof}

\begin{Lem} \label{Lem:ProperRepresentableLifting}
  If a map $f \colon \mathcal{F} \to \mathcal{G}$ of v-stacks is proper and
  representable by diamonds, then $f$ is proper*.
\end{Lem}

\begin{proof}
  This is an immediate consequence of \cite[Proposition~2.13]{ZhangThesis}. Note
  that properness is used to produce uniquely existing lifts along $\coprod_i
  \spa(C_i, \mathcal{O}_{C_i}) \to \coprod_i \spa(C_i, C_i^+)$ in the sense of \cite[Definition 2.1.8.(1)]{Companion}.
\end{proof}

\begin{Lem} \label{Lem:LiftsDiagonal}
  Let $f \colon \mathcal{F} \to \mathcal{G}$ be a map of presheaves of groupoids. Then $f$
  is separated* if and only if the diagonal
  $\Delta_f \colon \mathcal{F} \to \mathcal{F} \times_\mathcal{G} \mathcal{F}$
  is proper*.
\end{Lem}

\begin{proof}
  We note that there is a natural identification of groupoids
  \[
    \mathcal{F}(\ProdPts) \times_{(\mathcal{F}(\RkOne)
    \times_{\mathcal{G}(\RkOne)} \mathcal{G}(\ProdPts))} \mathcal{F}(\ProdPts)
    \simeq \mathcal{F}({\textstyle\RkOne}) \times_{(\mathcal{F}(\RkOne)
    \times_{\mathcal{G}(\RkOne)} \mathcal{F}(\RkOne))} (\mathcal{F}(\ProdPts)
    \times_{\mathcal{G}(\ProdPts)} \mathcal{F}(\ProdPts)),
  \]
  under which $\lambda_{\Delta_f}$ is identified with the diagonal
  $\Delta_{\lambda_f}$ of $\lambda_f$. The statement now follows from the fact
  that a morphism $\lambda$ of groupoids is fully faithful if and only if its
  diagonal $\Delta_\lambda$ is an equivalence.
\end{proof}
\begin{Lem} \label{Lem:LiftsTwoOutOfThree}
  Let $f \colon \mathcal{F} \to \mathcal{G}$ and $g \colon \mathcal{G} \to
  \mathcal{H}$ be maps between presheaves of groupoids on $\perf$. If $g$ is separated*
 and $g \circ f$ is proper*, then $f$ is proper*. The converse holds if $g$ is proper*.
\end{Lem}

\begin{proof}
  We note that $\lambda_{g \circ f}$ can be identified with the composition
  \[
    \mathcal{F}(\ProdPts) \xrightarrow{\lambda_f}
    \mathcal{F}({\textstyle\RkOne}) \times_{\mathcal{G}(\RkOne)}
    \mathcal{G}(\ProdPts) \xrightarrow{\mathrm{id} \times \lambda_g}
    \mathcal{F}({\textstyle\RkOne}) \times_{\mathcal{H}(\RkOne)}
    \mathcal{H}(\ProdPts),
  \]
  and so $\lambda_g$ being fully faithful implies that $\lambda_f$ is an
  equivalence if $\lambda_{g \circ f}$ is an equivalence. If $\lambda_g$ is moreover an equivalence, then $\lambda_f$ is an
  equivalence if and only if $\lambda_{g \circ f}$ is an equivalence; the lemma follows. 
\end{proof}

\begin{Lem} \label{Lem:UniqueLiftsFiberProd}
  Let
  \[ \begin{tikzcd}
    \mathcal{A}_1 \arrow{d}{f} \arrow{r} & \mathcal{B}_1 \arrow{d}{g} &
    \mathcal{C}_1 \arrow{l} \arrow{d}{h} \\ \mathcal{A}_2 \arrow{r} &
    \mathcal{B}_2 & \mathcal{C}_2 \arrow{l}
  \end{tikzcd} \]
  be a 2-commutative diagram of presheaves of groupoids on $\perf$, where $f, g,
  h$ are all separated*. Then the induced map
  \[
    f \times_g h \colon \mathcal{A}_1 \times_{\mathcal{B}_1} \mathcal{C}_1 \to
    \mathcal{A}_2 \times_{\mathcal{B}_2} \mathcal{C}_2
  \]
  also is separated*.
\end{Lem}

\begin{proof}
  This follows from the fact that the map $\lambda$ corresponding to $f \times_g
  h$ may be computed by taking the fiber product of both rows of the diagram
  \[ \begin{tikzcd}[column sep=small]
    \mathcal{A}_1(\ProdPts) \arrow{r} \arrow{d}{\lambda_f} &
    \mathcal{B}_1(\ProdPts) \arrow{d}{\lambda_g} & \mathcal{C}_1(\ProdPts)
    \arrow{l} \arrow{d}{\lambda_h} \\ \mathcal{A}_1(\RkOne)
    \times_{\mathcal{A}_2(\RkOne)} \mathcal{A}_2(\ProdPts) \arrow{r} &
    \mathcal{B}_1(\RkOne) \times_{\mathcal{B}_2(\RkOne)} \mathcal{B}_2(\ProdPts)
    & \mathcal{C}_1(\RkOne) \times_{\mathcal{C}_2(\RkOne)}
    \mathcal{C}_2(\ProdPts), \arrow{l}
  \end{tikzcd} \]
  together with the fact that fiber products of fully faithful morphisms of
  groupoids are fully faithful.
\end{proof}

\subsubsection{Proof of Theorem \ref{Thm:HodgeMain}}
We now prove that certain comparison maps are separated*. We note that there is a commutative diagram
\begin{equation} \label{Eq:CommutativeDiagramShimuraIgusa}
  \begin{tikzcd}
    \scrsdpreinf \ar[r] \ar[d,"i^{\diamond, \mathrm{pre}}"] & F^{\mathrm{pre}}
    \ar[d]\\ \scrsdvpreinf \ar[r] & F_{V}^{\mathrm{pre}},
  \end{tikzcd}
\end{equation}
where we write $F_V^{\mathrm{pre}} = \igsvpreinf \times_{\bun_{G_V}} \shtgvmu$.
Here, the map $F^\mathrm{pre} \to F_V^\mathrm{pre}$ is induced from taking the
fiber product of the maps
\[
  \igspreinf \to \igsvpreinf, \quad \bun_G \to \bun_{G_V}, \quad \shtgmuone \to
  \shtgvmu
\]
appearing in \eqref{Eq:TheCubeInt}.

\begin{Prop} \label{Prop:UniqueLiftsShimura}
  The map $i^{\diamond, \mathrm{pre}} \colon \scrsdpreinf \to \scrsdvpreinf$ is proper*.
\end{Prop}

\begin{proof}
  By Lemma~\ref{Lem:DiamondOfFormalScheme} together with
  \cite[Proposition~1.6]{GleasonSpecialization}, which says that a product
  of geometric points is totally disconnected in the sense of
  \cite[Definition~7.1]{EtCohDiam}, we see that $\scrsdpreinf$ and $\scrsdinf$
  take the same value on products of geometric points. Similarly, the presheaves of groupoids
  $\scrsdvpreinf$ and $\scrsdvinf$ take the same value on products of geometric
  points. Therefore it suffices to check that $i^\diamond \colon \scrsdinf \to
  \scrsdvinf$ is proper*.

  Using Lemma~\ref{Lem:ProperRepresentableLifting}, we reduce to verifying that
  $i^\diamond$ is proper and representable by diamonds. For any morphism of adic
  spaces $T = \spa(R^\sharp, R^{\sharp+}) \to
  \hatscrsgvinf^\mathrm{ad}$, the fiber product
  \[
    X = T \times_{\hatscrsgvinf^\mathrm{ad}}
    \hatscrsginf^\mathrm{ad}
  \]
  is a quasi-compact quasi-separated analytic adic space, as the map
  $i^\mathrm{ad}$ is an adic morphism. It follows that $X^\lozenge$ is a
  quasi-compact quasi-separated diamond, and therefore $i^\diamond$ is
  quasi-compact, quasi-separated, and representable by diamonds. To check
  properness, we use \cite[Proposition~18.3]{EtCohDiam} to reduce to a statement
  about lifting maps along $\spa(K, \mathcal{O}_K) \to \spa(K, K^+)$, where $K$
  is a perfectoid field. Again using Lemma~\ref{Lem:DiamondOfFormalScheme}, we
  reduce to lifting along $\spf \mathcal{O}_K \to \spf K^+$. This now follows
  from $i$ being a limit of proper maps.
\end{proof}

\begin{Prop} \label{Prop:UniqueLiftsBunG}
  The map $\bun_G \to \bun_{G_V}$ is separated*.
\end{Prop}

\begin{proof}

  By Lemma~\ref{Lem:LiftsDiagonal}, we may instead show that the diagonal
  $\Delta \colon \bun_G \to \bun_G \times_{\bun_{G_V}} \bun_G$ is proper*.
  We claim that the diagonal is
  a closed immersion; the result then follows from
  Lemma~\ref{Lem:ProperRepresentableLifting}, as closed immersions are proper
  and representable. Let $S$ be an arbitrary perfectoid space in characteristic
  $p$ and let $S \to \bun_G \times_{\bun_{G_V}} \bun_G$ be a morphism. We need
  to check that the pullback $T$ along the diagonal $\Delta$ defines a closed
  immersion $T \hookrightarrow S$.

  As in Section \ref{Sec:EtaleTensors}, we choose a finite collection of tensors $\{t_\alpha\}_{\alpha \in \mathscr{A}}
  \subset V^\otimes$ with the property that the intersection of the stabilizers
  of $t_\alpha$ inside $\mathrm{GL}(V)$ is $G$. Then a $G$-bundle on the
  Fargues--Fontaine curve $X_S$ gives rise to a vector bundle $\mathcal{E}$
  together with global sections $\{t_{\alpha,\mathrm{FF}}\}_{\alpha \in
  \mathscr{A}} \subset H^0(X_S, \mathcal{E}^\otimes)$. The map $S \to \bun_G
  \times_{\bun_{G_V}} \bun_G$ corresponds to a pair of $G$-bundles on $X_S$
  together with an isomorphism between the
  underlying $G_V$-bundles. Writing $\mathcal{E}$ for the common underlying
  vector bundle on $X_S$, we obtain two collections of tensors
  $\{t_{1,\alpha,\mathrm{FF}}\}_{\alpha \in \mathscr{A}},\{t_{2,\alpha,\mathrm{FF}}\}_{\alpha \in
  \mathscr{A}} \in H^0(X_S, \mathcal{E}^\otimes)$. The pullback $T$ is
  identified with the locus in $S$ over which $t_{1,\alpha,\mathrm{FF}} = t_{2,\alpha,\mathrm{FF}}$ for
  all $\alpha \in \mathscr{A}$.

  Since finite intersections of closed immersions are closed immersions, it
  suffices to prove that the locus of $t_{1,\alpha,\mathrm{FF}} = t_{2,\alpha,\mathrm{FF}}$ is a closed
  immersion for a single $\alpha$. This comes down to showing that for integers
  $n,m \ge 0$ and two $S$-points $t_1,t_2$ of the relative Banach--Colmez space
  $\operatorname{BC}(\mathcal{E}^{\otimes n} \otimes
  (\mathcal{E}^{\ast})^{\otimes m}) \to S$, the locus of $t_1=t_2$ defines a
  closed immersion. This follows since $\operatorname{BC}(\mathcal{E}^{\otimes n}
  \otimes (\mathcal{E}^{\ast})^{\otimes m}) \to \ProdPts$ is separated, which
  is proven in \cite[Proposition II.2.16]{FarguesScholze}.
\end{proof}

\begin{Prop} \label{Prop:UniqueLiftsShtuka}
  The map $\shtgmuone \to \shtgvmu$ is proper*.
\end{Prop}

\begin{proof}
  Let $S$ be as in Definition~\ref{Def:LiftsProductOfPoints}. \def\ProdRkOne{s}
  Decompose $\RkOne \to \ProdPts$ into
  \[
    {\textstyle\RkOne} \to \ProdRkOne = \spa({\textstyle (\prod_i
    C_i^+)[\varpi^{-1}], \prod_i \mathcal{O}_{C_i}}) \to \ProdPts.
  \]
  It follows from the discussion after \cite[Lemma~2.4.4]{PappasRapoportShtukas} that the structure maps
  $\shtgmuone \to \spd \mathcal{O}_E$ and $\shtgvmu \to \spd\zp$ have uniquely
  existing lifts for $\ProdRkOne \to \ProdPts$. On the other hand,
  \cite[Remark~11.11]{ZhangThesis} tells us that $\shtgmuone \to \spd
  \mathcal{O}_E$ and $\shtgvmu \to \spd\zp$ have uniquely existing lifts (as in \cite[Definition 2.1.8.(1)]{Companion}) for
  $\RkOne \to \ProdRkOne$. Here to get the bounds on the legs, we are using the
  fact that $\mathbb{M}_{\mathcal{G},\mu}^{\mathrm{v}} \hookrightarrow
  \operatorname{Gr}_{\mathcal{G},\spd \mathcal{O}_E}$ is a closed immersion,
  hence has uniquely existing lifts for $\RkOne \to \ProdRkOne$ by
  Lemma~\ref{Lem:ProperRepresentableLifting}. This shows that both maps have
  uniquely existing lifts for $\RkOne \to \ProdPts$.

  Now consider the diagram
  \[ \begin{tikzcd}
    \shtgmuone \arrow{r} \arrow{d} & \shtgvmu \arrow{d} \\ \spd \mathcal{O}_E
    \arrow{r} & \spd \zp.
  \end{tikzcd} \]
  Both vertical maps are proper*, and so is $\spd \mathcal{O}_E \to \spd \zp$.
  Lemma~\ref{Lem:LiftsTwoOutOfThree} now implies that $\shtgmuone \to \shtgvmu$
  also is proper*.
\end{proof}

\begin{Prop} \label{Prop:UniqueLiftsIgusa}
  The map $\igspreinf \to \igsvpreinf$ is separated*.
\end{Prop}

\begin{proof}
  Let $\RkOne \to \ProdPts$ be as in Definition~\ref{Def:LiftsProductOfPoints}. Let $x_1, x_2 \in \igspreinf(\ProdPts)$ be two points and let $\psi \colon \iota(x_1) \to \iota(x_2)$ be a morphism in
  $\igsvpreinf(\ProdPts)$, such that the restriction of $\psi$ to
  $\igsvpreinf(\RkOne)$ is (uniquely) induced from a morphism
  $x_1 \vert_{\RkOne} \to x_2 \vert_{\RkOne}$ in $\igspreinf(\RkOne)$. We wish
  to verify that $\psi$ is (uniquely) induced from a morphism $x_1 \to x_2$ in
  $\igspreinf(\ProdPts)$. The condition for a morphism in $\igsvpreinf(\ProdPts)$
  to come from a morphism in $\igspreinf(\ProdPts)$ is, by definition, that the
  induced morphism in $\bun_{G_V}(\ProdPts)$ comes from a morphism in
  $\operatorname{Bun}_{G}(\ProdPts)$. The result now follows from
  Proposition~\ref{Prop:UniqueLiftsBunG}.
\end{proof}

\begin{Prop} \label{Prop:UniqueLiftsF}
  The map $F^{\mathrm{pre}} \to F_V^{\mathrm{pre}}$ is separated*.
\end{Prop}

\begin{proof}
  This follows from Lemma~\ref{Lem:UniqueLiftsFiberProd} in combination with
  Propositions \ref{Prop:UniqueLiftsBunG}, \ref{Prop:UniqueLiftsIgusa} and
  \ref{Prop:UniqueLiftsShtuka}.
\end{proof}

We also record the following crucial ingredient from \cite{ZhangThesis}.

\begin{Thm}[{\cite[Proposition~8.14]{ZhangThesis}}] \label{Thm:SiegelCartesian}
  For $\ProdPts$ a product of geometric points in characteristic $p$, the map
  $\scrsdvpreinf(\ProdPts) \to F_V^{\mathrm{pre}}(\ProdPts)$ is a bijection.
\end{Thm}

\begin{Rem}
  Although \cite{ZhangThesis} works with Shimura varieties at a finite
  level $M_p M^p$, it can be readily checked that the isomorphisms are
  compatible with appropriate transition maps as $M^p$ shrinks along a
  neighborhood basis of $1 \in \mathsf{G}_V(\afp)$, and hence induce an
  isomorphism at infinite level away from $p$.\footnote{Here we use
  Lemma~\ref{Lem:ComparisonZhang} to identify $\operatorname{Igs}_{M}\gvx$
  with the Igusa stack studied in \cite{ZhangThesis}.}
\end{Rem}

We are finally ready to prove Theorem~\ref{Thm:HodgeMain}.

\begin{Prop} \label{Prop:ProdOfGeomPointsBijective}
  For $\ProdPts$ a product of geometric points in characteristic $p$, the map
  $\scrsdpreinf(\ProdPts) \to F^{\mathrm{pre}}(\ProdPts)$ is a bijection.
\end{Prop}

\begin{proof}
  We first prove surjectivity. Let $\RkOne \to \ProdPts$ be as in
  Definition~\ref{Def:LiftsProductOfPoints}, and let $f \colon \ProdPts \to
  F^{\mathrm{pre}}$ be a morphism. By Theorem~\ref{Thm:SiegelCartesian}, the
  composition $\ProdPts \xrightarrow{f} F^\mathrm{pre} \to F_V^\mathrm{pre}$
  induces a morphism $g \colon \ProdPts \to \scrsdvpreinf$. On the other hand,
  the map $\scrsdpreinf(\RkOne) \to F^{\mathrm{pre}}(\RkOne)$ is a bijection by
  Corollary~\ref{Cor:RankOneGeometricBijective}, and so the restriction of $f$
  to $\RkOne$ lifts uniquely to a map $h \colon \RkOne \to \scrsdpreinf$. By the
  commutativity of \eqref{Eq:CommutativeDiagramShimuraIgusa} together with
  Theorem~\ref{Thm:SiegelCartesian} applied to each $s_i$, the restriction of
  $g$ to $\RkOne$ agrees with $i^{\diamond,\mathrm{pre}} \circ h$, as they agree
  after composing with $\scrsdvpreinf \to F_V^\mathrm{pre}$. By
  Proposition~\ref{Prop:UniqueLiftsShimura}, this means that there is a unique
  lift $\tilde{h} \colon \ProdPts \to \scrsdpreinf$. We summarize the above paragraph with the diagram
  \[ \begin{tikzcd}
    \RkOne \arrow{r}{h} \arrow{d} & \scrsdpreinf \arrow{d} \arrow{r}{\alpha} &
    F^\mathrm{pre} \arrow{d} \\ \ProdPts \arrow{r}[']{g} \arrow[ur, densely dotted, "\tilde{h}"]
    \arrow{rru}[',near end]{f} & \scrsdvpreinf \arrow{r} & F_V^\mathrm{pre}.
  \end{tikzcd} \]
We claim that the diagram is commutative, that is, that $f = \alpha \circ \tilde{h}$. To see this, we note that both $f$ and $\alpha \circ
  \tilde{h}$ are lifts of $F^\mathrm{pre} \to
  F_V^\mathrm{pre}$ along $\RkOne \to \ProdPts$, and hence Proposition~\ref{Prop:UniqueLiftsF} implies that
  $f = \alpha \circ \tilde{h}$.

  We now prove injectivity. Let $f_1, f_2$ be $\ProdPts$-points of
  $\scrsdpreinf$ that agree after composition with $\scrsdpreinf \to
  F^{\mathrm{pre}}$. By the commutativity of
  \eqref{Eq:CommutativeDiagramShimuraIgusa} and
  Theorem~\ref{Thm:SiegelCartesian}, it follows that $i^{\diamond, \mathrm{pre}}
  \circ f_1=i^{\diamond, \mathrm{pre}} \circ f_2$. From
  Corollary~\ref{Cor:RankOneGeometricBijective} it follows that $f_1$ and $f_2$
  agree after restriction to $\RkOne$. It now follows from
  Proposition~\ref{Prop:UniqueLiftsShimura} that $f_1=f_2$.
\end{proof}

\begin{proof}[Proof of Theorem~\ref{Thm:HodgeMain}]
  This follows from Proposition~\ref{Prop:ProdOfGeomPointsBijective} combined with Lemma \ref{Lem:ProperStarIsomorphismRankOne}.
\end{proof}
}

%% file: Section7_Current.tex
{
\def\gvp{\mathsf{G}_{V'}}
\def\gvpp{\mathsf{G}_{V''}}
\def\gvxp{(\mathsf{G}_{V'},\mathsf{H}_{V'})}
\def\gvxpp{(\mathsf{G}_{V''},\mathsf{H}_{V''})}

\def\igsinfp{\mathrm{Igs}_{\Xi^\prime}\gxp}
\def\igsinfpp{\mathrm{Igs}_{\Xi^{\prime\prime}}\gx}
\def\igspreinfp{\mathrm{Igs}^{\mathrm{pre}}_{\Xi^\prime}\gxp}
\def\igspreinfpp{\mathrm{Igs}^{\mathrm{pre}}_{\Xi^{\prime\prime}}\gx}

\def\igsinftwo{\mathrm{Igs}_{\Xi^\prime}\gx}
\def\gtheta{\mathsf{G}_{\Theta}}
\def\gxtheta{(\mathsf{G}_{\Theta}, \mathsf{X}_{\Theta})}

\def\shtgonemuonerat{\mathrm{Sht}_{\mathcal{G}_{1},\mu,\delta=1,E}}
\def\shtgtwomuonerat{\mathrm{Sht}_{\mathcal{G}_{2},\mu,\delta=1,E}}
\def\shtgmupoe{\mathrm{Sht}_{\mathcal{G}^\prime,\mu', \mathcal{O}_{E}}}

In our construction of the Igusa stack $\igsinf$ we made a choice of Hodge embedding $\iota:\gx \to \gvx$, stabilizer Bruhat--Tits group scheme $\mathcal{G}$ and self-dual lattice $V_{\zp}$, signified by the notation $\Xi=(\mathcal{G}, \iota, V_{\zp})$. In this section we show that morphisms $f:\gx \to \gxp$ induce (unique) morphisms $\igsinf \to \igsinfp$, in particular showing that our construction does not depend on $\Xi$. To be precise, here we take the Igusa stack for $\gx$ with respect to a place $v$ of $\mathsf{E}$ above $p$, and the Igusa stack for $\gxp$ with respect to the induced place $v'$ of $\mathsf{E}' \subset \mathsf{E}$.

\subsection{Functoriality in parahoric group schemes} \label{Sec:IntegralFunctoriality}

We begin by proving a lemma. 
\begin{Lem} \label{Lem:Multiplier}
  Let $\gx$ be a Hodge type Shimura datum, and let $\iota_1 \colon \gx \to \gvx$
  and $\iota_2 \colon \gx \to \gvxp$ be two morphisms of Shimura data. If $c_1
  \colon \gv \to \mathbb{G}_m$ and $c_2 \colon \gvp \to \mathbb{G}_m$ are the
  similitude characters, then $c_1 \circ \iota_1 = c_2 \circ \iota_2$.
\end{Lem}

\begin{proof}
  Write $w \colon \mathbb{G}_m \to \g$ for the weight homomorphism, which is
  defined over $\mathbb{Q}$ as $\gx$ is of Hodge type. Then $w_1 = \iota_1 \circ
  w \colon \mathbb{G}_m \to \gv$ and $w_2 = \iota_2 \circ w \colon \mathbb{G}_m
  \to \gvp$ are the weight homomorphisms for $\gvx$ and $\gvxp$. We also observe
  that $c_1 \circ w_1$ and $c_2 \circ w_2$ are both given by $z \mapsto z^2$. 
	
  The morphisms $c_1 \circ \iota_1$ and $c_2 \circ \iota_2$ both factor through
  $\g \to \gab$ and moreover through the maximal $\mathbb{Q}$-split quotient $T$ of
  $\gab$. Thus it suffices to show that the induced morphisms $h_1, h_2 \colon T
  \to \mathbb{G}_m$ are equal. Since $\gx$ is of Hodge type, the
  $\mathbb{Q}$-split rank of the central torus $Z_{\mathsf{G}}^{\circ}$ is equal to $1$, and
  this implies that $T$ is of rank $1$. Consider the homomorphism
  \[
    g \colon \mathbb{G}_m \xrightarrow{w} \mathsf{G} \to T.
  \]
  Observe that $h_1 \circ g = h_2 \circ g$ as both $c_1 \circ w_1$ and
  $c_2 \circ w_2$ are given by $z \mapsto z^2$. This moreover shows that $g$ is nonzero, and hence a surjection since $T$ is abstractly isomorphic to $\mathbb{G}_m$. It follows that $h_1 = h_2$.
\end{proof}

\subsubsection{Setup} Suppose $f:\gx \to \gxp$ is a morphism of Shimura data of Hodge type, and let $\mathsf{E}$ and $\mathsf{E}'$ be the reflex fields of $\gx$ and $\gxp$, respectively. Fix a prime $p$ and a place $v$ of $\mathsf{E}$ above $p$ as before, let $v'$ be the induced place of $\mathsf{E}'$, and let $E' \subset E$ be the induced inclusion of completions. \smallskip

Let $\mathcal{G}$ and $\mathcal{G}'$ be stabilizer Bruhat--Tits group schemes over $\zp$ with generic fibers $G = \mathsf{G}_{\qp}$ and $G' = \mathsf{G}'_{\qp}$. Let $K_p = \mathcal{G}(\zp)$ and $K_p' = \mathcal{G}'(\zp)$. If $f(K_p) \subset K_p'$ then there is a morphism of Shimura varieties
\begin{align} \label{Eq:GenericFibreMorphism}
    \mathbf{Sh}_{K_p}\gx \to \mathbf{Sh}_{K_p'}\gxp_{E}.
\end{align}
To get a morphism of integral models, we need to assume that $f(\mathcal{G}(\zpbr)) \subset \mathcal{G}'(\zpbr)$, or equivalently (see \cite[Corollary 2.10.10]{KalethaPrasad}) that $f$ extends (necessarily uniquely) to a morphism $\mathcal{G} \to \mathcal{G}'$. The following result is \cite[Proposition 4.1.10]{Companion}, but is essentially due to Pappas--Rapoport, see \cite[Theorem 4.3.1]{PappasRapoportShtukas}.
\begin{Lem} \label{Lem:Functoriality}
    If $f_{\qp}$ extends to a morphism $\mathcal{G} \to \mathcal{G}'$, then \eqref{Eq:GenericFibreMorphism} extends uniquely to a morphism
\begin{align}
    \scrs_{K_p}\gx \to \scrs_{K'_p}\gxp_{\mathcal{O}_{E}},
\end{align}
which moreover sits in a $2$-commutative diagram
\begin{equation}
    \begin{tikzcd}
        \scrs_{K_p}\gx^{\diamond} \arrow{r} \arrow{d} & \scrs_{K'_p}\gxp^{\diamond}_{\mathcal{O}_{E}} \times_{\spd \mathcal{O}_{E'}} \spd \mathcal{O}_{E} \arrow{d} \\
        \shtgmu \arrow{r} & \shtgmupoe.
    \end{tikzcd}
\end{equation}
\end{Lem}

Fix Hodge embeddings
\begin{align*}
	\iota: \gx \to \gvx \text{ and } \iota': \gxp \to \gvxp. 
\end{align*}
We assume as in Section \ref{subsub:Zarhin} that there are self-dual $\zp$-lattices $V_\zp \subset V_\qp$ and $V_\zp' \subset V_\qp'$ such that $\mathcal{G}(\zpbr)$ is the stabilizer of $V_{\zpbr}:=V_\zp \otimes_{\zp}\zpbr$ in $G(\qpbr)$, and such that $\mathcal{G}'(\zpbr)$ is the stabilizer of $V'_{\zpbr}=V'_\zp \otimes_\zp \zpbr$ in $G'(\qpbr)$. Then as explained in \cite[Section 4.3]{PappasRapoportShtukas}, this means that $\iota:G \to G_V$ extends to a dilated immersion $\mathcal{G} \to \operatorname{GSp}(V_{\zp})$ and that $\iota':G' \to G_{V'}$ extends to a dilated immersion $\mathcal{G}' \to \operatorname{GSp}(V_{\zp}')$. We let $V_{\zlocp} = V_\zp \cap V$,  $V'_{\zlocp} = V_\zp' \cap V'$ and set $M_p=\operatorname{GSp}(V_{\zp})(\zp)$, $M_p'=\operatorname{GSp}(V_{\zp}')(\zp)$.

\subsubsection{The diagonal embedding} Let $V''_{\zlocp} = V'_{\zlocp} \oplus V_{\zlocp}$, equipped with the direct sum of the symplectic forms on $V'_{\zlocp}$ and $V_{\zlocp}$, and let $V'' = V_{\zlocp}'' \otimes_{\zlocp} \mathbb{Q}$; set $M_p'' = \operatorname{GSp}(V_{\zlocp}'')(\zp)$.  Observe that there is an idempotent $\Theta \in \operatorname{End}_{\zlocp}(V''_{\zlocp})$ projecting to $V'_{\zlocp}$, and note that $\Theta'=1-\Theta$ is an idempotent projecting to $V_{\zlocp}$. 

Let $\mathcal{O}_B \subset \operatorname{End}_{\zlocp}(V''_{\zlocp})$ be the $\zlocp$-algebra generated by $\Theta$ and let $B = \mathcal{O}_B \otimes_{\zlocp} \mathbb{Q}$. Then $(B,\mathcal{O}_B, V''_{\zlocp},\ast,h)$ is an unramified integral Shimura PEL datum in the sense of \cite[Definition 5.2]{ShinIgusa}. Here $\ast$ is the involution on $V''_{\zlocp}$ coming from the symplectic pairing and $h$ is a point of $\mathsf{X}_V \times \mathsf{X}_{V'}$. The corresponding similitude group is given by 
\begin{align}
    \gtheta=\gv \times_{\mathbb{G}_m} \gvp \subset \gvpp
\end{align}
and we similarly define $\mathcal{G}_{\Theta}=\operatorname{GSp}(V_{\zp}) \times_{\mathbb{G}_m} \operatorname{GSp}(V'_{\zp})$.
This PEL datum induces a Shimura datum $\gxtheta$. By Lemma \ref{Lem:Multiplier}, the natural map $(\iota,\iota'\circ f): \g \to \gv \times \gvp$ factors through $\gtheta$, and this in fact induces a morphism of Shimura data $\gx \to \gxtheta$. The composition of this map with $\gxtheta \to \gvxpp$ defines a Hodge embedding 
\begin{align*}
	\iota'':\gx \to \gvxpp.
\end{align*}
The assumption that $f_{\qp}$ extends to a morphism $\mathcal{G} \to \mathcal{G}'$ tells us that $\mathcal{G}(\zpbr)$ is the stabilizer of $V_{\zpbr}''$ in $G(\qpbr)$. We can thus use the embedding $\iota''$ to define an Igusa stack $\igsinfpp$, where we write $\Xi =
(\mathcal{G}, \iota, V_{\mathbb{Z}_{p}})$, $\Xi' =
(\mathcal{G}', \iota', V_{\mathbb{Z}_{p}}')$, and $\Xi'' =
(\mathcal{G}, \iota'', V_{\mathbb{Z}_{p}}'')$.
\begin{Prop}\label{Prop:FunctorialitySpecialCase} Assume that $f_{\qp}$ extends to a morphism $\mathcal{G} \to \mathcal{G}'$. Then in the situation described above, the projections from $V_{\zlocp}''$ onto $V_{\zlocp}$ and $V'_{\zlocp}$ induce morphisms
	\begin{align}\label{Eq:SummandFunctoriality}
		\igsinf \leftarrow \igsinfpp \to \igsinfp
	\end{align}
	which are compatible with the 2-Cartesian diagrams in \eqref{Eq:TheDiagram} for $(G,\mathsf{X},\mathcal{G})$ and $(G',\mathsf{X}',\mathcal{G}')$.
	Moreover, the left-hand morphism in \eqref{Eq:SummandFunctoriality} is an isomorphism. 
\end{Prop} 
\begin{proof}
    Let us first show the existence of the morphism on the right. By \cite[Lemma 2.1.2]{KisinModels}, the map $\iota''$ induces an embedding
	\begin{align*}
		\mathbf{Sh}_{K_p}\gx \hookrightarrow \mathbf{Sh}_{M_p''}\gvxpp_E.
	\end{align*}
	By \cite[Corollary 4.1.13]{Companion}, see \cite[Theorem 4.5.2]{PappasRapoportShtukas}, the integral model $\scrs_{K_p}\gx$ is independent of the Hodge embedding used to define it. Thus $\scrs_{K_p}\gx$ can be equivalently defined as the normalization of the closure of $\mathbf{Sh}_{K_p}\gx_E$ in $\scrs_{M_p''}\gvxpp_{\mathcal{O}_E}$, and the embedding above extends to a morphism of integral models
	\begin{align*}
		\scrs_{K_p}\gx \to \scrs_{M_p''}\gvxpp_{\mathcal{O}_E}.
	\end{align*} 
     Recall the auxiliary Shimura datum of PEL type $\gxtheta$ and the parahoric model $\mathcal{G}_\Theta$, and let $K_{\Theta,p}= \mathcal{G}_\Theta(\zp)$. Let $\scrs_\Theta$ be the moduli space over $\mathcal{O}_E$ of tuples $(A, \lambda, \eta^p, \Theta_A)$, where $(A, \lambda, \eta^p) \in \scrs_{M_p''}\gvxpp$, and $\Theta_A \in \operatorname{End}(A)$ is an endomorphism compatible with $\Theta$ via the isomorphism $\varepsilon: \mathcal{V}^p \xrightarrow{\sim} V'' \otimes \underline{\afp}$. Then $\scrs_\Theta$ is an integral model for $\mathbf{Sh}_{K_{\Theta,p}}(\mathsf{G}_\Theta, \mathsf{X}_\Theta)_E$, and moreover it is (the base change to $\mathcal{O}_E$ of) \textit{the} canonical integral model in the sense of \cite{PappasRapoportShtukas}. Indeed, it is (the base change to $\mathcal{O}_E$ of) the smooth integral model constructed by Kottwitz corresponding to the unramified integral Shimura PEL datum  $(B,\mathcal{O}_B, V''_{\zlocp},\ast,h)$, see \cite[Section 5]{ShinIgusa}. By \cite[Lemma 3.2]{XuPEL}, it agrees with the one constructed by Kisin in \cite{KisinModels}, which in turn agrees with the model constructed by Pappas--Rapoport in \cite{PappasRapoportShtukas} because the constructions are the same. 

By functoriality of canonical integral models (see Lemma \ref{Lem:Functoriality}), we have a commutative diagram 
	\begin{equation}\label{Eq:scrs_theta}
		\begin{tikzcd}
			\scrs_{K_p}\gx 
				\arrow[r] \arrow[d, "f_\scrs"]
			& \scrs_\Theta 
				\arrow[d]
			\\ \scrs_{K_p'}\gxp_{\mathcal{O}_E} 
				\arrow[r, "\iota'"]
			& \scrs_{M_p'}\gvxp_{\mathcal{O}_E}.
		\end{tikzcd}
	\end{equation}
	uniquely extending the morphisms on the generic fibers. Note that the map $\scrs_\Theta \to
  \scrs_{M_p'}\gvxp_{\mathcal{O}_{E}}$ is given by sending $(A, \lambda, \eta^p, \Theta_A)$ to
  $\Theta_A(A,\lambda, \eta^p)$. \smallskip
	
	Now we construct a morphism of v-stacks $\igsinfpp \to \igsinfp$. It suffices to construct instead $\igspreinfpp \to \igspreinfp$, where the latter objects are  the presheaves of groupoids introduced in Section \ref{Subsec:IgusaDefinition}. Fix an affinoid perfectoid space $S = \spa(R,R^+)$ in characteristic $p$. To an object $x:\spf R^{\sharp+} \to \scrs_{K_p}\gx$ in $\igspreinfpp(S)$, we associate the object $x'$ given by composition
	\begin{align*}
 \spf R^{\sharp+} \to \scrs_{K_p}\gx \xrightarrow{f_\scrs} \scrs_{K_p'}\gxp_{\mathcal{O}_{E}}.
	\end{align*} 
	On morphisms the process is more complicated: Suppose we are given a pair of untilts $(R^{\sharp_1}, R^{\sharp_1+}),(R^{\sharp_2}, R^{\sharp_2+})$ over $\mathcal{O}_E$, morphisms $x:\spf R^{\sharp_1+} \to \scrs_{K_p}\gx, y: \spf R^{\sharp_2+} \to \scrs_{K_p}\gx$ over $\spf \mathcal{O}_E$, and a morphism $\psi: x \to y$. So $\psi: A_x \dashrightarrow A_y$ is a formal quasi-isogeny such that $\eta_y \circ \psi = \eta_x$ as morphisms $\mathcal{V}^p_{x} \to V\otimes \underline{\afp}$ and such that the induced map $\mathcal{E}(A_x) \to \mathcal{E}(A_y)$ of vector bundles on $X_S$ is induced by a (necessarily unique) isomorphism of $G$-bundles $\mathbb{L}_{\mathrm{crys},x} \to \mathbb{L}_{\mathrm{crys},y}$.
	
	Now, since $\iota'':\scrs_{K_p}\gx \to \scrs_{M_p''}\gvxpp_{\mathcal{O}_E}$ factors through $\scrs_\Theta$, we obtain objects $(A_x, \lambda_x, \eta_x, \Theta_x)$ and $(A_y, \lambda_y, \eta_y, \Theta_y)$ in $\widehat{\scrs}_\Theta(\spf R^{\sharp_i+})$, $i=1,2$, which map to $\iota''(x)$ and $\iota''(y)$, respectively. On the other hand, by the commutativity of the diagram \eqref{Eq:scrs_theta}, we see that $(\iota' \circ f_\scrs)(x) = \Theta_x(A_x, \lambda_x, \eta_x)$ and $(\iota' \circ f_\scrs)(y) = \Theta_y(A_y, \lambda_y, \eta_y)$. Let $x' = f_\scrs(x)$ and $y' = f_\scrs(y)$, so $\Theta_x(A_x, \lambda_x, \eta_x) = (A_{x'}, \lambda_{x'}, \eta_{x'})$ and $\Theta_y(A_y, \lambda_y, \eta_y) = (A_{y'}, \lambda_{y'}, \eta_{y'})$.
	
	We claim that $\psi$ induces a formal quasi-isogeny $\psi':A_{x'} \dashrightarrow A_{y'}$. For this it is enough to show that $\Theta_y \circ \psi = \psi \circ \Theta_x$. In turn, by  the proof of Lemma \ref{Lem:Discrete}, it is enough to show this commutativity for the induced morphisms $\mathcal{V}^p_{x} \to \mathcal{V}^p_{y}$. But this follows from the compatibilities of $\Theta_x$ with $\eta_x$ and $\Theta_y$ with $\eta_y$, along with the identity $\eta_y \circ \psi = \eta_x$. Note also that the compatibility of $\Theta_x$ with $\eta_x$ and $\Theta_y$ with $\eta_y$, along with the fact that $\Theta_x(\mathcal{V}^p_{x}) = \mathcal{V}^p_{x'}$, implies that $\eta_{y'} \circ \psi' = \eta_{x'}$.

  It remains only to check that there is a morphism of $G'$-bundles $\psi:\mathbb{L}_{\mathrm{crys},x'} \to \mathbb{L}_{\mathrm{crys},y'}$ on $X_S$ which induces $\mathcal{E}(\psi'): \mathcal{E}(A_{x'}) \to \mathcal{E}(A_{y'})$. By functoriality of $\pi_{\mathrm{crys}}$, we have a natural isomorphism $\mathbb{L}_{\mathrm{crys},x} \times^G G' \xrightarrow{\sim} \mathbb{L}_{\mathrm{crys},x'}$, and similarly for $y'$. Thus the morphism $\psi_{\mathrm{crys}}:\mathbb{L}_{\mathrm{crys},x} \to \mathbb{L}_{\mathrm{crys},y}$ coming from $\psi: x \to y$ induces a morphism of $G'$-bundles $\psi'_{\mathrm{crys}}:\mathbb{L}_{\mathrm{crys},x'} \to \mathbb{L}_{\mathrm{crys},y'}$. It now suffices to check that this induces the map $\mathcal{E}(\psi')$ of $\operatorname{GL}(V')$-bundles on $X_S$. But this follows from the 2-commutativity of the cube
\begin{equation}\label{Eq:scrs_theta2}
\begin{tikzcd}[column sep=tiny, row sep=tiny]
    & \bun_{G} \arrow[rr] \arrow[dd]& &\bun_{G_{\Theta}} \arrow[dd]
    \\ \scrs_{K_p}\gx^{\diamond}
        \arrow[rr, crossing over] \arrow[ur] \arrow[dd, "f_\scrs^\diamond"]
    & & \scrs_\Theta^{\diamond} \arrow[ur]
    \\ & \bun_{G'} \arrow[rr] & & \bun_{G_{V'}} 
    \\ \scrs_{K_p'}\gxp^{\diamond}_{\mathcal{O}_E} 
        \arrow[rr, "\iota'"] \arrow[ur]
    & & \scrs_{M_p'}\gvxp^{\diamond}_{\mathcal{O}_E} \arrow[ur] \arrow[from=uu, crossing over]
\end{tikzcd}
\end{equation}
 coming from the fact that all the integral models in equation \eqref{Eq:scrs_theta} satisfy the Pappas--Rapoport axioms, and applying Lemma \ref{Lem:Functoriality}. 
 
 This completes the construction of the map $\igspreinfpp \to \igspreinfp$. Let $\igsinfpp \to \igsinfp$ be the induced map. As a consequence of the construction we see that the following cube is $2$-commutative.
 \begin{equation}
 \begin{tikzcd}[column sep=tiny, row sep=tiny] \label{Eq:IgusaFunctorialityCube}
     & \igsinfpp \arrow{rr} \arrow{dd} && \igsinfp \arrow{dd} \\ 
     \scrs_{K_p}\gx^{\diamond} \arrow[rr, crossing over] \arrow{dd} \arrow{ur} & &\scrs_{K_p'}\gxp^{\diamond}_{\mathcal{O}_E} \arrow{ur}
     \\ 
     & \bungmu \arrow{rr} && \bun_{G',{\mu'}^{-1}}. \\
     \shtgmu \arrow{rr} \arrow{ur} && \shtgmupoe \arrow{ur} \arrow[from=uu, crossing over]
 \end{tikzcd}
 \end{equation}
To construct the left morphism in the proposition, we follow the same argument as above, replacing $\Theta$ by the endomorphism $\Theta'$ of $V''_{\zlocp}$ whose image is $V_{\zlocp}$. This gives us a morphism $\igsinfpp \to \igsinf$. To show this is an isomorphism it is enough to pass to the v-cover $\operatorname{BL}:\shtgmuonerat \to \bungmu$ of Corollary~\ref{Cor:VSurjectiveBL}. But by commutativity of the cube in \eqref{Eq:IgusaFunctorialityCube} and Theorem \ref{Thm:HodgeMain} for both $\igsinfpp$ and $\igsinf$, the morphism 
	\begin{align*}
		\igsinfpp \times_{\bun_G} \shtgmuonerat \to \igsinf \times_{\bun_G} \shtgmuonerat
	\end{align*}
	is identified with the identity map $\mathbf{Sh}_{K_p}\gx^{\circ, \lozenge} \to \mathbf{Sh}_{K_p}\gx^{\circ, \lozenge}$, so we are done.
\end{proof}

\subsubsection{Conclusion} Suppose $f:\gx \to \gxp$ is a morphism of Shimura data of Hodge type, and let $\mathsf{E}$ and $\mathsf{E}'$ be the reflex fields of $\gx$ and $\gxp$, respectively. Fix a prime $p$ and a place $v$ of $\mathsf{E}$ as before, let $v'$ be the induced place of $\mathsf{E}'$, and let $E' \subset E$ be the induced inclusion of completions. Make choices $\Xi =
(\mathcal{G}, \iota, V_{\mathbb{Z}_{p}})$ and $ \Xi' =
(\mathcal{G}', \iota', V_{\mathbb{Z}_{p}}')$ as in Section \ref{Subsec:IgusaDefinition}.
\begin{Cor} \label{Cor:IntegralFunctoriality}
If $f_{\qp}$ extends to a morphism $\mathcal{G} \to
  \mathcal{G}'$, then $f$ induces a unique morphism of v-stacks
	\begin{align*}
		\igsinf \to \igsinfp
	\end{align*}
    fitting in a commutative diagram
    \begin{equation}
    \begin{tikzcd}
        \scrs_{K_p}\gx^{\diamond} \arrow{r} \arrow{d} & \scrs_{K_p^\prime}\gxp^{\diamond} \arrow{d} \\
        \igsinf \arrow{r} & \igsinfp.
    \end{tikzcd}
    \end{equation}
\end{Cor}
\begin{proof}
The uniqueness follows from the v-surjectivity of the vertical arrows in the diagram, and existence is given by inverting the first arrow in \eqref{Eq:SummandFunctoriality} of Proposition \ref{Prop:FunctorialitySpecialCase}.
\end{proof}

\subsection{Change-of-parahoric morphisms} Let $(\mathsf{G},\mathsf{X})$ be a Shimura datum of Hodge type, with a prime $p$ and $v\mid p$ a place of the reflex field $\mathsf E$ as before. Make choices $\Xi =
(\mathcal{G}_1, \iota, V_{\mathbb{Z}_{p}})$ and $ \Xi' =
(\mathcal{G}_2, \iota', V_{\mathbb{Z}_{p}}')$ as in Section \ref{Subsec:IgusaDefinition}. Assume that the identity map of $G$ extends to a morphism $\mathcal{G}_1 \to \mathcal{G}_2$ of quasi-parahoric group schemes. This is equivalent, by \cite[Corollary 2.10.10]{KalethaPrasad}, to assuming that $\mathcal{G}_1(\zpbr) \subset \mathcal{G}_2(\zpbr)$. Then by Corollary~\ref{Cor:IntegralFunctoriality} above, we have a morphism of the corresponding Igusa stacks.
\begin{Prop} \label{Prop:Forgetful}
    The map $\igsinf \to \igsinftwo$ constructed in Corollary~\ref{Cor:IntegralFunctoriality} is an isomorphism.
\end{Prop}
\begin{proof}
It suffices to check the map is an isomorphism after base change via the v-cover $\shtgonemuonerat \to \bungmu$, see Corollary~\ref{Cor:VSurjectiveBL}. By Lemma \ref{Lem:PotentiallyCrystalline} and the fiber product formulas for both Igusa stacks, see Theorem \ref{Thm:HodgeMain}, this base changed map can be identified with the natural map
\begin{align} \label{Eq:FiberProductMap}
\mathbf{Sh}_{K_{p,1}}\gx^{\circ,\lozenge} \to \mathbf{Sh}_{K_{p,2}}\gx^{\circ,\lozenge} \times_{\shtgtwomuonerat} \shtgonemuonerat.
\end{align}
To show that this is an isomorphism, we need to show that the diagram
\begin{equation} \label{Eq:CartesianDiagram}
    \begin{tikzcd}
        \mathbf{Sh}_{K_{p,1}}\gx^{\circ,\lozenge} \arrow{r} \arrow{d} & \shtgonemuonerat \arrow{d} \\
        \mathbf{Sh}_{K_{p,2}}\gx^{\circ,\lozenge} \arrow{r} & \shtgtwomuonerat      
    \end{tikzcd}
\end{equation}
is $2$-Cartesian. Using the description of the generic fiber of the stacks of shtukas from equation \eqref{Eq:GenericFibreShtukas}, we can identify \eqref{Eq:CartesianDiagram} with the diagram
\begin{equation}
    \begin{tikzcd}
        \mathbf{Sh}_{K_{p,1}}\gx^{\circ,\lozenge} \arrow{r} \arrow{d} & \left[\operatorname{Gr}_{G,\mu^{-1}} /\underline{K_{p,1}}\right] \arrow{d} \\
        \mathbf{Sh}_{K_{p,2}}\gx^{\circ,\lozenge} \arrow{r} & \left[\operatorname{Gr}_{G,\mu^{-1}} /\underline{K_{p,2}}\right]. 
    \end{tikzcd}
\end{equation}
The latter diagram is visibly $2$-Cartesian, for instance because both vertical maps are surjective and finite \'etale of degree $\#(K_{p,1} \backslash K_{p,2})$. Indeed, the diagram
\begin{equation}
\begin{tikzcd}
    \mathbf{Sh}_{K_{p,1}}\gx^{\circ,\lozenge} \arrow{r} \arrow{d} & \mathbf{Sh}_{K_{p,1}}\gx^{\lozenge} \arrow{d} \\
    \mathbf{Sh}_{K_{p,2}}\gx^{\circ,\lozenge} \arrow{r} & \mathbf{Sh}_{K_{p,2}}\gx^{\lozenge}
\end{tikzcd}    
\end{equation}
is Cartesian by \cite[Corollary~5.29]{ImaiMieda}, and the right-hand side is finite \'etale of degree $\#(K_{p,1} \backslash K_{p,2})$. 
\end{proof}

\begin{Cor} \label{Cor:CuriousCartesian}
If the identity map of $G$ extends to a map $\mathcal{G}_1 \to \mathcal{G}_2$, then the diagram
    \begin{equation}
        \begin{tikzcd}
            \scrs_{K_{p,1}}\gx^{\diamond} \arrow{r} \arrow{d} & \operatorname{Sht}_{\mathcal{G}_{1}, \mu,\delta=1} \arrow{d} \\
             \scrs_{K_{p,2}}\gx^{\diamond} \arrow{r} & \operatorname{Sht}_{\mathcal{G}_{2}, \mu,\delta=1}
        \end{tikzcd}
    \end{equation}
is $2$-Cartesian.
\end{Cor}
\begin{proof}
    This is a direct consequence of Theorem \ref{Thm:HodgeMain} and Proposition \ref{Prop:Forgetful}.
\end{proof}
\begin{Rem}
If one takes the reduction of the $2$-Cartesian diagram in Corollary \ref{Cor:CuriousCartesian} to $\affperf$, then one recovers \cite[Theorem 4.4.1]{vH} as a special case. 
\end{Rem}

\subsection{Independence of choices} Let $\gx$ be a Shimura datum of Hodge type with reflex field $\mathsf{E}$. Let $p$ be a prime number and let $v$ be a prime of $\mathsf{E}$ above $p$ and let $E$ be the $v$-adic completion of $\mathsf{E}$. Let $\Xi =
(\mathcal{G}, \iota, V_{\mathbb{Z}_{p}})$ be a choice as in Section \ref{Subsec:IgusaDefinition}, and let $\mathbf{Sh}\gx^{\circ,\lozenge}$ be the potentially crystalline locus introduced in Section \ref{Sec:PotCrys}. Then by Lemma~\ref{Lem:PotentiallyCrystalline}, Corollary~\ref{Cor:VSurjectiveBL} and \eqref{Eq:GenericFibreShtukas}, the natural map
\begin{align}
    \mathbf{Sh}\gx^{\circ,\lozenge} \to \igsinf
\end{align}
is a surjection in the v-topology, establishing $\igsinf$ as a quotient of $\mathbf{Sh} \gx^{\circ, \lozenge}$. This gives us an equivalence relation $R(\Xi) \subset \mathbf{Sh}\gx^{\circ,\lozenge} \times\mathbf{Sh}\gx^{\circ,\lozenge}$ defined as $\mathbf{Sh}\gx^{\circ,\lozenge} \times_{\igsinf} \mathbf{Sh} \gx^{\circ,\lozenge}$. 
\begin{Prop} \label{Prop:IndependenceOfParahoric}
    If $\Xi = (\mathcal{G}, \iota, V_\zp)$ and $\Xi' = (\mathcal{G}', \iota', V_\zp')$ are two choices as in Section \ref{Subsec:IgusaDefinition}, then $R(\Xi)=R(\Xi')$. 
\end{Prop}
\begin{proof}
Since $\mathcal{G}$ and $\mathcal{G}'$ are stabilizer Bruhat--Tits group schemes, there are points $x$ and $x'$ in $\mathcal{B}(G,\qp)$ such that 
\begin{align*}
    \mathcal{G}(\zpbr) = G(\qpbr)^1 \cap \mathrm{Stab}(x), \quad \text{ and } \quad \mathcal{G}'(\zpbr) = G(\qpbr)^1 \cap \mathrm{Stab}(x').
\end{align*}
  Let $\mathcal{F}$ and $\mathcal{F}'$ denote the facets of $\mathcal{B}(G,\qp)$ containing $x$ and $x'$, respectively, satisfying $\mathcal{G}(\zpbr) \supset G(\qpbr)^1_{\mathcal{F}}$ and $\mathcal{G}'(\zpbr) \supset G(\qpbr)^1_{\mathcal{F}'}$. From Lemma~\ref{Lem:StabilizerToPointwiseStabilizer}, we know that $G(\qpbr)^1_{\mathcal{F}}$ and $G(\qpbr)^1_{\mathcal{F}'}$ are also stabilizers of points in $\mathcal{B}(G,\qp)$; let us write $\mathcal{G}_1$ and $\mathcal{G}_2$ for the corresponding quasi-parahoric group schemes. Then $\mathcal{G}_1$ and $\mathcal{G}_{2}$ upgrade to choices $\Xi_1$ and $\Xi_2$ as in Section~\ref{Subsec:IgusaDefinition} because they are stabilizer Bruhat--Tits group schemes. By Proposition~\ref{Prop:Forgetful} we have $R(\Xi)=R(\Xi_1)$ and $R(\Xi')=R(\Xi_2)$. \smallskip 

Using Proposition \ref{Prop:Forgetful} and the above argument inductively, we see that the result follows if we can find a chain of facets $\mathcal{F}_3, \cdots, \mathcal{F}_n$ 
whose associated quasi-parahoric subgroups satisfy the following inclusion relations 
\begin{equation}
    \begin{tikzcd}[column sep=0.5]
        & G(\qpbr)_{\mathcal{F}_3}^1 \arrow[dl,hook] \arrow[dr,hook] & & \cdots \arrow[dl, hook] \arrow[dr, hook] & & G(\qpbr)_{\mathcal{F}_n}^1 \arrow[dr,hook] \arrow[dl, hook] \\
        G(\qpbr)_{\mathcal{F}_1}^1 & & G(\qpbr)_{\mathcal{F}_4}^1 & & G(\qpbr)_{\mathcal{F}_{n-1}}^1 & & G(\qpbr)_{\mathcal{F}_2}^1.
    \end{tikzcd}
\end{equation}
But such a chain exists by the contractibility of the Bruhat--Tits building $\mathcal{B}(G,\qp)$, see e.g., \cite[Axiom 4.1.1, Corollary 4.2.9]{KalethaPrasad}. 
\end{proof}

From now on we will write $\operatorname{Igs}\gx$ for $\igsinf$ and
$\operatorname{Igs}_{K^p}\gx$ for $\igs$. We now prove the functoriality stated
in Theorem~\ref{Thm:IntroIgusaGeneric}. Let $f \colon \gx \to \gxp$ be a
morphism of Shimura data and let $v \mid p$ be a prime of the reflex field
$\mathsf{E} \supset \mathsf{E}'$; let $v'$ be the induced prime of
$\mathsf{E}'$. Then there is a morphism of infinite level Shimura varieties
\begin{align}
    \mathbf{Sh}\gx^{\lozenge} \to \mathbf{Sh}\gxp_{E}^{\lozenge}.
\end{align}

\begin{Thm}\label{Thm:RationalFunctoriality}
  There is a unique morphism $\operatorname{Igs}\gx \to \operatorname{Igs}\gxp$
  such that the following diagram commutes.
\begin{equation}
    \begin{tikzcd}
        \mathbf{Sh}\gx^{\circ, \lozenge} \arrow{d} \arrow{r} & \mathbf{Sh}\gxp_{E}^{\circ, \lozenge} \arrow{d} \\
        \operatorname{Igs}\gx \arrow{r} & \operatorname{Igs}\gxp.
    \end{tikzcd}
\end{equation}
\end{Thm}
\begin{proof}
  It suffices to show that the equivalence relation from
  Proposition~\ref{Prop:IndependenceOfParahoric} is functorial for morphisms of
  Shimura data. By Corollary~\ref{Cor:IntegralFunctoriality} and
  Proposition~\ref{Prop:IndependenceOfParahoric} it suffices to show that we can
  find stabilizer Bruhat--Tits models $\mathcal{G}$ of $G$ and $\mathcal{G}'$ of $G'$ such that $f$ extends to a morphism $\mathcal{G} \to
  \mathcal{G}'$. This statement is Lemma~\ref{Lem:LandVogt}. 
\end{proof}
}

%% file: Section8.tex
\label{Sec:Cohomology}

In this section, we will apply the six functor formalism of \cite{EtCohDiam} and the machinery of \cite{FarguesScholze} to the cohomology of Shimura varieties of Hodge type. We will first show that the Igusa stack is cohomologically smooth and compute its dualizing sheaf, see Theorem \ref{Thm:IgusaDualizingComplex}; this completes the proof of Theorem \ref{Thm:IntroIgusaGeneric}. In Section \ref{sub:Hecke}, we will construct a sheaf $\mathcal{F}$ on $\bungmuk$, where $k$ is an algebraic closure of $\fp$, which controls the cohomology of the Shimura variety, see Theorem \ref{Thm: WeilCohoShiVar}. We will use this to prove Mantovan's product formula for the cohomology of the Shimura variety, see Theorem \ref{Thm:MantovanFormula}. We then prove Theorem \ref{Thm:IntroPerversity}, see Theorem \ref{Thm: Perversity}. We start in Section \ref{Sec:SheavesBunG} by briefly recalling from \cite{EtCohDiam} and \cite{FarguesScholze} some of the six functor formalism that we will use. In Section \ref{Sec:WeilGroupActions} we will prove a technical result about Frobenii and Weil group actions.

\subsection{\'Etale sheaves on \texorpdfstring{$\bun_{G}$}{BunG}} \label{Sec:SheavesBunG} Fix a prime $p$ and a prime $\ell\neq p$. Let $\Lambda$ be a Noetherian $\mathbb{Z}_\ell$-algebra in which $\ell$ is nilpotent, which will serve as our coefficient ring. For a small v-stack $X$ we write $D(X,\Lambda)$ for the triangulated category $D_\text{\'et}(X,\Lambda)$ in \cite[Definition 14.13]{EtCohDiam}. By \cite[Lemma 17.1]{EtCohDiam}, there is a (natural choice of) presentable stable $\infty$-category $\mathcal{D}_\text{\'et}(X,\Lambda)$ whose homotopy category is equivalent (as a triangulated category) to $D_\text{\'et}(X,\Lambda)$. We review below some preliminaries of \'etale sheaves on $\bun_{G}$ from \cite{FarguesScholze}. Throughout this section, we fix an algebraic closure $k$ of $\fp$.

Let $[b] \in B(G)$ and let $i_{b}:\bungk^{[b]} \to \bungk$ denote the inclusion. Choose $b \in [b]$ and let $G_b$ denote the $\sigma$-centralizer of $b$, see \cite[Section 1.11]{RapoportRichartz}, which is a connected reductive group over $\qp$. Then by \cite[Proposition V.2.2]{FarguesScholze}, there is an equivalence of categories
\begin{align}
    D(\bungk^{[b]}, \Lambda) \simeq D(G_b(\qp), \Lambda),
\end{align}
where $D(G_b(\qp), \Lambda)$ denotes the (unbounded) derived category of the category of smooth representations of $G_b(\qp)$ on $\Lambda$-modules.

\label{Subsub:ComparisonHuber} If $X$ is a quasi-separated rigid space over a non-archimedean field $C$, then $X^{\lozenge}$ is a locally spatial diamond, see \cite[Lemma 15.6]{EtCohDiam}. We will consider the \'etale site $X^{\lozenge}_{\text{\'et}}$ of the locally spatial diamond $X^{\lozenge}$, see \cite[Definition 14.1]{EtCohDiam}. The derived category $D(X^{\lozenge}_{\text{\'et}}, \Lambda)$ of sheaves of $\Lambda$-modules on $X^{\lozenge}_{\text{\'et}}$ admits a fully faithful functor to $D(X^{\lozenge}, \Lambda)$, see \cite[Proposition 14.15]{EtCohDiam}. We will use \cite[Lemma 15.6]{EtCohDiam} to identify the \'etale site of $X$ with the \'etale site of $X^{\lozenge}$. This induces isomorphisms of derived categories of \'etale sheaves compatible with $\ast$-pullback, see \cite[Proposition 14.7]{EtCohDiam}, and thus $\ast$-pushforward because it is an adjoint. This comparison is moreover compatible with internal hom and tensor product, and also with $!$-pushforward and $!$-pullback since the constructions are the same (extension by zero to the canonical compactification, and then pushforward). Thus we may use results from Huber's book \cite{Huber96} for $X$, in the six-functor formalism of \cite{EtCohDiam} for $X^{\lozenge}$. We will do this in what follows without further comment.

\label{Sec:ULA} We will use the notion of universally locally acyclic (ULA) sheaves from \cite[Definition IV.2.31]{FarguesScholze}. By Theorem V.7.1 of loc.\ cit., a complex $A \in D(\bungk, \Lambda)$ is ULA with respect to $\bungk\to \spd k$, if and only if $M_{b}=i_{b}^{\ast} A \in D(G_b(\qp), \Lambda)$ is admissible in the sense that for any pro-$p$ compact open subgroup $K \subset G_b(\qp)$, the complex $M_{b}^K$ is (quasi-isomorphic to) a perfect complex of $\Lambda$-modules.

\label{Sec:PerverseTStructure} We recall the perverse $t$-structure on $D(\bungk, \Lambda)$ from \cite[Definition 4.11]{Hamann-Lee}. For $[b] \in B(G)$ we let $\nu_{b}$ be the conjugacy class of fractional cocharacters of $G$ over $\qpbr$ given by the Newton cocharacters of $b \in [b]$. We denote by $d_b=\langle 2\rho, \nu_{b} \rangle$ the pairing of $\nu_{b}$ with $2 \rho$, where $2\rho$ is the sum of all positive roots for some choice of Borel pair $T \subset B \subset G \otimes \qpbr$ for which $\nu_{b}$ factors through $T$ and is $B$-dominant.

\begin{Prop} \label{Prop:PerverseTStruct}
    For any open substack $U \subset \bungk$ there is a $t$-structure on $D(U, \Lambda)$ characterized by the conditions that $A\in{}^pD^{\leq 0}(U,\Lambda)$ if for all $[b] \in \lvert U \rvert$ we have
\begin{align}
    i_{b}^{\ast}A \in  D^{\le d_b}(\bungk^{[b]}, \Lambda),
\end{align}
and $A\in {}^pD^{\geq 0}(U,\Lambda)$ if for all $[b]\in \lvert U \rvert$
\begin{align}
    i_{b}^{!}A \in  D^{\ge d_b}(\bungk^{[b]}, \Lambda).
\end{align}
\end{Prop}

\begin{proof}
  This result is well-known to experts; we supply a proof because we could not
  locate one in the literature.

  We argue as in \cite[Proposition~VI.7.1]{FarguesScholze}. There is a
  presentable stable $\infty$-category $\mathcal{D}(U,\Lambda)$ whose homotopy
  category is $D(U,\Lambda)$, see \cite[Lemma~17.1]{EtCohDiam}. Consider the
  full subcategory ${}^p\mathcal{D}^{\leq 0}(U,\Lambda) \subseteq
  \mathcal{D}(U,\Lambda)$ whose objects are $A \in \mathcal{D}(U,\Lambda)$
  satisfying $i_b^\ast A \in \mathcal{D}^{\leq d_b}(\bungk^{[b]},\Lambda)$ for
  all $[b] \in \lvert U \rvert$. Once we show that ${}^p\mathcal{D}^{\leq
  0}(U,\Lambda) \subseteq \mathcal{D}(U,\Lambda)$ is stable under extensions and
  colimits, and generated by a small set of objects under extensions and
  colimits, it follows from \cite[Proposition~1.4.4.11]{LurieHA} that there
  exists a unique $t$-structure $({}^p\mathcal{D}^{\leq 0}(U,\Lambda),
  {}^p\mathcal{D}^{\geq 0}(U,\Lambda))$ on $\mathcal{D}(U,\Lambda)$. Stability
  under extensions and colimits is clear, and for generation by a small
  collection of objects, we use the objects
  $\{i_{b,!}c\text{-Ind}_K^{G_b(\qp)}\Lambda[n]\}$ as $[b]$ runs over points in
  $\lvert U \rvert$, $K$ runs over open subgroups of $G_b(\qp)$, and $n$ runs
  over integers that are greater than or equal to $-d_b$.  It is clear that such objects
  generate $i_{b,!}\mathcal{D}^{\leq d_b}(\bungk^{[b]}, \Lambda)$. We may now
  use excision together with the fact that if we write $U = \varinjlim_\alpha
  U_\alpha$ for $j_\alpha \colon U_\alpha \hookrightarrow U$ quasi-compact
  opens, then $A = \varinjlim_\alpha j_{\alpha,!} j_\alpha^\ast A$. This proves
  the existence of a $t$-structure.

  We now identify the full subcategory ${}^p\mathcal{D}^{\geq 0}(U,\Lambda)$
  with those $A \in \mathcal{D}(U,\Lambda)$ with the property that $i_b^! A \in
  D^{\geq d_b}(\bungk^{[b]},\Lambda)$, for all $[b] \in \lvert U \rvert$. By
  definition, see \cite[Remark~1.2.1.3]{LurieHA}, we have $A \in
  {}^p\mathcal{D}^{\ge 0}(U,\Lambda)$ if and only if
  $\Hom_{\mathcal{D}(U,\Lambda)}(B, A[-1]) = 0$ for all $B \in
  {}^p\mathcal{D}^{\le 0}(U,\Lambda)$.

  We first check that if $A \in {}^p\mathcal{D}^{\ge 0}(U,\Lambda)$ then $i_b^!
  A \in \mathcal{D}^{\ge d_b}(\bungk^{[b]},\Lambda)$. For this, we simply note
  that for all $[b] \in \lvert U \rvert$ and $B \in
  \mathcal{D}^{\le d_b}(\bungk^{[b]},\Lambda)$ we have $i_{b,!} B \in
  {}^p\mathcal{D}^{\le 0}(U,\Lambda)$ and hence
  \[
    \Hom(B, i_b^! A[-1]) = \Hom(i_{b,!} B, A[-1]) = 0.
  \]
  This implies $i_b^! A \in \mathcal{D}^{\ge d_b}(\bungk^{[b]},\Lambda)$ as
  desired.

  In the reverse direction, assume that $i_b^! A \in \mathcal{D}^{\ge
  d_b}(\bungk^{[b]},\Lambda)$ for all $[b] \in \lvert U \rvert$. It suffices to
  verify that $R\Hom(B, A) \in \mathcal{D}^{\ge 0}(\Lambda)$ for all $B \in
  {}^p\mathcal{D}^{\le 0}(U,\Lambda)$. For each $[b] \in \lvert U \rvert$ we
  have
  \[
    R\Hom(i_{b,!} i_b^\ast B, A) = R\Hom(i_b^\ast B, i_b^! A) \in
    \mathcal{D}^{\ge 0}(\Lambda)
  \]
  since $i_b^\ast B \in \mathcal{D}^{\le d_b}(\bungk^{[b]},\Lambda)$ and $i_b^! A
  \in \mathcal{D}^{\ge d_b}(\bungk^{[b]},\Lambda)$ by assumption.
  Next, on any quasi-compact open $j_\alpha \colon U_\alpha
  \hookrightarrow U$ we may write $j_{\alpha,!} j_\alpha^\ast B$ as an extension
  of $i_{b,!} i_b^\ast B$ and hence $R\Hom(j_{\alpha,!} j_\alpha^\ast B, A) \in
  \mathcal{D}^{\ge 0}(\Lambda)$. Finally once we write $U = \varinjlim_\alpha
  U_\alpha$ we have
  \[
    R\Hom(B, A) = R\Hom\Big( \varinjlim_\alpha j_{\alpha,!} j_\alpha^\ast B, A
    \Bigr) = \varprojlim_\alpha R\Hom(j_{\alpha,!} j_\alpha^\ast B, A) \in
    \mathcal{D}^{\ge 0}(\Lambda)
  \]
  as the full subcategory $\mathcal{D}^{\ge 0}(\Lambda) \subseteq
  \mathcal{D}(\Lambda)$ is stable under limits, see
  \cite[Corollary~1.2.1.6]{LurieHA}.
\end{proof}

\subsection{Weil group actions} \label{Sec:WeilGroupActions}

When proving Theorem~\ref{Thm:IntroWeilCohShimVar} in Section~\ref{sub:Hecke}, we will compute the cohomology of a Shimura variety with local reflex field $E$ using its structure morphism to $\Div = [\spd E / \phi^\mathbb{Z}]$. This endows the Shimura variety with an action of the Weil group $W_E$ which is a priori different from the natural one given by restricting the action of $\operatorname{Gal}(\bar{E}/E)$ on the cohomology. In this section we show that these two Weil group actions agree. The results of this section are well-known to experts, and we encourage the reader who is comfortable with this comparison to skip this section.

\subsubsection{Canonical Frobenius descent} Recall that every v-stack $X$ is equipped with an automorphism $\phi_X \colon X \to X$, called the absolute Frobenius of $X$ (see Section \ref{Sec:AbsoluteFrobenii}). Recall as well that any morphism $f \colon X \to Y$ of v-stacks is $\phi$-equivariant, in the sense that the diagram
\begin{equation}\label{Eq:phiEquivariance}
        \begin{tikzcd}
            X \arrow{r}{\phi_{X}} \arrow{d}{f} & X \arrow{d}{f} \\
            Y \arrow{r}{\phi_{Y}} & Y
        \end{tikzcd}
    \end{equation}
is 2-commutative. We begin by sharpening our understanding of this $\phi$-equivariance.

\begin{Lem} \label{Lem:RelativeFrobenius}
For any morphism $f \colon X \to Y$ of v-stacks, the $2$-commutative diagram \eqref{Eq:phiEquivariance} is $2$-Cartesian.
\end{Lem}
\begin{proof}
Let $Z$ be another v-stack, and suppose we are given 1-morphisms $\alpha \colon Z \to X$ and $\beta \colon Z \to Y$ and a 2-isomorphism $f \circ \alpha \simeq \phi_Y \circ \beta$. Then $\alpha \circ \phi_Z^{-1} \colon Z \to X$ is the unique (up to isomorphism) 1-morphism making the respective triangles 2-commute.
\end{proof}

\begin{Lem} \label{Lem:CanonicalPhiDescent}
  For a small v-stack $X$ and an object $\mathcal{F} \in \mathcal{D}(X, \Lambda)$, there is a canonical descent datum $\phi_{\mathrm{can},\mathcal{F}} \colon \phi_X^\ast \mathcal{F} \xrightarrow{\sim} \mathcal{F}$.
\end{Lem}
\begin{proof}
  This follows from Lemma~\ref{Lem:RelativeFrobenius}; see the arguments in \cite[Lemma 03SR]{stacks-project}. More concretely, for every object $T \in X$ of $X_v$, we may describe $(\phi^\ast \mathcal{F})(T) = \mathcal{F}(T \xrightarrow{\phi_T} T \to X)$, and this is identified with $\mathcal{F}(T \to X)$ via $\phi_T^\ast$.
\end{proof}

\begin{Lem} \label{Lem:CanonicalPhiDescentPullback}
  Let $f \colon X \to Y$ be a morphism of small v-stacks, inducing a morphism of v-stacks $f_\phi \colon [X/\phi_X^\mathbb{Z}] \to [Y/\phi_Y^\mathbb{Z}]$. Let $\mathcal{F} \in \mathcal{D}(Y, \Lambda)$ and let $\mathcal{F}_\phi \in \mathcal{D}([Y/\phi_Y^\mathbb{Z}], \Lambda)$ be the descent of $\mathcal{F}$ for the canonical descent datum. Then $f_\phi^\ast \mathcal{F}_\phi \in \mathcal{D}([X/\phi_X^\mathbb{Z}], \Lambda)$ is canonically identified with the descent of $f^\ast \mathcal{F}$ for the canonical descent datum on $X$.
\end{Lem}

\begin{proof}
  After unraveling the constructions, this is equivalent to the existence of a canonical coherence datum
  \[ \begin{tikzcd}[column sep=large]
    \phi_X^\ast f^\ast \mathcal{F} \arrow{r}{\phi_{\mathrm{can},f^\ast\mathcal{F}}} \arrow{d}{\sim} & f^\ast \mathcal{F} \arrow[equals]{d} \\ f^\ast \phi_Y^\ast \mathcal{F} \arrow{r}{f^\ast \phi_{\mathrm{can},\mathcal{F}}} & f^\ast \mathcal{F}.
  \end{tikzcd} \]
  This again follows formally from Lemma~\ref{Lem:RelativeFrobenius}.
\end{proof}

\begin{Lem} \label{Lem:DescentAndPushforward}
  Let $\pi \colon X \to Y$ be a qcqs morphism of small v-stacks, and let
  $\pi_\phi \colon [X / \phi_X^\mathbb{Z}] \to [Y / \phi_Y^\mathbb{Z}]$ be the
  induced map. Let $A \in \mathcal{D}^+(X, \Lambda)$ and let $A_\phi \in
  \mathcal{D}^+([X/\phi_X^\mathbb{Z}], \Lambda)$ be the canonical descent of
  $A$. Then $R\pi_{\phi,\ast} A_\phi \in \mathcal{D}([Y / \phi_Y^\mathbb{Z}],
  \Lambda)$ is the canonical descent of $R\pi_\ast A \in \mathcal{D}(Y,
  \Lambda)$.
\end{Lem}

\begin{proof}
  Because both $\pi$ and $\pi_\phi$ are qcqs, using
  \cite[Proposition~17.6]{EtCohDiam} we may identify $R\pi_\ast$ and
  $R\pi_{\phi,\ast}$ with the pushforward on the v-sites $R\pi_{\mathrm{v}\ast}$ and
  $R\pi_{\phi,\mathrm{v}\ast}$. Using qcqs base change, see \cite[Theorem 1.9.(i)]{EtCohDiam}, along
  \[ \begin{tikzcd}
    X \arrow{r} \arrow{d}{\pi} & \lbrack X / \phi_X^\mathbb{Z} \rbrack
    \arrow{d}{\pi_\phi} \\ Y \arrow{r} & \lbrack Y / \phi_Y^\mathbb{Z}
    \rbrack
  \end{tikzcd} \]
  we see that the pullback of $R\pi_{\phi,\mathrm{v}\ast} A_\phi$ to $Y$ agrees with $R\pi_{\mathrm{v}\ast} A$. It now suffices to check that the descent data agree, i.e., the map
  \[
    \phi_Y^\ast R\pi_{\mathrm{v}\ast} A \xrightarrow{\sim} R\pi_{\mathrm{v}\ast} \phi_X^\ast A
    \xrightarrow{\phi_{\mathrm{can},A}} R\pi_{\mathrm{v}\ast} A
  \]
  agrees with $\phi_{\mathrm{can},R\pi_{\mathrm{v}\ast} A}$, see Lemma~\ref{Lem:CanonicalPhiDescent}. This formally follows from the fact that
  $\phi^\ast$ can be computed by the Frobenius action on the site.
\end{proof}

\input{weil-group-extra}

\subsection{The dualizing sheaf of the Igusa stack} \label{Sub:DualizingSheaf}
Let $\gx$ be a Shimura datum of Hodge type with reflex field $\mathsf{E}$. Fix a prime number $p$ and a prime $v \mid p$ of $\mathsf{E}$, let $E$ be the completion of $\mathsf{E}$ at $v$, and let $K^p \subset \gafp$ be a neat compact open subgroup. We write $\mu$ for the Hodge cocharacter of the Shimura datum $\gx$, which we consider as a $G(\qpbar)$-conjugacy class of cocharacters of $G$ using the place $v$ of $\mathsf{E}$. Fix $k$ an algebraic closure of $k_E$. We let $\igsk=\igsk\gx$ be the Igusa stack constructed in Section \ref{Sec:HodgeType}, see Theorem \ref{Thm:HodgeMain}, base changed to $k$. In this section, we compute the dualizing complex in $D(\igsk,\Lambda)$. The main result is the following theorem.

\begin{Thm} \label{Thm:IgusaDualizingComplex}
    The stack $\igsk$ is an $\ell$-cohomologically smooth Artin v-stack of $\ell$-dimension 0. Moreover, the dualizing sheaf of $\igsk \to \spd k$ is isomorphic to $\Lambda[0]$.
\end{Thm}

By proving this theorem we will complete the proof of Theorem \ref{Thm:IntroIgusaGeneric}. For the theorem we need as input some understanding of the dualizing complex on $\bungk$, the Shimura variety and the flag variety. This is the content of the theorem and two lemmas that follow.

\begin{Thm} \label{Thm:DualizingBunG}
The v-stack $\bungk$ is an $\ell$-cohomologically smooth Artin v-stack of $\ell$-dimension $0$. Moreover, the dualizing sheaf of $\bungk \to \spd k$ is isomorphic to $\Lambda[0]$.
\end{Thm}
\begin{proof}
The first part is \cite[Theorem I.4.1.(vii)]{FarguesScholze}, and the second part is \cite[Proposition 3.18]{HamannImai}.
\end{proof}
For our fixed compact open subgroup $K^p \subset \gafp$ we consider
\begin{align}
\mathbf{Sh}_{K^p}^{\circ, \lozenge}=\mathbf{Sh}_{K^p}\gx^{\circ, \lozenge}&=\varprojlim_{U_p \subset G(\qp)} \mathbf{Sh}_{U_pK^p}\gx^{\circ, \lozenge}.
\end{align}
Let $d$ be the dimension of the Shimura variety $\mathbf{Sh}_{U_pK^p}\gx$.
\begin{Lem} \label{Lem: DualizingFlag}
     Consider the structure map
    \[q: [\operatorname{Gr}_{G, \mu^{-1}}/\underline{G(\qp)}]\rightarrow
    \spd E.\]
    Then there is an isomorphism $Rq^!\Lambda \simeq \Lambda[2d]( d )$.
\end{Lem}

\begin{proof}
    It is clear that the underlying complex of $Rq^!\Lambda$ is isomorphic to $\Lambda[2d](d)$ since $\operatorname{Gr}_{G, \mu^{-1}}$ is smooth of pure dimension $d$, see \cite[Proposition 7.5.3]{Huber96}. To show that the $G(\qp)$-equivariant structure is trivial, consider the Cartesian diagram
    \[\begin{tikzcd}
    {\operatorname{Gr}_{G, \mu^{-1}}} \ar[r,"\tilde{f}"] \ar[d,"\tilde{g}"] & \spd E \ar[d,"g"] \\
    {[\operatorname{Gr}_{G, \mu^{-1}}/\underline{G(\qp)}]}\ar[r,"f"] & {[\spd E/\underline{G(\qp)}]}.
    \end{tikzcd}\]

    Note that since $\tilde{f}$ is $\ell$-cohomologically smooth and $f$ has finite $\mathrm{dim.trg}$ (as all fibers have uniformly bounded $\mathrm{dim.trg}$, cf.\ the paragraph below \cite[Remark~21.8]{EtCohDiam}), the map $f$ is also $\ell$-cohomologically smooth by \cite[Proposition~23.15]{EtCohDiam}. Then we have
    \[\tilde{g}^\ast Rf^!\Lambda \simeq R\tilde{f}^!g^\ast \Lambda,\]
    as complexes with $G(\qp)$-equivariant structure by \cite[Proposition~23.16.(3)]{EtCohDiam}. But the right hand side complex has trivial $G(\qp)$-action. This implies that $Rf^!\Lambda \simeq \Lambda[2d](d)$ as complexes on $[\operatorname{Gr}_{G, \mu^{-1}}/\underline{G(\qp)}]$. Now since $G(\qp)$ is unimodular, the shriek pullback of $\Lambda$ along $[\spd E/\underline{G(\qp)}]\to \spd E$ is trivial, see \cite[Example~4.2.4, 4.2.5]{HansenKalethaWeinstein}. Hence $Rq^!\Lambda\simeq Rf^!\Lambda \simeq \Lambda[2d](d)$.
\end{proof}

\begin{Lem} \label{Lem: DualizingShi}
    Consider the structure map
    \[q': [\mathbf{Sh}_{K^p}^{\circ, \lozenge}/\underline{G(\qp)}]\rightarrow \spd E. \]
    Then there is an isomorphism $R{q'}^!\Lambda \simeq \Lambda[2d](d)$.
\end{Lem}
\begin{proof}
Fix a compact open subgroup $K_p \subset G(\qp)$ and let $K=K_pK^p$. We note that $\mathbf{Sh}_{K}^{\circ}$ is an open subspace of the smooth rigid space $\mathbf{Sh}_{K}^{\mathrm{an}}$ which is smooth of dimension $d$. Thus it follows from \cite[Proposition 7.5.3]{Huber96} (cf.\ \cite[Theorem 6.1.8]{Zavyalov}) that its dualizing complex (relative to $\spd E$) is isomorphic to $\Lambda[2d](d)$. Here we fix the trivialization
    \[\alpha_K: Rq_K^!\Lambda \xrightarrow{\sim} \Lambda[2d](d)\]
    to be the adjunction map to the trace map as defined in \cite[(7.4.1)]{Huber96}, where $q_K$ denotes the structure map to $\spa E$ of $\mathbf{Sh}_{K}^{\circ}$. Note that by \cite[7.3.4]{Huber96} the trace map is functorial, which will be crucial to our argument below.

    To lighten the notation, we write below $S_K$ for $\mathbf{Sh}_{K}^{\circ,\lozenge}$ and $S_{K^p}$ for $\mathbf{Sh}_{K^p}^{\circ,\lozenge}$. Consider the presentation
    \[S_K\times_{[S_{K^p}/\underline{G(\qp)}]} S_K \substack{\xrightarrow{p_1} \\[0em] \xrightarrow[p_2]{}} S_K\xrightarrow{\pi} [S_{K^p}/\underline{G(\qp)}].\]
    The last map is \'etale since $G(\qp)/K_p$ is a discrete topological space. Thus the dualizing complex $Rq_K^!\Lambda$ on $S_K$ is equipped with a natural descent datum
    \begin{align*}
    p_1^\ast Rq_K^!\Lambda &\simeq Rp_1^! Rq_K^!\Lambda\simeq R(q_K\circ p_1)^!\Lambda\\
    & = R(q_K\circ p_2)^!\Lambda\simeq Rp_2^! Rq_K^!\Lambda\simeq p_2^\ast Rq_K^!\Lambda,
    \end{align*}
    where the first and last isomorphisms come from the natural isomorphism $p_i^\ast \xrightarrow{\sim} Rp_i^!$, for $i=1,2$.
    Via the trivialization
    $\alpha_K$, it identifies with the natural descent datum\footnote{Here we are using the naturality of the trivializations of dualizing complexes on $S_K$ and $S_K\times_{[S_{K^p}/\underline{G(\qp)}]} S_K$, coming from that of the trace maps \cite[7.3.4~(Var 3)]{Huber96}.}
    \[p_1^\ast \Lambda[2d](d) \simeq \Lambda[2d](d) \simeq p_2^\ast \Lambda[2d](d),\]
    on $\pi^\ast \Lambda[2d](d)$ on $[S_{K^p}/\underline{G(\qp)}]$. Since by \cite[Definition IV.1.13]{FarguesScholze}, $R{q'}^!\Lambda$ is defined as the complex on $[S_{K^p}/\underline{G(\qp)}]$ given by the first descent datum, it follows that it is isomorphic to $\Lambda[2d](d)$.
\end{proof}
\begin{Rem}
    The argument here shows more generally that an \'etale quotient of a smooth rigid space of dimension $d$ has dualizing sheaf $\Lambda[2d](d)$.
\end{Rem}

\begin{proof}[Proof of Theorem \ref{Thm:IgusaDualizingComplex}]
Note that using Corollary~\ref{Cor:VSurjectiveBL} and \cite[Proposition~13.4.(iv)]{EtCohDiam}, we can argue as in the proof of \cite[Corollary 8.19]{ZhangThesis} to deduce that $\operatorname{Igs}_{K^p}$ is an Artin v-stack. To prove that $\igsk \to \spd k$ is $\ell$-cohomologically smooth of $\ell$-dimension zero, it suffices to do so after base changing $\igsk$ via the cohomologically smooth cover $\spd \ebreve \to \spd k$.\footnote{Note that $\spd \mathcal{O}_E \to \spd k_E$ is cohomologically smooth by \cite[Corollary IV.2.34]{FarguesScholze} which implies after base changing via $\spd k \to \spd k_E$ that $\spd \mathcal{O}_{\ebreve} \to \spd k$ is $\ell$-cohomologically smooth. Since $\spd \ebreve \to \spd \mathcal{O}_{\ebreve}$ is open, it follows that $\spd \ebreve \to \spd k$ is $\ell$-cohomologically smooth.} Similarly, to conclude that the dualizing complex of $\igsk \to \spd k$ is constant it suffices to do so after base changing via $\spd \ebreve \to \spd k$. Indeed, since the formation of the dualizing sheaf commutes with base change, see \cite[Proposition 4.14]{GulottaHansenWeinstein}, if the result holds over $\spd \ebreve$, then it holds over $\spd C$ for any complete algebraically closed field containing $\ebreve$, and the base change map
\begin{align}
    D(\igsk, \Lambda) \to D(\operatorname{Igs}_{K^p,C}, \Lambda)
\end{align}
is fully faithful by \cite[Theorem 1.13.(ii)]{EtCohDiam}. \smallskip

For the rest of this proof only, we will base change $\bun_{G}$ and $\igsk$ via $\spd \ebreve \to \spd k$ and we will base change $\mathbf{Sh}_{K}^{\circ,\lozenge}$ and $\operatorname{Gr}_{G, \mu^{-1}}$ via $\spd \ebreve \to \spd E$, \emph{without changing our notation}. The maps (over $\spd \ebreve$) induced by the Beauville--Laszlo map and the reduction map then still sit in a Cartesian diagram
    \begin{equation} \label{Eq:CartesianFiniteLevelOverEBreve}
    \begin{tikzcd}
      \mathbf{Sh}_{K}^{\circ,\lozenge} \ar[r,"\pi_{\mathrm{HT},K_p}"]\ar[d,"q_{\mathrm{Igs}}"] &
     {[\operatorname{Gr}_{G, \mu^{-1}}/\underline{K}_p]} \ar[d,"\mathrm{BL}"]\\
    \operatorname{Igs}_{K^p} \ar[r,"\overline{\pi}_\mathrm{HT}"] & \bungmu.
    \end{tikzcd}
    \end{equation}
 Moreover the Beauville--Laszlo map $\mathrm{BL}$ is a separated $\ell$-cohomologically smooth map of $\ell$-dimension $d$, see \cite[the proof of Theorem IV.1.19]{FarguesScholze}.\footnote{Since $[\operatorname{Gr}_{G, \mu^{-1}}/\underline{K}_p]\to [\operatorname{Gr}_{G, \mu^{-1}}/\underline{G(\qp)}]$ is cohomologically smooth, the argument in the proof of \cite[Theorem IV.1.19]{FarguesScholze} still holds when replacing $G(\qp)$ by $K_p$. This is also spelled out in \cite[Proposition 2.10]{HansenMiddle}.} Hence by \cite[Proposition 23.15]{EtCohDiam}, the base change $q_{\mathrm{Igs}}$ is also separated and $\ell$-cohomologically smooth of $\ell$-dimension $d$, and it is moreover a surjection by Corollary \ref{Cor:VSurjectiveBL}. Since $\mathbf{Sh}_{K}^{\circ,\lozenge}$ is itself an $\ell$-cohomologically smooth spatial diamond, it follows by definition that $\operatorname{Igs}_{K^p} \to \spd \ebreve$ is $\ell$-cohomologically smooth, see \cite[Definition IV.1.11]{FarguesScholze}. It is of $\ell$-dimension 0 in the sense of \cite[Definition IV.1.17]{FarguesScholze}, since both $\mathbf{Sh}_{K}^{\circ,\lozenge}$ and the map $q_{\mathrm{Igs}}$ are of $\ell$-dimension $d$. \smallskip

To conclude that the dualizing complex is constant, consider the Cartesian diagram
\begin{equation}\label{Eq:CartesianDiagrammoduloGQP}\begin{tikzcd}
    {[\mathbf{Sh}_{K^p}^{\circ, \lozenge}/\underline{G(\qp)}]}\ar[r,"\tilde{f}"]\ar[d,"\tilde{g}"] & {[\operatorname{Gr}_{G, \mu^{-1}}/\underline{G(\qp)}]} \ar[d,"g"]\\
    \operatorname{Igs}_{K^p} \ar[r,"f"] & \bungmu,
\end{tikzcd}
\end{equation}
    where we have relabeled the maps for simplicity. Now since the dualizing complex on $\bun_{G}$ is trivial by Theorem \ref{Thm:DualizingBunG}, the dualizing complex on $\operatorname{Igs}_{K^p}$ can be computed as $Rf^!\Lambda$.\footnote{The existence of the exceptional pullback is in \cite[Theorem 1.4]{GulottaHansenWeinstein}. Note that by the Cartesian diagram \eqref{Eq:CartesianFiniteLevelOverEBreve}, the map $f$ is ``fine'' in the sense of loc.\ cit.} By cohomological smoothness of the map $\tilde{g}$, we have
    \begin{align*}
        \tilde{g}^\ast Rf^!\Lambda & \simeq R\tilde{g}^! Rf^!\Lambda \otimes^{\mathbb{L}} (R\tilde{g}^!\Lambda)^{-1} \simeq \Lambda[2d](d) \otimes^{\mathbb{L}} (R\tilde{g}^!f^\ast\Lambda)^{-1}\\
        & \simeq \Lambda[2d](d) \otimes^{\mathbb{L}} (\tilde{f}^\ast Rg^!\Lambda)^{-1} \simeq \Lambda[2d](d) \otimes^{\mathbb{L}} (\tilde{f}^\ast \Lambda[2d](d))^{-1} \simeq \Lambda[0],
    \end{align*}
    where for the first isomorphism, we used the natural isomorphism
    \[R\tilde{g}^! \simeq \tilde{g}^\ast \otimes^{\mathbb{L}}R\tilde{g}^!\Lambda,\]
    see \cite[Proposition~23.12.(i)]{EtCohDiam}. The second isomorphism follows from Lemma \ref{Lem: DualizingShi}
    and the observation that the dualizing sheaf of $[\mathbf{Sh}_{K^p}^{\circ, \lozenge}/\underline{G(\qp)}]$ can be computed as $R(f \circ \tilde{g})^{!} \Lambda$ because the dualizing sheaf of $\bun_{G}$ is trivial. The third isomorphism uses the base change result from \cite[Proposition 23.12.(iii)]{EtCohDiam} and the fact that $g$ is separated, representable in locally spatial diamonds and $\ell$-cohomologically smooth. The fourth isomorphism follows from Lemma \ref{Lem: DualizingFlag}. \smallskip

    Now we get by adjunction a map
    \[Rf^!\Lambda \rightarrow R\tilde{g}_\ast \Lambda.\] Applying the truncation $\tau_{\leq 0}$ to both sides we obtain a map
    \begin{equation} \label{Eq:TruncationEquation}
    Rf^!\Lambda = \tau_{\leq 0} Rf^!\Lambda \rightarrow \tau_{\leq 0} R\tilde{g}_\ast \Lambda \xrightarrow{\sim} \Lambda[0].
    \end{equation}

    Here the last isomorphism needs some explanation: the fiber of
    \[[\operatorname{Gr}_{G, \mu^{-1}}/\underline{G(\qp)}]\rightarrow \bun_{G}\]
    over a point $b\in B(G)\simeq \lvert \bun_{G} \rvert$ is the moduli space of modifications
    \[\mathcal{E}_b\dashrightarrow \mathcal{E}\] that are bounded by $\mu$, for which the modified bundle $\mathcal{E}$ is geometrically-pointwise trivial (although not trivialized). Hence it identifies with the $b$-admissible locus on $\operatorname{Gr}_{G, \mu^{-1}}$, which is geometrically connected by \cite[Theorem 1.1]{GleasonLourenco}.

    One may check that the map in \eqref{Eq:TruncationEquation} is an isomorphism by pulling back along $\tilde{g}$, where it becomes
    \[\tilde{g}^\ast Rf^!\Lambda \rightarrow \tilde{g}^\ast \tau_{\leq 0}R\tilde{g}_\ast \Lambda= \tau_{\leq 0}\tilde{g}^\ast R\tilde{g}_\ast \Lambda \xrightarrow{\sim} \tilde{g}^\ast \Lambda[0].\]
    But this composition is an isomorphism, and hence so is the first arrow.
\end{proof}

\subsection{The cohomology of Shimura varieties} \label{sub:Hecke} Let the notation be as in Section \ref{Sub:DualizingSheaf} above. In particular, $k$ is a fixed algebraic closure of $k_{E}$ and $\Lambda$ is a Noetherian $\zl$-algebra in which $\ell$ is nilpotent. We assume additionally that $\Lambda$ contains a square root of $p$ and fix a choice of $\sqrt{p} \in \Lambda$. The goal of this section is to prove Theorem~\ref{Thm:IntroWeilCohShimVar}, which gives a formula for the cohomology of Hodge type Shimura varieties in terms of the action of a Hecke operator on the relative cohomology of the Igusa stack over $\bungk$. Later in Section \ref{Sec:EichlerShimura}, we will reinterpret this result using the spectral action. This formula allows us to use techniques from \cite{FarguesScholze} to study the cohomology of Shimura varieties. It thus plays a central role in Sections \ref{Sec:Cohomology} through \ref{Sec:TorsionVanishing}.

Recall from \cite[Section IX.1]{FarguesScholze} that we have the derived $\infty$-category of smooth representations of $G(\qp)$ on $\Lambda$-modules $\mathcal{D}(G(\qp),\Lambda)$, which is equivalent to $\mathcal{D}_{\text{\'et}}(\bungk^{[1]}, \Lambda)$ by \cite[Theorem V.1.1]{FarguesScholze}. We use this equivalence to give $\mathcal{D}(G(\qp),\Lambda)$ the structure of a condensed $\infty$-category. Following Section~\ref{Sec:WeilGroupActions}, we can therefore consider the category of $W_E$-equivariant objects $\mathcal{D}(G(\qp),\Lambda)^{B W_E}$.

We let $C$ be the completion of an algebraic closure of $E$ with ring of integers $\mathcal{O}_C$. We also fix an identification of $k$ with the residue field of $\mathcal{O}_C$. Let $\ebreve$ denote the completion of the maximal unramified extension of $E$ inside $C$, whose residue field is also identified with $k$, inducing a map $\spd \ebreve \to \spd k$.

\subsubsection{Hecke operators}
\label{subsub:HeckeOperators}
Let $\mathrm{Div}^1$ be the v-sheaf $[\spd \qp/\phi^\mathbb{Z}]$, where $\phi$ is the absolute Frobenius as before. It follows from the proof of \cite[Proposition VI.1.2]{FarguesScholze} that for $S \in \perf$ there is a functorial injection from $\mathrm{Div}^1(S)$ into the set of closed Cartier divisors of $X_S$. Consider the local Hecke stack
\[\mathcal{H}ck_{G,\mathrm{Div}^1_{X_\qp}}\rightarrow \mathrm{Div}^1\]
as defined in \cite[Definition VI.1.6]{FarguesScholze}. It is the moduli space sending $S \to \mathrm{Div}^1$ to a pair of $G$-bundles $\mathcal{E}_1, \mathcal{E}_2$ over the completion of $X_S$ in the Cartier divisor $D_S$, together with an isomorphism $\mathcal{E}_1 \to \mathcal{E}_2$ away from $D_S$. We will omit the subscript $X_\qp$ when the context is clear. We pull back along $\Div:=[\spd E / \phi^{\mathbb{Z}}] \to \mathrm{Div}^1$ and consider the Schubert cell
\[j_\mu: \mathcal{H}ck_{G,\Div, \mu}=\mathcal{H}ck_{G,\Div, \leq \mu}\hookrightarrow \mathcal{H}ck_{G,\Div},\]
which sends an affinoid perfectoid $S\to \Div$ to the groupoid of pairs of $G$-bundles $\mathcal{E}_1$, $\mathcal{E}_2$ over the completion of $(X_S)_E$ at the given Cartier divisor, together with a meromorphic isomorphism $\mathcal{E}_1\dashrightarrow \mathcal{E}_2$, with a pole at the given divisor that is bounded by $\mu$.\footnote{Note that for us $G$ is not necessarily split and $\mu$ is a $G_{\qpbar}$-conjugacy class of homomorphisms $\mathbb{G}_m\to G_{\qpbar}$, while in \cite[Definition VI.2.2]{FarguesScholze} $G$ is assumed to be split and $\mu$ is a dominant cocharacter with respect to some Borel pair. To compare the definitions, we note that on the one hand, the condition of being bounded by $\mu$ can be checked geometric pointwise on test objects. On the other hand, for any choice of a Borel pair $T \subset B\subset G_{\qpbar}$, there is a bijection between the set of conjugacy classes of homomorphisms $\mathbb{G}_m\to G_{\qpbar}$ and the set of dominant cocharacters $X_\ast(T)^+$.}

Applying \cite[Definition VI.7.8]{FarguesScholze} (the version for $\mathrm{Div}^1$ instead of $\mathrm{Div}^1_\mathcal{Y}$) to $S=\Div$, we get the Satake category
$\mathrm{Sat}(\mathcal{H}ck_{G,\Div},\Lambda)$. We follow the notation in \cite[Proposition VI.7.9]{FarguesScholze} and consider the flat perverse sheaf
\[\mathcal{S}_\mu\coloneqq \mathrm{Im}(\phantom{}^p\mathcal{H}^0
(j_{\mu,!}\Lambda[d](\tfrac{d}{2})) \to \phantom{}^p\mathcal{H}^0
(Rj_{\mu,\ast}\Lambda[d](\tfrac{d}{2}))).
\]
Here compared to the definition in loc.\ cit., we have included a Tate twist by $\frac{d}{2}$, which makes use of the fact that $\mathcal{H}ck_{G,\Div, \mu}$ lives over $\Div$ (the Tate twist is a rank one local system on $\Div$) and that we have chosen $\sqrt{p}\in \Lambda$. Adding this twist has the benefit of making $\mathcal{S}_\mu$ Verdier self-dual.

\begin{Rem} \label{Rem:IntersectionCohomologyTrivial}
In our case, the cocharacter $\mu$ is minuscule and hence $j_\mu$ is a closed
embedding. It follows that $j_{\mu,!} = j_{\mu,\ast} = Rj_{\mu,\ast}$ by
\cite[Theorem~22.5]{EtCohDiam}. On the other hand, the sheaf $j_{\mu, \ast}
\Lambda[d](\tfrac{d}{2})$ is perverse, since it is supported on
$\mathcal{H}ck_{G,\Div,\leq \mu}$ and $j_{\mu}^\ast j_{\mu,\ast}
\Lambda[d](\tfrac{d}{2}) \simeq \Lambda[d](\tfrac{d}{2})$ is concentrated in
degree $-d$. Therefore
\[
  \mathcal{S}_\mu \simeq j_{\mu,!}\Lambda[d](\tfrac{d}{2}) \simeq
  j_{\mu,\ast}\Lambda[d](\tfrac{d}{2}).
\]
\end{Rem}

\label{subsub: HeckeCorr} It follows from the proof of \cite[Proposition VI.1.2]{FarguesScholze} that for $S \in \perf$ there is a functorial injection from $\Div(S)$ into the set of closed Cartier divisors of $X_S \times_{\qp} E$.\footnote{We are \emph{not} using the Fargues--Fontaine curve for the local field $E$ here, but rather the base change to $E$ of the Fargues--Fontaine curve for $\qp$.} We consider the global Hecke stack
\[
  \mathrm{Hck}_{G, \le\mu} \to \Div
\]
that parametrizes two $G$-bundles $\mathcal{E}_0, \mathcal{E}_1$ on $X_S$
together with a meromorphic map $\mathcal{E}_0 \vert_{(X_{S})_{E}} \dashrightarrow
\mathcal{E}_1 \vert_{(X_{S})_{E}}$ at the given divisor, bounded by $\mu$. This fits in a Hecke correspondence (where $h_1$ remembers $\mathcal{E}_0$, while $h_2$ remembers $\mathcal{E}_1$ and the Cartier divisor)
\begin{equation} \label{Eq:HeckeDiagramI}\bun_{G} \xleftarrow{h_1} \mathrm{Hck}_{G, \le \mu} \xrightarrow{h_2} \bun_{G} \times \Div.
\end{equation}
Recall from \cite[beginning of Section III.3]{FarguesScholze} that $[\operatorname{Gr}_{G, \mu^{-1}}/\phi^{\mathbb{Z}}] \to \Div$ represents the functor sending $S \to \Div$ to the set of isomorphism classes of pairs consisting of a $G$-bundle $\mathcal{E}_1$ on $X_S$ and a modification $\mathcal{E}^0 \vert_{(X_{S})_{E}} \dashrightarrow \mathcal{E}_1 \vert_{(X_{S})_{E}}$ bounded by $\mu^{-1}$, where $\mathcal{E}^0$ is the trivial bundle. This admits a map $\tilde{i}_1:[\operatorname{Gr}_{G, \mu^{-1}}/\phi^{\mathbb{Z}}] \to \mathrm{Hck}_{G,\leq \mu}$ sending the modification to its inverse (which makes it bounded by $\mu$).
\begin{Lem} \label{Lem:HeckeDiagram}
There is a $2$-Cartesian diagram
    \begin{equation} \label{Eq:HeckeDiagramII}
\begin{tikzcd}
\left[\operatorname{Gr}_{G, \mu^{-1}}/ (\phi^{\mathbb{Z}} \times \underline{G(\qp)})\right] \ar[d,"\tilde{h}_2"]\ar[r, hook, "\tilde{i}_1"] & \mathrm{Hck}_{G,\leq \mu}\ar[d,"h_2"]\\
\bun^{[1]}_{G}\times \Div \ar[r, hook, "i_1"] & \bun_{G}\times \Div
\end{tikzcd}
\end{equation}
under which $h_1 \circ \tilde{i}_1$ is identified with the Beauville--Laszlo map.
\end{Lem}
\begin{proof}
    This is a straightforward consequence of the moduli interpretations of the involved objects.
\end{proof}

There is a commutative diagram
\begin{equation}
    \begin{tikzcd}[column sep=small]
        \mathrm{Hck}_{G,\leq \mu} \arrow[rr, "q"] \arrow[dr] & & \mathcal{H}\mathrm{ck}_{G,\Div,\leq \mu} \arrow[dl]\\
         & \Div, &
    \end{tikzcd}
\end{equation}
where $q$ is the map that sends $(\mathcal{E}_0, \mathcal{E}_1, \mathcal{E}_0
\vert_{(X_{S})_{E}} \dashrightarrow \mathcal{E}_1 \vert_{(X_{S})_{E}})$ to its
restriction to the formal completion around the given divisor. We denote still by $\mathcal{S}_{\mu}$ the pullback of $\mathcal{S}_{\mu}$ along $q$. If we pull back along $\spd k \to \spd \fp$ then we find $\Divk \simeq[\spd \ebreve / \varphi_q^{\mathbb{Z}}]$ where $\varphi_q$ is now the $q$-Frobenius acting on $\spd \ebreve = \spd E \times_{\spd \mathbb{F}_q} \spd k$ via the first factor, or equivalently, this is the partial Frobenius $\phi^r \times 1$ from Lemma \ref{Lem:lbreve}.
We obtain the Hecke correspondence over $\spd k$ (see \cite[Section IX.2]{FarguesScholze})
\begin{equation} \label{Eq:HeckeDiagram}\bungk \xleftarrow{h_{1,k}} \mathrm{Hck}_{G,\leq \mu,k}\xrightarrow{h_{2,k}} \bungk \times_k \Divk.
\end{equation}
Then the Hecke operator $T_\mu$ is defined by the following formula\footnote{In fact, it can already be defined before base change to $k$.}
\begin{align} \label{Eq:HeckeOperatorDef}
    T_{\mu}:D(\bungk, \Lambda) &\to D(\bungk \times_k \Divk, \Lambda) \\
    A &\mapsto Rh_{2,k,\ast}\left(h_{1,k}^{\ast} A \otimes^{\mathbb{L}}_{\Lambda} \mathcal{S}_{\mu} \right)\simeq Rh_{2,k,\ast}(h_{1,k}^{\ast} A[d](\tfrac{d}{2})).
\end{align}
Let $W_E$ be the Weil group of $E$. Recall from \cite[Proposition IV.7.1]{FarguesScholze} that $D(\bungk \times_k \Divk, \Lambda)$ contains $D(\bungk \times [\spd k / \underline{W_{E}}], \Lambda)$ as a full subcategory. According to \cite[Corollary IX.2.3]{FarguesScholze}, the essential image of $T_\mu$ lands in this subcategory.
Working $\infty$-categorically, one can identify
\[\mathcal{D}(\bungk \times [\spd k/\underline{W_E}],\Lambda) \simeq
\mathcal{D}( \bungk,\Lambda)^{BW_E}\]
as the $W_E$-equivariant objects of the condensed $\infty$-category $\mathcal{D}(\bungk,\Lambda)$, where as explained before, $W_E$ is considered as a condensed group in the last expression.

\label{subsub:HeckeFactorization} We consider the following diagram.
\[ \begin{tikzcd}
  & \lbrack \mathrm{Gr}_{G,\mu^{-1}} / (\phi^\mathbb{Z} \times
  \underline{G(\qp)}) \rbrack \arrow[hook]{d}{\tilde{i}_1}
  \arrow{r}{\tilde{h}_2} \arrow[dashed]{ldd}[']{g} & \bun_G^{[1]} \times \Div
  \arrow[hook]{d}{i_1} \\ & \mathrm{Hck}_{G,\le\mu} \arrow{d}{h_1}
  \arrow{r}{h_2} & \bun_G \times \Div \\ \bungmu \arrow[hook]{r}{j_\mu} & \bun_G
\end{tikzcd} \]
Since the composition $h_1 \circ \tilde{i}_1$ lands in $\bungmu$, we write $h_1
\circ \tilde{i}_1 = j_\mu \circ g$. Using smooth base change
\cite[Proposition~23.16.(i)]{EtCohDiam} along the open embedding $i_1$, we
compute for every $A \in \mathcal{D}(\bungk, \Lambda)$ that
\begin{align}
  i_{1,k}^\ast T_\mu A &\simeq i_{1,k}^\ast Rh_{2,k,\ast} (h_{1,k}^\ast
  A[d](\tfrac{d}{2})) \\ &\simeq R\tilde{h}_{2,k,\ast} \tilde{i}_{1,k}^\ast
  (h_{1,k}^\ast A[d](\tfrac{d}{2})) \simeq R\tilde{h}_{2,k,\ast} (g_k^\ast
  j_{\mu,k}^\ast A)[d](\tfrac{d}{2}).
\end{align}
This motivates the following definition.

\begin{Def} \label{Def:HeckeFactorization}
  We define the functor
  \begin{align}
    T_\mu^{[1]} \colon \mathcal{D}(\bungmuk, \Lambda) &\to
    \mathcal{D}(\bun_{G,k}^{[1]} \times_k \Divk, \Lambda) \simeq
    \mathcal{D}(G(\qp), \Lambda)^{BW_E} \\ A &\mapsto R\tilde{h}_{2,k,\ast}
    (g_k^\ast A[d](\tfrac{d}{2}))
  \end{align}
  so that $T_\mu^{[1]} j_{\mu,k}^\ast = i_{1,k}^\ast T_\mu$.
\end{Def}

\subsubsection{The Igusa sheaf} Recall that we have a morphism $\overline{\pi}_\mathrm{HT} \colon \igs \to \bungmu$. We define the sheaf
\[
  \mathcal{F} = R\overline{\pi}_{\mathrm{HT},k,\ast}\Lambda \in
  \mathcal{D}(\bungmuk, \Lambda),
\]
which carries a prime-to-$p$ Hecke action. The following gives a formula for the
cohomology of the Shimura variety in terms of $\mathcal{F}$. Note that we may
consider $R\Gamma(\mathbf{Sh}_{K^p,C}, \Lambda)$ as an element of
$\mathcal{D}(G(\qp),\Lambda)^{BW_E}$ using the natural action of $G(\qp)$ and of
$\operatorname{Gal}_E$, see Section \ref{subsub:EquivariantObjectsLocallyProfinite}.

\begin{Thm}
\label{Thm: WeilCohoShiVar}
There is an isomorphism
  \[ R\Gamma(\mathbf{Sh}_{K^p,C}, \Lambda) \simeq T_\mu^{[1]}
  (\mathcal{F}[-d])(-\tfrac{d}{2}) \]
in $\mathcal{D}(G(\qp),\Lambda)^{BW_E}\simeq \mathcal{D}(\bungk^{[1]},\Lambda)^{BW_E}$. The isomorphism is moreover equivariant for the natural prime-to-$p$ Hecke actions on both sides.
\end{Thm}

We have only constructed the Igusa stack for the good reduction locus of the Shimura variety. Luckily this is sufficient for the purpose of computing cohomology of Hodge type Shimura varieties, due to the following proposition. Write $\mathcal{S}^\circ_{K^p}\subset \mathcal{S}_{K^p}$ for the inclusion
\begin{align}
    \mathbf{Sh}_{K^p}^{\circ, \lozenge} \subset \mathbf{Sh}_{K^p}^{\lozenge}.
\end{align}
We consider the natural Hecke and $\mathrm{Gal}(\overline{E}/E)$ equivariant map $R \Gamma(\mathbf{Sh}_{K^p,C}, \Lambda) \to R\Gamma(\mathcal{S}_{K^p,C},\Lambda) \to R\Gamma(\mathcal{S}^\circ_{K^p,C},\Lambda)$.
\begin{Prop}\label{Prop:CohoGoodRed}
The natural map $R \Gamma(\mathbf{Sh}_{K^p,C}, \Lambda) \to R\Gamma(\mathcal{S}^\circ_{K^p,C},\Lambda)$ is an isomorphism.
\end{Prop}
\begin{proof}
The argument in \cite[Proposition 2.6.4]{CaraianiScholzeNoncompact} can be adapted to our situation, for the convenience of the readers we give the details. It suffices to show that for all compact open subgroups $K_p \subset G(\qp)$ the natural map 
\begin{align}
    R \Gamma(\mathbf{Sh}_{K,C}, \Lambda) \to R\Gamma(\mathcal{S}^\circ_{K,C},\Lambda)
\end{align}
is an isomorphism. Indeed, after taking the inverse limit, we then get an isomorphism
\begin{align}
    \varinjlim_{K_p} R \Gamma(\mathbf{Sh}_{K,C}, \Lambda) \to \varinjlim_{K_p} R\Gamma(\mathcal{S}^\circ_{K,C},\Lambda).
\end{align}
Applying \cite[Tag 09YQ]{stacks-project} to the left hand side and \cite[Proposition 14.9]{EtCohDiam} to the right hand side, we can identify this with
\begin{align}
    R \Gamma (\mathbf{Sh}_{K^p,C}, \Lambda) \to R\Gamma(\mathcal{S}^\circ_{K^p,C},\Lambda),
\end{align}
and the proposition follows. \smallskip

For $K_p \subset G(\qp)$ a compact open subgroup, we construct an ad hoc integral model $\mathscr{S}_{K}$ over $\CO_E$ for the Shimura variety $\mathbf{Sh}_{K}$, by choosing a Hodge embedding $\gx \hookrightarrow \gvx$ and taking a relative normalization of the integral model of the latter at a suitable level, see \cite[Introduction]{Madapusi}. Using Zarhin's trick, the level of the auxiliary Siegel Shimura variety can be chosen to be hyperspecial at $p$. Then we can apply \cite[Corollary 5.20]{LanStroh}, see \cite[Corollary 4.6]{LanStrohII} in the (Hdg) case\footnote{Note that in \cite{LanStrohII}, \cite{LanStroh}, the case (Hdg) has the more restrictive assumption that the level at $p$ is exactly the pullback of a hyperspecial level subgroup of $\mathrm{GSp}(\qp)$, but this is not necessary. This assumption was made due to reliance on the construction of toroidal compactifications of Hodge type Shimura varieties in \cite{Madapusi}, whose first version assumed this condition. However, in the published version of the cited reference, this assumption is removed. One can moreover check that the qualitative description of good compactifications in \cite[Proposition 2.2]{LanStroh} is satisfied by the compactifications constructed in \cite{Madapusi}, see \cite[Theorem 4.1.5, Theorem 5.2.11]{Madapusi}. We thank Kai-Wen Lan for explaining this to us.} and obtain
    \[R\Gamma(\mathbf{Sh}_{K,\overline{E}},\Lambda)\simeq R\Gamma(\mathscr{S}_{K,k}, R\Psi\Lambda),\]
    where $R\Psi$ is the nearby cycle functor. On the other hand, since $\mathcal{S}^\circ_{K,E}$ is the adic generic fiber of $\mathscr{S}_{K}$, by Lemma \ref{Lem:PotentiallyCrystalline} and its proof, the natural map
    \[R\Gamma(\mathcal{S}^\circ_{K,C},\Lambda)\rightarrow R\Gamma(\mathscr{S}_{K,k}, R\Psi\Lambda)\]
    is always an isomorphism by \cite[Theorem 3.5.13]{Huber96}. One concludes by invariance of \'etale cohomology under extension of algebraically closed fields.
\end{proof}
\begin{Rem}
    It is not clear to us that $R\Gamma(\mathcal{S}_{K^p,C},\Lambda)$ can be computed as the direct limit of $R\Gamma(\mathcal{S}_{K,C},\Lambda)$, because the adic spaces $\mathcal{S}_{K,C}$ are not quasi-compact.
\end{Rem}

\begin{proof}[Proof of Theorem \ref{Thm: WeilCohoShiVar}]
 The last statement about prime-to-$p$ Hecke action is clear once we can prove
  the earlier statements. Using Lemma~\ref{Lem:HeckeDiagram} and the Cartesian
  diagram from \eqref{Eq:CartesianDiagrammoduloGQP}, where we quotient the top
  row out by $\phi^{\mathbb{Z}}$ noting that $\phi$ acts trivially on the bottom
  row by Proposition~\ref{Prop:FrobeniusTrivial} and
  Proposition~\ref{Prop:AbsoluteFrobeniusBunG}, we get the following
  $2$-commutative diagram (where we have relabeled the maps for simplicity).
  \[ \begin{tikzcd} \label{Eq:NotSoHugeDiagram}
    \lbrack \mathcal{S}_{K^p}^{\circ} / (\underline{G(\qp)} \times
    \phi^\mathbb{Z}) \rbrack \arrow{r}{\pi} \arrow{d}{\tilde{g}} & \lbrack
    \mathrm{Gr}_{G,\mu^{-1}}^\lozenge / (\underline{G(\qp)} \times
    \phi^\mathbb{Z}) \rbrack \arrow{d}{g} \arrow{r}{\tilde{h}_2} & \bun_G^{[1]}
    \times_{\fp} \Div \\ \igs \arrow{r}{\overline{\pi}} & \bungmu
  \end{tikzcd} \]

  Using the formula in Definition~\ref{Def:HeckeFactorization}, we compute
  \[
    T_\mu^{[1]} (\mathcal{F}[-d])(-\tfrac{d}{2}) \simeq R\tilde{h}_{2,k,\ast}
    g_k^\ast \mathcal{F} = R\tilde{h}_{2,k,\ast} g_k^\ast
    R\overline{\pi}_{k,\ast} \Lambda \simeq R\tilde{h}_{2,k,\ast} R\pi_{k,\ast}
    \Lambda,
  \]
  where in the last step we use qcqs base change
  \cite[Proposition~17.6]{EtCohDiam} along $\overline{\pi}_k$.\footnote{The map $\overline{\pi}_k$ is qcqs by \cite[Proposition 10.11.(o)]{EtCohDiam}, using Corollary \ref{Cor:VSurjectiveBL}.} We now observe
  that the composition $\tilde{h}_2 \circ \pi$ is exactly the structure map
  \[
    \lbrack \mathcal{S}_{K^p}^{\circ}/(\underline{G(\qp)} \times
    \phi^\mathbb{Z}) \rbrack \to \lbrack \Div/\underline{G(\qp)} \rbrack.
  \]
  This identifies
  \[
    R (\tilde{h}_2 \circ \pi)_{k,\ast} \Lambda \simeq
    R\Gamma(\mathcal{S}^\circ_{K^p}, \Lambda)
  \]
  as objects of $\mathcal{D}(\bungk^{[1]},\Lambda)^{BW_E}$, where we use
  Proposition~\ref{Prop:GaloisRestriction} to identify the Weil group actions.
  The desired statement now follows from Proposition~\ref{Prop:CohoGoodRed}.
\end{proof}

\label{subsub:Coefficients}
Let $K_{\ell}$ be the image of $K^p$ in $G(\ql)$ and let $V$ be a continuous representation of $K_{\ell}$ on a finite free $\Lambda$-module (e.g.\ the base change to $\Lambda$ of a $K_{\ell}$-stable $\zl$-lattice in an algebraic representation of $G(\ql)$ over $\ql$). Then $V$ defines an automorphic local system $\mathbb{V}$ on the Shimura variety of level $K^p$ which in fact descends to $\igsk$. Defining $\mathcal{F}_{\mathbb{V}}=R \overline{\pi}_{\mathrm{HT}, k, \ast} \mathbb{V}$, the proof of Theorem \ref{Thm: WeilCohoShiVar} shows that there is a $G(\qp) \times W_E$-equivariant isomorphism
\begin{align}
    T_\mu^{[1]} \mathcal{F}_{\mathbb{V}} \xrightarrow{\sim} R\Gamma(\mathbf{Sh}_{K^p,C}, \mathbb{V})[d](\tfrac{d}{2}).
\end{align}

\subsection{Mantovan's product formula}
Our goal in this section is to prove a version of Mantovan's product formula for the cohomology of Hodge type Shimura varieties, which we deduce from Theorem \ref{Thm: WeilCohoShiVar}.

We use the same notation as before. In particular, $\Lambda$ is a Noetherian $\zl$-algebra in which $\ell$ is nilpotent; $\mathcal{S}^\circ_{K^p}$ is the good reduction locus of the Shimura variety with infinite level at $p$, and $\mathcal{F}$ is the complex $R\overline{\pi}_{\mathrm{HT},\ast}\Lambda$ on $\bungk$. For $b \in G(\qpbr)$, fix a left invariant Haar measure on $G_b(\qp)$ and write $\mathcal{H}(G_b)\coloneqq C_c(G_b(\qp))$ for the algebra of compactly supported $\Lambda$-valued functions on $G_b(\qp)$ (with multiplication given by convolution).
We define the local Shimura variety at infinite level attached to the local Shimura datum $(G,b,\mu)$ as the fiber product
\begin{equation}
    \begin{tikzcd}
        \mathcal{M}_{G,b,\mu,\infty} \arrow{r} \arrow{d} & \operatorname{Gr}_{G, \mu^{-1}}^{[b]} \arrow{d}{\operatorname{BL}} \\
        \spd k \arrow{r}{b} & \bung^{[b]},
    \end{tikzcd}
\end{equation}
where $\operatorname{Gr}_{G, \mu^{-1}}^{[b]}$ is the preimage of $\bung^{[b]}$ under $\mathrm{BL}$.

The map $\mathcal{M}_{G,b,\mu,\infty} \to \operatorname{Gr}_{G, \mu^{-1}}^{[b]}$ is a $\tilde{G}_b$-torsor, and the action of $\underline{G(\qp)}$ on $\operatorname{Gr}_{G, \mu^{-1}}^{[b]}$ lifts to an action on $\mathcal{M}_{G,b,\mu,\infty}$ which commutes with the $\tilde{G}_b$-action.

For any $[b]\in \bgmu$, let $i_b \colon \bungk^{[b]} \hookrightarrow \bungmuk$
be the inclusion of the corresponding stratum and we write the restriction of
$\overline{\pi}_\mathrm{HT}$ to this stratum as $\overline{\pi}_b$.

Let $C$ be as in Section \ref{sub:Hecke}. Assume we are given a $\spa C$-point of $\bungk$ corresponding to some $[b]$, and a lift of it to a $\spa C$-point $x$ of $\mathrm{Gr}_{G,\mu^{-1}}$. We choose a representative $b\in [b]$ and consider as in Section \ref{Sub:VSheafIgusa} the $v$-sheaf Igusa variety $\operatorname{Ig}^{b,\mathrm{v}}\gx \to \spd \fpbar$ equipped with commuting actions of $\tilde{G}_b$ and $\underline{\gafp}$; we write $\operatorname{Ig}^{b,\mathrm{v}}_{K^p}$ for its quotient by $\underline{K^p}$. We also consider the perfect Igusa variety $\mathrm{Ig}^b_{K^p}$ of level $K^p$ from Section \ref{Sec:PerfectIgusa} and recall that there is a natural open immersion
\begin{align}
    \mathrm{Ig}^{b,\diamond}_{K^p} \to \operatorname{Ig}^{b,\mathrm{v}}_{K^p}
\end{align}
by Corollary \ref{Cor:CanonicalCompactification}.
\begin{Lem} \label{ActionExists}
    The $\tilde{G}_b$ action on $\operatorname{Ig}^{b,\mathrm{v}}_{K^p}$ restricts to a $\tilde{G}_b$ action on $\mathrm{Ig}^{b,\diamond}_{K^p}$.
\end{Lem}
\begin{proof}
When $\gvx=\gx$, the group  $\tilde{G}_{V,b}$ is canonically isomorphic to $\mathbf{Aut}_{\gv}(\tilde{\mathbb{X}}_b)^{\diamond}$ by \cite[Corollary 9.46]{ZhangThesis}. The group $\mathbf{Aut}_{\gv}(\tilde{\mathbb{X}}_b)$ acts on $\mathrm{Ig}_{K^p}^{b}\gvx$ by \cite[Corollary 4.3.5]{CaraianiScholzeCompact}. For general $\gx$ of Hodge type, after choosing a suitable Hodge embedding $\gx \to \gvx$ and a lift $b$ to a map $\spd \fpbar \to \shtgmu$ so that we may reinterpret the Igusa varieties as in \eqref{Eq:IgbDiagram}, we get a commutative diagram
\begin{equation}
    \begin{tikzcd}
         \operatorname{Ig}_{K^p}^b\gx \arrow{r} \arrow{d} &  \operatorname{Ig}_{M^p}^b\gvx \arrow{d} \\
          \scrs_{K}\gx^{\mathrm{perf}}_{\fpbar} \arrow{r} & \scrs_{M}\gvx^{\mathrm{perf}}_{\fpbar}.
    \end{tikzcd}
\end{equation}
It follows from the construction that the natural map
\begin{align}
     \operatorname{Ig}_{K^p}^b\gx \to \operatorname{Ig}_{M^p}^b\gvx \times_{\scrs_{M}\gvx^{\mathrm{perf}}_{\fpbar}} {\scrs_{K}\gx^{\mathrm{perf}}_{\fpbar}}
\end{align}
is a closed immersion (it is a reduction of structure of torsors for profinite groups). We now consider the commutative diagram
\begin{equation}
    \begin{tikzcd}
         \operatorname{Ig}_{K^p}^b\gx^{\diamond} \arrow{d} \arrow{r} &  \operatorname{Ig}_{M^p}^b\gvx^{\diamond} \times_{\scrs_{M}\gvx^{\diamond}_{\fpbar}} {\scrs_{K}\gx^{\diamond}_{\fpbar}}   \arrow{r} \arrow{d} &\operatorname{Ig}_{M^p}^b\gvx^{\diamond} \arrow{d} \\
         \operatorname{Ig}^{b,\mathrm{v}}_{K^p}\gx \arrow{r} & \operatorname{Ig}_{M^p}^{b,\mathrm{v}}\gvx \times_{\scrs_{M}\gvx^{\diamond}_{\fpbar}} {\scrs_{K}\gx^{\diamond}_{\fpbar}}  \arrow{r}  & \operatorname{Ig}^{b,\mathrm{v}}_{M^p}\gvx.
    \end{tikzcd}
\end{equation}
From the discussion above, we deduce that the group $\tilde{G}_b$ acts compatibly on all objects in the outer square of the diagram, except possibly for $\operatorname{Ig}_{K^p}^b\gx^{\diamond}$. It thus suffices to show that the outer square is Cartesian. A small diagram chase shows that the right square is Cartesian. To show that the left square is Cartesian, we observe that the objects in the bottom row are the canonical compactifications toward $\scrs_{K}\gx^{\diamond}_{\fpbar}$ of the objects in the top row; this follows from Corollary \ref{Cor:CanonicalCompactification} and the fact that the formation of canonical compactifications commutes with base change \cite[Corollary 18.8 (iv)]{EtCohDiam}. Since the top horizontal arrow is a closed immersion and hence partially proper, it follows from the definition \cite[Proposition 18.6]{EtCohDiam} that the left square is Cartesian. This concludes the proof.
\end{proof}
We moreover consider the perfectoid Igusa variety $\mathrm{Ig}^b_{K^p,C}$ given by the adic generic fiber of the formal scheme $W(\mathrm{Ig}^{b}_{K^p})\times_{\spf W(k)}\spf \CO_C$, where $W(\mathrm{Ig}^{b}_{K^p})$ is the Witt vector lift of the perfect scheme $\mathrm{Ig}^b_{K^p}$. Note that there is a natural isomorphism
\begin{align}
    \mathrm{Ig}^{b,\diamond}_{K^p} \times_{\spd k} \spd C \xrightarrow{\sim} \left(\mathrm{Ig}^b_{K^p,C}\right)^{\lozenge},
\end{align}
and we will use $\mathrm{Ig}^{b,\lozenge}_{K^p,C}$ to denote the right hand
side, which is different from the base change of
$\mathrm{Ig}^{b,\lozenge}_{K^p}$ to $\spd C$.

\begin{Lem}\label{Lem:CohoIgsb}
There is an inclusion
\[\left[\mathrm{Ig}^{b,\lozenge}_{K^p,C}/\widetilde{G}_b\right]\hookrightarrow \mathrm{Igs}_{K^p,C}^{[b]}\]
which induces an isomorphism between their canonical compactifications. In particular, we have isomorphisms in $D(G_b(\qp),\Lambda)$
\[i_{b}^\ast\mathcal{F}\simeq R\overline{\pi}_{b,\ast}\Lambda \simeq R\Gamma(\mathrm{Ig}^b_{K^p,C},\Lambda)\simeq R\Gamma(\mathrm{Ig}^b_{K^p},\Lambda).\]
\end{Lem}

\begin{proof}
    The first claim follows from the identification of the fibers of the
    Hodge--Tate period map, see the proof of Proposition~\ref{Prop: HTfiber}.
    The isomorphism $i_{b}^\ast \mathcal{F} \simeq
    R\overline{\pi}_{b,\ast}\Lambda$ follows from \cite[Proposition~17.6]{EtCohDiam}
    and the fact that $\overline{\pi}_\mathrm{HT}
    \colon \mathrm{Igs}_{K^p} \to \bungmu$ is quasi-compact; this can be tested
    v-locally, and $\mathcal{S}_{K^p}^\circ \to
    \mathrm{Gr}_{G,\mu^{-1}}^\lozenge$ is indeed quasi-compact. Identification of
    the latter with the cohomology of $\mathrm{Ig}^b_{K^p,C}$ uses the fact that
    canonical compactifications do not change the cohomology, see
    \cite[Lemma~4.4.2]{CaraianiScholzeCompact}. The last isomorphism follows
    from \cite[Lemma~4.4.3]{CaraianiScholzeCompact}.
\end{proof}

\begin{Cor} \label{Cor:ULA}
  The sheaf $\mathcal{F}$ is universally locally acyclic with respect to
  $\bungmuk \to \spd k$.
\end{Cor}

\begin{proof}
  We first explain that it suffices to check that $j_{\mu,k,!} \mathcal{F}$ is ULA
  with respect to $\bungk \to \spd k$. By choosing a $\ell$-cohomologically
  smooth chart $X \to \bungk$ and restricting to $\bungmuk$, it follows from
  \cite[Definition~IV.2.31]{FarguesScholze} and
  \cite[Proposition~IV.2.13.(ii)]{FarguesScholze} that ULA sheaves $A \in
  \mathcal{D}(\bungk, \Lambda)$ pull back to ULA sheaves $j_{\mu,k}^\ast A \in
  \mathcal{D}(\bungmuk, \Lambda)$. Hence if $j_{\mu,k,!} \mathcal{F}$ is ULA,
  then so is $\mathcal{F} = j_{\mu,k}^\ast j_{\mu,k,!} \mathcal{F}$.

    For $[b]\in \bgmu$, Lemma \ref{Lem:CohoIgsb} identifies $i_{b}^\ast\mathcal{F}$ with $R\Gamma(\operatorname{Ig}^{b}_{K^p},\Lambda) \in D(G_b(\qp), \Lambda)$ under the equivalence of categories
\begin{align}
    D(\bun^{[b]}_{G,k},\Lambda) \simeq D(G_b(\qp), \Lambda).
\end{align}
According to \cite[Proposition VII.7.9]{FarguesScholze}, it suffices to show $R\Gamma(\operatorname{Ig}^{b}_{K^p},\Lambda)$ is admissible in the sense that for any pro-$p$ open subgroup $U_p\subset G_b(\qp)$, the complex $R\Gamma(\operatorname{Ig}^{b}_{K^p},\Lambda)^{U_p}$ is a perfect complex of $\Lambda$-modules. But the latter identifies with $R\Gamma(\operatorname{Ig}^{b}_{K^p}/U_p,\Lambda)$, which is the cohomology of (the perfection of) a finite type $k$-scheme, hence it is a perfect complex, see the proof of \cite[Proposition 8.21]{ZhangThesis}.
\end{proof}

Before stating the Mantovan product formula, we need the following computation from \cite{HamannImai}, where we have identified $D(\bungk^{[b]},\Lambda)$ with $D(G_b(\qp),\Lambda)$ as in \cite[Theorem I.5.1.(ii)]{FarguesScholze}.
\begin{Prop}[{\cite[Proposition 3.15]{HamannImai}}]\label{Prop:DualizingComplexBunGb}
    The dualizing complex on $\bungk^{[b]}$ is isomorphic to $\delta_b^{-1}[-2d_b]$, where $\delta_b$ is the character as in Definition 3.14 of loc.\ cit.\ and $d_b\coloneqq \langle 2\rho, \nu_b\rangle$ as before.
\end{Prop}
To state the following result, we remark that the structure map $p: \bungk^{[b]}\to \spd k$ is a fine map of decent Artin $v$-stacks in the sense of \cite[Definition 1.1, Definition 1.3]{GulottaHansenWeinstein}; this can be deduced from \cite[Example IV.1.9]{FarguesScholze} and \cite[Proposition IV.1.22]{FarguesScholze}. Therefore there is a well-defined exceptional pushforward $R p_{!}$, see \cite[Theorem 1.4]{GulottaHansenWeinstein}.
\begin{Cor}\label{Cor:HomologyGb}
Under the identification $D(\bungk^{[b]},\Lambda)\simeq D(G_b(\qp),\Lambda)$,
there is a natural isomorphism of functors from $D(G_b(\qp),\Lambda)$ to $D(\Lambda)$
 \[ Rp_{!}(-)\simeq (-\otimes_\Lambda \delta_b[2d_b])\otimes^\mathbb{L}_{\mathcal{H}(G_b)}\Lambda,\]
 where $\Lambda$ is the trivial representation of $G_b(\qp)$, considered as a $\mathcal{H}(G_b)$-module via $f \cdot 1=\int_G f(y)dy$.
\end{Cor}
\begin{proof}
    From Proposition~\ref{Prop:DualizingComplexBunGb} we see that if we write $Rp_{!}$ in two steps
    \[\bun^{[b]}_{G,k}\simeq [\spd k /\widetilde{G}_b]\xrightarrow{\alpha}[\spd k /G_b(\qp)]\xrightarrow{\beta} \spd k,\]
    then the equivalence between $D(\bun^{[b]}_{G,k},\Lambda)$ and $D(G_b(\qp),\Lambda)$ is induced by $\alpha^\ast$, $\alpha_\ast$. We have $R\alpha_!(-)\simeq (-)\otimes_\Lambda\delta_b[2d_b]$. Indeed, one can check this by computing using cohomological smoothness of $\alpha$ that
    \begin{align}
    \Hom_\Lambda(R\alpha_! A, B) &\simeq \Hom_\Lambda(A, R\alpha^! B) \\ &\simeq \Hom_\Lambda(A, B \otimes_\Lambda \delta^{-1}_b[-2d_b]) \simeq \Hom_\Lambda(A \otimes_\Lambda \delta_b[2d_b], B)
    \end{align}
    and applying Yoneda's lemma.
    See the computation in \cite[Lemma 7.4]{KoshikawaGeneric} and \cite[Corollary 1.9.(1)]{HamannImai} for the identification $R\alpha^!\Lambda \simeq \delta^{-1}_b[-2d_b]$. Meanwhile, $R\beta_!$ is group homology twisted by the module of Haar measures on $G_b(\qp)$, as explained in \cite[Example 4.2.4]{HansenKalethaWeinstein}, cf.\ \cite[Section 5.1.3]{Mantovan}. Hence the desired formula follows.
\end{proof}
The following result, which first appeared in \cite{Mantovan}, \cite{MantovanPEL} for PEL type Shimura varieties, is often referred to as Mantovan's product formula, cf.\ \cite[Theorem 7.1]{KoshikawaGeneric}, \cite[Corollary 3.17]{Hamann-Lee}. Note that $\mathcal{M}_{G,b,\mu,\infty}$ admits structure maps to both $\spd k$ and $\spd E$ and therefore to $\spd E \times_{\spd \mathbb{F}_q} \spd k$. This structure map admits a descent $\mathcal{M}_{G,b,\mu,\infty,\phi} \to \Divk$ defined as the fiber product
\begin{equation}
    \begin{tikzcd}
        \mathcal{M}_{G,b,\mu,\infty,\phi} \arrow{r} \arrow{d} & \lbrack\operatorname{Gr}_{G, \mu^{-1}}/\phi^{\mathbb{Z}}\rbrack \arrow{r} \ar[d]& \Div \\
        \spd k \arrow{r}{b} & \bungmu.
    \end{tikzcd}
\end{equation}
This equips the exceptional pushforward $R\Gamma_c(\mathcal{M}_{G,b,\mu,\infty},\delta_b)$ along the structure map $\mathcal{M}_{G,b,\mu,\infty} \to \spd C$ with an action of $G_b(\qp)\times G(\qp)\times W_E$.

\begin{Thm} \label{Thm:MantovanFormula}
There exists a filtration on $R\Gamma(\mathbf{Sh}_{K^p,\overline{E}},\Lambda)$ by complexes of smooth representations of $G(\qp)\times W_E$, labeled by $B(G,\mu^{-1})$, with graded pieces given by
\[i_1^\ast T_\mu (Ri_{b!}i_b^\ast \mathcal{F}[-d])(-\tfrac{d}{2}) \simeq R\Gamma(\mathrm{Ig}^b_{K^p},\Lambda)^\mathrm{op}\otimes_{\mathcal{H}(G_b)}^\mathbb{L} R\Gamma_c(\mathcal{M}_{G,b,\mu,\infty},\delta_b)[2d_b].\]
\end{Thm}
\begin{proof}
    By Proposition~\ref{Prop:CohoGoodRed}, it suffices to consider the cohomology of $\mathcal{S}^\circ_{K^p,C}$. The category $D(\bungk,\Lambda)$ has a semi-orthogonal decomposition into $D(\bun^{[b]}_{G,k},\Lambda)$'s via excision triangles \cite[Theorem I.5.1]{FarguesScholze}; it follows that $\mathcal{F}$ has a filtration whose graded pieces are $Ri_{b!}i_b^\ast\mathcal{F}$. Now Theorem~\ref{Thm: WeilCohoShiVar} tells us that $R\Gamma(\mathcal{S}^\circ_{K^p,C},\Lambda)$ has a filtration with graded pieces given by $i_1^\ast T_\mu (Ri_{b!}i_b^\ast \mathcal{F}[-d])(-\tfrac{d}{2})$. Therefore it suffices to identify this with the right-hand side of the exhibited equation.

    But this follows from restricting the Cartesian diagram in Theorem \ref{Thm:HodgeMain} to the Newton stratum labeled by $[b]$. More precisely, we have a diagram
    \[ \begin{tikzcd}[column sep = small]
      \left[ \mathrm{Gr}_{G,\mu^{-1},\ebreve}^{[b]} / (\underline{G(\qp)} \times
      \varphi_q^{\mathbb{Z}}) \right] \arrow{d}{g_b} \arrow[bend left=15]{rr}{q}
      \arrow{r}{z} & \left[\mathrm{Gr}_{G,\mu^{-1},\ebreve} /
      (\underline{G(\qp)} \times \varphi_q^{\mathbb{Z}}) \right] \arrow{d}{g}
      \arrow{r} & \left[\spd k / \underline{G(\qp)} \right] \times_k
      \mathrm{Div}_{E,k}^1 \\ \bungk^{[b]} \arrow{r}{i_b} & \bungmuk \\
      \igsk^{[b]} \arrow{r} \arrow{u}{\overline{\pi}_b} & \igsk
      \arrow{u}{\overline{\pi}_{\mathrm{HT}}}
    \end{tikzcd} \]
    where both squares are Cartesian. It follows from proper base change, see \cite[Proposition~22.8]{EtCohDiam}, along the top square that
    \begin{align}
        g^{\ast} (Ri_{b!}i_b^\ast \mathcal{F}) \simeq R z_{!} g_b^{\ast} i_b^{\ast} \mathcal{F}
    \end{align}
    and then from qcqs base change, see \cite[Proposition~17.6]{EtCohDiam}, along the bottom square that furthermore
    \begin{align}
        R z_{!} g_b^{\ast} i_b^{\ast} \mathcal{F} \simeq R z_{!} g_b^{\ast} R\overline{\pi}_{b\ast}\Lambda.
    \end{align}
    If we combine this with the proof of Theorem~\ref{Thm: WeilCohoShiVar}, we find that
    \begin{align}
     i_1^\ast T_\mu (Ri_{b!}i_b^\ast \mathcal{F}[-d])(-\tfrac{d}{2})\simeq Rq_! g_b^\ast (R\overline{\pi}_{b\ast}\Lambda).
    \end{align}
    Notice that the map $q$ factors as a composition
    \begin{align}
        [\mathcal{M}_{G,b,\mu,\infty, \phi}/(\widetilde{G}_b\times \underline{G(\qp)})]
    &\simeq \left[\mathrm{Gr}_{G,\mu^{-1},\ebreve}^{[b]} /
      (\underline{G(\qp)} \times \varphi_q^{\mathbb{Z}}) \right] \\ &\xrightarrow{q'} [\spd k /(\widetilde{G}_b\times \underline{G(\qp)})] \times_{\spd k} \Divk \\
      &\xrightarrow{p_2} [\spd k/\underline{G(\qp)}] \times_{\spd k} \Divk.
    \end{align}
    Hence
    \begin{align}
        Rq_! g_b^\ast (R\overline{\pi}_{b\ast}\Lambda)
        & \simeq  Rp_{2,!} Rq'_{!} g_b^\ast (R\overline{\pi}_{b\ast}\Lambda)\\
        & \simeq  Rp_{2,!} Rq'_{!} {q'}^\ast p_1^\ast (R\overline{\pi}_{b\ast}\Lambda)\\
        & \simeq  Rp_{2,!} (p_1^\ast R\overline{\pi}_{b\ast}\Lambda\otimes_{\Lambda}^\mathbb{L} Rq'_{!}\Lambda),
    \end{align}
    where $p_1: [\spd k /(\widetilde{G}_b\times \underline{G(\qp)})] \times_{\spd k} \Divk \to [\spd k /\widetilde{G}_b] \times_{\spd k} \Divk$ is the natural projection, and the last isomorphism uses the projection formula, see \cite[Proposition~22.23]{EtCohDiam}. Now we note that
    \begin{align}
    & R\overline{\pi}_{b\ast}\Lambda\simeq R\Gamma(\mathrm{Ig}^b_{K^p,C},\Lambda),\\
    & R\Gamma_c(\mathcal{M}_{G,b,\mu,\infty},\Lambda)\coloneqq Rq'_{!}\Lambda,\\
    & Rp_{2,!}(-)\simeq (-\otimes_\Lambda \delta_b[2d_b])\otimes^\mathbb{L}_{\mathcal{H}(G_b)}\Lambda,
    \end{align}
    Here the first isomorphism uses Lemma~\ref{Lem:CohoIgsb} and the last isomorphism uses Corollary~\ref{Cor:HomologyGb}. Also, by the projection formula, see \cite[Proposition 22.11]{EtCohDiam}, we have
    \[Rq'_{!}\Lambda \otimes_\Lambda \delta_b\simeq Rq'_{!}{q'}^\ast \delta_b=: R\Gamma_c(\mathcal{M}_{G,b,\mu,\infty}, \delta_b).\]
    Now combine these with \cite[Proposition 5.12]{Mantovan} to conclude the desired formula.
\end{proof}

In this section we compare our definition of $R
\Gamma_c(\mathcal{M}_{G,b,\mu,\infty,C}, \Lambda)$ with a potentially different definition in the literature, see e.g. \cite[discussion before
Theorem~1.13]{Hamann-Lee} or \cite[page~90]{Hamann22}. We consider the structure
map $a \colon \mathcal{M}_{G,b,\mu,\infty,C} \to \spd C$. For a compact open
subgroup $K \subset G(\qp)$ we consider the quotient $\mathcal{M}_{G,b,\mu,K,C}
:= \mathcal{M}_{G,b,\mu,\infty,C}/\underline{K}$, which is the local Shimura
variety of level $K$ associated to $(G,b,\mu)$; this is a smooth rigid space
over $\spa C$, see \cite[Section 24.1]{ScholzeWeinsteinBerkeley}.

Let us use $R\Gamma_c$ to denote the shriek pushforward along the structure map
to $\spd C$. By definition, see \cite[Definition~22.13]{EtCohDiam}, we have
\[
  R\Gamma_c(\mathcal{M}_{G,b,\mu,K,C}, A) = \varinjlim_U R\Gamma_c(U, A
  \vert_U),
\]
where $U \subseteq \mathcal{M}_{G,b,\mu,K,C}$ ranges over quasi-compact opens
(see the discussion before \cite[Theorem~IX.3.1]{FarguesScholze}). Moreover, on
\cite[page~90]{Hamann22}, Hamann considers the complex
\[
  R\Gamma_c(G,b,\mu) := \varinjlim_{K} R\Gamma_c(\mathcal{M}_{G,b,\mu,K,C},
  \mathcal{S}_\mu),
\]
where $\mathcal{S}_\mu$ is the perverse sheaf in Section~\ref{Sec:Cohomology}.

\begin{Prop} \label{Cor:CompactlySupportedDirectLimit}
  There is a natural isomorphism
  \[
    R\Gamma_c(G,b,\mu) \xrightarrow{\sim}
    R\Gamma_c(\mathcal{M}_{G,b,\mu,\infty,C}, \mathcal{S}_\mu).
  \]
\end{Prop}

\begin{proof}
  For a compact open subgroup $K \subseteq G(\qp)$ and a quasi-compact open $U
  \subset \mathcal{M}_{G,b,\mu,K}$ we consider its inverse image $U_{\infty}
  \subset \mathcal{M}_{G,b,\mu,\infty}$ (which is again a quasi-compact open)
  with projection map $a_{U,\infty} \colon U_{\infty} \to U$. Choose a
  decreasing sequence $\{K_i\}_{i \in I}$ of compact open subgroups of $K$ with
  $\bigcap_i K_i=\{1\}$, and consider $a_{U,i} \colon U_i :=
  U_{\infty}/\underline{K_i} \to U$ and $b_{U,i} \colon U_\infty \to U_i$.

  Note that $a_{U,i}$, $b_{U,i}$, $a_{U,\infty}$ are all pro-finite \'{e}tale,
  hence proper and of zero $\mathrm{dim.trg.}$. Thus the natural maps
  $\mathcal{S}_\mu \to Rb_{U,i,\ast} \mathcal{S}_\mu = Rb_{U,i,!}
  \mathcal{S}_\mu$ induce a map
  \[
    \varinjlim_i Ra_{U,i,!} \mathcal{S}_\mu \to Ra_{U,\infty,!} \mathcal{S}_\mu.
  \]
  It follows from \cite[Proposition~14.9]{EtCohDiam} (or rather its relative
  version used in the proof of \cite[Proposition~20.7]{EtCohDiam}) that this is an
  isomorphism. Applying the (derived) shriek pushforward along $U \to \spd C$,
  and using that it commutes with all colimits,
  \cite[Proposition~22.20]{EtCohDiam}, we get an isomorphism
  \[
    \varinjlim_i R\Gamma_c(U_i, \mathcal{S}_\mu) \to R\Gamma_c(U_{\infty},
    \mathcal{S}_\mu).
  \]
  We now take the direct limit over $U$ to obtain the desired
  isomorphism.
\end{proof}

\subsection{Perversity} \label{Sec:PerversityMain} We continue using the notation in Section \ref{Sub:DualizingSheaf}. In this section we prove Theorem~\ref{Thm:IntroPerversity}, which is a perversity result for the relative cohomology of the Igusa stack over $\bungmuk$. This will be a key input to our study of the generic part of the torsion cohomology of compact Hodge type Shimura varieties in Section~\ref{Sec:TorsionVanishing}. We make the following assumption for the remainder of Section \ref{Sec:PerversityMain}.

\begin{Assump}
    The map $\mathbf{Sh}_{K} \to \spec E$ is proper.
\end{Assump}

Under this assumption, the good reduction locus agrees with the whole Shimura variety, see \cite[Example 5.20]{ImaiMieda}. Therefore the Cartesian diagram in Theorem \ref{Thm:HodgeMain} becomes (after base change to $k$)
\begin{equation} \label{Eq:Cartesian}
\begin{tikzcd}
    \mathbf{Sh}_{K^p, \ebreve}^{\lozenge} \ar[r,"\pi_\mathrm{HT}"]\ar[d,"q_{\mathrm{Igs}}"] &
    \operatorname{Gr}_{G, \mu^{-1}, \ebreve} \ar{d}{\operatorname{BL}}\\
    {\igsk} \ar[r,"\overline{\pi}_\mathrm{HT}"] & \bungmuk.
\end{tikzcd}
\end{equation}
Let us write $\mathcal{F}\coloneqq R\overline{\pi}_{\mathrm{HT},\ast}\Lambda =
R\overline{\pi}_{\mathrm{HT},!}\Lambda$ as before.

\begin{Prop} \label{Prop:semiperverse}
If $\mathbf{Sh}_{K}$ is proper over $\spec E$, then
\[ i_{b}^\ast \mathcal{F} \in D^{\leq d_b}(\bun^{[b]}_{G,k}, \Lambda). \]
\end{Prop}

\begin{proof}
Using Lemma \ref{Lem:CohoIgsb}, it suffices to show that $R\Gamma(\operatorname{Ig}^{b}_{K^p}, \Lambda)$ is concentrated in degrees $\le d_b$. We deduce from Proposition \ref{Prop:CentLeafDim} below that central leaves in the Newton stratum associated to $b$ are equidimensional of dimension $d_b$. It follows from \cite[Proposition 5.14.(4)]{MaoCompact} that they are moreover affine (using the properness of $\scrs_{K}$, which follows from the properness of $\mathbf{Sh}_K$ by a theorem of Madapusi \cite[Corollary 4.1.7]{Madapusi}).

Thus the Igusa variety is a pro-finite \'etale cover of the perfection of an affine $k$-scheme of dimension $d_b$, see Corollary \ref{Cor:InverseLimitOfFiniteType}. Since the \'etale site is invariant under perfection, the Igusa variety is also the perfection of an inverse limit of $d_b$-dimensional $k$-schemes, with finite \'etale transition maps. Using \cite[Tag 09YQ]{stacks-project}, cohomology of the Igusa variety is a colimit of that of those $k$-schemes in the inverse system. Hence the cohomological dimension bound follows from Artin vanishing, see \cite[Corollaire 3.5, Partieme XIV, Tome 3]{SGA4}.
\end{proof}

The main result of this section is the following.

\begin{Thm}\label{Thm: Perversity}
  If $\mathbf{Sh}_K$ is proper over $\spec E$, then
  $\mathcal{F}$ is perverse for the perverse $t$-structure on $D(\bungmuk,
  \Lambda)$.
\end{Thm}

\begin{proof}
  Since $\overline{\pi}_\mathrm{HT}$ is proper we have
  $R\overline{\pi}_{\mathrm{HT},!} \Lambda = R\overline{\pi}_{\mathrm{HT},\ast}
  \Lambda$. Let us write $\mathbb{D}_X$ for the Verdier duality functor with
  respect to $X \to \spd k$. Then it follows from relative Verdier duality, see
  \cite[Proposition~23.3.(i)]{EtCohDiam}, we have
  \[
    \mathbb{D}_{\bungmuk} \mathcal{F} \simeq \mathbb{D}_{\bungmuk}
    R\overline{\pi}_{\mathrm{HT},!} \Lambda \simeq
    R\overline{\pi}_{\mathrm{HT},\ast} \mathbb{D}_{\igsk} \Lambda \simeq
    \mathcal{F}.
  \]
  Here the last isomorphism follows from
  Theorem~\ref{Thm:IgusaDualizingComplex}. Recall from
  Proposition~\ref{Prop:semiperverse} that $i_b^\ast \mathcal{F} \in D^{\le
  d_b}(\bun^{[b]}_{G,k}, \Lambda)$ for $[b] \in \bgmu$. On the other hand, again
  by relative Verdier duality for $i_b$, see
  \cite[Proposition~22.3.(ii)]{EtCohDiam}, we have
  \[
    Ri_b^! \mathcal{F} \simeq Ri_b^! \mathbb{D}_{\bungmuk} \mathcal{F} \simeq
    \mathbb{D}_{\bungk^{[b]}} i_b^\ast \mathcal{F} \in D^{\ge
    d_b}(\bungk^{[b]}, \Lambda)
  \]
  as the dualizing sheaf on $\bungk^{[b]}$ is concentrated in degree $2d_b$,
  see Proposition~\ref{Prop:DualizingComplexBunGb}.
\end{proof}

\subsection{Dimensions of central leaves} We now compute the dimensions of the central leaves $C^b$ introduced in Section \ref{Sec:CentralLeaves}. Let $\Lambda$ be a $\zl$ algebra in which $\ell$ is nilpotent, let $K_p \subset G(\qp)$ be a parahoric subgroup and let $K=K_pK^p$.

\begin{Prop} \label{Prop:CentLeafDim}
  For each $b \colon \spec \fpbar \to \shtglocmu$, the dualizing sheaf (with
  coefficients in $\Lambda$) of the central leaf
  $C_K^{b,\diamond} \to \spd \fpbar$ is invertible and concentrated in
  degree $-2d_b$. In particular, the central leaf $C_K^b$ has pure dimension
  $d_b = \langle 2\rho, \nu_b \rangle$.
\end{Prop}

Let $\widehat{\mathscr{S}}_{K_p}\gx_{\mathcal{O}_{\ebreve}}^{[b]}$ be the formal completion of the locally closed subscheme
  $\mathscr{S}_{K_p}\gx_{\fpbar}^{[b]} \hookrightarrow
  \mathscr{S}_{K_p}\gx_{\mathcal{O}_{\ebreve}}$. We consider 
  \[\mathbf{Sh}_K\gx_{\ebreve}^{\circ,\lozenge,]b[}:= (\widehat{\mathscr{S}}_{K_p}\gx^{[b]})_{\ebreve}^\lozenge,\]
the diamond attached to the adic generic fiber of $\widehat{\mathscr{S}}_{K_p}\gx_{\mathcal{O}_{\ebreve}}^{[b]}$.\footnote{Warning: This is different from $\mathbf{Sh}_K\gx_{\ebreve}^{\circ,\lozenge,[b]}$, which is defined using Equation~\eqref{Eq:NewtonStrat}!} We have an induced inclusion $\mathbf{Sh}_K\gx_{\ebreve}^{\circ,\lozenge,]b[}\subseteq \mathbf{Sh}_K\gx_{\ebreve}^{\circ,\lozenge}$.

\begin{Lem} \label{Lem:Specialization}
  There is an inclusion
  $(\widehat{\mathscr{S}}_{K_p}\gx_{\mathcal{O}_{\ebreve}}^{[b]})^\lozenge
  \subseteq \mathscr{S}_{K_p}\gx_{\mathcal{O}_{\ebreve}}^{\diamond,[b]}$.
\end{Lem}

\begin{proof}
 It
  suffices to consider rank $1$ geometric points since for any perfectoid space
  $S$ and map $S \to \bung$, the $[b]$-stratum $S^{[b]} \subseteq S$ is stable
  under vertical specializations. Let $C$ be an algebraically closed complete
  non-archimedean field with residue field $l = \mathcal{O}_C/\mathfrak{m}_C$,
  and let $y \colon \spf \mathcal{O}_C \to \mathscr{S}_K\gx$ be a map whose
  restriction $x \in \mathscr{S}_K\gx(l)$ lies in the $[b]$-Newton stratum. Upon
  choosing a section $l \to \mathcal{O}_C^\flat$ using
  Lemma~\ref{Lem:ResidueFieldSection}, the map $y$ factors through
  \[
    \tilde{y} \colon \spf \mathcal{O}_C \to
    \widehat{\mathscr{S}_K\gx_{W_{\mathcal{O}_E}(l)}}_{/x}.
  \]
  By the Pappas--Rapoport axioms
  \cite[Conjecture~4.2.2.(c)]{PappasRapoportShtukas},
  \cite[Definition~4.1.2.(iv)]{Companion}, the right hand side is identified
  under $\Theta_x$ with a subsheaf of
  $\mathcal{M}_{\mathcal{G},b_x,\mu}^\mathrm{int}$. Then the induced map $\spd
  \mathcal{O}_C \to \bung$ lands in $\bung^{[b]}$, which is the stratum
  corresponding to $[b_x]$, cf.\ Lemma~\ref{Lem:UnderlyingUniformization}.
\end{proof}

Consider the following diagram, where the bottom horizontal map is induced by Corollary \ref{Cor:ProductFormula} and the left vertical map by Corollary \ref{Cor:CanonicalCompactification}.
\begin{equation} \label{Eq:ProductFormulaLemmaDiagram}
    \begin{tikzcd}
    \mathrm{Ig}^b\gx^\diamond \times_{\spd \fpbar} \mintgmu \arrow[hook]{d} \arrow[dashed]{r} & \bigl(\widehat{\mathscr{S}}_{K_p}\gx_{\mathcal{O}_{\ebreve}}^{[b]} \bigr)^\lozenge \arrow[hook]{d} \\
         \igvinfbgx \times_{\spd \fpbar} \mintgmu \arrow{r} &
    \mathscr{S}_{K_p}\gx_{\mathcal{O}_{\ebreve}}^{\diamond,[b]}.
    \end{tikzcd}
\end{equation}
We now prove that there exists a dashed arrow making the diagram commute.

\begin{Lem} \label{Lem:CentLeafOpenSub}
The map $\mathrm{Ig}^b\gx^\diamond \times_{\spd \fpbar} \mintgmu \to \mathscr{S}_{K_p}\gx_{\mathcal{O}_{\ebreve}}^{\diamond,[b]}$ of \eqref{Eq:ProductFormulaLemmaDiagram} factors through
$\bigl(\widehat{\mathscr{S}}_{K_p}\gx_{\mathcal{O}_{\ebreve}}^{[b]} \bigr)^\lozenge$.
\end{Lem}

\begin{proof}
Because $\bigl(\widehat{\mathscr{S}}_{K_p}\gx_{\mathcal{O}_{\ebreve}}^{[b]} \bigr)^\lozenge$ is an open subsheaf, see \cite[Proposition 4.22]{GleasonSpecialization}, it suffices to check the statement on geometric points. Let $C \supset \fpbar$ be an algebraically closed complete non-archimedean field and let $C^+ \subset C$ be a bounded open valuation subring. Given a $\spa(C, C^+)$-point of the left hand side, we first claim that it extends to a $\spd(C^+, C^+)$-point. Indeed, the extension to $\spd(C^+, C^+) \to \mathrm{Ig}^b\gx^\diamond$ is clear from the construction of the functor $(-)^\diamond$ from Section~\ref{Sub:AdicSpaces}. The extension on $\mintgmu$ follows from the fact that any framed $\mathcal{G}$-shtuka over an untilt $(C^\sharp, C^{\sharp+})$ extends to a $\mathcal{G}$-Breuil--Kisin--Fargues module by the extension result of Ansch\"utz \cite[Theorem~9.10]{AnschuetzExtension}, cf.\ \cite[Theorem~2.10]{GleasonShtukas}.

  Hence, given a $\spa(C, C^+)$-point of $\mathrm{Ig}^b\gx^\diamond
  \times_{\fpbar} \mintgmu$, it maps to a $\spa(C, C^+)$-point of the target
  $\mathscr{S}_{K_p}\gx_{\mathcal{O}_{\ebreve}}^{\diamond,[b]}$ that further
  extends to a $\spd(C^+, C^+)$-point. We claim that every map $f \colon
  \spd(C^+, C^+) \to
  \mathscr{S}_{K_p}\gx_{\mathcal{O}_{\ebreve}}^{\diamond,[b]}$ automatically
  factors through the open subsheaf
  $(\widehat{\mathscr{S}}_{K_p}\gx_{\mathcal{O}_{\ebreve}}^{[b]})^\lozenge$. We
  first algebraize $f$. The restriction $f \vert_{\spa(C, C^+)}$, corresponds to
  a map $f^\sharp \colon \spf C^{\sharp+} \to
  \widehat{\mathscr{S}}_{K_p}\gx_{\mathcal{O}_{\ebreve}}$ by construction of
  $(-)^\diamond$ and Lemma~\ref{Lem:DiamondOfFormalScheme}, and this induces a
  map $(f^\sharp)^\lozenge \colon \spd(C^{\sharp+}, C^{\sharp+}) \to
  \mathscr{S}_{K_p}\gx_{\mathcal{O}_{\ebreve}}^\diamond$. Both $f$ and
  $(f^\sharp)^\lozenge$ are extensions of the same $(C, C^+)$-point, and so
  \cite[Proposition~4.9, Proposition~4.17]{GleasonSpecialization} shows that $f
  = (f^\sharp)^\lozenge$. Then $f^\sharp \vert_{\spec (C^+/C^{\circ\circ})}
  \colon \spec C^+/C^{\circ\circ} \to \mathscr{S}_{K_p}\gx_{\fpbar}$ has image
  contained in $\mathscr{S}_{K_p}\gx_{\fpbar}^{\diamond,[b]}$ upon applying
  $(-)^\diamond$. This means that the corresponding $G$-isocrystal on
  $C^+/C^{\circ\circ}$ factors through $\gisoc^{[b]}$, and hence $f^\sharp$
  sends $\spec(C^+/C^{\circ\circ}) \to \mathscr{S}_{K_p}\gx_{\fpbar}^{[b]}$.
  Therefore $f^\sharp \colon \spf C^{\sharp+} \to
  \widehat{\mathscr{S}}_{K_p}\gx_{\mathcal{O}_{\ebreve}}^{[b]}$, and hence $f =
  (f^\sharp)^\lozenge$, has image lying in
  $(\widehat{\mathscr{S}}_{K_p}\gx_{\mathcal{O}_{\ebreve}}^{[b]})^\lozenge$.
\end{proof}

\begin{proof}[Proof of Proposition~\ref{Prop:CentLeafDim}]
  We first prove that the dualizing sheaf of $C_K^{b,\diamond} \to \spd \fpbar$
  is invertible and concentrated in degree $-2d_b$. Recall from the proof of
  Corollary~\ref{Cor:InverseLimitOfFiniteType} that we have a profinite
  $\Gamma_b$-torsor $\mathrm{Ig}^b_{K^p}\gx \to C_K^b$ of perfect schemes, where
  $\Gamma_b \subseteq G_b(\qp)$ is a compact open subgroup, and this induces a
  $\underline{\Gamma_b}$-torsor $\mathrm{Ig}^b_{K^p}\gx^\diamond \to
  C_K^{b,\diamond}$. On the other hand, we have an open embedding
  $\mathrm{Ig}^b_{K^p}\gx^\diamond \hookrightarrow
  \mathrm{Ig}^{b,\mathrm{v}}_{K^p}\gx$ of v-sheaves by
  Corollary~\ref{Cor:CanonicalCompactification}.

  Using the section $\underline{G_b(\qp)} \to \tilde{G}_b$, we consider the
  v-sheaves
  \begin{equation} \label{eq:CentLeafDual}
    [\mathrm{Ig}^{b,\mathrm{v}}_{K^p}/\underline{\Gamma_b}] \xleftarrow{f}
    \mathrm{Ig}^{b,\mathrm{v}}_{K^p} \times_{\spd \fpbar}^{\underline{\Gamma_b}}
    \mathcal{M}_{G,b,\mu,K_p} \xrightarrow{g} \mathrm{Ig}^{b,\mathrm{v}}_{K^p}
    \times_{\spd \fpbar}^{\tilde{G}_b} \mathcal{M}_{G,b,\mu,K_p} =
    \mathbf{Sh}_K\gx_{\ebreve}^{\circ,\lozenge,[b]},
  \end{equation}
  where the last identification comes from
  Corollary~\ref{Cor:ProductFormula}. The map $f$ is a base change of
  $[\underline{\Gamma_b} \backslash \spd \fpbar] \leftarrow [\underline{\Gamma_b}
  \backslash \mathcal{M}_{G,b,\mu,K_p}]$, which is v-surjective and
  $\ell$-cohomologically smooth of dimension $1 + \langle 2\rho, \mu \rangle$ by
  \cite[Proposition~23.15, Proposition~24.4, Proposition~24.5]{EtCohDiam}
  together with the smoothness of $\mathcal{M}_{G,b,\mu,K_p}$. By applying
  \cite[Proposition~23.15]{EtCohDiam} again, we deduce that $f$ is v-surjective
  and $\ell$-cohomologically smooth of dimension $1 + \langle 2\rho, \mu
  \rangle$ as well.

  On the other hand, the map $g$ is a base change of $[\spd
  \fpbar/\underline{\Gamma_b}] \to [\spd \fpbar/\tilde{G}_b]$, which is a
  composition of an \'etale map $[\spd \fpbar/\underline{\Gamma_b}] \to [\spd
  \fpbar/\underline{G_b(\qp)}]$ with the map $[\spd \fpbar/\underline{G_b(\qp)}]
  \to [\spd \fpbar/\tilde{G}_b]$ whose fiber is $\tilde{G}_b /
  \underline{G_b(\qp)} \cong \tilde{G}_b^{>0}$, see
  \cite[Section~III.5.1]{FarguesScholze}. As we see from
  \cite[Proposition~II.2.5, Proposition~III.5.1]{FarguesScholze}, the unipotent
  part $\tilde{G}_b^{>0}$ is an iterated extension of $\spd
  \fpbar[[t^{1/p^\infty}]]$, and hence $\ell$-cohomologically smooth of
  dimension $\langle 2\rho, \nu_b \rangle$ by \cite[Theorem~24.1]{EtCohDiam}.
  Again by applying \cite[Proposition~23.15]{EtCohDiam} twice, we deduce that
  $g$ is $\ell$-cohomologically smooth of dimension $\langle 2\rho, \nu_b
  \rangle$.

  By Lemma~\ref{Lem:CentLeafOpenSub} and Corollary~\ref{Cor:CanonicalCompactification}, we have open v-subsheaves and maps
  \[
    C_K^{b,\diamond} = [\mathrm{Ig}^{b,\diamond}_{K^p} / \underline{\Gamma_b}]
    \xleftarrow{f_0} \mathrm{Ig}^{b,\diamond}_{K^p}
    \times_{\spd \fpbar}^{\underline{\Gamma_b}} \mathcal{M}_{G,b,\mu,K_p}
    \xrightarrow{g_0} \mathbf{Sh}_K\gx_{\ebreve}^{\circ,\lozenge,]b[}
  \]
  obtained by restricting the diagram \eqref{eq:CentLeafDual}. Since we have
  restricted $\ell$-cohomologically smooth morphisms to open subsheaves, the
  maps $f_0$ and $g_0$ are still $\ell$-cohomologically smooth of dimensions $1
  + \langle 2\rho, \mu \rangle$ and $\langle 2\rho, \nu_b \rangle$,
  respectively. The v-sheaf $\mathbf{Sh}_K\gx_{\ebreve}^{\circ,\lozenge,]b[}$ is
  representable by an open rigid analytic subvariety of
  $\mathbf{Sh}_K\gx_{\ebreve}^\circ$, and hence is $\ell$-cohomologically smooth
  over $\spd \fpbar$ of dimension $1 + \langle 2\rho, \mu \rangle$. Therefore
  \[
    f_0^\ast (\text{dualizing sheaf of } C_K^{b,\diamond} \to \spd \fpbar)
  \]
  is invertible and concentrated in degree $-2 \langle 2\rho, \nu_b \rangle =
  -2d_b$. As $f_0$ is v-surjective (as $f$ is) we conclude that the dualizing
  sheaf of $C_K^{b,\diamond} \to \spd \fpbar$ is invertible and concentrated in
  degree $-2d_b$.

  We now transfer this information to the scheme side. Since $C_K^b \to \spec
  \fpbar$ is perfectly of finite type, there exists a dense open perfectly
  smooth locus $C_K^{b,\mathrm{sm}} \subseteq C_K^b$ by
  \cite[Lemma~056V]{stacks-project} for example. Choose an arbitrary connected
  component $C \subseteq C_K^{b,\mathrm{sm}}$ of dimension $d$, so that it
  suffices to prove $d = d_b$. Further choose a closed point $x \in C(\fpbar)$
  so that $C_x^\wedge \cong \spf \fpbar[[t_1^{1/p^\infty}, \dotsc,
  t_d^{1/p^\infty}]]$. Because $(C_x^\wedge)^\lozenge \subseteq
  C_K^{b,\diamond}$ is open, we see that $(C_x^\wedge)^\lozenge \to \spd \fpbar$
  also has dualizing sheaf invertible and concentrated in degree $-2d_b$. On the
  other hand, \cite[Theorem~24.1]{EtCohDiam} implies that its dualizing sheaf is
  invertible and concentrated in degree $-2d$. This proves $d = d_b$ as desired.
\end{proof}

%% file: weil-group-extra.tex
{
\def\rrarrows{\rightrightarrows}
\def\rrrarrows{\rightrightarrows\hspace{-1em}\to}
\def\rrrrarrows{\raisebox{0.07em}{$\rightrightarrows$}\hspace{-1em}\raisebox{-0.07em}{$\rightrightarrows$}}
\def\lbreve{\breve{L}}

\subsubsection{Condensed \texorpdfstring{$\infty$}{infinity}-categories}
The derived $\infty$-category $\mathcal{D}(\Lambda)$ can be promoted to a condensed
$\infty$-category in the sense of \cite[Section~IX.1]{FarguesScholze} by considering
the assignment sending a profinite set $S$ to
the derived $\infty$-category $\mathcal{D}(S, \Lambda)$ of sheaves of
$\Lambda$-modules on the topological space $S$. For every algebraically closed
perfectoid field $C$, this agrees with the condensed $\infty$-category
structure on $\mathcal{D}(\spd C, \Lambda)$ described in
\cite[Section~IX.1]{FarguesScholze}, because for every profinite set $S$ we have
an equivalence
\begin{equation} \label{Eq:DetOnTotallyDisconnected}
  \mathcal{D}(\spd C \times \underline{S}, \Lambda) \simeq \mathcal{D}((\spd C
  \times \underline{S})_\text{\'{e}t}, \Lambda) \simeq \mathcal{D}(S, \Lambda)
\end{equation}
as $\spd C \times \underline{S}$ is strictly totally disconnected, see
\cite[Definition~7.15, Definition~14.13]{EtCohDiam}.

\begin{Lem} \label{Lem:vDescentDet}
  Let $Y$ and $Y^\prime$ be small v-stacks, and let $f \colon Y^\prime \to Y$ be
  a $0$-truncated v-cover. Then pullback along $f$ induces an equivalence of
  $\infty$-categories
  \[
    \mathcal{D}(Y, \Lambda) \xrightarrow{\sim} \varprojlim(\mathcal{D}(Y^\prime,
    \Lambda) \rrarrows \mathcal{D}(Y^\prime \times_Y Y^\prime, \Lambda)
    \rrrarrows \mathcal{D}(Y^\prime \times_Y Y^\prime \times_Y Y^\prime,
    \Lambda) \rrrrarrows \dotsb),
  \]
  where the limit (in the sense of \cite[Section 3.3.3]{HigherTopos}) is over the \v{C}ech nerve for $f$.
\end{Lem}

\begin{proof}
  This is implicit in \cite[Remark~17.4]{EtCohDiam}. From
  \cite[Proposition~17.3]{EtCohDiam}, we have an analogous statement for
  $\mathcal{D}(Y_v, \Lambda)$, and then it remains to check that $A \in
  \mathcal{D}(Y_v, \Lambda)$ lies in $\mathcal{D}(Y, \Lambda)$ if and only if
  $f^\ast A \in \mathcal{D}(Y^\prime, \Lambda)$. This follows from
  \cite[Remark~14.14]{EtCohDiam}.
\end{proof}

For every condensed $\infty$-category $\mathcal{C}$ and a locally profinite
group $G$, we may consider the \v{C}ech nerve of $\ast \to BG$ and evaluate
on $\mathcal{C}$ to form the cosimplicial diagram of $\infty$-categories
\[
  \mathcal{C} \rrarrows \mathcal{C}(G) \rrrarrows \mathcal{C}(G^2) \rrrrarrows
  \dotsb,
\]
where to evaluate $\mathcal{C}(G^n)$ we first write $G^n$ as a disjoint union
$G^n = \coprod_i S_i$ of profinite sets and define $\mathcal{C}(G^n) =
\prod_i \mathcal{C}(S_i)$. We write $\mathcal{C}^{BG}$ for the limit of this
diagram of $\infty$-categories.\footnote{One may also construct $\mathcal{C}^{BG}$
by realizing $BG$ as a condensed $\infty$-category and considering the
$\infty$-category of condensed functors.}

\label{subsub:EquivariantObjectsLocallyProfinite}
For every locally profinite group $G$, together with an action of
$\underline{G}$ on $\spd C$, we construct an equivalence of $\infty$-categories
\begin{align}
  \mathcal{D}([\spd C / \underline{G}], \Lambda) &\simeq \varprojlim
  (\mathcal{D}(\spd C, \Lambda) \rrarrows \mathcal{D}(\spd C \times
  \underline{G}, \Lambda) \rrrarrows \mathcal{D}(\spd C \times \underline{G}^2, \Lambda)
  \rrrrarrows \dotsb) \\ &\simeq \varprojlim
  (\mathcal{D}(\Lambda) \rrarrows \mathcal{D}(G, \Lambda) \rrrarrows
  \mathcal{D}(G^2, \Lambda) \rrrrarrows \dotsb) = \mathcal{D}(\Lambda)^{BG}
\end{align}
using Lemma~\ref{Lem:vDescentDet} on the \v{C}ech nerve of $\spd C \to
[\spd C / \underline{G}]$ and \eqref{Eq:DetOnTotallyDisconnected}.
If $H$ is another locally profinite group with a
continuous group homomorphism $H \to G$, there is an induced morphism $f \colon
[\spd C / \underline{H}] \to [\spd C / \underline{G}]$. The associated pullback
functor
\[
  \mathcal{D}(\Lambda)^{BG} \simeq \mathcal{D}([\spd C / \underline{G}],
  \Lambda) \xrightarrow{f^\ast} \mathcal{D}([\spd C / \underline{H}], \Lambda)
  \simeq \mathcal{D}(\Lambda)^{BH}
\]
agrees with the functor restricting the $G$-action to an $H$-action.

\begin{Lem} \label{Lem:CanonicalDescentRep}
  Let $G$ be a locally profinite group with an action of $\underline{G}$ on
  $\spd C$. Then we have an isomorphism of stacks $[[\spd C / \underline{G}] /
  \phi_{[\spd C / \underline{G}]}^\mathbb{Z}] \simeq [\spd C / (\underline{G}
  \times \phi_C^\mathbb{Z})]$, and moreover the functor induced from the canonical
  descent datum (see Lemma~\ref{Lem:CanonicalPhiDescent})
  \[
    \mathcal{D}(\Lambda)^{BG} \simeq \mathcal{D}([\spd C / \underline{G}],
    \Lambda) \to \mathcal{D}([\spd C / (\underline{G} \times
    \phi_C^\mathbb{Z})], \Lambda) \simeq \mathcal{D}(\Lambda)^{B(G \times
    \mathbb{Z})}
  \]
  agrees with restriction along the projection map $G \times \mathbb{Z} \to G$.
\end{Lem}

\begin{proof}
  The first part follows immediately from the functoriality of $\phi$ together with the
  fact that $\phi$ acts by identity on $\underline{G}$. For the second part, we
  note that because the canonical descent is functorial for pullbacks, see
  Lemma~\ref{Lem:CanonicalPhiDescentPullback}, the descent functor is
  induced by taking the limit of both rows in the diagram
  \[ \begin{tikzcd}
    \mathcal{D}(\spd C, \Lambda) \arrow[shift left]{r} \arrow[shift right]{r}
    \arrow{d}{d_0} & \mathcal{D}(\spd C \times \underline{G}, \Lambda)
    \arrow[shift left]{r} \arrow[shift right]{r} \arrow{r} \arrow{d}{d_1} &
    \cdots \\ \mathcal{D}([\spd C / \phi^\mathbb{Z}], \Lambda) \arrow[shift
    left]{r} \arrow[shift right]{r} & \mathcal{D}([(\spd C \times \underline{G})
    / \phi^\mathbb{Z}], \Lambda) \arrow[shift left]{r} \arrow[shift right]{r}
    \arrow{r} & \cdots,
  \end{tikzcd} \]
  where the horizontal functors are pullbacks and the vertical functors
  $d_\bullet$ are canonical descents, see Lemma~\ref{Lem:CanonicalPhiDescent}.
  On the other hand, the quotient map $\spd C \to [\spd C /
  \phi^\mathbb{Z}]$ induces a morphism of simplicial v-sheaves $f_\bullet$ from
  the \v{C}ech nerve of $\spd C \to [\spd C / (\underline{G} \times
  \phi^\mathbb{Z})]$ to the \v{C}ech nerve of $[\spd C / \phi^\mathbb{Z}] \to
  [\spd C / (\underline{G} \times \phi^\mathbb{Z})]$. Pulling back along these maps induces another diagram
  \[ \begin{tikzcd}
    \mathcal{D}([\spd C / \phi^\mathbb{Z}], \Lambda) \arrow[shift left]{r}
    \arrow[shift right]{r} \arrow{d}{f_0^\ast} & \mathcal{D}([(\spd C \times
    \underline{G}) / \phi^\mathbb{Z}], \Lambda) \arrow[shift left]{r}
    \arrow[shift right]{r} \arrow{r} \arrow{d}{f_1^\ast} & \cdots \\
    \mathcal{D}(\spd C, \Lambda) \arrow[shift left]{r} \arrow[shift right]{r} &
    \mathcal{D}(\spd C \times \underline{G} \times \underline{\mathbb{Z}},
    \Lambda) \arrow[shift left]{r} \arrow[shift right]{r} \arrow{r} & \cdots,
  \end{tikzcd} \]
  where taking the limit of both rows induces the identity functor on
  $\mathcal{D}([\spd C / (\underline{G} \times \phi^\mathbb{Z})], \Lambda)$.

  It now suffices to check that the composite functor
  \[
    \mathcal{D}(G^n, \Lambda) \simeq \mathcal{D}(\spd C \times \underline{G}^n,
    \Lambda) \xrightarrow{f_n^\ast \circ d_n} \mathcal{D}(\spd C \times
    \underline{G}^n \times \underline{\mathbb{Z}}^n, \Lambda) \simeq
    \mathcal{D}(G^n \times \mathbb{Z}^n, \Lambda)
  \]
  agrees with pulling back by the continuous projection map $G^n \times
  \mathbb{Z}^n \to G^n$. For this we note that $f_n$ can be written as the
  composition
  \[
    \spd C \times \underline{G}^n \times \underline{\mathbb{Z}}^n
    \xrightarrow{p_n} \spd C \times \underline{G}^n \xrightarrow{q_n} [(\spd C
    \times \underline{G}^n) / \phi^\mathbb{Z}],
  \]
  where $p_n$ is the projection. Now $q_n^\ast \circ d_n$ is the identity
  functor by construction, and hence $f_n^\ast \circ d_n = p_n^\ast$ agrees with
  pullback along $G^n \times \mathbb{Z}^n \to G^n$ as desired.
\end{proof}

\subsubsection{Twisted away the Frobenius}
Let $L$ be a finite extension of $\qp$ with residue field $\mathbb{F}_q$ and write $q=p^r$.
Later we will take $L$ to be $E$, so this slight conflict of notation shall not
become confusing. We let $C$ be the completion of an algebraic closure of
$L$, let $\lbreve$ be the completion of the maximal unramified subextension
of $L$ in $C$, and let $k$ be the residue field of
$\lbreve$. Let $\sigma$ denote the Frobenius on $\lbreve =
W_{\CO_L}(k)[\tfrac{1}{p}]$ obtained as the Teichm\"uller lift of the
$q$-Frobenius on $k$.

\begin{Lem}\label{Lem:lbreve}
    There is a canonical isomorphism
    \[\spd \lbreve \xrightarrow{\sim} \spd L\times_{\spd \mathbb{F}_q} \spd k\]
    under which $\sigma$ is identified with $\mathrm{id}\times \phi^r$.
\end{Lem}

\begin{proof}
    Let $S=\spa(R,R^+)\in \perf_{\mathbb{F}_q}$ be an affinoid perfectoid test object. A map $S\to \spd \lbreve$ is determined by an untilt of $S$ over $\spa \lbreve$, i.e., a pair $(S^\sharp\xrightarrow{\alpha} \spa \lbreve, \iota: S^{\sharp\flat}\simeq S)$. Given such a pair, the map $\alpha^\ast: \CO_{\lbreve}\to R^{\sharp+}$ defines a unique $k$-algebra structure on the perfect $\mathbb{F}_q$-algebra $R^+$ by adjunction between $W_{\CO_L}(-)$ and $(-)^\flat$, giving therefore a map $\spd \lbreve \to 
    \spd k$. This together with the natural projection $\spd \lbreve \to \spd L$ gives the desired map. To construct an inverse to it, assume we are given an $S$-point of the fiber product, i.e., a pair $(S^\sharp\xrightarrow{\beta} \spa L, \iota: S^{\sharp\flat}\simeq S)$, a map $k\xrightarrow{\gamma} R^+$ such that there is a commutative diagram
    \[\begin{tikzcd}
        W(k) \ar[r,"W(\gamma)"] & W(R^+)\ar[r, two heads] & R^{\sharp +}\\
        W(\mathbb{F}_q) \ar[u] \ar[r] & \CO_L \ar[ur, "\beta"].&
    \end{tikzcd}\]
    This commutative diagram leads to a map $W_{\CO_L}(k)\to R^{\sharp+}$, which gives rise to an $S$-point of $\spd \lbreve$. One verifies readily that the two constructions are inverse to each other. Since $W(\mathrm{Frob}_q)=\sigma$, the second claim is clear.
\end{proof}

Let $\gal_L$ be the absolute Galois group of $L$. We consider the action of
$\underline{\gal_L} \times \underline{\mathbb{Z}}$ on $\spd C$, where $(\tau,
n)$ acts by $\tau \circ \phi_C^n$. We can identify
\[
  [\spd C / (\underline{\gal_L} \times \{0\})] \simeq \spd L, \quad [\spd C /
  (\underline{\gal_L} \times \underline{\mathbb{Z}})] \simeq [\spd L /
  \phi_L^\mathbb{Z}].
\]
Denote by $W_L \subset \gal_L$ the Weil group of $L$. There is a twisted embedding
\[
  \iota \colon W_L \hookrightarrow \gal_L \times \mathbb{Z}; \quad \tau \mapsto
  (\tau, -r \deg\tau),
\]
where $\deg$ is the projection $W_L \to (\phi^r)^\mathbb{Z} \simeq \mathbb{Z}$, cf.\ \cite[Section~IV.7]{FarguesScholze}. Then the natural map $\spd C \to \spd k$ is $\underline{\iota(W_L)}$-equivariant when we
give $\spd k$ the trivial action.

Let $\pi \colon X \to \spd L$ be a qcqs morphism of v-stacks. Taking cohomology
defines a sheaf
\[
  R\pi_\ast \Lambda \in \mathcal{D}(\spd L, \Lambda) = \mathcal{D}([\spd C /
  (\underline{\gal_L} \times \{0\})], \Lambda) \simeq \mathcal{D}(\Lambda)^{B\gal_L}.
\]
On the other hand, we may quotient by the absolute Frobenius action on both
sides to define $\pi_\phi \colon [X/\phi_X^{r\mathbb{Z}}] \to [\spd L /
\phi_L^{r\mathbb{Z}}]$, and then base change along $\mathbb{F}_q \to k$ to
define
\[
  \pi_{\phi,k} \colon [X / \phi^{r\mathbb{Z}}] \times_{\spd \mathbb{F}_q} \spd k
  \to [\spd L / \phi^{r\mathbb{Z}}] \times_{\spd \mathbb{F}_q} \spd k.
\]
Using Lemma~\ref{Lem:lbreve}, we may identify the target with
\[
  [(\spd L \times_{\spd \mathbb{F}_q} \spd k) / (\phi^{r\mathbb{Z}} \times
  \mathrm{id})] = [\spd C / \underline{\iota(W_L)}].
\]
Therefore we obtain
\[
  R\pi_{\phi,k,\ast} \Lambda \in \mathcal{D}([\spd C / \underline{\iota(W_L)}], \Lambda)
  \simeq \mathcal{D}(\Lambda)^{BW_L}.
\]

\begin{Prop} \label{Prop:GaloisRestriction}
  Let $\pi \colon X \to \spd L$ be as above. Under the restriction functor
  \[
    \mathcal{D}([\spd C / (\underline{\gal_L} \times \{0\})], \Lambda) \simeq
    \mathcal{D}(\Lambda)^{B\gal_L} \to \mathcal{D}(\Lambda)^{BW_L} \simeq
    \mathcal{D}([\spd C / \underline{\iota(W_L)}], \Lambda)
  \]
  along the inclusion $W_L \hookrightarrow \gal_L$, the object $R\pi_\ast \Lambda$ is sent to $R\pi_{\phi,k,\ast} \Lambda$.
\end{Prop}

\begin{proof}
  By Lemma~\ref{Lem:DescentAndPushforward}, the object
  \[
    R\pi_{\phi,\ast} \Lambda \in \mathcal{D}([\spd L / \phi_L^{r\mathbb{Z}}],
    \Lambda) \simeq \mathcal{D}(\Lambda)^{B(\gal_L \times r\mathbb{Z})}
  \]
  is the canonical descent of $R\pi_\ast \Lambda \in
  \mathcal{D}(\Lambda)^{B\gal_L}$. This implies by
  Lemma~\ref{Lem:CanonicalDescentRep} that $R\pi_{\phi,\ast} \Lambda$ is
  obtained from $R\pi_\ast \Lambda$ by pulling back the action along the
  projection $\gal_L \times r\mathbb{Z} \to \gal_L$.

  On the other hand, qcqs base change for the diagram
  \[ \begin{tikzcd}[column sep = small]
    \lbrack X / \phi^{r\mathbb{Z}} \rbrack \times_{\spd \mathbb{F}_q} \spd k
    \arrow{r} \arrow{d}{\pi_{\phi,k}} & \lbrack X / \phi^{r\mathbb{Z}} \rbrack
    \arrow{d}{\pi_\phi} \\ \lbrack \spd L / \phi^{r\mathbb{Z}} \rbrack
    \times_{\spd \mathbb{F}_q} \spd k = \lbrack \spd C / \underline{\iota(W_L)} \rbrack
    \arrow{r}{g} & \lbrack \spd L / \phi^{r\mathbb{Z}} \rbrack = \lbrack \spd C
    / (\underline{\gal_L} \times \underline{r\mathbb{Z}}) \rbrack
  \end{tikzcd} \]
  implies that $g^\ast R\pi_{\phi,\ast} \Lambda \simeq R\pi_{\phi,k,\ast}
  \Lambda$, where $g^\ast \colon \mathcal{D}(\Lambda)^{B(\gal_L \times r\mathbb{Z})} \to
  \mathcal{D}(\Lambda)^{BW_L}$ is given by restriction along $\iota \colon W_L
  \hookrightarrow \gal_L \times \mathbb{Z}$. Putting the two together, we
  conclude that $R\pi_\ast \Lambda \in \mathcal{D}(\Lambda)^{B\gal_L}$ is sent
  to $R\pi_{\phi,k,\ast} \Lambda$ in $\mathcal{D}(\Lambda)^{BW_L}$ under
  restriction along the composition
  \[
    W_L \xrightarrow{\iota} \gal_L \times r\mathbb{Z} \xrightarrow{\mathrm{id}
    \times 0} \gal_L,
  \]
  which is the natural inclusion.
\end{proof}

}

%% file: Eichler-Shimura.tex
{
\def\Hck{\mathrm{Hck}}
\def\zspec{\mathcal{Z}^{\mathrm{spec}}}

In this section, we study the complex $\mathcal{F}\coloneqq R\overline{\pi}_{\mathrm{HT},\ast}\Lambda$ using the spectral action of \cite{FarguesScholze}. This gives refined information about the cohomology of Shimura varieties as representations of the Weil group. In particular, we prove Theorem~\ref{Thm:IntroEichlerShimura} and Theorem \ref{Thm:IntroCompatibilityFarguesScholze}.

\subsection{Preparations}

\label{subsub:NotationSection9} As before, we fix a Hodge-type Shimura datum $(\mathsf{G},\mathsf{X})$, with Hodge cocharacter $\mu$, and reflex field $\mathsf{E}$; we let $E$ be the completion of $\mathsf{E}$ at a place $v \mid p$ and $k_E=\mathbb{F}_q$ be the residue field of $E$. For a fixed compact open subgroup $K^p\subset \mathsf{G}(\afp)$ and varying $K_p\subset G(\qp)$, we set $K=K_pK^p$ and write $\mathbf{Sh}_{K}\coloneqq\mathbf{Sh}_K\gx$ for the base change of the corresponding Shimura variety to $E$.

\label{subsub:RMu} Let $\dualgrp{G}$ be the dual group of $G$ over $\zl[\sqrt{p}]$ equipped with its action of the absolute Galois group of $\qp$ and therefore an action of the Weil group $W_\qp$. One can then define the $L$-group ${}^LG$ as the semi-direct product $\dualgrp{G}\rtimes W_\qp$, see \cite[Section~2.1]{BuzzardGee}. Under the geometric Satake equivalence of Fargues--Scholze \cite[Theorem VI.0.2]{FarguesScholze}, the sheaf $\mathcal{S}_{\mu}$ of Section~\ref{subsub:HeckeOperators} corresponds to a representation $r_\mu$ of ${}^LG_E\coloneqq \dualgrp{G}\rtimes W_E$ over $\zl[\sqrt{p}]$, where we use the fixed $\sqrt{p} \in \zl[\sqrt{p}]$ to trivialize the cyclotomic twist. The following proposition uniquely characterizes $r_{\mu}$. It is well known to experts; we have included a proof for the sake of completeness.
\begin{Prop} \label{Prop:Computationrmu}
    The restriction of the representation $r_{\mu}$ to $\dualgrp{G} \subset {}^LG_E$ is given by the representation $V_{\mu}$ of highest weight $\mu$. Moreover, the Weil group $W_E$ acts trivially on the highest weight vector in $V_{\mu}$.
\end{Prop}
\begin{proof}
We first recall the equivalence $\mathrm{Sat}(\mathcal{H}ck_{G,\Div},\zl[\sqrt{p}]) \xrightarrow{\sim} \operatorname{Rep}_{{}^LG_E}(\zl[\sqrt{p}])$ of \cite[Theorem VI.0.2]{FarguesScholze}. There is an exact symmetric monoidal functor
\begin{align}
    F: \mathrm{Sat}(\mathcal{H}ck_{G,\Div},\zl[\sqrt{p}]) &\to \operatorname{Rep}_{W_E}(\zl[\sqrt{p}]), \\ A &\mapsto
    \bigoplus_i \mathcal{H}^i(R \pi_{G, \Div} A),
\end{align}
defined in \cite[Definition/Proposition VI.7.10, Proposition VI.10.1]{FarguesScholze}.
Using $F$ and \cite[Proposition VI.10.2]{FarguesScholze}, Fargues--Scholze define a Tannaka group scheme $\widecheck{G}$ over $\zl[\sqrt{p}]$ equipped with an action of $W_E$. They then prove, see \cite[Theorem VI.11.1]{FarguesScholze}, that $\widecheck{G}$ is $W_E$-equivariantly isomorphic to $\dualgrp{G}$, where $\dualgrp{G}$ is equipped with a certain Tate-twisted version of the usual $W_E$ action. \smallskip

Using $\sqrt{p} \in \zl[\sqrt{p}]$, we can consider the exact symmetric monoidal functor
\begin{align}
    F': \mathrm{Sat}(\mathcal{H}ck_{G,\Div},\zl[\sqrt{p}]) &\to \operatorname{Rep}_{W_E}(\zl[\sqrt{p}]), \\ A &\mapsto
    \bigoplus_i \mathcal{H}^i(R \pi_{G, \Div} A)(\tfrac{i}{2}).
\end{align}
Using $F'$ and \cite[Proposition VI.10.2]{FarguesScholze}, we get a Tannaka group scheme $\widecheck{\widecheck{G}}$ over $\zl[\sqrt{p}]$ equipped with an action of $W_E$. It follows from the proof of \cite[Theorem VI.11.1]{FarguesScholze} that there is a natural $W_E$-equivariant isomorphism $\widecheck{\widecheck{G}} \xrightarrow{\sim} \dualgrp{G}$, where $\dualgrp{G}$ is equipped with the natural (untwisted) $W_E$ action. The only difference with the proof of \cite[Theorem VI.11.1]{FarguesScholze} is the computation that happens for $G=\operatorname{PGL}_{2}$. If $A \in \mathrm{Sat}(\mathcal{H}ck_{G,\Div},\zl[\sqrt{p}])$ is the intersection cohomology sheaf of the minuscule Schubert cell for $\mathbb{P}^1$ (denoted by $A_{\mu}$) in loc.\ cit., then $F(A)=\zl \oplus \zl(-1)$ and it is proved in loc.\ cit.\ that $\widecheck{G} \xrightarrow{\sim} \operatorname{SL}(\zl \oplus \zl(-1))$. Using $F'$, we find that $F'(A)=\zl(-\tfrac{1}{2}) \oplus \zl(-\tfrac{1}{2})$ so that $\widecheck{\widecheck{G}} \xrightarrow{\sim} \operatorname{SL}(\zl(-\tfrac{1}{2}) \oplus \zl(-\tfrac{1}{2})) \xrightarrow{\sim}\operatorname{SL}(\zl \oplus \zl)$. \smallskip

By \cite[Proposition VI.10.2]{FarguesScholze}, the category $\mathrm{Sat}(\mathcal{H}ck_{G,\Div},\zl[\sqrt{p}])$ is isomorphic to the category of representations of $\widecheck{\widecheck{G}} \xrightarrow{\sim} \dualgrp{G}$ in $\operatorname{Rep}_{W_E}(\zl[\sqrt{p}])$. Unwinding the definitions, this is equivalent to the category of representations of ${}^LG_E$ in finite projective $\zl[\sqrt{p}]$-modules.\footnote{The Satake category is also isomorphic to the category of representations of $\widecheck{G} \rtimes W_E$ on finite projective $\Lambda$-modules. The induced isomorphism $\widecheck{G} \rtimes W_E \xrightarrow{\sim} \widecheck{\widecheck{G}} \rtimes W_E$ is given by $(g,w) \mapsto (g \cdot c(w)^{-1}, w)$, where $c(w) = 2\widecheck{\rho}(\chi_{\mathrm{cyc}}(w)^{\tfrac{1}{2}})$ with $\chi_{\mathrm{cyc}}$ the cyclotomic character and $2\widecheck{\rho}$ the sum of the positive coroots of $\widecheck{G}$.} We now identify the representation of ${}^LG_E$ corresponding to $\mathcal{S}_{\mu}$. Since $\mu$ is minuscule, it follows that the representation of $\dualgrp{G} \subset {}^LG_E$ is given by the representation with highest weight $\mu$ (see \cite[top of p. 237]{FarguesScholze}). It remains to compute the $W_E$-action on the highest weight vector. Since the Borel in $\widecheck{\widecheck{G}}$ is induced by the cohomological grading (see \cite[page 236]{FarguesScholze}), the highest weight vector is simply the line given by $\mathcal{H}^{d}(R \pi_{G, \Div} \mathcal{S}_{\mu})(\tfrac{d}{2}) = H^{2d}(\operatorname{Gr}_{G,\mu,\overline{E}}, \zl[\sqrt{p}])(d)$, where $\operatorname{Gr}_{G,\mu}$ is the Schubert variety corresponding to $\mu$. Now $W_E$ acts trivially on $H^{2d}(\operatorname{Gr}_{G,\mu,\overline{E}}, \zl[\sqrt{p}])(d)$ since $\operatorname{Gr}_{G,\mu}$ is (geometrically) connected.
\end{proof}

\subsection{Spectral action} We denote by $\loc$ the stack of $L$-parameters over $\zl[\sqrt{p}]$ as in \cite{DHKMModuli}, \cite{ZhuCoherent}, and \cite{FarguesScholze}. This is the stack quotient of the moduli space $\cocycle_{\zl[\sqrt{p}]}$ of condensed 1-cocycles (in $\zl[\sqrt{p}]$-algebras) by the conjugation action of $\dualgrp{G}$. Let $\mathrm{Perf}(\loc)$ be the $\infty$-category of perfect complexes on $\loc$, see \cite[Section~VIII.5]{FarguesScholze}. We have the following main theorem of loc.\ cit., which is the combination of \cite[Theorem X.0.2]{FarguesScholze} and \cite[Theorem IX.0.1]{FarguesScholze}. We will base change $\loc$ to $\Lambda$ along $\zl[\sqrt{p}] \to \Lambda$ using our fixed $\sqrt{p} \in \Lambda$, without changing the notation. We will use the same mild abuse of notation for other objects defined over $\zl[\sqrt{p}]$.

\begin{Thm}[{\cite[Theorem~X.0.2, Theorem~V.4.1]{FarguesScholze}}] \label{Thm:FSSpectralAction}
    If $\ell$ is coprime to the order of $\pi_0(Z(G))$, then there exists a natural $\Lambda$-linear action of $\mathrm{Perf}(\loc)$ on $
    \mathcal{D}(\bungk,\Lambda)$, preserving the compact objects.
\end{Thm}

In particular, for every object $W$ of $\mathrm{Perf}(\loc)$ there is a $\Lambda$-linear functor $T_W:\mathcal{D}(\bungk,\Lambda) \to \mathcal{D}(\bungk,\Lambda)$ preserving the compact objects, which we write as $A \mapsto W \ast A$. For every morphism $W \to W'$ in $\mathrm{Perf}(\loc)$ there is a natural transformation of functors $T_W \to T_W'$ and there are specified coherence data for compositions and commutative diagrams of morphisms. The following is our primary example.

\begin{Example}\label{Eg:SpectralAction}
    Let $r:\dualgrp{G}\to \mathrm{GL}(V)$ be an algebraic representation of $\dualgrp{G}$. The trivial bundle $\CO_{\cocycle}\otimes_\Lambda V$ on $\cocycle$, equipped with the descent datum given by $r$, descends to a vector bundle $\mathcal{V}$ on $\loc$. Suppose $r$ extends to a representation
    \[\tilde{r}: {}^LG_E\coloneqq\dualgrp{G}\rtimes W_E\to \GL(V)\]
    for some finite extension $E/\qp$. Then $\mathcal{V}$ acquires a $W_E$-action descending that on $\CO_{\cocycle}\otimes_\Lambda V$ through $\tilde{r} \circ \phi^\mathrm{univ}\mid_{W_E}$, where $\phi^\mathrm{univ}$ denotes the universal 1-cocycle. By construction, the spectral action of $\mathcal{V}$ on $\mathcal{D}(\bungk,\Lambda)$ is naturally (in $V$) identified with the Hecke operator $T_V$ defined in \cite[First page of Chapter IX, Theorem IX.0.1]{FarguesScholze}, i.e., $\mathcal{V} \ast (-) = T_V$. When $V=V_{\mu}$, we recover the Hecke operator $T_{\mu}$ of \eqref{Eq:HeckeOperatorDef}.
\end{Example}

\begin{Rem}\label{Rem:SpectralWeilAction}
For any $A \in \mathcal{D}(\bungk,\Lambda)$, the object $T_V A$ inherits an action of $\mathrm{End}_{\loc}(\mathcal{V})$, and in particular an action of the Weil group $W_E$. We will refer to this $W_E$-action on $T_V A$ as the \textit{spectral $W_E$-action}. Recall that, by definition, the Hecke operators naturally carry Weil group actions, coming from the structure map of the Hecke stack to $\Div$, see \eqref{Eq:HeckeOperatorDef} or more generally \cite[Corollary~IX.2.3]{FarguesScholze}. These two Weil group actions are compatible, as can be seen by unraveling the definitions.
\end{Rem}

We apply the above example to the representation $r_\mu$ and write $\mathcal{V}_\mu$ for the corresponding vector bundle on $X_{\dualgrp{G}}$. This allows us to rephrase Theorem~\ref{Thm: WeilCohoShiVar} in terms of the spectral action. As before, we write $\mathcal{F}$ for $R\overline{\pi}_{\mathrm{HT},\ast}\Lambda$ and $d$ for the dimension of the Shimura variety. To state it, we need to introduce some notation.

Following \cite[Definition IX.0.2]{FarguesScholze}, we let $\zspec(G,\Lambda)$ be the spectral Bernstein center, which is the ring $\Gamma(\loc, \mathcal{O})$ of global functions on $\loc$, and let $\mathcal{Z}(G(\qp), \Lambda)$ be the Bernstein center of the (abelian) category of smooth $G(\qp)$-representations on $\Lambda$-modules. If $\ell$ is coprime to the order of $\pi_0(Z(G))$, then by \cite[Corollary IX.0.3]{FarguesScholze}, there is a morphism
    $\Psi_G:\zspec(G,\Lambda)\to \mathcal{Z}(G(\qp), \Lambda)$.
If $\chi$ is a character of $\mathcal{Z}(G(\qp), \Lambda)$ valued in $\flbar$ or $\qlbar$, the Fargues--Scholze $L$-parameter $\phi_{\chi}$ is the semisimple $L$-parameter corresponding under \cite[Proposition VIII.3.2]{FarguesScholze} to the morphism $\chi \circ \Psi_G$. Let us write $\pi_0\operatorname{End}_{\loc}(\mathcal{V}_{\mu})$ for the endomorphism ring of $\mathcal{V}_{\mu}$ considered as a vector bundle on $\loc$ (as opposed to a perfect complex).
\begin{Thm} \label{Thm:SpectralAction}
If $\ell$ is coprime to the order of $\pi_0(Z(G))$, then one has
\[R\Gamma(\mathbf{Sh}_{K^p,\overline{E}}, \Lambda) \simeq i_1^\ast \left( \mathcal{V}_\mu\ast (j_{\mu,k,!} \mathcal{F}[-d])(-\tfrac{d}{2})\right)\]
in $\mathcal{D}(G(\qp),\Lambda)^{BW_E}\simeq \mathcal{D}(\bungk^{[1]},\Lambda)^{BW_E}$. In particular, for all compact open subgroups $K \subset \gaf$, there are maps of $\Lambda$-algebras 
\[\pi_0\operatorname{End}_{\loc}(\mathcal{V}_{\mu})\to \pi_0\mathrm{End}_{D(G(\qp),\Lambda)}(R\Gamma(\mathbf{Sh}_{K^p,\overline{E}}, \Lambda))\to \pi_0\mathrm{End}_{D(\Lambda)}(R\Gamma(\mathbf{Sh}_{K,\overline{E}}, \Lambda)).\]
Moreover, the spectral and the usual $W_E$-action on $R\Gamma(\mathbf{Sh}_{K,\overline{E}}, \Lambda)(\tfrac{d}{2})$ agree. Furthermore, the action of $\zspec(G, \Lambda) \subset \pi_0\operatorname{End}_{\loc}(\mathcal{V}_{\mu})$ on $R\Gamma(\mathbf{Sh}_{K^p,\overline{E}}, \Lambda)$ factors through the natural action of $\mathcal{Z}(G(\qp), \Lambda)$ via $\Psi_G$.
\end{Thm}

\begin{proof}
The first part and the compatibility of spectral and usual Weil group actions follow from Example \ref{Eg:SpectralAction} and Theorem \ref{Thm: WeilCohoShiVar}, taking into account Definition \ref{Def:HeckeFactorization} and Section \ref{subsub:HeckeFactorization}. The first map in the penultimate line exists by Remark~\ref{Rem:SpectralWeilAction} and the second map is induced by the functor
\[R\Gamma(K_p,-)=R\Hom(c\text{-Ind}_{K_p}^{G(\qp)}\Lambda,-): D(G(\qp),\Lambda)\to D(\Lambda). \qedhere \]
\end{proof}

By passing to the inverse limit, we get a similar statement for integral coefficients. More precisely, we have the following corollary:

 \begin{Cor} \label{Cor:EllAdicSpectralAction}
     Let $\ell$ be coprime to the order of $\pi_0(Z(G))$, and $\Lambda$ be the ring of integers in a finite extension of $\ql$, containing a fixed square root of $p$. Then $\pi_0\operatorname{End}_{\loc}(\mathcal{V}_{\mu})$ acts on $R\Gamma(\mathbf{Sh}_{K,\overline{E}}, \Lambda)$ for all compact open subgroups $K \subset \gaf$. Moreover, the spectral and the usual $W_E$-action on $R\Gamma(\mathbf{Sh}_{K,\overline{E}}, \Lambda)(\tfrac{d}{2})$ agree. Furthermore, the action of $\zspec(G, \Lambda) \subset \pi_0\operatorname{End}_{\loc}(\mathcal{V}_{\mu})$ on $R\Gamma(\mathbf{Sh}_{K,\overline{E}}, \Lambda)$ factors through the natural action of $\mathcal{Z}(G(\qp), \Lambda)$ via $\Psi_G$.
 \end{Cor}

\begin{proof}
Applying Theorem \ref{Thm:SpectralAction} to $\Lambda/\ell^n \Lambda$ for all $n$ and taking the inverse limit gives an action of $\varprojlim_n \pi_0\operatorname{End}_{X_{\dualgrp{G}, \Lambda/\ell^n \Lambda}}(\mathcal{V}_{\mu})$ on $R\Gamma(\mathbf{Sh}_{K, \overline{E}}, \Lambda) = \varprojlim_n R\Gamma(\mathbf{Sh}_{K, \overline{E}}, \Lambda/\ell^n \Lambda)$. This receives a natural map from $\pi_0\operatorname{End}_{\loc}(\mathcal{V}_{\mu})$, giving us the desired action. The rest of the corollary now follows from Theorem \ref{Thm:SpectralAction}.
\end{proof}

\subsection{Proof of Theorem \ref{Thm:IntroCompatibilityFarguesScholze}} \label{Sub:SpectralActionII}
We use the same notation as in Section~\ref{subsub:NotationSection9}, except that in the rest of this section we take $\Lambda$ to be \textit{the ring of integers in a finite extension of $\ql$, containing a fixed square root of $p$}. We will now prove Theorem \ref{Thm:IntroCompatibilityFarguesScholze}, whose statement and proof are inspired by \cite[Theorem 1.3]{Koshikawa}. Let us first recall its content: For a compact open subgroup $K_p \subset G(\qp)$, we write $\mathcal{H}_{K_p}\coloneqq \Lambda[G(\qp)\sslash K_p]$ for the Hecke algebra of level $K_p$ with coefficients in $\Lambda$ and $\mathcal{Z}_{K_p}$ for its center. Recall moreover that $\mathcal{H}_{K_p}$ has a natural (right) action on $R\Gamma(\mathbf{Sh}_{K, \overline{E}}, \Lambda)$. Let $L \in \{\flbar, \qlbar\}$, let $\chi: \mathcal{Z}_{K_p} \to L$ be a character, and consider the $W_E$-representation $W^i(\chi):=H^i(\mathbf{Sh}_{K, \overline{E}}, \Lambda)(\tfrac{d}{2}) \otimes_{\mathcal{Z}_{K_p}, \chi} L$ for some $i$.
\begin{Thm}[Theorem~\ref{Thm:IntroCompatibilityFarguesScholze}] \label{Thm:CompatibilityFarguesScholze}
If $\ell$ is coprime to the order of $\pi_0(Z(G))$, then each irreducible $L$-linear $W_E$ representation occurring as a subquotient of $W^i(\chi)$ also occurs as a subquotient of $r_{\mu} \circ {\phi_{\chi}}\mid_{W_E}$.
\end{Thm}

\begin{Lem} \label{Lem:KoshikawaJacobsonDensity}
    Let $L \in \{\flbar, \qlbar\}$ and let $\sigma_1, \cdots, \sigma_n$ be finite dimensional irreducible representations of $W_E$ over $L$ which are pairwise non-isomorphic. Then the action map
    \[L[W_E] \to \prod_{i=1}^n \operatorname{End}_{L} \sigma_i\]
    is surjective.
\end{Lem}

\begin{proof}
This follows from Jacobson's density theorem, \cite[Theorem 3.2.2.(ii)]{RepresentationTheoryBook}.
\end{proof}

\begin{proof}[Proof of Theorem \ref{Thm:CompatibilityFarguesScholze}]
    We follow \cite[Section 5]{Koshikawa}. Let $\sigma_1, \cdots, \sigma_{n-1}$ be the pairwise non-isomorphic irreducible $L$-linear $W_E$-subquotients of $r_{\mu} \circ \restr{\phi_{\chi}}{W_E}$. Suppose that there is an irreducible $L$-linear $W_E$-representation $\sigma_n$ occurring in $W^i(\chi)$ which is not isomorphic to any of the representations $\sigma_1, \cdots, \sigma_{n-1}$. Then by Lemma \ref{Lem:KoshikawaJacobsonDensity} there is an element $e \in L[W_E]$ which acts by zero on $\sigma_1, \cdots, \sigma_{n-1}$ and as the identity on $\sigma_n$. By assumption, every power of the element $e$ acts nontrivially on $W^i(\chi)$, and from this we will derive a contradiction.

    By construction and Corollary \ref{Cor:EllAdicSpectralAction}, we have the following commutative diagram
    \begin{equation}
        \begin{tikzcd}
            \zspec(G, \Lambda) \arrow{d}{\Psi_{G}} \arrow{r} &\pi_0\operatorname{End}_{\loc}(\mathcal{V}_{\mu}) \arrow{d} \\
            \mathcal{Z}(G(\qp), \Lambda) \arrow{r} \arrow{d} & \pi_0 \operatorname{End}_{D(\Lambda)}\left( R\Gamma(\mathbf{Sh}_{K^p,\overline{E}}, \Lambda)\right) \arrow{d} \\
            \mathcal{Z}_{K_p} \arrow{d}{\chi} \arrow{r} & \pi_0 \operatorname{End}_{D(\Lambda)}\left(R\Gamma(\mathbf{Sh}_{K,\overline{E}}, \Lambda) \right) \arrow{d} \\
            L \arrow{r} & \operatorname{End}_L(W^i(\chi)).
        \end{tikzcd}
    \end{equation}
By the assumption that the order of $\pi_0(Z(G))$ is coprime to $\ell$, the formation of $\zspec(G, -)$ commutes with base change, see \cite[Theorem VIII.0.2]{FarguesScholze}. Therefore there is an induced commutative diagram
\begin{equation}
    \begin{tikzcd}
        \zspec(G, L) \arrow{r} \arrow{d} & \pi_0\operatorname{End}_{X_{\dualgrp{G}, L}}(\mathcal{V}_{\mu,L}) \arrow{d} \\
        L \arrow{r} & \operatorname{End}_L(W^i(\chi)).
    \end{tikzcd}
\end{equation}
Let us write $X_{\dualgrp{G}, \chi}$ for the fiber over the closed point corresponding to $\chi$ under $X_{\dualgrp{G},L} \to \spec \zspec(G,L)$, and $\mathcal{V}_{\mu, \chi}$ for the restriction of $\mathcal{V}_{\mu,L}$ to that closed substack. Then there is moreover a commutative diagram
\begin{equation}
    \begin{tikzcd}
        \zspec(G, L) \arrow{r} \arrow{d} & \pi_0\operatorname{End}_{X_{\dualgrp{G}, \chi}}(\mathcal{V}_{\mu,\chi}) \arrow{d} \\
        L \arrow{r} & \operatorname{End}_L(W^i(\chi)).
    \end{tikzcd}
\end{equation}
Note that $W_E$ acts on $\mathcal{V}_{\mu, \chi}$ and $W^i(\chi)$-compatibly, see Corollary \ref{Cor:EllAdicSpectralAction}. The element $e$ defines an endomorphism of $\mathcal{V}_{\mu,\chi}$ (just take the $L$-linear extension of the action of $W_E$). We want to show that this endomorphism is nilpotent, which would imply that the endomorphism $e$ of $W^i(\chi)$ is nilpotent, thus leading to a contradiction.

Let us write $Z^1(W_{\qp}, \dualgrp{G})_{\chi}$ for the inverse image of $X_{\dualgrp{G}, \chi}$ in $Z^1(W_{\qp}, \dualgrp{G})_L$ under the natural map. It suffices to show that the endomorphism defined by $e$ on the pullback of $\mathcal{V}_{\mu,\chi}$ to $Z^1(W_{\qp}, \dualgrp{G})_{\chi}$ is nilpotent. Let $i:Z^1(W_{\qp}, \dualgrp{G})_{\chi}^{\mathrm{red}} \subset Z^1(W_{\qp}, \dualgrp{G})_{\chi}$ be the reduced closed subscheme and let $N$ be the rank of $\mathcal{V}_{\mu}$.
\begin{Claim}
    The endomorphism $e^N$ of $i^{\ast}\mathcal{V}_{\mu,\chi}$ is zero.
\end{Claim}
\begin{proof}
    It suffices to do this after pulling back to $L$-points of $Z^1(W_{\qp}, \dualgrp{G})_{\chi}^{\mathrm{red}}$. These $L$-points correspond to $L$-parameters $\phi:W_{\qp} \to {}^{L}G(L)$ with semisimplification (conjugate to) $\phi_{\chi}$, see \cite[Proposition VIII.3.2]{FarguesScholze}. For such $\phi$, the action of $e$ on $r_{\mu}\circ\phi$ is nilpotent (since it is zero on the semisimplification because the semisimplification is conjugate to $r_{\mu}\circ\phi_{\chi}$). We conclude that $e^N$ acts as zero on $r_{\mu}\circ\phi$, in other words, on the pullback of $\mathcal{V}_{\mu,\chi}$ via $\spec L \to Z^1(W_{\qp}, \dualgrp{G})_{\chi}^{\mathrm{red}}$.
\end{proof}
We conclude that the endomorphism $e^N$ is in the kernel of
\begin{align}
    \operatorname{End}(Z^1(W_{\qp}, \dualgrp{G})_{\chi}, \mathcal{V}_{\mu,\chi}) \to \operatorname{End}(Z^1(W_{\qp}, \dualgrp{G})_{\chi}^{\mathrm{red}}, i^{\ast}\mathcal{V}_{\mu,\chi}).
\end{align}
But this kernel is generated by the kernel $J$ of
\begin{align}
    \Gamma(Z^1(W_{\qp}, \dualgrp{G})_{\chi}, \mathcal{O}) \to \Gamma(Z^1(W_{\qp}, \dualgrp{G})_{\chi}^{\mathrm{red}}, \mathcal{O}),
\end{align}
which is nilpotent. We conclude that the endomorphism $e^N$ is nilpotent (it corresponds to a matrix which is the zero matrix modulo $J$, and thus a power of it is zero). This implies that the endomorphism $e$ of $W^i(\chi)$ is nilpotent, which contradicts our assumption on the existence of $\sigma_n$, proving the theorem.
\end{proof}

\subsection{The Cayley--Hamilton theorem} \label{Sec:CayleyHamilton}
As above, we let $\Lambda$ be the ring of integers in a finite extension of $\ql$, containing a fixed square root of $p$, and we let $\loc$ be the stack of $L$-parameters over $\Lambda$. To lighten notation, we write $\zspec\coloneqq \zspec(G,\Lambda)$ for the spectral Bernstein center. We will use the following version of the Cayley--Hamilton theorem for vector bundles on algebraic stacks, applied to $X_{\dualgrp{G}}$. The idea of applying the Cayley--Hamilton theorem to the stack of $L$-parameters has already appeared in \cite[Proposition 6.3.1]{XiaoZhuTwisted}, where we note that when $G$ is unramified, the stack $[G\sigma/G]$ is precisely the substack of $X_{\widehat{G}}$ of unramified Langlands parameters.
\begin{Thm}[Cayley--Hamilton] \label{Thm:CayleyHamilton}
    Let $\mathcal{V}$ be a vector bundle on $\loc$. For any endomorphism $\gamma \in \pi_0\mathrm{End}_{\loc}(\mathcal{V})$, there is a $\zspec$-linear homomorphism
    \[\zspec[X]/\det(X-\gamma)\to \pi_0\mathrm{End}_{\loc}(\mathcal{V}), \quad X\mapsto \gamma.\]
\end{Thm}
We let $V$ be a representation of $\dualgrp{G}$ together with an extension $r_V:{}^LG_E\to \mathrm{GL}(V)$ as in Example \ref{Eg:SpectralAction}, and assume $V$ is of rank $n$ over $\Lambda$. Let $\gamma\in W_E$ and let $\phi_E^\mathrm{univ}$ be the universal $1$-cocycle on $\cocycle$ restricted to $W_E\subset W_\qp$. Then $r_V\circ \phi_E^\mathrm{univ}(\gamma)$ defines an element in $\pi_0(\End_{\loc}(\mathcal{V}))$.

We will compute the coefficients of its characteristic polynomial in terms of excursion operators, see \cite[Definition VIII.4.2]{FarguesScholze} and the paragraph below. Apply the construction of loc.\ cit.\ to the excursion datum
\[D_{V,\gamma}\coloneqq(\{0,1\}, V\otimes V^\ast, \alpha: 1\xrightarrow{coev}V\otimes V^\ast, \beta: V\otimes V^\ast \xrightarrow{ev} 1, (\gamma,1)\in W_E^{\{0,1\}});\]
it gives rise to an excursion operator (here we use the $\mathcal{D}_\mathrm{lis}$ formalism of \cite[Section VII.6]{FarguesScholze} so that we can construct a global function on $\loc$ instead of its $\ell$-adic completion)
\[S_{V,\gamma}:T_1\xrightarrow{T_\alpha}T_{V\otimes V^\ast}\xrightarrow{(\gamma,1)}T_{V\otimes V^\ast}\xrightarrow{T_\beta}T_1=\mathrm{id}_{\mathcal{D}_\mathrm{lis}(\bungk, \Lambda)}.\]
This lifts to a global function on $\loc$, whose value at an $L$-parameter $\phi$ computes the trace of $r_V\circ \phi(\gamma)$ on $V$. We view it as an element in $\zspec$ and still denote it by $S_{V,\gamma}$. Similarly, replacing $V$ by its wedge powers $\wedge^iV$, $i=1,\dots,n$, we get excursion operators $S_{\wedge^iV,\gamma}$ whose attached functions take value at $\phi$ the trace of $\wedge^i (r_V\circ \phi(\gamma))$ on $\wedge^iV$. We use this to define a \emph{spectral Hecke polynomial} associated to $\gamma$.
\begin{Def} \label{Def:SpectralHeckePolynomial} For $\gamma \in W_E$ we define $H^{\mathrm{spec}}_{G,V,\gamma}(X)$ to be the polynomial
\[H^{\mathrm{spec}}_{G,V,\gamma}(X)\coloneqq \mathrm{det}(X-r_V\circ \phi_E^\mathrm{univ}(\gamma))=\sum_{i=0}^{n}(-1)^{n-i}S_{\wedge^{n-i}V,\gamma}X^i\in \zspec[X].\]
\end{Def}
The following result (cf.\ \cite[Proposition 2.1]{Koshikawa}) is immediate from Definition \ref{Def:SpectralHeckePolynomial} and Theorem \ref{Thm:CayleyHamilton}.
\begin{Cor}\label{Lem:HamiltonCayley}
    There is a $\zspec$-algebra homomorphism
    \[\zspec[X]/(H^{\mathrm{spec}}_{G,V,\gamma}(X))\rightarrow \pi_0(\End_{\loc}(\mathcal{V})), \quad X\mapsto r_V\circ \phi_E^\mathrm{univ}(\gamma).\]
\end{Cor}

Let $\mathcal{Z}(G(\qp), \Lambda)$ be the Bernstein center of smooth $G(\qp)$-representations on $\Lambda$-modules and $\Psi_G:\zspec(G, \Lambda)\to \mathcal{Z}(G(\qp), \Lambda)$ be the map of Bernstein centers as before (assuming $\ell$ is coprime to the order of $\pi_0(Z(G))$).
For $\gamma \in W_E$, we write $H_{G,V,\gamma}(X) \in \mathcal{Z}(G(\qp), \Lambda)[X]$ for the image of $H^{\mathrm{spec}}_{G,V,\gamma}(X)$ under $\Psi_G$. Now apply this to the representation $r_\mu$ and combine with Corollary~\ref{Cor:EllAdicSpectralAction} to obtain the following result. Recall that $\mathcal{Z}(G(\qp), \Lambda)$ acts on the (derived) $K_p$-fixed points of any $G(\qp)$-representation, through its image in the Hecke algebra at level $K_p$.
\begin{Cor} \label{Cor:InfiniteLevelEichlerShimura}
    If $\ell$ is coprime to the order of $\pi_0(Z(G))$, then for every compact open subgroup $K_p \subset G(\qp)$ and $\gamma \in W_E$ the endomorphism $H_{G,V_\mu,\gamma}(\gamma)$
    acts as zero on $R\Gamma(\mathbf{Sh}_{K,\overline{E}}, \Lambda)(\tfrac{d}{2})$.
\end{Cor}
\begin{proof}
    This follows from Corollary \ref{Lem:HamiltonCayley} in combination with Corollary \ref{Cor:EllAdicSpectralAction}.
\end{proof}

\subsection{Hecke polynomials} \label{Sub:HeckePolynomials} This is a standalone section on Hecke polynomials, which we will use to prove Theorem~\ref{Thm:IntroEichlerShimura}. The goal of this section is to compute the spectral Hecke polynomial of a lift of Frobenius, after restriction to the stack of unramified $L$-parameters. We thank Jean-Francois Dat for his helpful suggestions regarding this section.

Let $G$ be an unramified reductive group over $\qp$.\footnote{See Remark \ref{Rem:UnramifiedNecessary} for a discussion of this hypothesis.} Let $S \subset G$ be a maximal $\qp$-split torus with centralizer $T$, a maximal torus of $G$, and choose a Borel $B \supset T$. The torus $S$ determines an apartment in the Bruhat--Tits building of $G$. We choose a hyperspecial subgroup $K_p \subset G(\qp)$ corresponding to some vertex in this apartment. This determines an Iwahori subgroup $I$ contained in $K_p$. Let $\Lambda$ be the ring of integers in a finite extension of $\ql$ and fix $\sqrt{p} \in \Lambda$ as before. Throughout Section~\ref{Sub:HeckePolynomials}, we make the following assumption.

\begin{Assump}
    The prime number $\ell$ is coprime to the order of $\pi_0(Z(G))$.
\end{Assump}
Since the action of $W_{\qp}$ on $\dualgrp{G}$ factors through $W_{\qp}/I_{\qp}$, where $I_{\qp}$ is the inertia group, there is an inflation map (of $\Lambda$-schemes)
\begin{align}
    Z^1(W_{\qp}/I_{\qp}, \dualgrp{G}) \to Z^1(W_{\qp}, \dualgrp{G}).
\end{align}
This map is a closed immersion and we will denote its image by $Z^1_{\mathrm{ur}}(W_{\qp}, \dualgrp{G}) \subset Z^1(W_{\qp}, \dualgrp{G})$. We will write $\zspec_{\mathrm{ur}}(G, \Lambda)$ for the ring of global functions of $Z^1_{\mathrm{ur}}(W_{\qp}, \dualgrp{G}) \sslash \dualgrp{G}$, which receives a natural restriction map from $\zspec(G, \Lambda)$.
\begin{Lem} \label{Lem:UnramifiedVSUnipotent}
    The map $\zspec(G, \Lambda) \to \zspec_{\mathrm{ur}}(G, \Lambda)$ is surjective.
\end{Lem}
\begin{proof}
The inflation map $Z^1(W_{\qp}/I_{\qp}, \dualgrp{G}) \to Z^1(W_{\qp}, \dualgrp{G})$ has a section given by choosing a Frobenius element in $W_{\qp}$. This shows that the restriction map
    \begin{align}
        \mathcal{O}(Z^1(W_{\qp}, \dualgrp{G})) \to \mathcal{O}(Z^1_{\mathrm{ur}}(W_{\qp}, \dualgrp{G}))
    \end{align}
has a section, and thus the induced map
\begin{align}
    \mathcal{O}(Z^1(W_{\qp}, \dualgrp{G}))^{\dualgrp{G}} \to \mathcal{O}(Z^1_{\mathrm{ur}}(W_{\qp}, \dualgrp{G}))^{\dualgrp{G}}
\end{align}
is surjective.
\end{proof}

\subsubsection{Unipotent parameters} Let $I \subset K_p$ denote the Iwahori subgroup corresponding to $S$ and consider the representation
\begin{align}
    \pi=c\text{-}\mathrm{Ind}_{I}^{G(\qp)} \Lambda,
\end{align}
and let $\mathcal{H}_I\coloneqq\mathcal{H}_I(G, \Lambda)=\operatorname{End}_{G(\qp)}(\pi)$ be the Iwahori--Hecke algebra with coefficients in $\Lambda$, with center $\mathcal{Z}_I\coloneqq\mathcal{Z}_I(G(\qp), \Lambda)$.
\begin{Lem} \label{Lem:Factorization}
    There exists a map $\zspec_{\mathrm{ur}}(G, \Lambda) \to \mathcal{H}_I$ fitting into a commutative diagram
    \begin{equation}
        \begin{tikzcd}
            \zspec(G, \Lambda) \arrow{r}{\Psi_G} \arrow{d} & \mathcal{Z}(G(\qp), \Lambda) \arrow{d} \\
            \zspec_{\mathrm{ur}}(G, \Lambda) \arrow[r] & \mathcal{H}_I.
        \end{tikzcd}
    \end{equation}
\end{Lem}

\begin{proof}
The right vertical map has image in $\mathcal{Z}_I$ so it suffices to prove the lemma with $\mathcal{H}_I$ replaced by $\mathcal{Z}_I$. If we translate this into a diagram of schemes, then we are trying to show that the map
\begin{align}
    \spec \mathcal{Z}_I \to \spec \mathcal{Z}(G(\qp), \Lambda) \to \spec \mathcal{O}\left(Z^1(W_{\qp}, \dualgrp{G}) \sslash \dualgrp{G}\right)
\end{align}
factors through the closed subscheme\footnote{Since $Z^1(W_{\qp}, \dualgrp{G}) \sslash \dualgrp{G}$ is an infinite disjoint union of affine schemes, it is not isomorphic to the spectrum of its ring of global functions.}
\[Z^1_{\mathrm{ur}}(W_{\qp}, \dualgrp{G}) \sslash \dualgrp{G} \subset \spec \mathcal{O}\left(Z^1(W_{\qp}, \dualgrp{G}) \sslash \dualgrp{G}\right),\]
see Lemma \ref{Lem:UnramifiedVSUnipotent}. It follows from \cite[Theorem 6.4.1]{Boumasmoud} that the scheme $\spec \mathcal{Z}_I$ is integral and $\Lambda$-flat. We may thus check the factorization after base change to $\Lambda[1/\ell]$ and we may moreover check it on $\qlbar$-points.

On the level of $\qlbar$-points, we are trying to prove that Iwahori spherical $G(\qp)$-representations have unramified Fargues--Scholze parameters. This follows for example from the fact that they are quotients of parabolic inductions of unramified characters of $T(\qp)$, in combination with the fact that the construction of Fargues--Scholze parameters is compatible with parabolic induction, see \cite[Corollary IX.7.3]{FarguesScholze} and \cite[Proposition IX.6.5]{FarguesScholze}. Alternatively, this follows from the main results of \cite{MR4651101}.
\end{proof}

We write $W_0$ for the relative Weyl group of $S$, which acts on $\dualgrp{S}$ and $\dualgrp{T}$. We write $N$ for the normalizer of $\dualgrp{T}$ in $\dualgrp{G}$ and $N_0$ for the inverse image of $W_0$ in $N$. Then there are natural maps
\begin{align}
    \dualgrp{S} \sslash W_0 \leftarrow \dualgrp{T} \sslash \operatorname{Ad}_{\sigma} N_0 \to \dualgrp{G} \sslash \operatorname{Ad}_{\sigma} \dualgrp{G},
\end{align}
where $\sigma$ is the Frobenius on $\qpbr$ and we write $\operatorname{Ad}_\sigma$ for the $\sigma$-twisted conjugation action. We have the following lemma.

\begin{Lem}
  The natural map $\dualgrp{T} \sslash \operatorname{Ad}_{\sigma} N_0  \to \dualgrp{S} \sslash W_{0}$ is an isomorphism.
\end{Lem}
\begin{proof}
  It can be checked directly using cocharacter groups that the natural map
  $\dualgrp{T} \sslash \operatorname{Ad}_{\sigma} \dualgrp{T} \to
  \dualgrp{T}_\sigma$ is an isomorphism (where the latter is the $\qp$-split torus whose character group is the Galois coinvariants of $X^\ast(\widehat{T})$), and $\dualgrp{T}_\sigma \simeq
  \dualgrp{S}$ since $X_\ast(S) = X_\ast(T)^\sigma$. It follows that the map of
  the lemma is an isomorphism.
\end{proof}

\begin{Rem}
We expect that the natural map $\dualgrp{T} \sslash \operatorname{Ad}_{\sigma} N_0 \to \dualgrp{G} \sslash \operatorname{Ad}_{\sigma} \dualgrp{G}$ is also an isomorphism. This is true over $\ql$ by \cite[Lemma 6.4, 6.5]{BorelCorvallis} (cf.\ \cite[Conjecture 4.19, Theorem 4.21]{ZhuCoherent}).
\end{Rem}

We will write $\Theta_{\mathrm{Spec}}:\zspec_{\mathrm{ur}}(G, \Lambda) = \mathcal{O}(\dualgrp{G} \sslash \operatorname{Ad}_{\sigma} \dualgrp{G}) \to \mathcal{O}(\dualgrp{S} \sslash W_0)=\mathcal{O}(\dualgrp{S})^{W_0}$. By \cite[Theorem 6.4.1, Theorem 6.5.1]{Boumasmoud}, there is a canonical isomorphism
\begin{align}
    \Theta_{\mathrm{Bern}}:\mathcal{O}(\dualgrp{S} \sslash W_0) \xrightarrow{\sim} \mathcal{Z}_I(G(\qp), \Lambda)
\end{align}
coming from the Bernstein presentation of $\mathcal{H}_I$. We now want to understand how the map $\Psi_I^G:\zspec_{\mathrm{ur}}(G, \Lambda) \to \mathcal{Z}_I(G(\qp), \Lambda)$ constructed in the proof of Lemma \ref{Lem:Factorization} interacts with these isomorphisms.
\begin{Prop} \label{Prop:UnramifiedLocalLanglandsFamilies}
    The following diagram commutes
    \begin{equation}
        \begin{tikzcd}
            & \mathcal{O}(\dualgrp{S})^{W_0} \arrow[dr,"\Theta_{\mathrm{Bern}}"] & \\
            \zspec_{\mathrm{ur}}(G, \Lambda) \arrow[rr, "\Psi_I^G"]  \arrow[ur,
    "\Theta_{\mathrm{Spec}}"]
            & & \mathcal{Z}_I(G(\qp), \Lambda).
        \end{tikzcd}
    \end{equation}
\end{Prop}

\begin{proof}
    To prove the proposition we may invert $\ell$ and work with $\qlbar$ coefficients, since $\zspec_{\mathrm{ur}}(G, \Lambda)\simeq \mathcal{O}(\dualgrp{S})^{W_0}$ is torsion-free. We may then moreover prove the result (considered as a statement about morphisms of schemes) on the level of $\qlbar$-points since our rings are reduced. We are then trying to show that the Fargues--Scholze $L$-parameters of Iwahori--spherical smooth irreducible representations $\pi$ of $G(\qp)$ over $\qlbar$ can be computed by thinking of $\pi$ as a character of $\mathcal{Z}_I(G(\qp), \Lambda)$, applying the Bernstein isomorphism and pulling back along $\Theta_{\mathrm{Spec}}$. Both these methods of building $L$-parameters are compatible with parabolic induction, see \cite[Corollary IX.7.3]{FarguesScholze} for the Fargues--Scholze $L$-parameters and the discussion before \cite[Lemma 6.3]{HainesSatake} for the $L$-parameters constructed in the second way. This reduces the proposition to the case of tori, which is \cite[Proposition IX.6.5]{FarguesScholze}. Alternatively, this follows from the main results of \cite{MR4651101}.
\end{proof}

\begin{Rem} \label{Rem:UnramifiedNecessary}
Proposition \ref{Prop:UnramifiedLocalLanglandsFamilies} has been generalized to arbitrary quasi-split groups $G$ in recent work of van den Hove, see \cite[Proposition 4.4]{vdHoveEichlerShimura}. One ingredient in that work is a definition of the stack of spherical Langlands parameters for groups that are not necessarily unramified. Non-quasi-split groups are also covered by \cite[Proposition 4.4]{vdHoveEichlerShimura}; although the statement differs from Proposition \ref{Prop:UnramifiedLocalLanglandsFamilies} and involves the transfer homomorphisms of \cite[Section 12]{HainesSatake}.
\end{Rem}

\subsubsection{The Eichler--Shimura relation} We now prove Theorem~\ref{Thm:IntroEichlerShimura}; we first recall its statement. Let $S \subset T \subset B$, $I \subset K_p$ and $\mathcal{Z}_I\coloneqq\mathcal{Z}_I (G(\qp), \Lambda)$ be as above. For $K^p \subset \gafp$ a neat compact open subgroup we set $K'=IK^p$. Also $q$ is the cardinality of the residue field of $E$ as before. Choose any Frobenius lift $\sigma_E\in W_E$. We denote by $H_{\mu}$ the re-normalized Hecke polynomial
\begin{align}
 \operatorname{det}\left(X- q^{\tfrac{d}{2}} \left(r_{\mu} \circ \phi_E^{\operatorname{univ}}(\sigma_E)\right)\right) \in \mathcal{O}(\dualgrp{S})^{W_0}[X],
\end{align}
which we think of as a polynomial in $\mathcal{Z}_I[X]$ using the Bernstein isomorphism. Our definition of $H_{\mu}$ agrees with the one in \cite[Introduction]{Koshikawa}. Note that in \cite{LeeEichlerShimura}, the Hecke polynomial $H_{\mu^{-1}}$ is used, see \cite[Section 2.1.2]{LeeEichlerShimura}\footnote{Note that in \cite{LeeEichlerShimura} the symbol $d$ is used for the half integer $\tfrac{d}{2}$.}; we suspect the discrepancy is explained by \cite[Remark 1.2]{Koshikawa}. Let $\mathbf{Sh}_{K',\ebar}$ be the base change of the Shimura variety $\mathbf{Sh}_{K'}$ to $\ebar$, whose cohomology has an action of $\mathcal{Z}_I\times W_E$.

\begin{Thm} \label{Cor:EichlerShimura}
If the order of $\pi_0(Z(G))$ is coprime to $\ell$, then the inertia group $I_E \subset W_E$ acts unipotently on $H^i(\mathbf{Sh}_{K',\overline{E}}, \Lambda)$ for all $i$. Moreover, the action of any Frobenius lift $\sigma_E \in W_E$ on $R \Gamma(\mathbf{Sh}_{K',\overline{E}}, \Lambda)$ satisfies $H_{\mu}(\sigma_E)=0$.
\end{Thm}
\begin{proof}
Lemma~\ref{Lem:Factorization} implies that for any algebraically closed field $L$ over $\Lambda$ and any smooth $L$-representation of $G(\qp)$ with nonzero Iwahori fixed vectors, its attached Fargues--Scholze (semisimplified) $L$-parameter is unramified. This in particular shows that the characteristic polynomial $H^{\mathrm{spec}}_{G,V_{\mu},\gamma}(X)$ for any $\gamma\in I_E$ maps to $(X-\mathrm{id})^{\operatorname{rank}V_\mu}$ in $\mathcal{Z}_I[X]$. Hence $\gamma$ acts unipotently on $H^i(\mathbf{Sh}_{K',\overline{E}}, \Lambda)$ for all $i$.

If we identify the center of the Iwahori--Hecke algebra $\mathcal{Z}_I\coloneqq\mathcal{Z}_I(G(\qp), \Lambda)$ with $\mathcal{O}(\dualgrp{S})^{W_0}$ as above, then it follows from Proposition \ref{Prop:UnramifiedLocalLanglandsFamilies} and Lemma \ref{Lem:Factorization} that we may identify the image of $H_{G,V_{\mu}, \sigma_E}^{\operatorname{spec}}(X)$ in $\mathcal{Z}_I[X]$ with the polynomial
    \begin{align}
        \operatorname{det}(X-r_{\mu} \circ \phi_E^{\operatorname{univ}}(\sigma_E)) \in \mathcal{O}(\dualgrp{S})^{W_0}[X].
    \end{align}
Here we are taking the representation $r_{\mu}:\dualgrp{G} \rtimes W_E \to \operatorname{GL}(V_{\mu})$ of Section \ref{subsub:RMu}, restricting it to $\dualgrp{S} \rtimes W_E$ and evaluating the universal $L$-parameter over $Z^1_{\mathrm{ur}}(W_\qp,\dualgrp{S})$ at $\sigma_E$. By Corollary \ref{Cor:InfiniteLevelEichlerShimura}, the value of this polynomial at $\sigma_E$ acts trivially on (via the natural right action of $\mathcal{Z}_I$) 
    \begin{align}
        R\Gamma(\mathbf{Sh}_{K',\overline{E}}, \Lambda)(\tfrac{d}{2}).
    \end{align}
Twisting by $\tfrac{d}{2}$ (using our fixed $\sqrt{p} \in \Lambda$), we find that the re-normalized Hecke polynomial
    \begin{align}
        H_{\mu}=\operatorname{det}\left(X- q^{\tfrac{d}{2}} \left(r_{\mu} \circ \phi_E^{\operatorname{univ}}(\sigma_E)\right)\right) \in \mathcal{O}(\dualgrp{S})^{W_0}[X]
    \end{align}
evaluated at $\sigma_E$, acts as zero on
$R\Gamma(\mathbf{Sh}_{K',\overline{E}}, \Lambda)$.
    \end{proof}
}

%% file: Torsion-vanishing.tex
In this section we use Theorem \ref{Thm: Perversity} to prove Theorem~\ref{Thm:IntroTorsionVanishing}, which is a vanishing result for the generic part of the cohomology of (compact) Shimura varieties with torsion coefficients. We proceed following the ideas of \cite{CaraianiScholzeCompact}, \cite{CaraianiScholzeNoncompact},
\cite{Koshikawa}, \cite{Santos}, and more closely \cite{Hamann-Lee}. Throughout this section, we fix a prime $\ell \neq p$ and take $\Lambda$ to be $\flbar$, a fixed algebraic closure of $\mathbb{F}_\ell$. We also fix a choice $\sqrt{p}\in \flbar$ and follow Section~\ref{Sec:Cohomology} for all unspecified notation. We further assume that $\ell$ is coprime to the order of $\pi_0(Z(G))$.

\subsection{\texorpdfstring{$t$}{t}-exactness of Hecke operators}
\label{Sec:LocalizedCategories} Let $W_{\qp}$ be the Weil group of $\qp$ and $\phantom{}^{L} G$ be the $L$-group of $G$ as before. Consider a semisimple $L$-parameter $\phi:W_{\qp} \to \phantom{}^{L} G(\flbar)$. Under our assumption that $\ell$ is coprime to the order of $\pi_0(Z(G))$ we have a full subcategory
 \begin{equation} \label{Eq:InclusionCategory}
     D(\bungk,\flbar)_{\phi} \subset D(\bungk,\flbar),
 \end{equation}
see \cite[Definition 4.2]{Hamann-Lee}. Recall moreover that the inclusion \eqref{Eq:InclusionCategory} has a left adjoint, denoted by $\mathcal{L}_{\phi}: A \mapsto A_{\phi}$. Given $A \in D(G(\qp), \flbar) \simeq D(\bungk^{[1]}, \flbar) \subset D(\bungk, \flbar)$, we have that $A_{\phi} \in D(\bungk^{[1]}, \flbar)_\phi$, see \cite[Proposition A.2]{Hamann-Lee}.

If $G$ is quasi-split and splits over an unramified extension, then we have a notion of \emph{unramified} semisimple $L$-parameters, see \cite[Page 43]{ZhuCoherent}. For a fixed hyperspecial subgroup $K_p \subset G(\qp)$, the classical Satake isomorphism gives a bijection between conjugacy classes of unramified semisimple $L$-parameters $\phi$ and maximal ideals $\mathfrak{m}$ of the spherical Hecke algebra $\mathcal{H}_{K_p}$, see \cite[Section 4.3]{ZhuCoherent}. The following lemma is \cite[Lemma 4.3(3)]{Hamann-Lee}.
 \begin{Lem} \label{Lem:LocalizedCategories}
 If $\phi$ is unramified and $K_p$ is hyperspecial, then for $A \in D(\bungk^{[1]}, \flbar)$ there is a natural isomorphism
     \begin{align}
         R \Gamma ( K_p, A)_{\mathfrak{m}}= R \Gamma(K_p, A_{\phi}),
     \end{align}
     where we recall that $\mathcal{H}_{K_p}$ has a natural right action on $R \Gamma ( K_p, A)$.
 \end{Lem}

By \cite[Lemma 4.3.(2)]{Hamann-Lee}, the Hecke operator $T_\mu$ defined in Section \ref{subsub:HeckeOperators} restricts to a functor
\begin{align}
    T_{\mu}:D(\bungk,\flbar)_{\phi} \to D(\bungk,\flbar)^{BW_E}_{\phi}.
\end{align}
We consider the following assumption on the semisimple $L$-parameter $\phi$.
\begin{Assump} \label{Assump:tExact}
The functor
\begin{align}
  i_{1,k}^{\ast}T_{\mu} \colon D(\bungk,\flbar)^{\mathrm{ULA}}_{\phi} \to
  D(\bungk,\flbar)^{BW_E}_{\phi} \to D(\bungk^{[1]},\flbar)^{BW_E}_{\phi}
\end{align}
is $t$-exact with respect to the perverse $t$-structures from Proposition~\ref{Prop:PerverseTStruct}.
\end{Assump}
Assumption \ref{Assump:tExact} is known in some cases by work of Hamann--Lee, see \cite[Theorem 4.23]{Hamann-Lee}. See \cite[Conjecture 6.4]{Hamann-Lee} for a more general conjecture. We will prove a special case of their conjecture below, see Proposition \ref{Prop:tExact} below.

\subsubsection{Torsion vanishing} Let the notation be as in Section \ref{sub:Hecke}.
We will consider $R\Gamma(\mathbf{Sh}_{K^p,C}, \flbar)$ with its Hecke action as an object in $D(\bungk^{[1]},\flbar)$. Thus by Lemma \ref{Lem:LocalizedCategories} we can apply the functor $\mathcal{L}_\phi$ to it, for any semisimple $L$-parameter $\phi$.

Recall that we have the sheaf $\mathcal{F}\coloneqq R\overline{\pi}_{\mathrm{HT},\ast}\flbar \in  D(\bungmuk,\flbar)$. We can now prove the main result of this section, which directly implies Theorem \ref{Thm:IntroTorsionVanishing}.

\begin{Thm}\label{Thm: TorsionVanishing}
Assume that $\mathbf{Sh}_K \to \spec E$ is proper. If Assumption \ref{Assump:tExact} holds for $\phi$, then there is an isomorphism  
    \[R\Gamma(\mathbf{Sh}_{K^p,C}, \flbar)_{\phi}\simeq H^d(\mathbf{Sh}_{K^p,C}, \flbar)_{\phi}[-d],\] in $D(G(\qp),\flbar)$. If $\phi$ is unramified, $K_p$ is a hyperspecial subgroup and $\ell$ is coprime to the pro-order of $K_p$, then $H^i(\mathbf{Sh}_{K,C},\flbar)_{\mathfrak{m}} = 0$ unless $i = d$. 
\end{Thm}

\begin{proof}
We are going to show that
\[T_\mu^{[1]} \mathcal{F}_{\phi} \in D(G(\qp),\flbar)\]
is concentrated in degree $0$ using Assumption \ref{Assump:tExact}. We know that $j_{\mu,k,!} \mathcal{F}$ is ULA by the proof of Corollary \ref{Cor:ULA}, and it follows from the proof of Theorem \ref{Thm: Perversity} that $j_{\mu,k,!} \mathcal{F} \in {}^{p}D(\bungk,\flbar)^{\le 0}$. From \cite[Proposition A.4, Proposition A.5]{Hamann-Lee} it follows that $j_{\mu,k,!} \mathcal{F}_{\phi} =(j_{\mu,k,!} \mathcal{F})_{\phi}$ is also ULA and contained in ${}^{p}D(\bungk,\flbar)_{\phi}^{\le 0}$. It now follows from Assumption \ref{Assump:tExact} that $T_\mu^{[1]} \mathcal{F}_{\phi}=i_{1,k}^{\ast} T_{\mu} j_{\mu,k,!} \mathcal{F}_{\phi}$ is concentrated in degrees $\le 0$.

Next, we consider $R j_{\mu,k,\ast} \mathcal{F}$. By relative Verdier duality, see \cite[Proposition 23.3.(i)]{EtCohDiam} and the self-duality of $\mathcal{F}$, see the proof of Theorem \ref{Thm: Perversity}, we find that $Rj_{\mu,k,\ast} \mathcal{F}$ is the Verdier dual of $j_{\mu,k,!} \mathcal{F}$ and is thus ULA. By \cite[Theorem 1.9.(iii)]{EtCohDiam}, we have $i_{b,k}^! Rj_{\mu,k,\ast} \mathcal{F} = i_{b,k}^! \mathcal{F}$ for $[b] \in \bgmu$ and $i_{b,k}^! Rj_{\mu,k,\ast} \mathcal{F} = 0$ for $[b] \in B(G) \setminus \bgmu$, and so $Rj_{\mu,k,\ast} \mathcal{F}$ lies in ${}^{p}D(\bungk,\flbar)^{\ge 0}$ since $\mathcal{F}$ is perverse. From \cite[Proposition A.4, Proposition A.5]{Hamann-Lee} it follows that $Rj_{\mu,k,\ast} \mathcal{F}_{\phi} =(Rj_{\mu,k,\ast} \mathcal{F})_{\phi}$ is also ULA and contained in ${}^{p}D(\bungk,\flbar)_{\phi}^{\ge 0}$. It now follows from Assumption \ref{Assump:tExact} that $T_\mu^{[1]} \mathcal{F}_{\phi}=i_{1,k}^{\ast} T_{\mu} Rj_{\mu,k,\ast} \mathcal{F}_{\phi}$ is concentrated in degrees $\ge 0$. But this is
\begin{align}
    R\Gamma(\mathbf{Sh}_{K^p,C}, \flbar)_{\phi}(\tfrac{d}{2})[d],
\end{align}
by Theorem \ref{Thm: WeilCohoShiVar}, and this implies the desired result.

For the final assertion of the theorem, we note that since we assume $\ell$ is coprime to the pro-order of $K_p$, taking $K_p$-invariants is exact. The result now follows from the fact that
\begin{align}
    R\Gamma(\mathbf{Sh}_{K^p,C}, \flbar)^{K_p} \simeq R\Gamma(\mathbf{Sh}_{K,C}, \flbar)
\end{align}
and so that
\begin{align}
    R\Gamma(\mathbf{Sh}_{K^p,C}, \flbar)^{K_p}_{\phi} \simeq R\Gamma(\mathbf{Sh}_{K,C}, \flbar)_{\mathfrak{m}_{\phi}},
\end{align}
see Lemma \ref{Lem:LocalizedCategories}.
\end{proof}

\begin{Rem}
If Assumption \ref{Assump:tExact} holds for both $\phi$ and its Bernstein--Zelevinsky dual $L$-parameter $\phi^\vee$, then the condition that $\ell$ is coprime to the pro-order of $K_p$ in Theorem \ref{Thm:IntroTorsionVanishing} can be removed by applying Poincar\'e duality. Indeed, suppose $\phi$ is an $L$-parameter that corresponds to a maximal ideal $\mathfrak{m}$ of the spherical Hecke algebra $\mathcal{H}_{K_p}$ for some hyperspecial subgroup $K_p \subset G(\qp)$. Then, both
\begin{align}
    R\Gamma(\mathbf{Sh}_{K,C}, \flbar)_{\mathfrak{m}^{\vee}} \quad \text{ and } \quad R\Gamma(\mathbf{Sh}_{K,C}, \flbar)_{\mathfrak{m}}
\end{align}
are concentrated in degrees $\ge d$ by Theorem \ref{Thm: TorsionVanishing}. Here $\mathfrak{m}^{\vee}$ is the dual maximal ideal, corresponding to the $L$-parameter $\phi^{\vee}$, see \cite[Corollary A.7]{Hamann-Lee}. But since these two complexes are Poincar\'e-dual to each other, it follows that they are concentrated in degree $d$.
\end{Rem}

\subsection{On the t-exactness conjecture of Hamann--Lee} \label{Sec:tExact} In this section we will discuss under which conditions Assumption \ref{Assump:tExact} is known to hold. Throughout this section we assume that $G$ is quasi-split, and we fix $T \subset B \subset G$, where $T$ is a maximal torus, and $B$ is a Borel.

We let $X_{\ast}(T)$ denote the cocharacter lattice of $T$ over $\qpbar$, equipped with its action of $\Gamma_p=\gal(\qpbar/\qp)$. For $\alpha \in X_{\ast}(T)/\Gamma_p$ being the $\Gamma_p$-orbit of some coroot, we will consider the $\flbar$-valued representation $\alpha$ of $\phantom{}^{L} T$. Given a semisimple $L$-parameter $\phi_T:W_{\qp} \to \phantom{}^{L} T(\flbar)$ (a \emph{toral parameter}), we will consider the compositions $\alpha \circ \phi_{T}$.

We recall from \cite[Definition 3.6]{Hamann22} that
a toral $L$-parameter $\phi_T:W_{\qp} \to \phantom{}^{L} T(\flbar)$ is called of \emph{weakly Langlands--Shahidi type} if for all $\alpha \in X_{\ast}(T)/\Gamma_p$ corresponding to a $\Gamma_p$-orbit of coroots, the cohomology groups
\begin{align}
    H^2(W_{\qp}, \alpha\circ\phi_T), \quad H^2(W_{\qp}, \alpha\circ\phi_T^{\vee})
\end{align}
are trivial. We also recall from \cite[Definition 1.4]{Hamann22} that a toral $L$-parameter $\phi$ is called \emph{generic} if for all coroots $\alpha \in X_{\ast}(T)/\Gamma_p$, the complex
\begin{align}
    R \Gamma(W_{\qp}, \alpha\circ\phi_T)
\end{align}
is trivial. We call a semisimple $L$-parameter $W_{\qp} \to \phantom{}^{L} G(\flbar)$ \emph{generic} (resp.\ of weakly Langlands--Shahidi type) if it can be conjugated to factor through $\phantom{}^{L} T(\flbar) \subset \phantom{}^{L} G(\flbar)$ via a generic (resp.\ of weakly Langlands--Shahidi type) toral $L$-parameter $\phi_T$.

For $G=\operatorname{GL}_n$, a toral parameter $\phi$ is a direct sum of characters $\phi_{i}$ for $i=1, \dots, n$. The parameter $\phi$ is of weakly Langlands--Shahidi type if $\phi_i \not \simeq \phi_j(1)$ for $i \not=j$, where $(1)$ denotes twisting by the mod $\ell$ cyclotomic character of $W_{\qp}$. The parameter $\phi$ is generic if $\phi_i \not \simeq \phi_j(1), \phi_j$ for $i \not=j$, see \cite[Remark 1.3]{Hamann22}.

Hamann and Lee conjecture that Assumption \ref{Assump:tExact} holds for toral $L$-parameters $\phi$ of weakly Langlands--Shahidi type, see \cite[Conjecture 6.4]{Hamann-Lee}. We will now prove a special case of this.

Since $G$ is quasi-split, there is a distinguished element $[b_{\mu}] \in \bgmu$, called the $\mu$-ordinary element, which is the largest element in $\bgmu$ for the partial order defined in \cite[Section 2.3]{RapoportRichartz}, see \cite[Section 1.3.15]{KMPS}. Let $\bgmu_{\mathrm{un}}$ be the subset of unramified elements in $\bgmu$, i.e., the intersection of $\bgmu$ with the image of $B(T)$ in $B(G)$ under the natural map $B(T) \to B(G)$. If we choose $\mu^{-1}$ to be a representative of its conjugacy class, then it defines an element in $B(T)\simeq \pi_1(T)_\Gamma$ by taking the image under $X_\ast(T)\simeq\pi_1(T)\to\pi_1(T)_\Gamma$, whose image under $B(T)\to B(G)$ is precisely $[b_\mu]$. Therefore, the set $B(G,\mu^{-1})_\mathrm{un}$ always contains $[b_\mu]$. The result below is inspired by the arguments in \cite{Koshikawa}.  We thank Linus Hamann for suggesting its proof; any error is due to us.

\begin{Prop} \label{Prop:tExact}
Suppose that \cite[Assumption 6.5]{Hamann22} holds for $G$. If $\bgmu_{\mathrm{un}}$ consists of a single element, then Assumption \ref{Assump:tExact} holds for generic parameters $\phi$.
\end{Prop}
The assumption \cite[Assumption 6.5]{Hamann22} is stated as \cite[Assumption 4.5]{Hamann-Lee}. See \cite[Theorem 4.13]{Hamann-Lee} for a list of the known cases of this assumption (see also \cite{Peng}).
\begin{proof}
Let $\phi$ be a generic semisimple $L$-parameter and let $B \in D(\bungk,\flbar)^{\mathrm{ULA}}_{\phi}$. Then it follows from \cite[Proposition 4.6]{Hamann-Lee} that $i_{b,k}^{\ast} B = 0$ unless $[b]$ is unramified. Since $\bgmu_{\mathrm{un}}$ always contains the $\mu$-ordinary element, while by our assumption $\bgmu_{\mathrm{un}}$ contains a single element, it must be the $\mu$-ordinary element $[b_{\mu}]$. This reduces the statement we want to the t-exactness of the functor
\[i_{1}^{\ast}T_{\mu}i_{[b_{\mu}],!}: D(\bungk^{[b_\mu]}, \flbar)_\phi\rightarrow D(\bungk^{[1]}, \flbar)_\phi\simeq D(G(\qp),\flbar)_\phi.\]
In fact, this holds even before localizing at $\phi$, which follows from Proposition \ref{Lem: t-exact} below and exactness of the unnormalized parabolic induction. The latter is a consequence of the second adjointness, due to \cite[Corollary 1.3]{DHKM} in our setting.
\end{proof}

For the $\mu$-ordinary element $[b_\mu]$, we have $\langle 2\rho, \nu_{b_\mu}\rangle=\langle 2 \rho, \mu \rangle=d$ and the group $G_{b_\mu}$ is isomorphic to a Levi subgroup $M$ of $G$, the centralizer of the slope homomorphism $\nu_{b_\mu}$ in $G$. Let $P$ be the parabolic with Levi factor $M$ containing the Borel $B^-$ opposite to $B$, and denote by $\mathrm{Ind}_{P(\qp)}^{G(\qp)}$, the unnormalized parabolic induction for $P$. In this case we have a simple formula for the Hecke operator $i_{1}^{\ast}T_{\mu}i_{b_\mu,!}(-)$.
\begin{Prop}\label{Lem: t-exact} Let $[b]=[b_\mu]$ be the $\mu$-ordinary element. Under the isomorphism
    \[D(G_{b}(\qp), \flbar)\simeq D(\bungk^{[b]}, \flbar),\]
    the functor
    \[i_{1}^{\ast}T_{\mu}i_{b,!}(-): D(\bungk^{[b]}, \flbar)\rightarrow D(\bungk^{[1]}, \flbar)\simeq D(G(\qp),\flbar).\]
    is identified with $\mathrm{Ind}_{P(\qp)}^{G(\qp)}(-)[d](\tfrac{d}{2})$.
\end{Prop}
\begin{proof}
    Applying \cite[Proposition 10.3]{Hamann22} to $[b]$, we get the formula\footnote{See the discussion before Corollary \ref{Cor:CompactlySupportedDirectLimit} for the definition of $R\Gamma_c(G,b,\mu)$. Note that in our case $\mathcal{S}_\mu\simeq \flbar[d](\tfrac{d}{2})$ by Remark~\ref{Rem:IntersectionCohomologyTrivial}.}
    \begin{align}
    i_{1}^{\ast}T_{\mu} i_{b,!} \simeq R\Gamma_c(G,b,\mu)\otimes^{\mathbb{L}}_{\mathscr{H}(G_{b})} (-\otimes \delta_b)[2d],
    \end{align}
     where $\mathscr{H}(G_{b})=C^\infty_c(G_{b}(\qp),\flbar)$ is the usual smooth Hecke algebra, and $\delta_b$ is the character $G_{b}(\qp)\rightarrow \flbar^\ast$ as defined in \cite[Definition 3.14]{HamannImai}, such that the dualizing complex on $\bungk^{[b]}$ is $\delta_b^{-1}[-2d_b]$, see \cite[Proposition 4.1, Corollary 4.6.(1)]{HamannImai}.

    Note that the triple $([b], [1], \mu^{-1})$ is Hodge-Newton reducible in the sense of \cite[Definition 4.5]{ImaiGaisin}, cf.\ \cite[Definition 4.24]{RapoportViehmann}. Pick a $B$-dominant representative of $\mu^{-1}$ that factors through $M$, and write $\mu_M^{-1}$ for its $M(\qpbar)$-conjugacy class. Denote by $[b_M]$ the $\mu_M$-ordinary element. Then $([b_\mu], [1])$ is the image of $([b_M], [1])$ under the natural map $B(M)\to B(G)$. Now, according to \cite[Theorem 4.26]{ImaiGaisin}, whose proof works verbatim replacing $\qlbar$ by $\flbar$, the cohomology of local Shimura variety $\mathcal{M}_{G,b,\mu,\infty}$ is parabolically induced from $\mathcal{M}_{M,b_M,\mu_M,\infty}$ in the sense that
    \begin{align}
R\Gamma_c(\mathcal{M}_{G,b,\mu,\infty}, \flbar)
\simeq & \mathrm{Ind}^{G(\qp)}_{P(\qp)}\left(R\Gamma_c(\mathcal{M}_{M,b_M,\mu_M,\infty}, \flbar[-2d])\otimes^\mathbb{L}\delta_{P,b_M} \right).
    \end{align}
    Here $\delta_{P,b_M}$ denotes the character of $M$ (inflated to a character of $P$), defined in \cite[Section 3.1.(3)]{HamannImai}. The identification of $\delta_{P,b_M}$ with the character $\kappa$ in the original formula follows from \cite[Corollary 4.6.(2)]{HamannImai} for $\theta=b_M$. Plugging this into the expression for the Hecke operator we get
    \begin{align}
    i_{1}^{\ast}T_{\mu} i_{[b],!}(-) \simeq \mathrm{Ind}^{G(\qp)}_{P(\qp)}R\Gamma_c(\mathcal{M}_{M,b_M,\mu_M,\infty}, \flbar)\otimes^\mathbb{L}_{\mathscr{H}(G_{b})} (-\otimes \delta_b\otimes \delta_{P,b_M}) [d](\tfrac{d}{2}).
    \end{align}
    Since $\mu_M^{-1}$ is central in $M$, the space that parametrizes modifications  $\mathscr{E}_M^1\dashrightarrow \mathscr{E}$ of type $\mu_M^{-1}$ of the trivial $M$-bundle $\mathscr{E}_M^1$ is just a point. But $\mathcal{M}_{M,b_M,\mu_M,\infty}$ is the $M(\qp)$-torsor over this space which parametrizes trivializations of the bundle $\mathscr{E}$, so it is isomorphic to $M(\qp)=G_{b}(\qp)$. Hence
    \[R\Gamma_c(\mathcal{M}_{M,b_M,\mu_M,\infty}, \flbar)\simeq {\mathscr{H}(G_{b})}[0].\]
    On the other hand, in our case, the diamond groups $\widetilde{G}_b^{>0}$ and $\widetilde{G}_{b,U}^{>0}$ in \cite[Proposition 4.5]{HamannImai} agree. Therefore, comparing Proposition 4.4 and 4.5 of loc.\ cit., we see that $\delta_b^{-1}\simeq \delta_{P,b_M}$. So the formula above further simplifies to
     \begin{align}
    i_{1}^{\ast}T_{\mu} i_{[b],!}(-) \simeq \mathrm{Ind}^{G(\qp)}_{P(\qp)} (-) [d](\tfrac{d}{2}),
    \end{align}
    and this finishes the proof.
\end{proof}

\subsection{An example} \label{Sec:Example} In this section we apply Proposition \ref{Prop:tExact} and Theorem \ref{Thm: TorsionVanishing} in an example. Let $\mathsf{F}$ be a totally real field of degree $d>1$ and let $\tau_1, \cdots, \tau_d$ be the infinite places of $\mathsf{F}$. Let $H$ be an inner form of $\operatorname{GSp}_{4,\mathsf{F}}$ and suppose that
\begin{align}
    H \otimes_{\mathsf{F},\tau_i} \mathbb{R}
\end{align}
is isomorphic to $\operatorname{GSp}_{4, \mathbb{R}}$ for $i=1, \dots, r$ and has $\mathbb{R}$-points compact modulo center for $i=r+1, \cdots, d$, with $1<r<d$. Let $\mathsf{G}_2=\operatorname{Res}_{\mathsf{F}/\mathbb{Q}} H$ and let $\mathsf{X}_2= \prod_{i=1}^d \mathsf{X}_{2,i}$ be the Shimura datum where $\mathsf{X}_{2,i}$ is the usual Siegel Shimura datum for $H \otimes_{\mathsf{F},\tau_i} \mathbb{R}$ if $i \le r$, and where $\mathsf{X}_{2,i}$ is trivial if $i >r$. The Shimura datum $\gxtwo$ is of abelian type and the corresponding Shimura varieties have dimension $3 r$. They are moreover proper: Because $r<d$,  $\mathsf{G}_2^{\mathrm{ad}}$ is $\mathbb{Q}$-anisotropic.

Let $\ell$ be a prime number and let $S$ be a finite set of places containing $\ell$ and the finite places $\ell'$ such that $H_v$ is not isomorphic to $\operatorname{GSp}_{4,\mathsf{F}_v}$ for some place $v$ of $\mathsf{F}$ above $\ell'$. Let $K \subset \mathsf{G}_2(\af)$ be a neat compact open subgroup that is hyperspecial away from $S$ and let $\mathbb{T}^S$ be the spherical Hecke algebra away from $S$ with coefficients in $\flbar$. Note that $\mathbb{T}^S$ acts on $H^i(\mathbf{Sh}_{K}\gxtwo_{\overline{\mathsf{E}}}, \flbar)$ and let $\mathfrak{m}$ be a maximal ideal of $\mathbb{T}^S$. For a prime $p \not \in S$ this induces a maximal ideal $\mathfrak{m}_p$ of the spherical Hecke algebra of $G$ with respect to $K_p$. Let $\mathsf{E}$ be the reflex field of $\gxtwo$ and let $\overline{\mathsf{E}}$ be an algebraic closure. The following is a corollary of Theorem \ref{Thm: TorsionVanishing}.

\begin{Cor} \label{Cor:ExampleVanishing}
If there is a prime $p \not \in S$ that is totally split in $\mathsf{F}$ such that $\ell$ is coprime to the pro-order of $K_p$ and such that the maximal ideal $\mathfrak{m}_p$ corresponds (see Section \ref{Sec:LocalizedCategories}) to a toral $L$-parameter $\phi$ that is generic, then
    \begin{align}
        H^i(\mathbf{Sh}_{K}\gxtwo_{\overline{\mathsf{E}}}, \flbar)_{\mathfrak{m}_p}=0
    \end{align}
    unless $i=3r$.
\end{Cor}
\begin{Rem}
    The assumption that $\phi$ is generic can be rephrased in terms of certain ratios of the Frobenius eigenvalues of $\phi$ not being equal to $p$ or $1$. To be precise, by our assumption on $p$ we have $G=\prod_{i=1}^d \operatorname{GSp}_4$ and if we identify $\dualgrp{G} = \operatorname{GSp}_4^d$, then our L-parameter is a product $\phi=\prod_{i=1}^d \phi_i$ of unramified representations
    \begin{align}
        \phi_i:W_{\qp} \to \operatorname{GSp}_4(\flbar).
    \end{align}
By abuse of notation, we will also write $\phi_i$ for the composition of $\phi_i$ with the inclusion $\operatorname{GSp}_4(\flbar) \to \operatorname{GL}_4(\flbar)$. The assumption that $\phi$ is generic then comes down to the following: For all $i$ the element $\phi_i(\operatorname{Frob})$
has eigenvalues $\alpha_1, \cdots, \alpha_4$ which satisfy
\begin{align}
    \alpha_j \cdot \alpha_{j'}^{-1} \not \in \{1,p\}
\end{align}
for $j \not=j'$.
\end{Rem}

\begin{proof}[Proof of Corollary \ref{Cor:ExampleVanishing}]
Let $p$ be a prime as in the statement of the corollary. We will now explicitly describe an auxiliary Shimura datum $\gx$ of Hodge type for $\gxtwo$, following \cite[Section 12 of version 1]{KretShinSymplecticArxiv} (see also \cite[Section 6]{DeligneTravaux}). Let $L$ be a quadratic CM extension of $\mathsf{F}$ and let $x \mapsto x^c$ denote the nontrivial $\mathsf{F}$-automorphism of $L$. Let $\mathsf{T}=\operatorname{Res}_{\mathsf{F}/\mathbb{Q}} \mathbb{G}_m$, let $\mathsf{T}_L=\operatorname{Res}_{L/\mathbb{Q}} \mathbb{G}_m$ and let $\mathsf{T}^1_L=\operatorname{Res}_{\mathsf{F}/\mathbb{Q}} \operatorname{Res}^1_{L/\mathsf{F}} \mathbb{G}_m$, where
\begin{align}
    \operatorname{Res}^1_{L/\mathsf{F}} \mathbb{G}_m \subset \operatorname{Res}_{L/\mathsf{F}} \mathbb{G}_m
\end{align}
is the norm one subtorus. We define
\begin{align}
    \operatorname{sim}_L: \mathsf{G}_2 \times \mathsf{T}_L \to \mathsf{T} \times \mathsf{T}^1_L \\
    (g,t) \mapsto (\operatorname{sim}(g) t t^c, t/t^c),
\end{align}
where $\operatorname{sim}:\mathsf{G}_2 \to \mathsf{T}$ is the map identifying $\mathsf{T}$ as the maximal abelian quotient of $\mathsf{G}_2$. We note that the kernel of $\operatorname{sim}_L$ is given by the image of $j:\mathsf{T} \to \mathsf{G}_2 \times \mathsf{T}_L$ where $j$ is the product of the natural embedding of $\mathsf{T}$ as the center of $\mathsf{G}_2$ and the \emph{inverse} of the natural embedding $\mathsf{T} \to \mathsf{T}_L$. We define
\begin{align}
    \mathsf{G}_1=j(\mathsf{T}) \backslash \left(\mathsf{G}_2 \times \mathsf{T}_L \right)
\end{align}
and note that it comes with a natural surjection $\operatorname{sim}_L: \mathsf{G}_1 \to \mathsf{T} \times \mathsf{T}^1_L$, which identifies the target with the maximal abelian quotient of $\mathsf{G}_1$. We then define $\mathsf{G}$ to be the inverse image in $\mathsf{G}_1$ of $\mathbb{G}_m \times \mathsf{T}^1_L \subset \mathsf{T} \times \mathsf{T}^1_L$. We identify the base change of $\mathsf{G} \times \mathsf{T}_L$ to $\mathbb{R}$ with
\begin{align}
    \prod_{i=1}^d  H \otimes_{\mathsf{F},\tau_i} \mathbb{R} \times \prod_{i=1}^d \mathbb{S},
\end{align}
where $\mathbb{S}=\operatorname{Res}_{\mathbb{C}/\mathbb{R}} \mathbb{G}_m$ is the Deligne torus. This allows us to define a Shimura datum for this group, which induces a Shimura datum $\mathsf{X}$ for $\mathsf{G}$; see \cite[Section 12 of version 1]{KretShinSymplecticArxiv}, where it is also shown that $\gx$ is of Hodge type. \smallskip

It is explained in \cite[Section~5.2, Proposition 5.4]{Hamann-Lee} that torsion vanishing results for the Shimura varieties for $\gx$ and for generic parameters imply torsion vanishing results for the Shimura varieties for $\gxtwo$ for generic parameters.

Recall that we have assumed that $p$ is totally split in $\mathsf{F}$ and that $p \not \in S$. If we choose $L$ such that every prime of $\mathsf{F}$ above $p$ is (totally) split in $L$, then the group $G$ is a split reductive group. Note that this implies that $\bgmu_{\mathrm{un}}$ consists of a single element and so the desired torsion vanishing results for $\gx$ follow from Theorem \ref{Thm:IntroTorsionVanishing} and Proposition \ref{Prop:tExact} if we can show that \cite[Assumption 6.5]{Hamann22} holds for $G$.

By \cite[Theorem 4.13]{Hamann-Lee} we know that \cite[Assumption 6.5]{Hamann22} holds for $\operatorname{GSp}_{4}$, and it is straightforward to see that this implies that the assumption holds for $G_2 = \operatorname{GSp}_4^d$ and for
\begin{align}
    G_2 \times T_L = \operatorname{GSp}_4^d \times \mathbb{G}_m^{2d},
\end{align}
using the compatibility of the Fargues--Scholze local Langlands correspondence with products, see \cite[Proposition IX.6.2]{FarguesScholze}. To deduce \cite[Assumption 6.5]{Hamann22} for $G_1$, we argue as in the proof of \cite[Theorem 4.13.(3)]{Hamann-Lee}, using a variant of \cite[Lemma 4.10]{Hamann-Lee} for $G_1$. Finally, \cite[Assumption 6.5]{Hamann22} for $G$ is now a direct consequence of \cite[Proposition 4.9]{Hamann-Lee} because $G \subset G_1$ has the same derived and adjoint group as $G_1$.
\end{proof}